\documentclass[10pt, reqno]{amsart}

\usepackage{float, comment}
\usepackage{mathtools}
\usepackage{amsmath}
\usepackage{amsthm}
\usepackage{amssymb}
\usepackage{amsbsy}
\usepackage{amstext}
\usepackage{amsopn}
\usepackage{mathrsfs} 
\usepackage{enumerate}
\usepackage{scalerel}
\usepackage{microtype} 
\usepackage[bookmarks=true, bookmarksopen=true, bookmarksdepth=3,bookmarksopenlevel=2, colorlinks=true, linkcolor=blue, citecolor=blue, filecolor=blue, menucolor=blue, urlcolor=blue]{hyperref}
\usepackage{newtxtext} 
\usepackage{mathrsfs}
\usepackage{bbm}
\usepackage{tikz}
\usepackage{tikz-cd}
\usetikzlibrary{matrix,arrows.meta}
\newcommand\puteqnum{
  \refstepcounter{equation}\hfill\textup{(\theequation)}}

\usepackage{amssymb,latexsym,amsfonts}
\usepackage{amscd,mathabx}
\usepackage{calligra,mathrsfs}
\usepackage{amsopn}
\usepackage{scalerel}
\usepackage{color}
\usepackage[OT2, T1]{fontenc}
\usepackage{url}
\usepackage{array}
\usepackage{paralist}
\usepackage{mathabx}			
\usepackage{amscd}
\usepackage[all,cmtip]{xy}			
\usepackage{sseq}				
\usepackage{verbatim}
\usepackage{parskip}
\usepackage{microtype}
\usepackage[in]{fullpage}
\usepackage{mathrsfs} 			
\usepackage{cleveref}
\usepackage{enumitem}
\usepackage{moreenum}			
\usepackage[alphabetic,lite]{amsrefs} 	
\usepackage{colonequals}	
\usepackage{fontenc}
\usepackage{amsmath}
\usepackage{geometry}
\usepackage{graphicx}
\usepackage{stmaryrd}
\usepackage[all]{xy}
\usepackage[english]{babel}
\usepackage[latin1]{inputenc}
\usepackage[T1]{fontenc}
\usepackage[foot]{amsaddr}
\usepackage{subfiles}
\usepackage{ragged2e}
\usepackage[bbgreekl]{mathbbol}     
\DeclareSymbolFontAlphabet{\mathbb}{AMSb}   

\DeclareMathOperator{\Stab}{Stab}
\DeclareMathOperator{\Mod}{Mod}

\DeclareMathOperator{\ord}{ord}
\DeclareMathOperator{\TProj}{TProj}
\DeclareMathOperator{\Sym}{Sym}

\DeclareMathOperator{\Div}{Div}

\DeclareMathOperator{\IC}{IC}

\DeclareMathOperator{\bD}{\mathbb{D}}

\DeclareMathOperator{\Der}{Der}
\newcommand{\wbk}{\widetilde{\mathscr{B}}}
\DeclareMathOperator{\Spec}{Spec}

\DeclareMathOperator{\inv}{inv}

\DeclareMathOperator{\CT}{CT}
\DeclareMathOperator{\TopMod}{TopMod}
\DeclareMathOperator{\ISpec}{ISpec}
\DeclareMathOperator{\gr}{gr}

\DeclareMathOperator{\id}{id}
\DeclareMathOperator{\Hom}{Hom}
\DeclareMathOperator{\Bk}{{\mathscr{B}}}

\DeclareMathOperator{\Id}{Id}
\DeclareMathOperator{\End}{End}

\DeclareMathOperator{\Ima}{Im}

\DeclareMathOperator{\Gr}{Gr}

\DeclareMathOperator{\Aff}{Aff}
\DeclareMathOperator{\Proj}{Proj}
\DeclareMathOperator{\Sh}{Shv}

\DeclareMathOperator{\Rhom}{R\mathscr{H}\text{\kern -3pt {\calligra\large om}}\,}
\DeclareMathOperator{\EExt}{\mathscr{E}\text{\kern -3pt {\calligra\large xt}}\,}
\DeclareMathOperator{\Ind}{Ind}
\DeclareMathOperator{\Aut}{Aut}
\DeclareMathOperator{\Ad}{Ad}

\newcommand{\uug}{\underline{g}}
\DeclareMathOperator{\uucX}{\underline\cX}
\DeclareMathOperator{\uucY}{\underline\cY}
\DeclareMathOperator{\uucZ}{\underline\cZ}
\DeclareMathOperator{\uuphi}{\underline\phi}
\DeclareMathOperator{\uupsi}{\underline\psi}

\DeclareMathOperator{\Sat}{Sat}

\DeclareMathOperator{\thick}{thick}
\DeclareMathOperator{\Perv}{Perv}

\DeclareMathOperator{\Ker}{Ker}
\DeclareMathOperator{\Irr}{Irr}
\DeclareMathOperator{\Rep}{Rep}

\DeclareMathOperator{\car}{car}

\DeclareMathOperator{\Ext}{Ext}

\DeclareMathOperator{\Stk}{St}
\DeclareMathOperator{\PrStk}{PrSt}

\DeclareMathOperator{\Pro}{Pro}

\DeclareMathOperator{\AlgSp}{AlgSp}
\DeclareMathOperator{\PreSh}{PreSh}
\DeclareMathOperator{\StCat}{StCat}
\DeclareMathOperator{\QCoh}{QCoh}

\DeclareMathOperator{\PrCat}{PrCat}

\DeclareMathOperator{\Map}{Map}
\DeclareMathOperator{\dAff}{dAff}
\DeclareMathOperator{\dPrSt}{dPrSt}
\DeclareMathOperator{\PrSh}{PrShv}

\DeclareMathOperator{\Fact}{Fact}

\DeclareMathOperator{\cT}{\mathcal{T}}

\DeclareMathOperator{\act}{act}
\DeclareMathOperator{\hgt}{ht}
\DeclareMathOperator{\SL}{SL}
\DeclareMathOperator{\rk}{rk}
\DeclareMathOperator{\Lie}{Lie}

\DeclareMathOperator{\Sets}{Sets}

\DeclareMathOperator{\Fl}{Fl}
\DeclareMathOperator{\Cat}{Cat}

\DeclareMathOperator{\free}{free}
\DeclareMathOperator{\aff}{{{aff}}}
\DeclareMathOperator{\fin}{{{fin}}}
\DeclareMathOperator{\tor}{tor}

\DeclareMathOperator{\Pic}{Pic}

\DeclareMathOperator{\Vin}{Vin}
\DeclareMathOperator{\RGm}{R\Gamma}
\DeclareMathOperator{\ev}{{{ev}}}
\DeclareMathOperator{\ins}{ins}

\DeclareMathOperator{\quotlim}{``lim''}

\usepackage{amsthm}
\usepackage{hyperref}
\usepackage{cleveref}

\newcounter{stepcounter}[subsection]

\let\oldsubsubsection\subsubsection
\renewcommand{\subsubsection}[1]{%
  \refstepcounter{stepcounter}%
  \oldsubsubsection{#1}%
}


\newtheoremstyle{thms}%
 {6pt}{0pt}{\itshape}{}{\bfseries}{.}{ }{}

\newtheoremstyle{defs}%
  {6pt}{0pt}{}{}{\bfseries}{.}{ }{}

\theoremstyle{thms}
\newtheorem{thm}[stepcounter]{Theorem}
\Crefname{thm}{theorem}{theorems}
\newtheorem{lemma}[stepcounter]{Lemma}
\Crefname{lemma}{lemma}{lemmas}
\newtheorem{prop}[stepcounter]{Proposition}
\Crefname{prop}{proposition}{propositions}
\newtheorem{cor}[stepcounter]{Corollary}
\Crefname{cor}{corollary}{corollaries}
\newtheorem{conj}[stepcounter]{Conjecture}
\Crefname{conj}{corollary}{conjectures}

\theoremstyle{defs}
\newtheorem{defi}[stepcounter]{Definition}
\Crefname{defi}{definition}{definitions}
\newtheorem{exa}[stepcounter]{Example}
\Crefname{exa}{example}{examples}
\newtheorem{rmq}[stepcounter]{Remark}
\Crefname{rmq}{remark}{remarks}
\newtheorem*{rmqs-tweak}{Remarks}

\renewcommand{\b}{\textbf}

\newcommand{\brems}{\begin{rmqs} \hfill \begin{enumerate}[label=\b{\thenumberingbase.},ref=\thenumberingbase]}
\newcommand{\remi}{\item}
\newcommand{\erems}{\end{enumerate} \end{rmqs}}
\newcommand{\bexs}{\begin{exas} \hfill \begin{enumerate}[label=\b{\thenumberingbase.},ref=\thenumberingbase]}

\newcommand{\eexs}{\end{enumerate} \end{exas}}

\newcommand{\bsm}{\begin{smallmatrix}}
\newcommand{\esm}{\end{smallmatrix}}
\newcommand{\blem}{\begin{lemma}}
\newcommand{\elem}{\end{lemma}}

\newcommand{\bconj}{\begin{conj}}
\newcommand{\econj}{\end{conj}}
\newcommand{\bdefi}{\begin{defi}}
\newcommand{\edefi}{\end{defi}}
\newcommand{\bprob}{\begin{Problem}}
\newcommand{\eprob}{\end{Problem}}

\newcommand{\bq}{\begin{Q}}
\newcommand{\eq}{\end{Q}}
\newcommand{\wti}{\widetilde}
\newcommand{\wh}{\widehat}
\newcommand{\benum}{\begin{enumerate}[label={{\upshape(\alph*)}}]}
\newcommand{\benuma}{\begin{enumerate}[label={{\upshape(\arabic*)}}]}
\newcommand{\benumr}{\begin{enumerate}[label={{\upshape(\roman*)}}]}
\newcommand{\eenum}{\end{enumerate}}
\newcommand{\bitem}{\begin{itemize}}
\newcommand{\eitem}{\end{itemize}}
\newcommand{\bc}{}

\newcommand{\bexa}{\begin{exa}}
\newcommand{\eexa}{\end{exa}}

\newcommand{\bcl}{\begin{claim}}
\newcommand{\ecl}{\end{claim}}

\newcommand{\ft}{{\mathrm{ft}}}

\newcommand{\ad}{{\mathrm{ad}}}
\newcommand{\qcqs}{{\mathrm{qcqs}}}

\newcommand{\ba}{\begin{aligned}}
\newcommand{\ea}{\end{aligned}}
\newcommand{\hV}{\hat{V}}
\newcommand{\be}{\begin{align}}
\newcommand{\ee}{\end{align}}
\newcommand{\bpf}{\begin{proof}}
\newcommand{\epf}{\end{proof}}
\newcommand{\bthm}{\begin{thm}}
\newcommand{\ethm}{\end{thm}}
\newcommand{\bprop}{\begin{prop}}
\newcommand{\eprop}{\end{prop}}
\newcommand{\bcor}{\begin{cor}}
\newcommand{\ecor}{\end{cor}}
\newcommand{\brem}{\begin{rmq}}
\newcommand{\llb}{\llbracket}
\newcommand{\rrb}{\rrbracket}
\newcommand{\llp}{(\!(}	
\newcommand{\rrp}{)\!)}
\newcommand{\erem}{\end{rmq}}

\newcommand{\Hyp}{{\mathrm{Hyp}}}
\newcommand{\form}{{\mathrm{f}}}
\newcommand{\mono}{{\mathrm{mon}}}
\newcommand{\ind}{{\mathrm{ind}}}
\newcommand{\BM}{{\mathrm{BM}}}
\newcommand{\Top}{{\mathrm{top}}}

\providecommand{\qxq}[1]{\quad\text{#1}\quad}

\newcommand{\ov}{\overline}
\newcommand{\cH}{\mathcal{H}}
\newcommand{\uula}{\underline{\lambda}}
\newcommand{\ct}{\mathcal{T}}
\newcommand{\cE}{\mathcal{E}}
\newcommand{\cQ}{\mathcal{Q}}
\newcommand{\cW}{\mathcal{W}}

\newcommand{\wco}{\hat{\mathcal{O}}}

\newcommand{\cU}{\mathcal{U}}

\newcommand{\ti}{\tilde}
\newcommand{\cM}{\mathcal{M}}

\newcommand{\bP}{\mathbb{P}}

\newcommand{\thra}{\twoheadrightarrow}
\newcommand{\co}{\mathcal{O}}
\newcommand{\cO}{\mathcal{O}}
\newcommand{\cS}{\mathcal{S}}
\newcommand{\kq}{\mathfrak{q}}

\newcommand{\hkg}{\hat{\mathfrak{g}}}
\newcommand{\hkb}{\hat{\mathfrak{b}}}
\newcommand{\cB}{\mathcal{B}}

\newcommand{\f}{\phi}
\newcommand{\cD}{\mathcal{D}}
\newcommand{\kp}{\mathfrak{p}}

\newcommand{\la}{\lambda}
\newcommand{\km}{\mathfrak{m}}

\newcommand{\bO}{\mathbb{O}}

\newcommand{\eps}{\epsilon}

\newcommand{\cI}{\mathcal{I}}
\newcommand{\hkn}{\hat{\mathfrak{n}}}
\newcommand{\cZ}{\mathcal{Z}}

\newcommand{\hS}{\hat{S}}
\newcommand{\hOm}{\hat{\Omega}}
\newcommand{\bZ}{\mathbb{Z}}

\newcommand{\ra}{\rightarrow}
\newcommand{\hra}{\hookrightarrow}

\newcommand{\bql}{\overline{\mathbb{Q}}_{\ell}}

\newcommand{\ab}{\mathbb{A}}
\newcommand{\bF}{\mathbb{F}}

\newcommand{\NN}{\mathbb{N}}
\newcommand{\al}{\alpha}
\newcommand{\hP}{\hat{P}}
\newcommand{\hG}{\hat{G}}

\newcommand{\hI}{\hat{I}}
\newcommand{\hU}{\hat{U}}
\newcommand{\hX}{\hat{X}}

\newcommand{\hu}{\hat{u}}
\newcommand{\hT}{\hat{T}}
\newcommand{\hB}{\hat{B}}
\newcommand{\bA}{\mathbb{A}}
\newcommand{\bG}{\mathbb{G}}

\newcommand{\red}{{\mathrm{red}}}
\newcommand{\fppf}{{\mathrm{fppf}}}
\newcommand{\re}{{\mathrm{re}}}
\newcommand{\op}{{\mathrm{op}}}

\newcommand{\bC}{\mathbb{C}}
\newcommand{\bN}{\mathbb{N}}
\newcommand{\bB}{\mathbb{B}}
\newcommand{\bQ}{\mathbb{Q}}
\newcommand{\cY}{\mathcal{Y}}
\newcommand{\cC}{\mathcal{C}}

\newcommand{\cJ}{\mathcal{J}}
\newcommand{\cA}{\mathcal{A}}

\newcommand{\cF}{\mathcal{F}}
\newcommand{\cX}{\mathcal{X}}
\newcommand{\Spc}{{\mathrm{Sp}}}
\newcommand{\La}{\Lambda}

\newcommand{\kj}{\mathfrak{j}}
\newcommand{\kl}{\mathfrak{l}}
\newcommand{\kg}{\mathfrak{g}}
\newcommand{\kn}{\mathfrak{n}}
\newcommand{\kb}{\mathfrak{b}}
\newcommand{\kt}{\mathfrak{t}}
\newcommand{\Gm}{\Gamma}
\newcommand{\wK}{\hat{K}}
\newcommand{\bX}{\mathbb{X}}
\newcommand{\Cofib}{Cofib}
\newcommand{\uu}{\underline{u}}
\newcommand{\uw}{\underline{w}}
\newcommand{\khn}{\hat{\mathfrak{n}}}
\newcommand{\Bun}{{Bun}}
\newcommand{\str}{str}
\newcommand{\Om}{\Omega}

\newcommand{\colim}{\operatornamewithlimits{colim}} 
\newcommand{\Sp}{{\mathrm{Sp}}}
\newcommand{\LLa}{{\check\Lambda}}

\newcommand{\Lal}{{\check\al}}
\newcommand{\Lbe}{{\check\beta}}
\newcommand{\Lrho}{{\check\rho}}
\newcommand{\LDelta}{{\check\Delta}}
\usepackage[title]{appendix}

\def\overbigdot#1{\overset{\hbox{\tiny$\circ$}}{#1}}
\numberwithin{equation}{subsection}

\begin{document}

\title{On the geometric Satake equivalence for Kac-Moody groups}
\author{Alexis Bouthier and Eric Vasserot}
\address{Sorbonne Université, Institut de Math\'{e}matiques de Jussieu-Paris Rive Gauche,, 4 place Jussieu, 75005 Paris, France.}
\address{Université Paris Cité, Institut de Math\'{e}matiques de Jussieu-Paris Rive Gauche, Bâtiment Sophie Germain, 8 place Aurélie Nemours, F-75013 Paris, France.}
\maketitle
\begin{center}
\textit{In memoriam Gérard Laumon (1952-2025)}
\end{center}

\begin{abstract}
This article establishes a geometric Satake equivalence for affine Kac-Moody groups as an equivalence of abelian semisimple categories over 
algebraically closed fields. We define a well-behaved category of equivariant sheaves on the $\infty$-stack $\Gr_{G}$ that we equip with 
a $t$-structure. We obtain an Braden's hyperbolic localization theorem for such a stack and prove that the constant term functor is $t$-exact using dimension 
estimates for affine MV-cycles. We then deduce the sought-for equivalence and prove that the $\IC$-complexes match with the 
irreducible highest weight representations of the dual group $G^{\vee}$.
\end{abstract}

\tableofcontents

\bigskip

 \section{Introduction}
\subsection{Motivations and brief outline}

Let $G$ be a connected reductive group.
The affine Grassmannian $\Gr_{G}=G\llp s\rrp/G\llb s\rrb$  
has a natural stratification by $G\llb s\rrb$-orbits indexed by dominant cocharacters $\la$.
 The geometric Satake equivalence for $G$ relates the $\IC$-complex of the closure 
$\overline{\Gr}_{\la}$ of  $\Gr_\lambda=G\llb s\rrb\cdot s^{\la}$ in $\Gr_{G}$ 
with the irreducible representation  $L(\la)$ of highest weight $\la$ of the Langlands 
dual group $G^{\vee}$. It has been initiated by Lusztig in \cite{Lusz} and culminates with the work of Mirkovic-Vilonen \cite{MV} with 
contributions of Ginzburg \cite{Gi} and Beilinson-Drinfeld \cite{BD}.

The search for an analog for Kac-Moody groups  was initiated by I. Frenkel and I. Grojnowski 20 years ago
and has since then been an active open problem. Because of its highly infinite nature and the lack of technology to 
treat such spaces, the current approach, initiated by Braverman-Finkelberg in \cite{BFI}, \cite{BFII} and \cite{BFIII}, 
is to formulate a Satake equivalence in terms on some analog of finite dimensional transversal slices to orbits in  the affine Grassmannian
that are supposed to encapsulate the 
behaviour of intersection cohomology complexes when restricted to smaller strata. 
This was done in affine type A, in which case the slices are related to moduli 
spaces of instantons and Nakajima's quiver varieties, yielding the
partial results of Braverman-Finkelberg-Nakajima in \cite{BFG}, \cite{BFN1}, \cite{Nak18}. 
It also led, to a definition of these slices in terms of Coulomb branches for symmetrizable 
Kac-Moody algebras, see \cite{ICMFinkl}. Beyond these constructions, the link with the Satake equivalence is still far-reaching.

In this work, we bypass completely this approach and  deal directly with the affine grassmannian of a Kac-Moody group $G_{\aff}$, as 
infinite as it might be. To approach such a space, the  techniques developed in \cite{BKV} are extremely useful. Instead of trying to put 
some geometric structures on a space like $\Gr_{G_{\aff}}$ that do not exist, the key idea is to treat it just as a prestack, focus rather on 
the relations between the various functors involved and prove  finiteness and representability results on smaller pieces that we can glue 
after  by general formalism.

This way, we obtain an equivalence of semisimple abelian categories between a category of equivariant perverse sheaves 
$\Perv_{G_{\aff}\llb s\rrb}(\Gr_{G_{\aff}})$ on the affine grassmannian $\Gr_{G_{\aff}}$ and the category $\Ind(\Rep(G_{\aff}^{\vee}))$ of 
representations of the Langlands dual $G_{\aff}^{\vee}$ that we detail below. Note that in this work, because of the amount of technical 
difficulties that have to be overcome, we do not investigate the monoïdal structure that is intended to be studied in a future work.
Let us explain our approach in details. 

\subsection{Kac-Moody groups}

Before diving in the affine Grassmannian of the Kac-Moody (KM) group, several foundational results on KM groups must be settled. 
We do not restrict to the untwisted affine case, because the Langlands duality exchanges twisted and untwisted types. 
For that purpose, we deal with the general KM case.
In addition we work either over an arbitrary algebraically closed field or over $\bZ$, 
to pass from finite field to complex numbers.
In this setting few things are known on KM groups.

Associated to a KM root datum $\cD$, there is a well-defined $\bZ$-form $\kg_{\bZ}$ of the KM $\bC$-algebra 
$\kg$ as well as a completed version $\hat\kg_{\bZ}$. The situation for groups is more delicate.
Over $\bC$, there are two group ind-schemes: 
the minimal one $G_{\bC}$, defined by Kumar in \cite{Kum}, which is of ind-finite type, and 
the formal one $\hat{G}_{\bC}$, defined by Mathieu \cite{Mat}, which is an ind-pro-scheme.
The formal KM group is defined over an arbitrary ring.
We need the minimal KM group, for which no construction is known over integers.
Over $\bZ$, the only available object is an abstract group functor introduced by Tits \cite{Tits}, which carries no geometric structure.
This group functor is well behaved on fields or on Euclidean rings only. In the reductive case it does not recover the usual 
Chevalley group schemes.
Elementary questions, such as the  formal smoothness of $\hG$, the computation of its Lie algebra over $\bZ$ 
or its behaviour with respect to base change were not considered.
We summarize  Propositions  \ref{Bcmut}, \ref{minZ}, \ref{g-ind} and \ref{U-rep} in the following theorem.

\bthm\label{fond1}
Let $\cD$ be a simply connected KM root datum \eqref{r-dim} and $R$ be a ring.
\hfill
\begin{enumerate}[label=$\mathrm{(\alph*)}$,leftmargin=8mm]
\remi
We have
$\hG_{\bZ}\times_{\Spec(\bZ)}\Spec(R)\cong \hG_{R}$
and $\Lie(\hG_{R})=\hkg_{\bZ}\widehat{\otimes}_{\bZ}R$. 
\item
$\hG_{\bQ}$ is formally smooth. 
\remi
There is a group ind-scheme $G_{R}$ ind-(affine finitely presented) and ind-normal over $\Spec(R)$ such that 
$\Lie(G_{R})\cong\kg_{\bZ}\otimes_{\bZ}R$
and an ind-closed embedding $G_{R}\hra \hG_{R}$. 
\remi
The minimal group $G$ has a closed Borel subgroup $B$ that splits as $B=U\rtimes T$, for $T$ a maximal torus. 
We have $G/B\cong\hG/\hB$, the latter being  ind-projective and the morphism $G\ra G/B$ is Zariski locally trivial.
\eenum
\ethm

We abbreviate $\hG=\hG_{\bZ}$ and  $G=G_{\bZ}$.
We expect that formal smoothness should  hold over $\bZ$, but we were not able to prove it. Our argument over $\bQ$ 
relies on a generalization for group ind-schemes of Cartier's Theorem \ref{cart}.
In addition to the construction of the minimal group, we provide a construction of the Kashiwara flag scheme $\hG/B^{-}$ over $\bZ$, as 
a formally smooth separated scheme in Theorem \ref{fs-Kash}.

\subsection{The geometry of the affine grassmannian}
\subsubsection{Cartan and Iwasawa decomposition}

Consider now the minimal KM group ind-scheme $G$, 
associated with a simply connected KM root datum over $\bZ$. Let $T\subset G$ be a 
maximal torus, $X_*(T)$ and $X^{*}(T)$ the lattices of cocharacters and characters. 
Let $\{\al_{i}\}_{i\in I}$ be the set of simple roots.
Consider an element $\rho\in X^*(T)$ such that  $\left\langle \rho,\alpha^{\vee}_i\right\rangle=1$ for all $i\in I$ with 
$\al_i^{\vee}$ the corresponding coroot, see \S\ref{r-KM}.
We form the quotient stack of the polynomial loops:
\[\Gr_G=G[s^{\pm1}]/G[s],\]
 where we sheafify for the étale topology. 
 We switch to Laurent polynomials instead of Laurent series, because 
 $G[s^{\pm1}]$ and $G[s]$ are ind-finite type schemes over $\bZ$.
On the contrary, the group $G\llb s\rrb$ is a pro-ind-object and $G\llp s\rrp$ is even worse 
 and we cannot apply the techniques of \cite{BKV} to them.
In the reductive case, the Beauville-Laszlo theorem yields an isomorphism of ind-schemes
\[G[s^{\pm1}]/G[s]\cong G\llp s\rrp/G\llb s\rrb,\]
and one can check that Cartan and Iwasawa decompositions match on both sides. 
So, one can formulate a
Geometric Satake theorem
for reductive groups using Laurent polynomials. 
Before going to sheaf theory, we need basic geometric properties, that is to say, 
the Cartan and Iwasawa decompositions and the description of the closure of Schubert cells and semi-infinite orbits. 
For any field $k$, we have an Iwasawa decomposition, see Proposition \ref{iwa},
\[G(k[s^{\pm1}])=B(k[s^{\pm1}]) G(k[s])=B^{-}(k[s^{\pm1}]) G(k[s]).\]
This decomposition is not new since it already appears in \cite{BFK} or \cite{GR08}. 
However  \cite{BFK}, although geometric, holds only in the untwisted affine case only, whereas \cite{GR08} is only group theoretical
but holds for general types. Here we treat uniformly general types in a geometric way.

In particular, we can define for $\nu\in X_{*}(T)$ the quotients
\[S_{\nu}=s^{\nu}\cdot\Gr_{U},\quad T_{\nu}=s^{\nu}\cdot\Gr_{U^{-}}\subset\Gr_{G}.\]
 Although both sides are not representable, the inclusion $S_{\nu},T_{\nu}\ra\Gr_{G}$ 
 are both finitely presented (fp) locally closed immersions.
For $G$ symmetrizable,  Proposition \ref{ST-comp1} yields
\begin{equation}
\ov{S_{\nu}}=\bigsqcup_{\mu\leq\nu} S_{\mu},\quad 
\ov{T_{\nu}}= \bigsqcup_{\mu\geq\nu} T_{\mu}.
\label{s-clo}
\end{equation}
 Here the closures have to be understood as the closures of their inverse images in $G[s^{\pm1}]$, which is an ind-scheme, 
 quotiented by $G[s]$. The symmetrizability assumption appears because we need some fact about coweights \eqref{roots}
  to reduce to the $SL_2$ 
 case.

For the Cartan decomposition, we work over a base field $k$.
It is well-known \cite[Appendix A]{BKP} that the Cartan decomposition does not hold for the full $G[s^{\pm1}]$ 
but rather for a smaller closed subsemigroup, see Proposition \ref{Cart},
\begin{equation}
G_{c}=\bigsqcup\limits_{\la\in X_*(T)^{+}}G[s]\cdot s^{\la}\cdot G[s]\subset G([s^{\pm 1}]).
\label{gc}
\end{equation}
For dominant cocharacters $\la\in X_{*}(T)^{+}$, we define in Lemma  \ref{Schub1} a fp closed substack
$\ov{\Gr}_{\la}\subset\Gr_{G}$
and a quasi-compact dense open $\Gr_{\la}\subset\ov{\Gr}_{\la}$.
If $G$ is symmetrizable, Proposition \ref{Cartan}
yields a decomposition  
\begin{equation}
\ov{\Gr}_{\la}=\bigsqcup\limits_{\mu\leq\la}\Gr_{\mu}.
\label{c-clo}
\end{equation}
We do not have Chevalley's theorem to prove that orbits are locally closed.
We must prove by hand, in Proposition \ref{unif}, that $\Gr_{\la}$ is a $G[t]$-orbit in a schematic way and not only on $K$-points, for 
$K/k$ a field extension. In order to do that, we reduce to the formal case, i.e., replacing $G$ by $\hG$.
The corresponding 
Schubert cell $\widehat{\Gr}_{\la}$ has better representability properties, although it is by a nasty ind-scheme,
see Proposition \ref{sch-form}(d). On the 
contrary, the functor $\Gr_{\la}$ is not representable but has a better cohomological behaviour.

 \subsection{Affine MV-cycles and the classical Satake}

The function counterpart of the geometric Satake for KM groups has already been established, first in the untwisted affine case by 
 Braverman-Kazhdan in \cite{BK} and then in the general case by 
 Gaussent-Rousseau in \cite{GR2}. The affine 
 case enables to use geometry, based on the link with moduli spaces of $G_{\fin}$-torsors on surfaces with $G_{\fin}$ a
 reductive group. 
 The general case is more combinatorial, using masures, that are partial analogs of Bruhat-Tits buildings for KM groups. 
Nevertheless, even when geometry is used, it is only at the level of $k$-points of some moduli spaces, where $k$ is a field.
Every 
statement there, involves sets, which is why classical Satake is more amenable than the geometric one.

Once the stratifications are introduced, the first test for the geometric Satake equivalence is to study the interaction between 
Schubert cells and semi-infinite orbits that deserve to be called affine Mirkovic-Vilonen (MV)-Cycles, or open affine MV cycles. 
 Here a new feature appears, opposed to the reductive case. In the reductive case, the intersections $\Gr_{\la}\cap S_{\nu}$ and 
 $\Gr_{\la}\cap T_{\nu}$ play a symmetric role and have the same dimension.
In the Kac-Moody case, the intersection $\Gr_{\la}\cap T_{\nu}$ is finite dimensional whereas  $\Gr_{\la}\cap S_{\nu}$ is finite 
codimensional.
It is important there to take $\Gr_{\la}$ instead of $\ov{\Gr}_{\la}$, where we loose representability.
For that reason, the first intersection was much more studied.
Over a finite field, it  was proved  that the set $(\Gr_{\la}\cap T_{\nu})(\bF_q)$ is finite 
in \cite[Thm.~1.9]{BKGP} in the untwisted affine case and in \cite[Thm.~5.6]{Heb2} in the general case, 
which is a weak shadow of the finite dimensionality, known as Gindinkin-Karpelevich finiteness.
The other intersection was never considered.

In both works assertions, the finiteness follows from the finiteness of central fibers of the affine Zastavas, that are intersections $S_{\la}\cap T_{\nu}$. 
We follow a similar strategy but we work schematically.
The following key theorem is proved in  Proposition \ref{Z-fin},  Theorem \ref{thm:GT} and Lemma \ref{GrS3}.

\bthm\label{MV-cyc}
Let $G$ be a minimal simply connected KM group over $\bZ$. Let $\la,\mu,\nu\in X_*(T)$ be cocharacters.
\hfill
\begin{enumerate}[label=$\mathrm{(\alph*)}$,leftmargin=8mm]
\remi
The intersection $(S_{\mu}\cap T_{\nu})_\red$ is a finite type scheme over $\bZ$.
\remi
Assume that $G$ is symmetrizable. The following holds:
\benumr
\item 
the intersection $(\Gr_{\la}\cap T_{\nu})_\red$ is a finite type scheme over $\bZ$ with relative dimension $\left\langle \rho,\la-\nu\right\rangle$,
\item
the number of top dimension irreducible components of  $\Gr_{\la}\cap T_{\nu}$ is  $\dim L(\la)_{\nu}$, where
$L(\la)_{\nu}$ is the weight $\nu$-multiplicity space in the irreducible representation of highest weight $\la$ of $G^{\vee}$,
\item
the intersection $(\Gr_{\la}\cap S_{\mu})_\red$ is equidimensional of codimension
$\left\langle \rho,\la-\mu\right\rangle$.
\eenum
\eenum
\ethm

The proof of (a) relies on an other interpretation of these intersections as some central fibers of  mapping spaces  
$\Hom(\bP^{1},[U^{-}\backslash G/B])$.
We  prove a general representability result for these spaces in Proposition \ref{Dr-rep},
 and obtain that the connected components of such are smooth in Proposition \ref{Zfond2}, inspired by \cite[Prop.~2.25]{BFG}.
After that, we need to prove that both definitions of the Zastavas match, which is delicate, because we lack of results such as Beauville-Laszlo 
for group ind-schemes or the fact that $(G/U)$ is a presheaf quotient, i.e for any ring $R$
\[(G/U)(R)\cong G(R)/U(R),\]
that holds in the untwisted affine case using results of \v{C}esnavi\v{c}ius (see \ref{U-rep}). To circumvent these difficulties, we use rather a
Beauville-Lazslo gluing assertion for $\hG/U^{-}$, which is a separated scheme (even though not noetherian), to pass from $\bP^1$ to a formal 
disc $\Spec(\bZ\llb s\rrb)$ for the formal group $\hG$. After that, we need to say that the Zastavas obtained over with Laurent series and the 
group $\hG$ agree with the ones obtained with Laurent polynomials and the group $G$.

The next two assertions use representability of Zastavas as well as a counting argument. The symmetrizability appears because we need to 
use the affine Mac-Donald formula (\cite[§7.8]{BKP} for the untwisted affine case and \cite[Thm. 7.3]{BGR} for the symmetrizable) to relate the 
Satake transform of $1_{\Gr_{\la}}$ with the intersections $(\Gr_{\la}\cap T_{\nu})(\bF_q)$ and deduce the dimension and number of top 
irreducible components by working asymptotically when $q$ goes to $\infty$. A similar trick was used by Zhu in \cite[Cor. 2.8]{ZhuII}, for the 
mixed characteristic version of Geometric Satake.
Note that in order to compare our intersections with the ones of Gaussent-Rousseau, 
we must go from $\bF_q[s^{\pm1}]$ to $\bF_q\llp s\rrp$ (\ref{comp-Gr}), identify our maximal compact with theirs (\ref{int-comp}) and convert 
their formula which involves Hecke paths to the affine MV cycles \eqref{sat0}.

The last assertion is the most difficult, since we cannot rely on previous results, and even the notion of finite codimensionality is already 
hard to make precise. When we pass to the formal group, we have a reasonable notion of 
equidimensionality of codimension $\left\langle \rho,\la-\nu\right\rangle$ but not quite for the minimal group. 
We only obtain a weaker 
form, which essentially says that after pulling back by some ind-group and divide by an other ind-group, we end up with 
something of the right codimension \eqref{diag2}. 
Also, note that in the formal version, our proof uses the standard argument of cutting by successive effective Cartier divisors.
As opposed to the reductive case where there is a finite number of $\mu's$ smaller to $\la$, here it is no longer the case. 
However, as suggested to us by T. Van den Hove, we can reduce to weights $\mu\leq\la$ such that 
$w\la\leq\mu\leq\la$, for some $w\in W$. 
In particular, we essentially sandwich $\Gr_{\la}\cap S_{\mu}$ between $\Gr_{\la}$ and $\Gr_{\la}\cap S_{w\la}$ 
and for the latter, we compute it by hand, because it is a nice and easy intersection. 
Indeed, in the reductive case it is just an affine 
space $\ab^{\left\langle \rho,\la+w\la\right\rangle}$.
 
 \subsection{The \texorpdfstring{$t$}{t}-structure on the affine grassmannian}
We now move to the cohomological part. We fix a prime $\ell$ different for the characteristic of the base field. It is important to work with $G[t]$-equivariant sheaves on $\Gr_{G}$, because it is 
only there that the $t$-structure exists.
The techniques of \cite{BKV} allow us to obtain an $\ell$-adic $\infty$-category $\cD_{G[s]}(\Gr_{G})$ that satisfies gluing, i.e., 
we have the  four 
functors $i^{*},$ $i^{!},$ $i_{*},$ $i_{!}$ for any fp-locally closed immersion $i$. 
By the Cartan decomposition, 
we work on the smaller 
category  $\cD_{G[s]}(\Gr_{G_c})$, with $\Gr_{G_{c}}=G_{c}/G[s]$, where $G_c$ is introduced in \eqref{gc}, 
that still satisfies gluing for a `reasonable' stratification by the Cartan orbits (Lem. \ref{b-constr}).
In particular, to obtain a $t$-structure on $\cD_{G[s]}(\Gr_{G_c})$ it is enough to get one on $\cD_{G[s]}(\Gr_{\la})$. 
Note nevertheless that it will be on the big categories. Now as $\Gr_{\la}$ is a $G[s]$-orbit, we have
\[\cD_{G[s]}(\Gr_{\la})=\cD(\bB K_{\la}),\]
for $K_{\la}\subset G[s]$ the stabilizer of $\la$. This is a group ind-scheme, thus it is not clear how to define a $t$-structure on 
$\cD(\bB K_{\la})$. Nevertheless the loop action contracts $K_{\la}$ to a parabolic $P_{\la}$ of $G$ and, if $G$ is affine 
(twisted or untwisted), the parabolic further contracts to its Levi factor $L_{\la}$, which is reductive. 
In particular, in Proposition \ref{t-strat} we prove  that
\[\cD(\bB K_{\la})\cong\cD(\bB L_{\la}),\]
and the latter has a well-defined $t$-structure. So far, it is the only place where we need our KM group $G$ to be affine, 
because otherwise it can happen that the Levi $L_{\la}$ is still an ind-group and thus $\bB L_{\la}$ does not a priori have a $t$-structure. Nevertheless, some ongoing work of Y. Varshavsky should allow to remove this assumption.
In particular, for any dominant cocharacter $\la$, we define an intersection complex
\[\IC_{\la}\in\cD_{G[s]}(\Gr_{G_c}).\]
The next step is to construct a fiber functor $F$ such that
\[F(\IC_{\la})=L(\la),\]
where $L(\la)$ is the irreducible representation of highest weight $\la$ of the dual group. 
In order to do that, we must establish an hyperbolic 
localization theorem.
 
 \subsection{Hyperbolic localization and Geometric Satake}
\subsubsection{Hyperbolic localization}

Let $G$ be a minimal simply connected Kac-Moody group over an algebraically closed field of characteristic zero or affine in arbitrary characteristics. 
We consider the $\bG_m$-action on $\Gr_{G}$ given by $2\rho^{\vee}$.
We prove in Proposition \ref{id-att} that there is an hyperbolic diagram
$$\xymatrix{&\Gr_{B}\ar[dr]^{p^{+}}\\\Gr_{T}\ar[ur]^{i^{+}}\ar[dr]_{i^{-}}&&\Gr_{G}\\&\Gr_{B^{-}}\ar[ur]_{p^{-}}}$$
that is to say that, $\Gr_{T}$ identifies with the functor of fixed points, $\Gr_{B}$ with the attractor and $\Gr_{B}^{-}$ the repulsor. In 
addition, the maps $i^{\pm}$ are fp closed immersions and $p^{\pm}$ are fp locally closed immersion when restricted to each 
connected component. In particular,  the functors $(p^{\pm})^{*}$, $(p^{\pm})^{!}$ and $(i^{\pm})^{*}$, $(i^{\pm})^{!}$ are well-defined.
We can thus consider by the constant term functor
\[\CT_{*}=(i^{+})^{*}(p^{+})^{!}: \cD(\Gr_{G})\ra\cD(\Gr_{T}).\]
It decomposes along the connected components as
\[\CT_{*}=\bigoplus\limits_{\nu\in X_{*}(T)}\CT_{*,\nu}.\]
 Based on the usual geometric Satake, we prove the following  in Theorems \ref{Braden} and \ref{t-exact}.
 
\bthm\label{sum-Hyp}
\hfill
\begin{enumerate}[label=$\mathrm{(\alph*)}$,leftmargin=8mm]
\remi
There is a morphism of functors \eqref{adj-map}
\begin{equation}
\CT_{*}\ra\CT_{!}^{-}=(i^{-})^{!}(p^{-})^{*}.
\label{hyp-map}
\end{equation}
It is an equivalence when restricted to the category of $\bG_m$-monodrodromic objects.
\remi
If $G$ is symmetrizable, then the normalized constant term functor
\[\CT_{*}[\deg]=\bigoplus\limits_{\nu\in X_{*}(T)}\CT_{*,\nu}[2\left\langle \rho,\nu\right\rangle],\]
restricted to $\cD_{G[t]}(\Gr_{G_c})$ is $t$-exact.
\eenum
\ethm

In order to obtain (a), one cannot rely on the existing versions of Braden's hyperbolic localization. In the literature, there are essentially 
two ways to get it: one is due to Richarz \cite{Ric}, which, by a series of devissages, reduces to the case of an affine scheme of finite 
type and the other one, due to Drinfeld-Gaitsgory \cite{DG1}, that reinterprets Braden's theorem as a statement on adjoint pairs.

In our situation, there does not seem to exist a way to reduce to known situations, so we must reprove it by hand.
In that 
perperspective, the Drinfeld-Gaitsgory's approach is much more flexible. 
First, we prove a contraction principle in Proposition \ref{adj} that allows to reformulate, as them, our problem to a question on adjoint 
pairs with the map \eqref{hyp-map} as the co-unit of the adjunction. The difficult part is then to construct the unit map. This is done 
using an  $\ab^1$-family $\widetilde{\Gr_{G}}$ that interpolates $\Gr_{G}$ and $\Gr_{B}\times_{\Gr_{T}}\Gr_{B^{-}}$, with an action of 
$\bG_m$. 
The key assertion is that there is an fp-locally immersion by Corollary \ref{cor-quot}
\begin{equation}
\la:\widetilde{\Gr_{G}}\ra\ab^{1}\times\Gr_{G}\times\Gr_{G},
\label{la-Hyp}
\end{equation}
which is $\bG_m$-equivariant for the diagonal action on the right hand side given by $2\rho^{\vee}$.
In particular, the pushforward $\la_*$ is defined and by considering the specialization map in Proposition \ref{sp} associated with
$\la_{*}\omega_{\widetilde{\Gr_{G}}}$, we obtain the desired unit map.
After that, it is essentially formal to get Braden's theorem and we follow Drinfeld-Gaitsgory's argument 
provided that we have the key geometric statement in Proposition \ref{fib-prod}. Moreover, in order to get the assertion on 
\eqref{la-Hyp}, similarly to \cite{DG2}, we relate this interpolation $\widetilde{\Gr_{G}}$ to a natural hyperbolic monoid $\Hyp_{G}$, 
obtained from the Vinberg monoid, whose construction is generalized for KM groups in \S\ref{vinberg}.

Proving the $t$-exactness of the functor  $\CT_{*}[\deg]$ amounts to compute the image by this functor of $(i_\la)_*\omega_{\Gr_{\la}}$ 
and $(i_\la)_*\omega_{\Gr_{\la}}$, for $i_{\la}:\Gr_{\la}\ra\Gr_{G}$ the obvious immersion. 
The first image is related to the Borel-Moore homology of the affine MV cycle $\Gr_{\la}\cap T_{\nu}$, 
the second one to the cohomology with compact support of $\Gr_{\la}\cap S_{\nu}$. 
We then use the estimates in Theorem \ref{MV-cyc} to obtain $t$-exactness.

\subsubsection{Affine Geometric Satake}

We now reach the final step of our work in Proposition \ref{ic-calc} and Theorem \ref{satake-fin}. 
Assume that $G$ is  affine, simply connected, over an algebraically closed field of arbitrary characteristic.
By exchanging roots and coroots in the KM root datum, we consider $G^{\vee}$ the minimal group associated to the dual root datum, 
defined over $\bql$.
We introduce $\Rep(G^{\vee})$ to be the category of integrable $\kg$-modules in the category $\cO$, see \S\ref{O-cat}.
It is abelian semisimple by \cite[\S10]{Kac}.

 \bthm
The category $\Perv_{G[t]}(\Gr_{G_{c}})$ is semisimple.
The normalized constant term functor
\[\CT_{*}[\deg]:\Perv_{G[t]}(\Gr_{G_{c}})\ra\Ind(\Rep(G^{\vee}))\]
induces an equivalence of abelian semisimple categories that sends $\IC_{\la}$ to $L(\la)$.
\ethm

To prove the semisimplicity, we consider an exact sequence in $\Perv_{G[t]}(\Gr_{G_{c}})$
\begin{equation}
\Delta_{\la}\ra\IC_{\la}\ra\nabla_{\la},
\label{ct-A}
\end{equation}
where $\Delta_{\la}$ is standard and $\nabla_{\la}$ costandard. 
We have $$\CT_{*}[\deg](\nabla_{\la})=L(\la),$$ as $X_{*}(T)$-graded vector spaces.
We prove that the functor $\CT_{*}$ is conservative.
Using the perverse inequalities of the $\IC$-complex, a spectral sequence argument yields
\[\CT_{*}[\deg](\Delta_{\la})\cong\CT_{*}[\deg](\IC_{\la})\cong L(\la).\]
Since the normalized constant term functor is conservative, all maps in \eqref{ct-A} are isomorphisms. 
In particular, when we restrict $\IC_{\la}$ to $\ov{\Gr}_{\la}\setminus\Gr_{\la}$, it gives a stronger perversity inequalities 
which allows us to prove that
\[\Hom(\IC_{\la},\IC_{\mu}[1])=0,\]
yielding the semisimplicity. 
Since the functor $\CT_{*}[\deg]$ is exact, faithful, gives a bijection on simple objects, and the categories on both sides
are semisimple and stable by direct sums, we get the desired equivalence. 
 \subsection{Remarks on the hypotheses}
Throughout this article we work in the simply connected case, essentially for a matter of reference and convenience; removing this 
assumption is not a problem, just  makes the presentation a lot heavier. As we already explained, the work of Varshavsky should 
permit to have a $t$-structure in the general case. Nevertheless, we would still need symmetrizability in order to get the closure 
relations \eqref{c-clo} and \eqref{s-clo} and because the affine Mac-Donald formula is only known in this case. Finally, the characteristic 
zero assumption essentially appears because we want to have formal smoothness of $G$ that is needed for hyperbolic localization.
Nevertheless, in many statements, at the cost of replacing isomorphisms by topological equivalences, that are enough from 
cohomological point of view, it might be possible to discard this assumption. But of course, it would be much more interesting to have 
formal smoothness in general. In conclusion, with few improvements we expect to get the equivalence in the symmetrizable case in arbitrary characteristic.
Also, we work with $\bql$-coefficients, but most of this work should also hold with $\bZ/\ell^{n}$-coefficients and should shed light on modular representations of KM group for which little is known.

\subsection{Conventions and notations}

\subsubsection{\texorpdfstring{$\infty$}{Infinity}-categories}
All categories in this work are $\infty$-categories, all functors are $\infty$-functors,
and all limits and colimits are the homotopical ones. In particular, ordinary categories are viewed as $\infty$-categories.
A morphism in an $\infty$-category is an isomorphism if it is an isomorphism in the homotopy category.

Our conventions regarding $\infty$-categories follow those of \cite[§.0.6]{DG0}.
We shall use \cite{Lu1} as a basic reference.
Let $\StCat$ be the $\infty$-category of small $\infty$-categories, stable and linear over $\bql$ for a field $L$. In the sequel, $L$ is essentially $\bql$ for a prime $\ell$ prime to the residual characteristics of the object we consider.
The morphisms are the exact functors, i.e., the functors that preserve finite limits and colimits. 
By \cite[Thm.~1.1.4.4, Prop.~1.1.4.6]{Lu2}, the category $\StCat$ has all small filtered colimits and all small limits.
The $\infty$-categories we will encounter all  belong to $\StCat$.

Most of the $\infty$-categories we will encounter are cocomplete, i.e., contain all small colimits.
A functor between cocomplete $\infty$-category is continuous if it commutes with all small colimits.
Continuous functors are exact.
An $\infty$-category is presentable if it is cocomplete and is generated under colimits by a set of compact objects, i.e., 
objects for which the corresponding corepresentable functor commutes with filtered colimits, see \cite[Def.~5.5.0.1]{Lu1}.
A left adjoint functor is always continuous.
In presentable categories, a right adjoint functor whose left adjoint preserves compact objects
is continuous, see \cite[Prop.~5.5.7.2]{Lu1}.
A corollary of the adjoint functor theorem implies that presentable $\infty$-categories also contain
all small limits, see \cite[Cor.~5.5.2.4]{Lu1}.
Let $\PrCat$ be the $\infty$-category of presentable $\infty$-categories which are stable and linear over a field $L$.
The morphisms are the continuous functors.
This category bicomplete by \cite[\S I.1, Cor.~5.3.4]{GRI}.

Given  $\cC\in\StCat$,  we form its category of ind-objects $\Ind(\cC)$, see \cite[Def.~5.3.5.1]{Lu1}.
It is stable and presentable, see \cite[Prop.~1.1.3.6]{Lu2}.
This yields a functor
$\Ind:\StCat\rightarrow\PrCat$
which commutes with small filtered colimits, see \cite[Prop.~5.3.5.10]{Lu1}, \cite[\S 1.9.2]{DG} and \cite{Roz}.

An $\infty$-category $\cC$ is accessible if it is small and idempotent complete, see \cite[Cor.~5.4.3.6]{Lu1}.
Let $\cC$ be an accessible $\infty$-category with finite limits. 
Let $\Spc$ be the $\infty$-category of spaces (=$\infty$-groupoids).
The pro-category $\Pro(\cC)$ is the opposite of the category of accessible functors  $F:\cC\ra\Spc$ 
which commute with finite limits, see \cite[Def.~A.8.1.1]{Lu3}.
By \cite[Rmk.~A.8.1.2]{Lu3}, we have $\Pro(\cC)=\Ind(\cC^{\op})^{\op}$. 
By \cite[Rmk.~A.8.1.3]{Lu3}, the Yoneda embedding yields a fully faithful functor $\cC\hra\Pro(\cC)$.
If $\cC$ is stable then $\Pro(\cC)$ is also stable by \cite[Rmk.~1.1.1.13, Prop.~1.1.3.6]{Lu2}. 
By \cite[Lem.~6.1.2.(a)]{BKV}, if $\cC$ is equipped with a $t$-structure, then $\Pro(\cC)$ is also equipped with a $t$-structure
such that $\Pro(\cC)^{\leq 0}=\Pro(\cC^{\leq 0})$ and $\Pro(\cC)^{\geq 0}=\Pro(\cC^{\geq 0})$.
By \cite[App.~A3]{DG1}, for any continuous functor $\cC\to\cD$ of cocomplete $\infty$-categories,
we can define its left adjoint as a functor $\cD\to\Pro(\cC)$. 

\subsubsection{\texorpdfstring{$\infty$}{Infinity}-stacks}
Let $k$ be a commutative ring.
Let $\Aff_k$ be the category of affine $k$-schemes, equipped with the \'etale topology.
We abbreviate qc=quasi-compact, qs=quasi-separated,
lfp=locally finitely presented,  fp=finitely presented,  lft=locally finite type, and ft=finite type.
Let $\AlgSp_k^\ft$ be the category of $k$-algebraic spaces of ft, and
$\AlgSp_k^\qcqs$ the category of qcqs $k$-algebraic spaces.
We follow \cite[Tag.~01IO]{Sta} and call an immersion, a morphism of schemes that can be factored as $j\circ i$ with $i$ 
a closed immersion and $j$ an open one.

An $\infty$-prestack over $k$ is a functor $\Aff^{\op}_k\to\Spc$.
Let $\PrStk_k$ be the $\infty$-category of $\infty$-prestacks over $k$.
Let $\Stk_k$ be the $\infty$-category of $\infty$-stacks over $k$, i.e., 
the full subcategory of $\PrStk_k$ of sheaves for the \'etale topology on $\Aff_k$.
We write
\begin{align}\label{PSt}
\PrStk_k=\PrSh(\Aff_k)
,\quad
\Stk_k=\Sh(\Aff_k).
\end{align}
The sheafification functor $\PrStk_k\to\Stk_k$ is the left adjoint to the inclusion
$\Stk_k\subset\PrStk_k$. 
All quotients in $\Stk_k$ are made for the étale topology, unless explicitely stated.

Let (P) be a class of morphisms $f:\cX\to Y$ from an $\infty$-stack to an affine scheme which is closed under pullbacks.
A morphism $f:\cX\to\cY$ of $\infty$-stacks is (P)-representable, if for every
morphism $Y\to\cY$ where $Y\in\Aff_k$, the pullback $\cX\times_\cY Y\to Y$ belongs to (P).
We say that $f$ is representable / schematic / affine
if it is (P)-representable, where (P) is the class of all morphisms $\cX\to Y$, 
where $\cX$ is an algebraic space / scheme / affine scheme.
We say that $f$ is (fp) open / (fp) closed / (fp) locally closed immersion
if $(P)$ is the class of (fp) open / (fp) closed / (fp) locally closed immersions of schemes.

For any $\infty$-stack $\cX$ let $\cX_\red$ be the corresponding reduced $\infty$-stack.
We say that $\cX$ is reduced if the counit map $\cX_\red\to \cX$ is invertible.
More precisely, let $\iota:\Aff_{\red,k}\hra\Aff_k$ be the inclusion of the category of reduced affine schemes over $k$.
The \'etale topology on $\Aff_k$ restricts to the \'etale topology on $\Aff_{\red,k}$, hence we can consider the $\infty$-category
$\Stk_{\red,k}=\Sh(\Aff_{\red,k})$.
The restriction functor $\iota^*:\Sh(\Aff_k)\to\Sh(\Aff_{\red,k})$ has a fully faithful left adjoint $i_!$.
We define $\cX_\red=\iota_!\iota^*\cX$.
See \cite[\S 1.4]{BKV} for details. 
We say that a morphism of $\infty$-stacks $f$ is a topological equivalence if the morphism $f_\red$ is an equivalence, and that 
$f$ is a topological fp locally closed immersion if $f_\red$ is an fp locally closed immersion.

 \subsection{Acknowledgments}

At first, we are extremely grateful to Pierre Baumann and Thibaud Van den Hove. The first, for countably many answers, examples and counter-examples and the second, for this very nice trick that cleared an important obstacle on MV cycles.
We also thank, Alexander Braverman, Kestutis \v{C}esnavi\v{c}ius, Michael Finkelberg, Dennis Gaitsgory, Stéphane Gaussent, Syu Kato, Shrawan Kumar, Jo\~ao Lourenço, Olivier Mathieu, Ivan Mirkovic, Dinakar Muthiah, Sam Raskin, Simon Riche, Guy Rousseau, Yakov Varshavsky and Xinwen Zhu for helpful correspondence and conversations.

\section{Foundations on Kac-Moody groups}

\subsection{Ind-schemes}
For any $k$-algebra $R$ we write
$$K_R=R[s,s^{-1}]
,\quad
\cO_R=R[s]
,\quad
\cO_R^-=R[s^{-1}]
,\quad
\hat K_R=R\llp s\rrp
,\quad
\hat\cO_R=R\llb s\rrb.$$
We abbreviate
$$K=K_k
,\quad
\cO=\cO_k
,\quad
\cO^-=\cO_k^-
,\quad
\hat K=\hat K_k
,\quad
\hat\cO=\hat\cO_k.$$

\subsubsection{Ind-schemes}\label{sec-ind-sch}

\bdefi\label{d-indsch} 
\hfill
\begin{enumerate}[label=$\mathrm{(\alph*)}$,leftmargin=8mm]
\item
A weak ind-scheme (resp.~weak ind-algebraic space) over $k$ is a space $\cX$ over $k$  
which admits a presentation as a filtered 
colimit $X\cong\colim X_a$ where $X_a$ are $k$-schemes (resp.~algebraic spaces) and transition maps are closed immersions. 
If in addition the schemes (resp.~the algebraic spaces) are qcqs, we say that $X$ is an ind-scheme 
(resp.~ind-algebraic space).
\item
 We say that $X$ is a reasonable ind-scheme if transition maps are fp closed immersions.
\item
If (P) is a class of schemes which is stable by closed immersions, we say that $X$ is ind-(P), if each $X_a$ is (P).
If (P) is a class of morphisms schemes which is stable under base change, then a morphism $f:X\to Y$ of ind-schemes
is ind-(P) if there is a presentation $f_{a,b}:X_a\to Y_b$ of $f$ by morphisms in (P)
with $X\cong\colim X_a$ and $Y\cong\colim Y_b$.
\end{enumerate}
\edefi

\brem
\hfill
\begin{enumerate}[label=$\mathrm{(\alph*)}$,leftmargin=8mm]
\item
A morphism $f:\cX\ra\cY$ of prestacks is ind-representable if for any map $\Spec(R)\ra\cY$
the fiber product $\cX_R=\Spec(R)\times_{\cY}\cX$ is representable by an ind-algebraic space.
\item
The colimit is taken in $\PrStk_k$ and is computed componentwise, i.e., $X(R)\cong\colim X_a(R)$ for any 
$k$-algebra $R$, see \cite[Prop.~5.1.2.2]{Lu1}.
Since a filtered colimit of $n$-truncated groupoids is $n$-truncated,
an ind-scheme takes values into $\Sets$ (take $n=0$), see \cite[Rmk.~5.5.8.26]{Lu1} and  \cite[Ex.~7.3.4.4]{Lu1}.
\item
By Gabber's theorem  \cite[Tag.~0APL]{Sta}, an ind-algebraic space is a stack for the fpqc topology. 
Indeed, each $X_a$ is a fpqc sheaf and the 
finite limits involved in the fpqc descent commute with filtered colimits.
\eenum
\erem

\blem\label{lem:immersion}
Let $i:X\ra Y$ be an immersion of ind-schemes over $k$.
If $Y$ is of ind-ft, then $i$ is a quasi-compact immersion.
If further $Y$ is ind-affine, then $i$ is also quasi-affine.
\elem

\bpf
For any ft-closed subscheme $Z\subset Y$, 
the morphism $i_{Z}:X\times_{Y}Z\ra Z$ is an immersion of finite type, being locally closed in a ft $k$-scheme. Thus $i$ is an immersion of finite type, thus quasi-compact.
\epf

An $\infty$-stack $\cX$ is a space if the $\infty$-groupoid $\cX(S)$ is isomorphic to a set for each $S\in\Aff_k$.
Let $H$ be a group space acting on an $\infty$-stack $\cX$.
A morphism of $\infty$-stacks $f:\cX\ra\cY$ is an $H$-torsor if $f$
is an epimorphism in the \'etale topology, i.e., there are sections locally for the \'etale topology, 
and the action map $H\times\cX\to\cX\times_\cY\cX$ is an isomorphism.
We can form the quotient $[\cX/H]$.
It is an $\infty$-stack such that,
for each affine scheme $S$, the $\infty$-groupoid $[\cX/H](S)$ classifies pairs consisting of an étale $H$-torsor $E\to S$ 
and an $H$-equivariant map $\phi: E\ra\cX$.
As in the classical case, the quotient map $\cX\to[\cX/H]$ is an $H$-torsor, and if $f:\cX\ra\cY$ is an $H$-torsor then
the induced map $[\cX/H]\to\cY$ is an isomorphism.
See \cite[\S1.2.6]{BKV} for more details.

\begin{lemma}\label{quotH}
Let $f:X\ra Y$ be a morphism of ind-schemes which is equivariant for the action of a group ind-scheme $H$, 
with $X$ ind-separated. Let $[f/H]:[X/H]\ra[Y/H]$ be the corresponding morphism of stacks.
Then $[f/H]$ is ind-representable. Consider one the following properties $(P)$
\hfill
\begin{enumerate}[label=$\mathrm{(\alph*)}$,leftmargin=8mm]
\item
quasi-affine,
\item
finitely presented,
\item
qc immersion,
\item
ind-affine or ind-quasi-affine.
\end{enumerate}
If $f$ is (P), so does $[f/H]$.
If in addition $H$ is ind-affine, then $[X/H]$ has ind-affine diagonal.
\end{lemma}

\begin{proof}
Given $\phi:\Spec(R)\ra [Y/H]$, we set $\cX_{R}=\Spec(R)\times_{[Y/H]}[X/H].$
Recall we use the quotient for the étale topology.
Hence, after an étale cover $\Spec(R')\ra\Spec(R),$ 
the map $\phi$ lifts to $Y$.
Further,  the functor $\cX_{R'}$ is representable by an ind-scheme. Moreover the 
diagonal $\cX_{R}$ is representable by a closed immersion, as it is the case after an étale cover, by effectivity of descent of affine 
morphisms. Then, it follows from \cite[Lem.~3.12]{HR} that the functor $\cX_{R}$ is representable by an ind-algebraic space.
To conclude, it is thus enough to prove that all these properties are étale local for algebraic spaces and this follows
from \cite[Tag 0423, 041V, 0424]{Sta}. Finally for (d), it follows again by effectivity of descent for affine/quasi-affine schemes.
Let us now prove the statement on the diagonal.
We have the Cartesian diagram
$$\xymatrix{X\times H\ar[d]\ar[rr]^{\Delta'}&&X\times X\ar[d]\\
[X/H]\ar[rr]^-{\Delta_{[X/H]}}&&[X/H]\times[X/H]}$$
with $\Delta'(x,h)=(x,h.x)$. But $\Delta'$ is the composite of the map
$$\Delta_{X}\times\id_{H}:X\times H\ra X\times X\times H$$ 
which is a closed immersion, because $X$ is ind-separated, followed by the map
$$\act:X\times X\times H\to X\times X,\quad(x_1,x_2,h)\mapsto (x_1,h\cdot x_2)$$ which is ind-affine, because $H$ is.
\end{proof}

We will need the following small variant of the above.

\blem\label{quotH2}
Let $S$ be a scheme and $X$ be an $S$-scheme with an action of an fppf affine group scheme $H$ over $S$.
Then $[X/H]$ has a schematic separated diagonal (here we sheafify for fppf topology). 
If in addition $X$ is quasi-separated, then the diagonal $\Delta_{[X/H]}$ is qc.
\elem

\bpf
By \cite[Tag.~04UI, Tag.~04TB, Tag.06DB]{Sta}, 
the functor $[X/H]$ is an algebraic stack.
Thus, by \cite[Tag.~04YQ]{Sta}, the morphism $\Delta_{[X/H]}$ is schematic and separated. 
If $X$ is quasi-separated, then in the proof of Lemma  \ref{quotH} the map
$\Delta'$ is qc, because it is the composition of a qc morphism followed by an affine one, 
as $H$ is affine thus qc.
\epf

\bexa\label{ind-tors}
Let $H$ be an ind-affine or ind-quasi-affine group ind-scheme over a base scheme $S$. 
Let $T$ be an $S$-scheme.
Each étale-locally trivial $H$-torsor over $T$ is representable by an ind-affine scheme (resp.~ind-quasi-affine) over $T$. 
Indeed by Lemma \ref{quotH}, the morphism $S\ra\bB H$ is ind-affine (resp.~ind-quasi-affine) and if $E$ is an étale-locally trivial $H$-torsor over $T$, 
there is a map $T\ra\bB H$, such that $E\cong T\times_{\bB H}S$ and $E$ is ind-affine (resp.~ind-quasi-affine) over $S$.
\eexa

\subsubsection{Arcs and loops}

We use the polynomial loops functor because it has better finiteness properties.
For any prestack $\cX$ over a ring $k$, we consider the functors 
$\cX_\cO$, $\cX_K$ on $k$-algebras given by
$$\cX_\cO:R\mapsto \cX(\co_R),\quad 
\cX_K:R\mapsto \cX(K_R).$$

\begin{lemma}\label{loop-rep}
\hfill
\begin{enumerate}[label=$\mathrm{(\alph*)}$,leftmargin=8mm]
\item
If $\cX$ is formally smooth then $\cX_\cO$ is also formally smooth.
\item
Let $X$ be an ind-affine scheme over $k$.
Then  $X_\cO$ and $X_K$ are representable by ind-affine schemes.
\item
Let $X$ be an ind-fp-affine scheme over $k$.
Then  $X_\cO$ and $X_K$ are representable by ind-fp-affine schemes.
\end{enumerate}
\end{lemma}

\begin{proof}
Part (a) is obvious.
To prove (b), since both loop functors commute with filtered colimits, we may assume that $X$ an affine $k$-scheme. 
Choose a closed embedding of $X$ in $\ab^{(I)}=\Spec(k[x_i\,;\,i\in I])$. Then $X_{\co}$ is closed inside
\[\ab_\cO^{(I)}=\colim_{d\geq 0} \Spec\Big(k[x_n\,;\,n\in I\times\bN]\,/\,(x_n\,;\,n\in I\times[d,\infty))\Big).\]
For $X_K$ the proof is similar. 
Now, let us prove (c).
Both loop functors commute with filtered colimits.
Thus we may assume that $X$ is fp and closed in $\ab^N$ for some positive integer $N$. 
Since both functors preserve fp closed immersions and finite products, 
we are reduced to $X=\ab^{1}$.
The claim is now obvious.
\end{proof}

An other basic result on loop torsors is the following.

\bprop\label{H-tors}
Let $G$ be an ind-affine group ind-scheme over a ring $k$, and $R$ be a $k$-algebra. The loop functor yields a  
bijection between $G$-torsors on $\ab^{1}_{R}$ (resp.~$\bG_{m,R}$) étale locally trivial over $R$, 
and étale locally trivial $G_\cO$-torsors (resp.~$G_K$-torsors) on $\Spec(R)$.
\eprop

\bpf
We treat the case of $G_\cO$, the case of $G_K$ is analog.
Let $E$ be a $G$-torsor over $\ab^{1}_{R}$ that is étale locally trivial over $\Spec(R)$.
Then $E_\cO\ra\Spec(R)$ is an étale locally trivial $G_\cO$-torsor as
\[E_\cO\times_{\Spec(R)}E_\cO\cong(E\times_{\ab_R^1}E)_\cO\cong E_\cO\times_{\Spec(R)} G_\cO.\]
Consider now the converse.
Let $E$ be $G_\cO$-torsor over $\Spec(R)$ 
that trivializes after an étale cover $\Spec(R')\ra\Spec(R).$
As $G_\cO$ is ind-affine, by Example \ref{ind-tors}, a $G_\cO$-torsor over $R$ trivializable 
over $R'$ is equivalent to the trivial $G_\cO$-torsor over $R'$ with a descent datum.
The descent datum amounts of an element 
$$g\in G_\cO(R'\otimes_{R}R')=G((R'\otimes_{R}R')[s])$$ that satisfies a cocycle condition.
Since the map
$\ab_{R'}^1\ra\ab_R^1$ 
is also étale surjective, we can use this descent 
datum for the trivial $G$-torsor on $\ab_{R'}^1$.
This yields a $G$-torsor $\ab_R^1$.
\epf

\brem
For Laurent power series rings $\hat K_R$ instead of $K_R$, the situation is more delicate 
but the result is still true for $G$ reductive with an embedding in $GL_n$, 
as we were informed by K. \v{C}esnavi\v{c}ius.
\erem

\subsubsection{Pro-rings and pro-modules}
Let  $k$ be a ring. Let $\Pro(\Mod_k)$ be the category of $k$-pro-modules, which consists of cofiltered projective systems 
of $k$-modules. It is an abelian category where short exact sequences are given by cofiltered 
systems of short exact sequences of $k$-modules.
For each ring morphism $k\ra R$, there is a base change functor 
$\Pro(\Mod_k)\ra\Pro(\Mod_R),$ $M\mapsto M\hat{\otimes}_{k}R.$
A pro-ring is a pro-object $k=(k_a)$ in the category of commutative rings. 
A pro-module over a ring $k_a$ is a  pro-object in the category of $k_a$-modules.
 Given a morphism $k_a\to k_b$ the base change of pro-modules is denoted by
$M_a\to M_a\widehat\otimes_{k_a}k_b.$
A pro-module over the pro-ring $k$ is a compatible system $M=(M_a)$ of pro-modules $M_a$ over $k_a$
such that $M_b\cong M_a\widehat\otimes_{k_a}k_b.$
It is countably generated or $\aleph_0$ if it can be represented by a compatible system 
whose index set is countable.
A pro-ideal $I$ of $k$ is a pro-module over $k$, where each $M_a=I_a$ is an ideal of $k_a$.
For any abelian category $\cC$, a pro-object $M=(M_a)$ in $\Pro(\cC)$
is Mittag-Leffler (ML) if it is equivalent to a filtered projective system $(M_{a})$ with surjective transition maps.
It is  strictly ML if it is ML and $\lim_{a\geq b} M_a\ra M_b$ is surjective for each $b$.

An ind-affine scheme $T$ over a ring $k$ is equivalent to an $k$-pro-algebra that is ML.
There is a  functor to topological $k$-modules
\[F:\Pro(\Mod_{k})\ra\TopMod_{k}\]
given by $(M_{a})\mapsto\lim M_a$ and a left adjoint
\[G:\TopMod_{k}\ra\Pro(\Mod_{k})\]
that sends $M$ to the projective system of its discrete quotients. 
In general, the functors $F$ and $G$ are not equivalences. 
It is convenient to specify a class of objects for which this is the case.
If $N\in\TopMod_k$, then $G(N)$ is strict ML.
Thus, if $M\in\Pro(\Mod_k)$  then
\[G\circ F(M)\cong M\iff M\ \text{is\ strict\ ML}\]
By \cite[Tag.0597]{Sta}, for any countably generated $k$-pro-module $M$, the conditions 
$M$ is ML and $M$ is strict ML are equivalent.
In particular, an $\aleph_0$-ind-affine ind-scheme  $T\cong\colim \Spec(R_a)$ over $\Spec(k)$ is the same as a $k$-complete 
topological algebra $R=\lim R_a$, whose topology is defined by a countable family of open ideals $I_a$.

\subsubsection{Algebraic smoothness and formal smoothness}

Consider an ind-scheme $X$ over a field $k$. 
For each $x\in X$ we can construct the following ind-affine ind-scheme
\[\ISpec(\cO_{X,x})=\colim_{x\in Z\subset X}\Spec(\cO_{Z,x}),\]
where $Z$ runs over closed subschemes of $X$. It is filtered because $X$ is an ind-scheme.
The ind-affine scheme $\ISpec(\cO_{X,x})$ corresponds to a pro-algebra $\cO_{X,x}$ represented by  the pro-system
$\cO_{X,x}\cong (\cO_{X_{a},x})$
for an ind-scheme presentation $X\cong\colim X_{a}$. 
The pro-local ring $\cO_{X,x}$ has a maximal ideal $\km_{x}=(\km_{x,X_a})$.
Similarly, we can consider the local completion at $x$ 
\[\ISpec(\cO^{\wedge}_{X,x})=\colim_{x\in Z\subset X}\Spec(\cO^{\wedge}_{Z,x}).\]
Note that we have isomorphims of projective systems 
\[\cO^{\wedge}_{X,x}\cong(\cO^{\wedge}_{X_{a},x})\cong\lim(\cO_{X_a,x}/\km^n_{X_{a},x})\cong\lim(\cO_{X,x}/\km_{X,x}^{n}).\]
Further, we have
\begin{equation}
\gr^{p}(\cO_{X,x})\cong\km_{x}^{p}/\km_{x}^{p+1}\cong\gr^{p}(\cO_{X,x}^{\wedge}),
\label{p-gr}
\end{equation}
The following terminology was introduced by Shafarevich \cite{Shaf} and used in \cite[Def.~4.3.1]{Kum}

\bdefi
Let $X$ be an ind-scheme over a field $k$. 
We say that $X$ is algebraically smooth at $x$ if
$\Sym^{p}(\km_x/\km_x^{2})\cong\km_{x}^{p}/\km_{x}^{p+1}$ for all $p\geq 0$
as $\cO_{X,x}$-pro-modules.
\edefi

For a $k$-pro-vector space $V$ which is  ML, we can form the ind-affine scheme
\[\ab(V^{\vee})=\colim_{V\thra V_{a}} \Spec(\Sym (V_a))\]
where $V_a$ is a vector space, and $V^{\vee}$ is the dual ind-vector space of $V$. 
This ind-affine scheme is formally smooth. If $V=\km_x/\km_x^2$ then 
\begin{equation}
\gr^{p}(\cO_{\ab(V^{\vee}),0})\cong\Sym^{p}(\km_x/\km_x^{2})\cong\gr^{p}(\cO_{\ab(V^{\vee}),0}^{\wedge}).
\label{p-gr2}
\end{equation}

We want to relate algebraic smoothness to formal smoothness.
\blem\label{cont-iso}
Let $X$ be an ind-scheme, and $x\in X$.
If $X$ is algebraically smooth at $x$, then we have an isomorphism of pro-algebras
$\cO_{\ab(T_{X,x}),0}^{\wedge}\cong\cO_{X,x}^{\wedge}$
and an isomorphism of ind-schemes 
$$\ISpec(\cO_{X,x}^{\wedge})\cong\ISpec(\cO_{\ab(T_{X,x}),0}^{\wedge}).$$
\elem

\bpf
By choosing lifts in $\km_x$ of a projective system of generators of $\km_{x}/\km_{x}^{2}$.
We construct a local morphism of pro-local rings
$\phi:\Sym(\km_x/\km_{x}^{2})\ra\cO_{X,x}$
that yields a map of pro-rings 
\[\hat{\phi}:\cO_{\ab(T_{X,x}),0}^{\wedge}\ra\cO_{X,x}^{\wedge}.\]
By algebraic smoothness and \eqref{p-gr}, \eqref{p-gr2}, it is an isomorphism on each graded pieces.
Thus it is an isomorphism.
\epf

\blem\label{inf-art}
An ind-ft scheme $T$ over a noetherian ring $k$
is formally smooth if it satisfies the infinitesimal lifting property for local complete $k$-algebras.
\elem

\bpf
By Gabber's theorem \cite[Thm.~2.2.3, 6.2.5]{Bt3}, 
it is enough to check formal smoothness for a ring $R$ after a fpqc cover.
Since $T$ commutes with filtered colimits, we reduce to $R$ a local $k$-algebra of ft and using the completion morphism, 
as $k$ is noetherian, to $R$ local complete.
\epf

\bprop\label{inf-art 2}
Consider an ind-ft-scheme $T$ over a field $k$. If $T$ is algebraically smooth, then it is formally smooth.
\eprop

\bpf
By Lemma \ref{inf-art}, it is enough to check the infinitesimal lifting criterion for $R$ a complete local ring. 
Let $I\subset R$ be a square zero ideal.
We must lift a map $\phi:\Spec(R/I)\ra T$.
Let $x\in X$ be the image of the closed point. 
Since $R$ is local complete, the map $\phi$ factors through $\ISpec(\cO_{T,x}^{\wedge})$ which is formally smooth by Lemma \ref{cont-iso}.
Thus the map $\phi$ lifts.
\epf
The next statement is a generalisation of Cartier's theorem. Even the case of an affine group scheme does not appear in the litterature.
\bthm\label{cart}
Let $k$ be a field of characteristic zero.
\begin{enumerate}[label=$\mathrm{(\alph*)}$,leftmargin=8mm]
\item
An affine group scheme over $k$ is formally smooth.
\item
An ind-affine group $\aleph_0$-ind-scheme $G$ over $k$ is formally smooth.
\eenum
\ethm

\bpf
By \cite{Per}, an affine $k$-group scheme $G$ is a projective limit of affine algebraic groups $G_i$. 
In characteristic zero, the groups $G_i$ are smooth 
and the cotangent complex $L_{G_i/k}$ is in degree zero and projective.
As the cotangent complex commutes with colimits, the complex
$L_{G/k}$ is concentrated in degree zero, isomorphic to $\Omega^{1}_{G/k}$ and flat. 
By \cite[Tag.~047I]{Sta}, the module $\Omega^{1}_{G/k}$ is free.
Thus  $G$ is formally smooth by \cite[Tag.~0D0L]{Sta}, proving (a).
Now we prove (b). By \cite[Thm.~4.3.7]{Kum} the group scheme $G$ is algebraically smooth. We conclude by Proposition \ref{inf-art 2}. 
The $\aleph_0$ assumption is needed to replace pro-modules by topological modules.
\epf

Finally, let us quote the following result for a future use.

\bprop\label{et-lift}
Let $k$ be a commutative ring.
Let $f:\cX\ra\cY$ and $g:\cY \ra\cZ$
be ind-fp-schematic morphisms of prestacks.
Assume that $f$ has sections étale locally and that $g  f$ is formally smooth.
Then $g$ is formally smooth.
\eprop

\bpf
Let $R$ be a $k$-algebra and $I\subset R$ an ideal with $I^2=0$. Set $\ov{R}=R/I$.
Let $z\in\cZ(R)$ such that its reduction $\bar{z}\in\cZ(\ov{R})$ lifts to $\bar{y}\in\cY(\ov{R})$.
After pulling back by $z$, we can assume that $\cZ=\Spec(R)$.
Then $\cX$, $\cY$ are ind-fp-schemes over $\cZ$.
By assumption, there is $\ov{R}\ra \ov{R'}$ étale surjective such that $\bar{y}$ lifts as $\bar{x}\in\cX(\ov{R'})$.
By \cite[Tag.~039R]{Sta}, there is an étale $R$-algebra $R'$ such that $R'/IR'\cong \ov{R'}$.
Since the map  $g  f$ is formally smooth, there is a lift $x\in\cX(R')$ of $z\vert_{R'}$, hence $g(x)\in\cY(R')$ lifts also $z\vert_{R'}$.
Thus the map $g$ satisfies the infinitesimal lifting after an étale cover $R\ra R'$.
By Gabber's theorem \cite[Thm.~6.2.5]{Bt3}, for ind-fp-schemes this is equivalent to the usual formal smoothness.
\epf

\subsubsection{Tangent bundles}
Let $X$ be an ind-scheme over a base scheme $S$ and $x\in X(S)$.
The tangent space $T_{x}X$ at $x$ is the sheaf which associates to any $S$-scheme $T$ 
the set of points $y\in X(T[\epsilon])$ such that $y=x_{T}\mod\eps$.
Since the functor commutes with colimits, by \cite[Prop.~5.2]{HLR}, we have 
\begin{equation}
T_{x}X(T)=\Hom_{\cO_{T},c}(x_T^{*}\Omega_{X_{T}/T}^1,\cO_T):=\colim\limits_{Z\subset X}\Hom_{\cO_{T}}(x_T^{*}\Omega^{1}_{Z_{T}/T},\cO_{T}),
\label{fond-T}
\end{equation}
where $Z$ runs over all closed subschemes of $X$ and the subscript 'c' stands for continuous.
In particular, $T_{x}X(T)$ has a structure of $\Gm(T,\cO_{T})$-module.
The next statement generalizes \cite[Lem.~5.5]{HLR} to ind-schemes.

\bprop\label{Z-tan}
Let $R$ be a Dedekind ring,  $S=\Spec(R)$, and $X\ra S$ an $\aleph_0$-ind-ft-scheme over $S$. 
Let $x\in X(R)$. For all $R$-algebras $R'$, the canonical map $T_{x}X(R)\otimes_{R} R'\ra T_{x}X(R')$ 
is injective. It is bijective if and only if $x^{*}\Omega^{1}_{X/S}$ is torsion free. 
Then  $x^{*}\Omega^{1}_{X/S}$ is dual to a projective $R$-module.
\eprop

\bpf
By localizing we may assume that $R$ is a DVR. Let $M=x^{*}\Omega^{1}_{X/S}$.
Write $X\cong\colim X_{a}$.
Then $M$ is a ML pro-set isomorphic to the system $(M_a)$ with $M_{a}=x^{*}\Omega^{1}_{X_{a}/S}$. Each $M_{a}$ splits as $M_{a}^{\free}\oplus M_{a}^{tor}$.
Let $X'=X_{R'}$ and $x'\in X(R')$ be the section given by $x$. 
Since $(x')^*\Omega^{1}_{X'/R'}=x^{*}\Omega^{1}_{X/R}\otimes R'$, see \cite[Tag.~01UV]{Sta}, we have
\begin{align*}
T_{x}X(R')&=\colim\Hom_{R'}(M_{a,R'}, R')\\
&\cong \colim (M_{a}^{\free}\otimes_RR')^*\oplus\colim\Hom_{R'}(M_a^{\tor}\otimes_RR',R').
\end{align*}
Since $R$ is torsion free, $\Hom_R(M_{a}^{\tor},R)=0$ and the injectivity follows.
The bijectivity is equivalent to 
$$\colim\Hom_R(M_{a}^{\tor}\otimes_RR',R')=0$$ for any $R$-algebra $R'$. 
Let us show that it implies that $(M_{a}^{\tor})=0$.
Taking $R'=R/\km$, we get that $(M_{a}^{\tor}/\km)=0$ and \cite[Lem.~5.4.9]{Bt3} gives the claim. We recall the argument here.
For every $a$, there exists $b\geq a$, such that $f_{ba}:M_{b}\ra M_a$ is the zero map modulo $\km$. 
By applying Nakayama  to $\Ima(f_{ba})$ we get $f_{ba}=0$, as wished.
The last claim follows from \cite[Prop.~5.2.11]{Bt3}: by Raynaud-Gruson, a pro-projective module of finite type, which is countably 
generated and is a  ML pro-set is dual to a  projective module.
\epf

\subsubsection{Untwisted affine Kac-Moody groups}\label{sec:AKM}
In this section, we introduce the untwisted affine KM groups.
Recall that for any ring $k$, 
a reductive group $G$ over $\Spec(k)$ is a smooth affine group with geometric fibers that are connected reductive.
Given a  split reductive group scheme $\bG_{f}$ over $k$, we say that a central extension $\widetilde{\bG_{f,K}}$ of $\bG_{f,K}$
is of KM type if we have a central extension
\[1\ra Z\ra \widetilde{\bG_{f,K}}\ra \bG_{f,K}\ra 1\]
that splits over $\bG_{f,\co}$, and $Z$ is a split $k$-torus.

\bdefi\label{aff-def}
Let $G$ be a minimal KM group over a commutative ring $k$. We say that
$G$ is untwisted affine type if there is a $k$-split torus $H$ and a split reductive group $\bG_{f}$ over $k$ such that 
$G\cong\widetilde{\bG_{f,K}}\rtimes H$, where $\widetilde{\bG_{f,K}}$ is a central extension of $\bG_{f,K}$ of KM type.
\edefi

The main feature of the untwisted affine case is that there is a parabolic $P=Z\times\bG_{f,\co}\rtimes H$ such that 
\begin{equation}
G/P\cong\Gr_{\bG_{f}}.
\label{part-flag}
\end{equation}
We recall the following theorem obtained in \cite{Ces} as a particular case.

\bthm\label{ces-spl}
Let $\bG$ be a split reductive group over a ring $k$, then $\Gr_{\bG}$ is a presheaf quotient, i.e., for any ring $R$ we have
$\Gr_{\bG}(R)=\bG(R[t,t^{-1}])/\bG(R[t]).$
\ethm

\bcor\label{ces-2}
Let $G$ be an  untwisted affine KM group over a ring $k$.
Let $\bG_f$ be as in Definition \ref{aff-def}. 
Let $\bB_{f}\subset\bG_f$ be a Borel subgroup of $\bG_f$, and $B$ the corresponding one in $G$. 
Let $R$ be a ring such that $\Pic(R)=0$.
We have $(G/B)(R)=G(R)/B(R).$
\ecor

\bpf
Let $x\in (G/B)(R)$.
There is a parabolic $P\subset G$ such that $G/P=\Gr_{\bG_f}$. 
Let $\ov{x}$ be the image of $x$ in 
$(G/P)(R)$. By Corollary \ref{ces-2}, we can lift $\ov{x}$ to $y\in G(R)$ such that 
$$h=y^{-1}x\in (P/B)(R)=(\bG_{f}/\bB_{f})(R),$$ 
where $\bB_{f}$ is the corresponding Borel of $\bG_{f}$. 
Since $\Pic(R)=0$, we can lift $h$ to $(\bG_{f}/U_{f})(R)$ for $U_{f}$ the unipotent radical, 
and then further to $\bG_{f}(R)$, which concludes.
\epf

We will need the following variant for the opposite parabolic sugroups.
Let $\bG$ be split reductive over a ring $k$.
We consider the thick flag variety $X_{\bG}^{\thick}$
given by the étale quotient
$$X_{\bG}^{\thick}=[\bG_{\hat K}/\bG_{\cO^{-}}],$$
see \S\ref{sec-ind-sch}. It is representable by a scheme, see, e.g., \cite[Rmk.~2.3.6]{Zhu}.

\bprop\label{ces-3}
If $\bG$ is a split reductive group over a ring $k$, then $X_{\bG}^{\thick}$ is a presheaf quotient.
\eprop

\bpf
Let $R$ be a ring.
Then $X(R)$ classifies the pairs $(E,\phi)$ of a $\bG$-torsor $E$ over $\bP_{R}^{1}$ and a trivialization $\phi$ 
on the formal disc $\Spec(R\llb t\rrb)$ at 0, see, e.g.,  \cite[Rmk.~2.3.6]{Zhu}. 
We need to prove that $E$ is trivial over $\ab^{1}_R=\bP^{1}_{A}\setminus\{0_{R}\}$.
This follows from \cite[Thm.~16]{Ces}.
\epf

\subsection{The formal Kac-Moody group}
In this section, we review the construction of formal Kac-Moody groups à la Mathieu, as group ind-schemes.

\subsubsection{The Kac-Moody root datum}\label{r-KM}
We follow  \cite{Rou}.
Consider a quadruple $\cD=(\La,\LLa,\Delta,\LDelta)$ such that
\begin{enumerate}[label=$\mathrm{(\alph*)}$,leftmargin=8mm]
\item
 $\La$ is a free $\bZ$-module of ft with $\bZ$-dual $\LLa$,
\item
 $\Delta\subset\La$ and $\LDelta\subset\LLa$ are finite sets vectors with a bijection $\Delta\to\LDelta$ 
 such that $\al\mapsto\Lal$.
\eenum

Consider the matrix $A_{\cD}=(\langle \beta,\Lal\rangle)$ with $(\Lal,\beta)\in\LDelta\times\Delta$.
The quadruple $\cD$ is called
a Kac-Moody (KM) root datum if  $A_{\cD}$ is a 
generalized Cartan matrix (GCM).
The KM root datum $\cD$ is free if the family of vectors of $\Delta$ is free in $\La$. 
It is cofree if the family of vectors of $\LDelta$ is free in  $\LLa$.
Set $Q=\bigoplus_{\al\in\Delta}\bZ \al$.
There is a morphism of Abelian groups $Q\ra\La$ that sends $\al$ to $\al$.
When $\cD$ is free, we identify $Q$ to a submodule of $\La$.
A KM root datum $\cD$ is cotorsion free if  $\LLa/\bZ\LDelta$ is torsion free.
The elements of $\Delta$ and $\LDelta$ are called simple roots and coroots.
A morphism of KM root datum 
$$\phi:\cD=(\La,\LLa,\Delta,\LDelta)\ra\cD_1=(\La_1,\La_1^{\vee},\Delta_1,\Delta_1^{\vee})$$
is the datum of a linear map $\phi^\vee:\LLa\ra\La_1^{\vee}$,
and  an injection $s:\Delta\ra\Delta_1$ such that $A_{\cD}=A_{\cD_1}\vert_{\Delta\times\Delta}$,
$\phi^\vee(\Lal)=s(\al)^{\vee}$ and $s(\al) \phi^\vee=\al$ for all $\al\in\Delta.$
Let $\phi:\La_1\ra\La$ the dual map.
We say that
\begin{enumerate}[label=$\mathrm{(\alph*)}$,leftmargin=8mm]
\item
$\cD$ is a $z$-extension of $\cD_1$ if $\phi^\vee$ is surjective and $s$ is a bijection.
\item
$\cD$ is a subroot datum of $\cD_1$ if $\phi^\vee$ is injective and $\La_1^{\vee}/\phi^\vee(\LLa)$ is torsion free.
\item
$\cD_1$ is a semi-direct extension of $\cD$ if $\cD$ is a subroot datum of $\cD_1$ and $s$ is a bijection. 
\eenum
We have the following proposition, see \cite[Prop.~1.3]{Rou}.

\bprop\label{rou}
Let $\cD$ be a KM root datum.
\begin{enumerate}[label=$\mathrm{(\alph*)}$,leftmargin=8mm]
	\item 
	There exists a free semi-direct extension $\cD^l$ of $\cD$.
	\item
	There exists a  cofree and cotorsion free $z$-extension $\cD^{sc}$ of $\cD$. If $\cD$ is free, then so is $\cD^{sc}$. 
The root datum $\cD^{scl}:=(\cD^l)^{sc}$ is free, cofree and cotorsion free.
	\item
	If $\cD$ is free, cofree and cotorsion free, then it is a semidirect extension of a KM root datum
	$\cD^{m}$ which is free, cofree, cotorsion free and of dimension $\rk\cD+\dim\Ker(A_{\cD})$. 
\qed
\end{enumerate}
\eprop

\subsubsection{The Kac-Moody Lie algebra}
Let $\cD=(\La,\LLa,\Delta,\LDelta)$ be a KM root datum.
We consider the torus $T_{\cD}=\Spec(\bZ[\La])$. 
The set of characters of $T_{\cD}$ identifies with $\La$.
Set $\kt_{\cD}=\LLa\otimes_{\bZ}\bC$ its complexified Lie algebra.
The complex KM algebra $\kg_{\cD}$ associated with $\cD$ is the $\bC$-algebra generated by $\kt_{\cD}$ 
and elements $(e_{\al},f_{\al})_{\al\in\Delta}$ submitted to relations \cite[\S 1]{Kac}, i.e., 
for $h,h'\in \kt_{\cD}$ and $\al\neq\beta$ we have
\begin{equation}
[h,h']=0,\quad [h,e_{\al}]=\al(h)e_\al,\quad [h,f_{\al}]=-\al(h)f_{\al}, \quad [e_{\al},f_{\al}]=-\Lal,\quad  [e_{\al},f_{\beta}]=0,
\label{kac1}
\end{equation}
\begin{equation}
(\ad e_{\al})^{1-a_{\al,\Lbe}}(e_{\beta})=(\ad f_{\al})^{1-a_{\al,\Lbe}}(f_{\beta})=0.
\label{kac2}
\end{equation}
We will abbreviate $\kg=\kg_\cD$ and $\kt=\kt_\cD$.
Let $R\subset Q\setminus\{0\}$ be the root system of $\kg$ 
and $R^{\vee}$ its coroot system.
We have a $Q$-grading
$$\kg=\kt\oplus\bigoplus_{\al\in R}\kg_{\al}.$$
We set $\Delta=\{\al_i\,;\,i\in I\}$ and
$$Q^{+}=\bigoplus_{\al\in\Delta}\NN\al
,\quad
R^{+}=R\cap Q^{+}
,\quad
R^{-}=-R^{+}.$$
Any $\al\in R^+\cup\{0\}$ can be written as $\sum_{i\in I} n_i\al_i$  with $n_i\in\bN$.
The height of $\al$ is the integer
$\hgt(\al)=\vert\sum_{i\in I} n_i\vert.$
The positive and negative Borel subalgebras are
\[\kb=\kt\oplus\kn
,\quad
 \kb^{-}=\kt\oplus\kn^-
,\quad
\kn=\bigoplus_{\al\in R^+} \kg_{\al}
 , \quad
 \kn^-=\bigoplus_{\al\in R^-} \kg_{\al}.\]
For any subset $J\subset\Delta$, we define the parabolic subalgebra
\begin{equation}
\kp_{J}=\kb\oplus\bigoplus_{\al\in R^{-}_J} \kg_{\al},
\label{jpar}
\end{equation}
for $R_{J}=R\cap(\bigoplus_{\al\in J}\bZ\al)$ and $R_{J}^{\pm}=R_{J}\cap R^{\pm}$.
The nilradical and the Levi component of $\kp_{J}$ are the subalgebras
\begin{equation}
\kn_{J}=\bigoplus_{\al\in R^{+}\setminus R^{+}_{J}}\kg_{\al}
,\quad
\kl_{J}=\kt\oplus\bigoplus_{\al\in R_J}\kg_{\al}.
\label{jlev}
\end{equation}
We say that $J$ is of ft if $\kl_{J}$ is finite dimensional.
Let us consider the cone of dominants characters
\[\Lambda_{+}=\{\la\in\La\,;\,\langle\la,\Lal\rangle\geq0\,,\,\forall~\al\in R^{+}\}\]
and the subset of regular dominant characters 
$$\La_{++}=\{\la\in\La_{+}\,;\,\langle\la,\Lal\rangle>0,\forall\al\in\LDelta\}.$$
The cone of dominant cocharacters is $\LLa_+\subset\LLa.$
For $\la,\mu\in \LLa$ we write $\mu\leq\la$ if $\la-\mu$ is a sum of positive coroots.
The Weyl group is the subgroup $W\subset\Aut_{\bZ}(Q)$ generated by the simple reflections $s_{\al}\in\Aut(Q)$ such that
$s_{\al}(\beta)=\beta-\beta(\Lal)\al$ for each $\beta\in Q$.
Finally, let $R_{\re}$ be the set of real roots.

\subsubsection{The formal Kac-Moody group}
Consider a free, cofree and cotorsion free KM root datum $\cD$ such that
\begin{equation}
\rk\La=\rk\cD+\dim\Ker A_{\cD}.
\label{r-dim}
\end{equation}
We refer to this situation as the simply connected case. 
By Proposition \ref{rou}, we can reduce to this setting.
In the rest of the paper, we will work constantly under this assumption, unless explicitely mentioned.
To ease the reading, we may omit the dependence on $\cD$, e.g, we abbreviate $T=T_{\cD}$. 
We have
$$\Lambda=X^*(T),\quad \LLa=X_*(T).$$
By \cite[\S 2]{Rou}, for any commutative ring $k$, there is an integral version $\kg_{k}$ of the KM algebra with a triangular decomposition 
$\kg_{k}=\kn_{k}\oplus\kt_{k}\oplus\kn_{k}^{-}$ and a root space decomposition.
A subset of positive roots $\Theta\subset R^{+}$ is closed if $\alpha+\beta\in\Theta$ whenever
$\alpha,\beta\in\Theta$ and $\alpha+\beta\in R^+$.
By \cite[\S3.1, Prop.~3.2]{Rou}, for any closed subset $\Theta\subset R^{+}$ 
there is a $k$-split pro-unipotent group $(\hU_\Theta)_k$ with a group and Lie algebra isomorphisms
\begin{equation}
x_\Theta:(\khn_\Theta)_k\cong(\hU_\Theta)_k
,\quad
\Lie((\hU_\Theta)_k)\cong(\khn_\Theta)_k=\prod\limits_{\al\in\Theta}(\kg_\al)_k,
\label{exp-comp}
\end{equation}
given by an integral version $[\exp]$ of the exponential map.
For $\Theta=R^+$ we abbreviate $\hU_k$ and $\hkn_k$ for the affine group scheme $(\hU_\Theta)_k$ and 
its Lie algebra $(\hkn_\Theta)_k$.
In particular, for any real root $\al$ we have a root group $\hU_{\alpha,k}$
with an isomorphism $x_\alpha:\bG_{a,k}\to (\hU_\al)_k$, see, e.g., \cite[\S1.3.6]{Kum}.
There is an natural action of $T_k$ on $(\hU_\Theta)_k$ 
coming from the corresponding action on the Lie algebra $(\hkn_\Theta)_k$.
We set 
\[\hB_{k}=\hU_{k}\rtimes T_k.\]
By \cite[\S3.5]{Rou} there is a minimal parabolic $(\hP_\al)_k$ with a Levi decomposition
\[(\hP_\al)_k=(\hU_\al)_k\rtimes (G_\al)_k.\]
where $(G_\al)_k$ is the unique split reductive group over $k$ with the root datum
$\cD_{\al}=(\La,\LLa, \al,\Lal)$, see  \cite[Thm.~10.1.1]{Spr}. 
The group $(G_\al)_k$ admits $T_k$ as a maximal torus and we have
$$\Lie((G_\al)_k)=\kt_k\oplus(\kg_\al)_k\oplus(\kg_{-\al})_k.$$
For each $\al\in\Delta$, we have an element 
$n_{\al}\in (\hP_\al)_k$
 that normalizes the $k$-torus $T_k$ and represents the simple reflection $s_{\al}\in W$.
All these data are obtained by base change from the case where $k=\bZ$. 
In the sequel, the constructions may depend on $k$.
Following \cite{Mat}, for any $w\in W$ and any reduced decomposition $\uw=s_{i_1}\cdots s_{i_n}$ of $w$ we set 
\[E(\uw)_{k}=(\hP_{\al_{i_1}})_k\times^{\hB_k}\dots\times^{\hB_{k}}(\hP_{\al_{i_n}})_k.\]
The corresponding Demazure scheme is
$D(\uw)_{k}=E(\uw)_{k}/\hB_{k}.$ 
It is smooth and projective. 
The affine scheme
\[\cB_{k}(w)=\Spec(\Gamma(E(\uw)_{k},\cO_{E(\uw)_{k}}))\]
is independent of the  choice of the reduced decomposition $\uw$. 
For $w'\leq w$ we have a closed immersion of affine schemes $\cB_{k}(w)\hra\cB_{k}(w')$,
see \cite[p.~130]{Mat2}.
We define
\[\hG_k=\colim_{w\in W}\cB_k(w).\]
If we complete with respect in the negative direction, we get
\[\Lie(\hU_k^{-})=\hkn^-:=\prod_{\al\in R^{-}}(\kg_\al)_k.\]
We define in a similar way the opposite formal KM group $\hG_k^{\op}$.

\blem[\cite{Mat}]\hfill
\begin{enumerate}[label=$\mathrm{(\alph*)}$,leftmargin=8mm]
\item 
The functor $\hG_k$ is an ind-affine group ind-scheme over $\Spec(k)$.
\item 
The affine group schemes $\hB_k$ and $\hU_k$ are closed subgroups of $\hG_k$.
\eenum
\qed
\elem

\brem\hfill
\begin{enumerate}[label=$\mathrm{(\alph*)}$,leftmargin=8mm]
\item 
We will see in Proposition \ref{Bcmut} that the formation of $\hG_k$ commutes with base change.
and prove in Theorem \ref{cart2} that $\hkg_{k}\cong\Lie(G_{k})$, both are assertions do not seem to appear in the litterature.
\item
In \cite{Kum}, Kumar constructs a group $\hG^K$ as an amalgamated product of the normalizer 
$N(\bC)$ and the minimal 
parabolic subgroups $P_{i}(\bC)$. 
By \cite[\S 3.20]{Rou}, we have $\hG^K=\hG(\bC)$.
\eenum
\erem

\subsection{Flags on Kac-Moody groups}

We keep the assumptions of the previous section and work over a commutative ring $k$. The goal of this section is to prove that the formation of thing flag manifold commutes with base change and give a construction of Kashiwara flag variety over integers, not considered before. 

\subsubsection{Representations of Kac-Moody groups}\label{EndInd}
For any $\Lambda$-graded $k$-module $M=\bigoplus_{\nu\in\Lambda}M_\nu$ with free finite rank weight spaces,
let $M^\wedge$ be the completion $M^\wedge=\prod_{\nu\in\Lambda}M_\nu$ of $M$,
$M^\vee$ the restricted dual of $M$, i.e., 
the direct sum of the duals of the weight spaces,
and $M^*$ its completion, i.e., the product of the duals of the weight spaces.
The dual $M^*$ has an obvious structure of a pro-$k$-module of ft.
In particular, the projective space $\bP(M^*)$ is a scheme, 
while $\bP(M)$ and $\bP(M^\vee)$ are ind-scheme, see Example \ref{ex:Pproj}.

For each dominant character $\omega$, let 
$\rho_\omega:\kg\to\mathrm{End}(L(\omega))$
be the integrable  
highest weight $\kg$-module with highest weight $\omega$.
Let $v_\omega$ be a generator of the weight subspace of $L(\omega)$ of weight $\omega$,
and $v_\omega^\vee$ a generator of the weight subspace of $L(\omega)^{\vee}$ of weight $-\omega$.
We abbreviate $v_i=v_{\omega_i}$ and $v_i^\vee=v_{\omega_i}^{\vee}$ for each $i\in I$.
Let 
$U(\kg)_{\bZ}$,
$U(\kb)_{\bZ}$
and
$U(\kn)_{\bZ}$ be the Tits integral forms (= the hyperalgebras) 
of the enveloping algebras of $\kg$, $\kb$ and $\kn$, see, e.g.,  \cite[\S 2.1]{Rou}.
Following \cite[\S 2.14]{Rou},
for any dominant character $\omega$ we set
\begin{equation}
 L(\omega)_{\bZ}=U(\kn^{-})_{\bZ}\cdot v_{\omega}
 ,\quad
L(\omega)_k=L(\omega)_{\bZ}\otimes_{\bZ}k
\label{A-om}
\end{equation}
We fix $v_\omega^\vee$ to be a generator of the rank one $\bZ$-submodule of $L(\omega)^{\vee}$ of weight $-\omega$.

Let $V$ be a highest weight representation of a KM group $G$. 
Then $V$ is an ind-scheme of ind-ft, but $\End(V)$ is not an ind-scheme in general. 
We consider a more reasonable subfunctor.
Following Solis \cite[Def.~4.4]{Sol}, we use the following definition.
Let $V=\colim_{a\in\NN} V_a$ an ind-scheme structure on $V$ and a corresponding one $G\cong \colim_{a\in\NN} G_a$ on $G$. 
The action map $G\times V\ra V$ is a morphism of ind-schemes, 
thus for any $a\in\NN$, there is an integer $n(a)\in\NN$, such that $G_{i}\times V_i\ra V$ factors through $V_{n(a)}$. 
In particular, for any $a$, $b$ we have a morphism of schemes
$$G_a\times V_{j}\subset G_{a+b}\times V_{a+b}\ra V_{n(a+b)}.$$
For any $a\in\NN$ and any $k$-algebra $R$ we  set
\begin{align*}
\End_a(V)(R)&=\{\phi\in \End(V)(R)\,;\,\phi(V_b(R))\subset V_{n(a+b)}(R)\}.
\end{align*}
Then, we define
\begin{align*}
\End^\ind(V)&=\colim_{a\in\NN}\End_a(V).
\end{align*}
By \cite[Lem.~4.5]{Sol} the functor
$\End_a(V)$ is representable by a $k$-vector space and $\End^\ind(V)$ by an ind-vector space.
In particular $\End^\ind(V)$ is ind-affine. 
This ind-scheme depends on the choice of the integers $n(a)$, but it won't matter for the rest.
Note that $\End^\ind(V)$ is not necessarily of ind-ft.
There is an obvious morphism of ind-schemes $G\to\End^\ind(V)$.
In particular, we write
\begin{align}\label{rho-omega}\rho_\omega:G\to\End^\ind(L(\omega))\end{align}

\subsubsection{ The thin flag manifold of a Kac-Moody group}\label{FM}
By \cite[p.~45]{Mat}, the quotient $\hG_k/\hB_k$ is representable by an 
ind-projective scheme over $\Spec(k)$ called the thin flag manifold. 
Let us review the construction.
By \cite[p.~40]{Mat}, for any element $w\in W$ with reduced composition $w=s_1\dots s_n$ we
have a morphism of ind-fp $k$-ind-schemes
\[\pi_{w}:D(\uw)_k\ra\bP(L(\omega)_k),\quad (g_{1},\dots,g_n)\mapsto g_1\dots g_n\cdot L_{w\omega},\] 
where $L_{w\omega}\subset L(\omega)_k$ is the line of weight $w\omega$. 
The source being fp, the map factors through some fp closed scheme.
Let $(S_{w,\omega})_k$ be the closure of the image.
The locally ringed space $((S_{w,\omega})_k\,,\,\pi_{*}\cO_{D(\uw)_k})$ is a scheme by \cite[Lem.~141]{Mat2}.
It is independent of $\omega$  by \cite[Lem.~141, \S XVIII.2]{Mat2}.
We abbreviate $(S_w)_k=(S_{w,\omega})_k$ and call it the Schubert cell.
It is a projective $k$-scheme.
For $u\leq w$ in $W$, the closed embedding $D(\uu)_k\hra D(\uw)_k$ yields a 
closed immersion $(S_u)_k\hra\ (S_w)_k$ by \cite[Lem.1]{Mat}.
By \cite[p.~45]{Mat}, we have an isomorphism of functors
\begin{equation}
\hG_k/\hB_k\cong\colim_{w\in W} (S_w)_k.
\label{Abflag}
\end{equation}
If the ring $k$ is normal and integral, then  $(S_w)_k$ is also normal and integral by \cite[p.~58]{Mat}.
The next proposition proves that the constructions commute with base changes.
Recall that a morphism of schemes $X\ra Y$ is normal if it is flat with geometrically normal fibers.
We use a more general definition than \cite[Tag.~0390]{Sta}, that requires that the fibers are locally noetherian schemes.

\bprop\label{Bcmut}
Let $k$ be a commutative ring.
\begin{enumerate}[label=$\mathrm{(\alph*)}$,leftmargin=8mm]
\item 
For every $w\in W$, we have an isomorphism
$(S_w)_\bZ\times_{\Spec(\bZ)}\Spec(k)\cong (S_w)_k.$
\item
The morphism $\hG_k/\hB_k\ra\Spec(k)$ commutes with base change.
It is surjective, ind-projective and ind-normal.
\item
We have an isomorphism
$\hG_{\bZ}\times_{\Spec(\bZ)}\Spec(k)=\hG_k.$
The morphism $\hG_{k}\ra\Spec(k)$ is ind-normal and surjective. 
\eenum
\eprop

\bpf
By \cite[Lem.~138]{Mat}, we have that $R^{q}(\pi_w)_*\cO_{D(\uw)_{\bZ}}=0$ for $q>0$.
Hence $(S_w)_\bZ$ is flat over $\bZ$ and commutes with base change. 
Since  $(S_w)_k$ is normal integral, 
the part (b) follows from \eqref{Abflag}.
Since the formation of $\hB_{\bZ}$ commutes with base change, the part (c)
follows from the corresponding statement for the flag variety.
\epf

\brem
The ind-scheme structure on $\hG_{\bC}/\hB_{\bC}$ is the same as in \cite[Def.~7.1.13, 7.1.19]{Kum}, 
because $(S_w)_\bC$ is normal integral.
\erem

\subsubsection{The opposite parabolics of a Kac-Moody group}\label{opp-J}
By Proposition \ref{Bcmut}, we may suppress the subscript $k$ or $\bZ$ and work over $k=\bZ$, as we do from now.
For instance, unless specified otherwise we simply write $U(\kg)$ for the hyperalgebra $U(\kg)_{\bZ}$.
For any subset $J\subset\Delta$ of ft, see \eqref{jlev},
we have the parabolic subgroup $\hP_{J}$ of $\hG$ associated with the Lie algebra $\kp_J$.
We now define the opposite parabolic $P_{J}^{-}$. 
Recall that for an ind-affine ind-scheme $X$ with a  $\bG_m$-stable presentation, one can define as subfunctors the attractor  $X^{+}$, 
the repeller  $X^{-}$, and the fixed point $X^{0}$ loci, see \S\ref{att-rep} below.
Since $X$ is ind-affine, all of them are representable by closed ind-affine ind-schemes, 
see \cite[Thm.~2.1]{HR2} and \cite[Lem.~1.9]{Ric}.
Fix a dominant cocharacter $\la$ of $T$  such that $\alpha \la=0$
if $\alpha\in J$ and $\alpha \la>0$ if $\alpha\in\Delta\setminus J$.
The cocharacter $\la$ acts by conjugacy on $\hG$.
Let $P_\la^-$ be the repeller locus and $\hP_\la$ the attractor one.
They do not depend on the choice of the cocharacter $\la$.
We have $\hP_\la=\hP_{J}$.
We define $P_{J}^{-}=P_\la^-$.
Let $\hU_\lambda=\hU_J$ be the unipotent radical of $\hat P_J$.
The fixed point locus $L_\la$ is identified with the Levi subgroup $L_{J}$ of $\hP_{J}$, see \cite[Proof of Lemma 4.1]{HLR}. 
We have a morphism of group ind-schemes 
$q^{-}:P_{J}^{-}\ra L_{J}$
given by evaluation at $\infty$.
Let $U_\la^-=U_{J}^{-}$ be the kernel of $q^{-}$. 
We have a decomposition  $P_{J}^{-}=L_{J}\ltimes U_{J}^{-}$.
The following is proved in \cite[Lem.~4.1]{HLR}.

\blem\label{op-cell}\hfill
\begin{enumerate}[label=$\mathrm{(\alph*)}$,leftmargin=8mm]
\item 
The group $U_{J}^{-}$ is representable by a closed ind-affine group ind-scheme over $\bZ$ of ind-ft.
\item
The multiplication gives an open immersion
$U_{J}^{-}\times\hP_{J}\ra\hG$.
\qed
\eenum
\elem

Note that the authors assume in loc.~cit.~that the KM group is symmetrizable but this is not used in their argument.
The proof of loc.~cit. implies also that for every field $k$, 
the group $U_{J}^{-}(k)$ is generated by the root groups $U_{-\al}(k)$ for $\al\in R_{J,\re}^{+}$. 
Hence its base change to $\bC$ coincides with the opposite ind-group $U_{J,\bC}^{-}$ considered in \cite[Thm.~6.2.8, Def.~7.3.6]{Kum}.
Since the map $\hG/\hB\ra\hG/\hP_{J}$ is smooth surjective with geometrically connected fibers, 
by Proposition \ref{Bcmut}, the partial flag varieties $\hG/\hP_{J}$ are surjective and 
ind-normal over $\Spec(\bZ)$. The same statement holds for $U_{J}^{-}$ which is identified with an open in $\hG/\hP_{J}$.
If $J=\emptyset$ we abbreviate 
$U^-=U_{\emptyset}^-.$
Taking the opposite formal group $\hG^{\op}$, we define similarly a group ind-scheme $U$ of ind-ft over $\bZ$, 
such that for any field $k$, the group $U(k)$ is generated by $U_{\al}(k)$ for all positive real roots $\al\in R_{\re}^+$.

\subsubsection{The Kashiwara flag manifold}\label{KFM}
Let $\cD$ be a simply connected root datum.
We now consider the quotient $\bX=\hG/B^-$, usually called the Kashiwara flag manifold or thick flag manifold. 
We also need the partial version $\bX_J=\hG/P_J^{-}$ for a  ft subset $J\subset \Delta$.

\bprop\label{w-cov}\hfill
\begin{enumerate}[label=$\mathrm{(\alph*)}$,leftmargin=8mm]
\item 
We have an open cover 
$\hG=\bigcup_{w\in W} w\hU B^{-}=\bigcup_{w\in W} wU^{-}\hB$.
\item
The maps $\hG\ra\hG/\hB$ and $\hG\ra\hG/B^{-}$ are Zariski locally trivial.
\eenum
\eprop

\bpf
Let $\Omega=U^{-} \hB/\hB$ and $\Omega_W=\bigcup_{w\in W}w\Omega$.
We must prove that 
\begin{equation}
\Omega_W=\hG/\hB.
\label{omeg-eq}
\end{equation}
The left hand side is open, and both sides are ind-ft ind-schemes over $\bZ$.
In particular they are Jacobson, 
and  their formation commutes with arbitrary base change by Proposition \ref{Bcmut}. 
If $x$ is a closed point in the complement in $\hG/\hB$, 
then it is sent by the Jacobson property to a closed point in $s\in\Spec(\bZ)$. 
Since $\Omega_W$ surjects to $\Omega_{W,s}$ (because $U^{-}(K)$ is generated by $U_{-\al_i}(K)$ for any field $K$), if we know the equality \eqref{omeg-eq} over any finite field, 
we would get a contradiction. 
Now over a field, the decomposition is a general fact for refined Tits systems \cite[Proof of Theorem 5.2.3, (16)]{Kum}.
Further, the tuple  $(\hG_k,N_k,\hU_k,U_k^{-},T_k,\Delta)$ is a refined Tits system over any field $k$ by \cite[\S 3.16]{Rou}.
\epf

By \eqref{w-cov}  we have an open cover
\begin{equation}
\hG=\bigcup_{w\in W} w\hU_{J}P_J^{-}.
\label{Pg-k}
\end{equation}
Thus the quotient $\bX_J$ has an open cover by schemes isomorphic to $\hU_{J}$ and the map
$\hG\ra\bX_J$
is Zariski locally trivial.
We now focus on the Borel case. We want to describe $\hG/B^{-}$ as a Proj algebra. 
In view of the description of $U^{-}$ in the next paragraph, we switch from $\hG/B^{-}$ to $\hG^{\op}/B$.

By \cite[p.~58 Cor.~2]{Mat}, for each $\la,\mu\in\La_+$ we have a surjection of $U(\kg)$-modules
\begin{equation}
L(\la)^{\vee}\otimes L(\mu)^{\vee}\ra L(\la+\mu)^{\vee}.
\label{surj-rep}
\end{equation}
It gives a commutative ring structure on the Abelian group 
$$R=\bigoplus_{\omega\in\Lambda_{+}} L(\omega)^{\vee}.$$ 
We set $\bX=\Proj^{\Lambda_{+}}(R).$
 Here the $\Proj$ is relative to the $\Lambda_{+}$-grading.
 By \cite[Def.~1.12]{Kato}, we have
$$\bX=\big(\Spec(R)\setminus\{x\in\Spec(R)\,;\,x\neq0\, \text{on}~L(\omega_i)^{\vee},\forall i\in I\}\big)/T$$
By \eqref{surj-rep} the sum $\bigoplus_{i\in I} L(\omega_i)^{\vee}$ 
generates $\bigoplus_{\omega\in\Lambda_{+}} L(\omega)^{\vee}$.
Hence there is a surjection of $\Lambda_{+}$-graded rings
\[\bigotimes\limits_{i\in I} S(L(\omega_{i})^{\vee})\thra\bigoplus_{\omega\in\Lambda_{+}} L(\omega)^{\vee}\]
where $S(V)$ is the symmetric algebra of an Abelian group $V$.
By \cite[Lem.~3.14, 3.16]{MR}, we get a closed immersion
\begin{equation}
\bX\hra\Proj^{\Lambda_{+}}\Big(\bigotimes\limits_{i\in I} S(L(\omega_{i})^{\vee})\Big)=\prod\limits_{i\in I}\bP(L(\omega_i)^{\wedge}).
\label{X-clo}
\end{equation}
To identify an open cell in $\bX$, we follow \cite[Prop.~ 1.18]{Kato}. 
Let $V(\omega)=U(\kg)\otimes_{U(\kb)}\bZ_{\omega}$ be the Verma module of highest weight $\omega$. 
From \cite[\S 3.1]{Rou}, we have
\begin{equation}
\bZ[\hU^{-}]=U(\kn^{-})^{\vee}=V(0)^{\vee}.
\label{verU}
\end{equation}
The surjection $V(\omega)\thra L(\omega)$ yields an injection
\begin{equation}
\bigoplus\limits_{\omega\in\La_{+}} L(\omega)^{\vee}\subset\bigoplus\limits_{\omega\in\La_{+}}V(\omega)^{\vee}\cong\bigoplus\limits_{\omega\in\La_{+}}\bZ[\hU^{-}]\otimes_{\bZ}\bZ_{-\omega}\subset\bZ[\hB^{-}],
\label{B-incl}
\end{equation}
The left hand side a subring of the right hand side. 
Hence the scheme $\bX$ is integral.
By inverting $v_\omega^\vee$, we obtain an isomorphism of algebras
\begin{equation}
\bigoplus\limits_{\omega\in\La_{+}} L(\omega)^{\vee}[(v_\omega^\vee)^{-1}]\cong \bZ[\hU^{-}]
\label{op-U-cell}
\end{equation}
We deduce that
\begin{equation}
\hU\cong\bX\setminus\{v_i^\vee=0\}_{i\in I}.
\label{U-bir}
\end{equation}

Since $N_{G}(T)$ acts on $\bX$, we deduce that $\bigcup_{w\in W}w\hU$ is open in $\bX$.
Hence
\begin{equation}
\bX=\bigcup_{w\in W}w\hU,
\label{X-wcov}
\end{equation}
because one has equality over $k$-points for any field $k$ by \cite[Thm. 1.23]{Kato}.
On the other hand, we have a  map
\begin{equation}
\phi:\hG^{\op}\ra\prod\limits_{i\in I}\bP(L(\omega_i)^{\wedge})
\label{phi-G}
\end{equation}
that sends $g$ to the tuple $(g\cdot v_i)_{i\in I}$. 

\blem\label{stab}
The group $B$ is the stabilizer of the lines $(\bZ v_i)_{i\in I}$. 
The map $\phi$ yields a monomorphism
\[\hG^{\op}/B\ra\prod\limits_{i\in I}\bP(L(\omega_i)^{\wedge}).\]
\elem

\bpf
Let $H$ be the stabilizer of the lines $(\bZ v_i)_{i\in I}$. 
Since $B$ is closed in $\hG^{\op}$, it is also closed in $H$. 
Recall that $\hU^{-}B$ is open in $\hG^{\op}$.
Further $B=H\cap(\hU^{-}B)$, because $H\cap\hU^{-}=\{1\}$.
Indeed, this is obvious on $k$-points and the Lie algebras match. 
Thus $B$ is also open in $H$. 
Finally, by the Bruhat decomposition of $\hG^{\op}$ the $k$-points are the same and
 we deduce that  $B=H$.
\epf

\bthm\label{fs-Kash}
\hfill
\begin{enumerate}[label=$\mathrm{(\alph*)}$,leftmargin=8mm]
\item
The map $\hG^{\op}/B\ra\prod\limits_{i\in I}\bP(L(\omega_i)^{\wedge})$ factors through $\bX$.
It yields an isomorphism $\hG^{\op}/B\cong\bX$.
\remi
Let $J$ be  a ft subset of $\Delta$.
The scheme $\bX_J=\hG^{\op}/P_J$ is formally smooth and separated over $\bZ$.
\eenum
\ethm

\bpf
For (a), recall that by \eqref{U-bir}, \eqref{Pg-k} and \eqref{U-bir}, the group $\hU^{-}$ sits as a dense open in 
$\bX$ and $\hG^{\op}/B$ and there is a commutative diagram
$$\xymatrix{\hG^{\op}/B\ar[r]^-{\phi}&\prod\limits_{i\in I}\bP(L(\omega_i)^{\wedge})\\
\hU^{-}\ar[u]\ar[r]&\bX\ar[u]}$$
Now $\hU^{-}$ is schematically dense in $\hG^{\op}/B$ and is reduced.
By \eqref{X-clo}, the schematic image of $\hG^{\op}/B$ factors through $\bX$.
Further, the diagram is $W$-equivariant.
Combining \eqref{Pg-k} and \eqref{X-wcov}, we deduce that it is an isomorphism.
Now, we prove (b). 
Since $\hU_{J}$ is formally smooth, because formal smoothness is Zariski local by \cite[Tag.~0D0F]{Sta} and
 \eqref{X-wcov}, we deduce that $\bX_J$ is formally smooth. 
For separatedness, the case $P=B$ follows from (a) and \eqref{X-clo}.
The general case follows by \cite[Tag.~09MQ]{Sta} because 
the map $\hG^{\op}/B\ra \hG^{\op}/P_J$ is projective and surjective.
\epf

We now move to  $\bC$. We will need it for Proposition \ref{Zfond2}.
By \cite[Thm.~1.23]{Kato}, the quotient $\bX_{\bC}$ coincides with the one in \cite[\S4]{Kash}. 
For each $w\in W$, we consider the $\hB_{\bC}^{-}$-orbit $\bO^{w}=\hB_{\bC}^{-}\cdot w$ in $\bX_{\bC}$. 
By \cite[Lem.~4.5.7, Cor.~4.5.8]{Kash} this orbit is affine locally closed of codimension $\ell(w)$.
Here, the codimension is defined in terms of tangent spaces as both schemes are formally smooth.
We obtain a stratification $\bX_{\bC}=\bigsqcup_{w\in W}\bO^{w}.$
By \cite[Prop.~4.5.11]{Kash}, the closure of $\bO^{w}$ in $\bX_{\bC}$ is
$\overline{\bO^{w}}=\bigsqcup_{w\leq v} \bO^{v}.$

\subsection{The structure of the ind-group $U$}

\subsubsection{The comparison of $U$ and $\hU$}\label{UU}
We first recall some facts on $\hG/\hB$. 
Recall that $\hG$ is simply connected and that $\hG/\hB$ is a colimit of Schubert varieties $S_{w}$, see \eqref{Abflag}. 
On each $S_{w}$ there is a $\hG$-equivariant line bundle $\cO_{S_{w}}(\omega)$ 
for $\omega\in \La_+$. By \cite[p.~253 Prop.~24]{Mat2}, we have
\[H^0(S_{w},\cO_{S_{w}}(\omega))=L_{w}(\omega)^{\vee}=(U(\kn)\cdot w\cdot v_{\omega})^{\vee},\quad
 H^{>0}(S_{w},\cO_{S_{w}}(\omega))=\{0\}.\]
 Note that $\cO_{S_{w}}(\omega)$ corresponds to $\tilde{\mathcal{L}}_{w}(-\omega)$ in loc.~cit.
Set 
\begin{align}\label{Rw}R_w=\bigoplus_{\omega\in \La_+}L_{w}(\omega)^{\vee}.\end{align}
The product  is given
 by the surjections deduced from \eqref{surj-rep} as in \cite[Cor.~1.14]{Kato}
\[L_{w}(\omega)^{\vee}\otimes L_{w}(\gamma)^{\vee}\ra L_{w}(\omega+\gamma)^{\vee},\quad \omega,\gamma\in \La_{+}\]
If $\omega\in\La_{++}$ there is a closed immersion $S_{w}\subset\bP(L_{w}(\omega))$ by \cite[p.~58 Cor.~1]{Mat}.
Thus, as in \cite[Cor.~1.16]{Kato}, we have
$S_{w}=\Proj^{\La_{+}}(R_w).$
We consider the topological ring
\[\widehat{R}=\lim_{w\in W}R_{w}.\]
We  have $\hG/\hB=\TProj(\widehat{R})$, where we define the topological Proj to be 
\begin{equation}
\TProj(\widehat R)=\colim_{w\in W}(\Proj(R_{w})).
\label{t-gb}
\end{equation}
Since $\lim_w L_{w}(\omega)^{\vee}=L(\omega)^{*}$, it is convenient to write
\[\widehat{R}=\widehat{\bigoplus\limits_{\omega\in\La_{+}}}L(\omega)^{*}.\]
Recall the groups $U$ and $U^-$ introduced in \S\ref{opp-J}.
We want to identify $U^{-}$ inside $\hG/\hB$ and compare it with $\hU^{-}$. 
There is an injection of rings
\[\bigoplus_{\omega\in\La^{+}}L(\omega)^{\vee}\ra \widehat{\bigoplus\limits_{\omega\in\La_{+}}}L(\omega)^{*}\]
such that for every $w\in W$, the composite with the projection
\[\bigoplus_{\omega\in\La^{+}}L(\omega)^{\vee}\ra
\widehat{\bigoplus\limits_{\omega\in\La_{+}}}L(\omega)^{*}\ra
\bigoplus\limits_{\omega\in \La_+}L_{w}(\omega)^{\vee}\]
is a surjection of graded rings.  
Using Theorem \ref{fs-Kash} and \cite[Lem. 3.14]{MR}, we get an ind-closed morphism
\begin{equation}
\hG/\hB\ra\hG^{\op}/B.
\label{thin-plong}
\end{equation}
and by considering the algebras, a Cartesian diagram whose horizontal maps are closed immersions
\begin{equation}
\begin{split}
\xymatrix{\hG/\hB\ar[d]\ar[r]&\prod\limits_{i\in I}\bP(L(\omega_i))\ar[d]\\\hG^{\op}/B\ar[r]&\prod\limits_{i\in I}\bP(L(\omega_i)^{\wedge})}
\label{cartB}
\end{split}
\end{equation}
In \eqref{minZ} we will prove that $G/B\cong\hG/\hB$, for $G$ the minimal group over $\bZ$.
Then the morphism \eqref{thin-plong} will be obtained through an embedding of $G$ in $\hG^{\op}$.

\bprop\label{fonct}
\hfill
\begin{enumerate}[label=$\mathrm{(\alph*)}$,leftmargin=8mm]
\remi
We have an isomorphism
$U^{-}\cong \hU^{-}\times_{\hG^{\op}/B}\hG/\hB.$

\item
The morphism $U^{-}\ra\hU^{-}$ is ind-closed and is a morphism of group ind-schemes.
\remi
We have $\Lie(U^{-}_k)\cong\Lie(U^{-})\otimes_{\bZ}k\cong\kn_k^-$ as $k$-Lie algebras for any ring $k$. 
In particular, the $\bZ$-module $e^{*}\Omega^{1}_{U^{-}/\bZ}$ is dual to a projective $\bZ$-module.
\eenum
\eprop

\bpf
Let first prove (a) and (b).
 By \eqref{op-U-cell}, we have 
 $$\bZ[\hU^{-}]=\bigoplus_{\omega\in\La_+} L(\omega)^{\vee}[(v_\omega^\vee)^{-1}].$$
Thus, by Theorem \ref{fs-Kash} and \eqref{t-gb}, it is enough to check that
\[\bZ[U^{-}]=\widehat{\bigoplus\limits_{\omega\in\La_{+}}}L(\omega)^{*}[(v_\omega^\vee)^{-1}].\]
This is equivalent to proving that 
\begin{equation}
U^{-}=(\hG/\hB)\setminus \bigcap_{\la\in\La_{+}}\{v_\omega^\vee=0\}.
\label{uminus}
\end{equation}
One inclusion is clear. Equality of opens can be checked over $k$-points for any field $k$ and this follows from the beginning of the proof of \cite[Cor.~4.3]{HLR}.
Finally, the morphism is a morphism of group ind-schemes, because the map
\[U^{-}\ra\prod\limits_{\omega\in\La_+}\bP(L(\omega)^{\wedge})\]
given by $g\mapsto (g\cdot v_{\omega})$ is equivariant by left action by $U^{-}$ and it factors through $\hU^{-}$.
Now, we prove (c).
From (a), (b) and \eqref{cartB} and since the formation of this objects commute with base change, 
we deduce that $U_{k}^{-}$ identifies schematically 
with 
$$\{g\in\hU_{k}^{-}\,;\, g\cdot L(\omega)_{k}\subset L(\omega)_{k}\,,\,\forall\omega\in\La_{+}\}.$$
Since $\Lie(U_{k}^{-})$ is a colimit of finite dimensional $T$-stable spaces, it is contained in 
the set of $T$-finite vectors of $\Lie(\hU_{k}^{-})=\hkn_{k}^-$ by \eqref{exp-comp}, and the later is
$\kn_{k}^-$ .
For the equality, note that $(\kg_\al)_k$ stabilizes $L(\omega)_{k}$ for any $\omega\in\La$ and
any $\al\in R^{-}$. 
The last claim follows from Proposition \ref{Z-tan}.
\epf
For every $w\in W$ we define the group schemes
$$\hU_{w}^{-}=\hU^{-}\cap w (\hU^{-})
,\quad
\hU^{-,w}=\hU^{-}\cap w(\hU)$$ 
and the group ind-schemes 
$$U_{w}^{-}=U^-\cap\hU_{w}^{-}
,\quad
U^{-,w}=U^-\cap\hU^{-,w}.$$

\bprop\label{uw-comp}
\hfill
\begin{enumerate}[label=$\mathrm{(\alph*)}$,leftmargin=8mm]
\remi
We have  $U^{-,w}=\hU^{-,w}$.
It is a smooth affine group scheme over $\bZ$.
\remi
The multiplication yields an isomorphism  
$U_{w}^{-}\times U^{-,w}\cong U^{-}$
of ind-schemes over $\bZ$.
\eenum
\eprop

\bpf
Using the decomposition \eqref{exp-comp} we deduce that the multiplication gives an isomorphism of schemes
\begin{equation}
\hU_{w}^{-}\times \hU^{-,w}\cong \hU^{-}.
\label{m-form}
\end{equation}
Similarly, we get an isomorphism
\[\prod\limits_{\al\in R^{-}\cap w(R^{+})} U_{\al}\cong\hU^{-,w},\]
where $R^{-}\cap wR^{+}$ is a finite set of real roots. 
Consequently, we have $U^{-,w}\cong\hU^{-,w}$ and $U^{-,w}$ is a smooth affine group scheme over $\bZ$.
Intersecting \eqref{m-form} with $U^{-}$ yields the isomorphism 
$U_{w}^{-}\times U^{-,w}\cong U^{-}$.
\epf

For each $n\in\NN$ the set $\Psi_n=\{\al\in R^{-}\,;\,\hgt(\al)\leq -n\}\cup\{R^{-}_{\textrm{im}}\}$
is a closed set of negative roots, because $R^{-}_{\textrm{im}}$  is closed in $R^-$ by 
\cite[Prop.~5.2, Ex.~5.16]{Kac}, see also \cite[Lem.~3.6]{Mar2} for details.
Consequently,  by \cite[\S3.1]{Rou}, one can define the pro-unipotent group 
$\hU^{-}_{(n)}=\hU^{-}_{\Psi_{n}}$.
The group of $R$-points of $\hU^{-}_{(n)}$ over a ring $R$ is the subgroup of $\hU^{-}(R)$ formed by products 
$\prod_{x}[\exp](\la_{x}x)$ where $\la_x\in R$ and
$x$ runs over a basis of $\bigoplus_{\al\in\Psi_n}\kg_{\al}$.
We set $U_{(n)}^{-}=U^{-}\cap\hU^{-}_{(n)}$.

\bprop\label{Un}
\hfill
\begin{enumerate}[label=$\mathrm{(\alph*)}$,leftmargin=8mm]
\remi
There is an ind-closed embedding $U^{-}\hra\hU^{-}$ over $\bZ$.
\item
The composed map $U^{-}\ra\hU^{-}\ra \hU^{-,(n)}$ is surjective. It yields an isomorphism of schemes over $\bZ$
\[U^{-}/U^{-}_{(n)}\cong\hU^{-}/\hU^{-}_{(n)}.\]
The map $U^{-}\ra U^{-}/U^{-}_{(n)}$ splits as a morphism of $\bZ$-ind-schemes.
\eenum
\eprop

\bpf
Part (a) is proved in Proposition \ref{fonct}. For (b), it is sufficient to note that the quotient $\hU^{-}\ra \hU^{-,(n)}$ is isomorphic to the 
product of $U_{\al}$'s for all negative real roots $\al$ such that $\hgt(\al)>n$.
Since $U^{-}$ contains all negative real root groups, the composition $U^{-}\ra \hU^{-}/\hU^{-}_{(n)}$ splits.
\epf

We define
$\hU^{-,(n)}=\hU^{-}/\hU^{-}_{(n)}$.
We will also use a normal subgroup without splitting.
For $n\in\NN$, let $\hV_{(n)}^-\subset\hU^{-}$ be the normal subgroup associated as above
with the set of negative roots  $\{\al\in R^{-}\,;\,\hgt(\al)\leq-n\}$.
The quotient $\hV^{-,(n)}=\hU^{-}/\hV_{(n)}^-$ is  a smooth unipotent group.
We set $V_{(n)}^{-}=U^{-}\cap \hV_{(n)}^-$.

\blem\label{Un2}
The composed map $U^{-}\ra \hU^{-}\ra \hU^{-}/\hV_{(n)}^-$ yields an isomorphism 
\[[U^{-}/V_{(n)}^{-}]_\fppf\cong \hU^{-}/\hV_{(n)}^{-}.\]
We sheafify the left hand side for the fppf topology.
\elem

\bpf
We first check the surjectivity  on $k$-points for any field $k$.
The proof of \cite[Cor.~7.3.8]{Kum} implies that the group
$\hV_{k}^{-,(n)}$ is generated by the image of the groups $(U_{-\al_i})_k$ for $i\in I$.
To prove the isomorphism above,  
note that the schemes $\hV^{-,(n)}$ and 
$X=\prod\limits_{i\in I} U_{-\al_i}$ are smooth over $\bZ$, and that
the multiplication gives a map
$m:X\ra\hV^{-,(n)}$.
Set 
$X^{1}=X\bigsqcup X$, and, for each $d>1$, define inductively
$$
X^{d}=X^{1}\times X^{d-1}
.$$
Let $m^{1}:X^{1}\ra \hV^{-,(n)}$ be the map given by $m^1=(m,\iota \circ m)$, where 
$\iota$ is the inversion, and set
\[m^{d}= \mu\circ (m^{1}\times m^{d-1}):X^{d}\ra \hV^{-,(n)},\]
where $\mu$ is the multiplication in $\hV^{-,(n)}$.
By the above, the map
\[\bigsqcup\limits_{d\geq 0} X^{d}\ra \hV^{-,(n)}\]
is surjective, beacuse it is surjective on $k$-points for any field $k$.
Thus, by  \cite[Exp.~6B,\,Prop.~7.6]{SGA3}, the sheaf $\hV^{-,(n)}$ is the fppf-sheafification of the presheaf
\[R\ra\left\langle U_{-\al_i}(R)\,;\,i\in I\right\rangle.\]
Consequently, the morphism $U^{-}\ra \hV^{-,(n)}$ is a surjection of fppf sheaves and it induces thus an isomorphism
\[U^{-}/V_{(n)}^{-}\cong\hU^{-}/\hV_{(n)}^-.\]
\epf

\subsubsection{Lie algebras of Kac-Moody groups and formal smoothness}
We want to compare $\Lie(\hG)$ and $\hat{\kg}$.
For all $i\in I$, the minimal parabolic $P_{i}\subset\hG$ yields a morphism $\SL_{2}\ra\hG$ over $\bZ$,
hence a morphism 
$\mathfrak{sl}_{2}\ra\Lie(\hG)$
of Lie algebras over $\bZ$,
and the inclusion $\hB\subset\hG$ yields a morphism  $\hat\kb\ra\Lie(\hG)$
of Lie algebras over $\bZ$.
Further, we have  $[e_{i},f_{j}]=0$ for $i\neq j$. 
Let  $\tilde{\kg}_{\bQ}$ be the quotient of the free Lie algebra over $\bQ$ by the relations \eqref{kac1}, completed in the $\kn$ direction.
By the above we have a morphism  $\kappa:\tilde{\kg}_{\bQ}\ra\Lie(\hG_{\bQ})$
of Lie algebras over $\bQ$.
We claim that this morphism factors through a morphism 
\begin{align}\label{kappa}\kappa:\hkg_{\bQ}\ra\Lie(\hG_{\bQ}).\end{align}
We must prove the relations \eqref{kac2}.
The first relation holds in $\hkb=\Lie(\hB)$. 
Since $\Lie(\hG_{\bQ})$ is an integrable $\mathfrak{sl}_{2,j}$-module $e_j$, $f_j$ 
act locally nilpotently, hence \cite[Lem.~1.3.9]{Kum} implies that the second relation in \eqref{kac2} also holds.

\bthm\label{cart2}
The map $\kappa$  restricts to an isomorphism of $\bZ$-Lie algebras $\hkg\to\Lie(\hG).$
\ethm

\bpf
The map $\kappa$ maps $\hkb$ isomorphically onto $\Lie(\hB)$ as $\bZ$-Lie algebras. 
In particular, it is a Lie algebra homomorphism which identifies the Cartan algebras. 
For weight reasons, it maps
 $\kn_{\bQ}^{-}$ into $\Lie(U^{-}_{\bQ})$.
By Proposition \ref{fonct} the embedding 
$$\theta:\Lie(U^{-})\ra\Lie(\hU^{-})=\hkn^-$$ factors through a $\bZ$-Lie 
algebra isomorphism $\Lie(U^{-})\cong\kn^{-}$.  
Hence, the composed map $\theta\circ \kappa$ restricts to a  $\bQ$-Lie algebra homomorphism 
$\kn_{\bQ}^{-}\to\kn_{\bQ}^{-}$ such that $f_i\mapsto f_i$. 
Thus it restricts to the identity of $\kn^{-}_{\bQ}$, and a fortiori of $\kn^-$ and $\kappa$ restricts 
to a $\bZ$-Lie algebra isomorphism $\kn\to\Lie(U^{-})$, from which 
we deduce that $\kappa$ is a Lie algebra isomorphism $\hkg\cong\Lie(\hG)$.
\epf

\bthm\label{fsm}
Let $k$ be a field of characteristic zero. Then $\hG_k$ and $\hG_k/\hB_k$ are formally smooth over $k$. 
If $G$ is affine over any algebraically closed field $k$, then $G$ is formally smooth.
\ethm

\bpf
The map $\hG_k\ra\hG_k/\hB_k$ is Zariski locally trivial and formally smooth as $\hB_k$ is.
Thus it is enough to prove that $\hG_k/\hB_k$ is formally smooth. 
By Lemma \ref{op-cell}, it is covered by translates of $U_k^{-}$, which is formally smooth by Theorem \ref{cart}.
For the  affine case, in the formal case, from the description of \cite[App. 2]{Tits2} we have that $\hG=\bG_{\hat K}$ for  a quasi-split 
reductive group scheme $\bG$ over $k\llp t\rrp$. From that, we get that $\hG/\hB=G/B$ is formally smooth and thus $U^{-}\subset\hG/\hB$ 
also. By considering $\hG^{\op}$, we have the corresponding assertion for $U$ and we get that $G$ is also formally smooth, as it is 
covered by translates of the open cell.
\epf

\brem
We expect that the same statement should hold over $\bZ$, or already over any field $k$. 
This would imply immediately all the geometric results of \S6.3.
\erem

\subsection{The Minimal Kac-Moody group}
In this section we give a construction of minimal KM groups, as group ind-schemes of ind-ft over $\bZ$. 
Over $\bZ$, the group functor $E_{\cD}$ introduced by Tits \cite[\S 3.6]{Tits} is known to be the wrong functor because,
already in the ft case, it does not recover the Chevalley group schemes, see \cite[\S1.2]{Tits2}. 
Only the points of $E_{\cD}$ over a field $k$ are well-behaved.
Over $\bC$, there is a construction due to Kumar \cite[\S 7]{Kum}.
Our construction is independent and gives the same answer over $\bC$. 
Let $\cD$ be a semisimple simply connected KM root datum.

\subsubsection{The definition of the minimal Kac-Moody group}
Let 
\begin{align}\label{Om}\Omega= U^-\times T\times U\subset\hOm=U^-\times T\times \hU.\end{align}
By Proposition \ref{fonct} applied to $U$, the subfunctor $\Omega$ is ind-closed in $\hOm$ and 
is representable by an ind-scheme of ind-ft over $\bZ$. Let $B=U\rtimes T$.
Let $G$ the subfunctor of $\hG$ defined as the Zariski sheafification of the presheaf
\begin{equation}
R\mapsto \bigcup_{w\in W}w\Omega(R).
\label{w-minZ}
\end{equation}
More generally, for an arbitrary commutative ring $k$, we can form the same way the subfunctor $G_{k}\subset\hG_{k}$, by considering translates of $U_{k}\times T_{k}\times U^{-}_{k}$ and sheafifying for Zariski topology.
\bprop\label{minZ}
\hfill
\begin{enumerate}[label=$\mathrm{(\alph*)}$,leftmargin=8mm]
\item
We have a canonical isomorphism $G\times_{\Spec(\bZ)} \Spec(k)\cong G_k$.
\item 
$\hOm\cap w\Omega\subset\Om$ as ind-schemes for each $w\in W$.
\remi
The map $G\ra \hG$ is ind-closed. The functor $G$ is representable by an ind-affine ind-scheme over $\bZ$.
\item
 $\hB\cap G=B$ and  $\hU\cap G=U$.
 \item
 The composed map 
$G\ra\hG\ra\hG/\hB$
factors through an isomorphism  $G/B\cong \hG/\hB.$
\item
The morphism $G\ra G/B$ is Zariski locally trivial.
The ind-scheme $G$ is ind-normal of ind-ft.
\eenum
\eprop

\bpf
(a) We have a canonical map $G\times_{\Spec(\bZ)} \Spec(k)\ra G_k$, that already exists at the level of presheaves. It is sufficient to 
prove the isomorphism Zariski locally, thus we are reduced to the claim for the open cell where it follows from the fact that attractors 
commute with base change and Proposition \ref{Bcmut}(b).

For (b), let $x\in (\hOm\cap w\Omega)(R)$ for a ring $R$. We have
\begin{equation}
v t \hu=wv_{1}t_{1}u_{1}
,\quad
 \hu\in\hU(R)
 ,\quad
 u_{1}\in U(R)
,\quad
 t,t_1\in T(R)
 ,\quad
 v,v_1\in U^{-}(R).
\label{dhO}
\end{equation}
Thus $\hu=wv't'u'$. 
Now we make this element act on the highest vector $v_{\omega}$ in $L(\omega)$ for various $\omega$. 
Since $\hU\cdot v_{\omega}=v_{\omega}$ by Lemma \ref{stab}, we get $w=t'=v'=1$ and $\hu=u'$.
For (c), to prove that $G\ra\hG$ is ind-closed, using Proposition \ref{w-cov}, 
it is enough to prove that the base change to $w\hOm$ is ind-closed and the claim follows from (b) and Proposition \ref{fonct}.
For (d), note that $\hB\cap G=B$ follows from (a). 
Part (e) is clear as $\hG/\hB$ is covered by $w$-translates of $U^{-}$.
Finally we prove (f). By (e), the functor $G/B$ is ind-projective over $\bZ$ and by (a), the map $G\ra G/B$ is Zariski locally trivial. 
Thus $G$ is of ind-ft as $B$ is.  By (e), Proposition \ref{Bcmut} and Lemma \ref{op-cell}, 
the functors $B$ and $G/B$ are ind-normal. Thus $G$ is also ind-normal. 
\epf

\bprop\label{g-ind}
The functor $G$ is an ind-fp-affine group ind-scheme.
The obvious morphism $G\ra\hG$ is a morphism of group ind-schemes over $\bZ$ and $\Lie(G)=\kg$.
\eprop

\bpf
Consider the composed map
$G\times G\ra\hG\times\hG\to\hG$
given by the multiplication in $\hG$.
We must prove that this map factors schematically through $G$.
Since $G$ is ind-closed in $\hG$ and the source is geometrically reduced by Proposition \ref{minZ} and \cite[Tag.~06DG]{Sta}, 
it is enough to prove that it factors on $k$-points for any field $k$ by  \cite[Tag.~0356]{Sta}.
Let $G_1(k)\subset\hG(k)$ be the subgroup generated by $T(k)$ by $U_{\al}(k)$ for all $\al\in R_{\re}$.
Let $U_1(k)$ and $U_1^-(k)$ be
the subgroups generated by $U_{\al}(k)$ for positive real roots and negative real roots respectively. 
By \S\ref{opp-J} and \eqref{w-minZ}, we have 
$$U(k)=U_1(k)
,\quad
U^{-}(k)=U_1^-(k)
,\quad
G(k)\subset G_1(k).$$
By \cite[Prop.~3.13]{Rou}, the group $G_1(k)$ is the Tits minimal group. 
By \cite[\S 1.5.4,\,\S 8.4.1]{Remy} the tuple 
$$(G_1(k)\,,\, N(k)\,,\,U_1(k)\,,\, U_1^-(k)\,,\, T(k)\,,\, S)$$ 
where $S\subset W$ is the set of simple reflexions is a refined Tits system.
Thus by \cite[Thm.~5.2.3]{Kum} we  have 
\begin{equation}
G_1(k)\subset\bigcup_{n\in N(k)}nU^{-}(k)U(k)=G(k).
\label{k-mar}
\end{equation}
The case of the inversion is analog and easier since we already know that $G(k)$ is a group.
The description of the Lie algebra follows from the description of the open cell in Proposition \ref{fonct} and Theorem \ref{cart2}.
\epf

\subsubsection{The set of \texorpdfstring{$\hat{\mathcal{O}}$}{\hat{O}}-points of the minimal KM group}\label{Elem}
Instead of an integral model of a minimal KM group, the references \cite{BKP} or \cite{GR2} use an abstract 
group $G_x$ which is defined by some generators.
We must compare the group $G_x$ with $G(\widehat{\cO})$, for $G$ the minimal group ind-scheme constructed above. 
Here it is important to 
work with $\wco$ rather than $\co$. The corresponding assertion for $\co$ is probably wrong.
By  \cite[\S3.4]{GR2}, we have
\[G_x=(\hU(\wco)\cap G(\wK))U^{-}_{0}N(\wco),\]
with $N$ the normalizer of $T$ and $U^{-}_{0}$ the group generated by $U_{\al}(\wco)$ for $\al\in R_{\re}^{-}$.
The following proposition is an illustration that our minimal group $G$ behaves well.

\bprop\label{int-comp}
We have $G_x=\langle U_{\pm\al_i}(\wco), T(\wco)\,;\,i\in I\rangle=G(\hat\cO)$.
\eprop

\bpf
By \cite[Prop.~3.1]{Heb}, we have $G_x=\langle U_{\pm\al_i}(\wco), T(\wco)\,;\,i\in I\rangle$ and  $G_x\subset G(\wco)$.
Conversely, let $g\in G(\wco)$. Replacing $g$ by $hg$ with $h\in G(k)$ a product by of elements of the middle group by \eqref{k-mar}, 
we can assume that $g=1$ modulo $s$. As $\wco$ is a local ring, we have $g\in(UTU^{-})(\wco)$.
Since $U(\wco)\subset\hU(\wco)\cap G(\hat K)$, we deduce that $G(\wco)\subset G_x$.
\epf

\subsection{Basic affine spaces of Kac-Moody groups}
\subsubsection{The positive basic affine space}
We assume that $G$ is simply connected.
Let the Schubert variety $S_w$ be as in \S\ref{FM}.
We define as in \eqref{X-clo} the closed embeddings of ind-ft-schemes $\bZ$:
\begin{equation}
S_{w}\hra G/B\hra\prod\limits_{i\in I}\bP(L(\omega_i)).
\label{gb-plong}
\end{equation}
There is a $T$-torsor over the right hand side given by
\begin{equation}
\prod\limits_{i\in I}( L(\omega_{i})\setminus\{0\})\ra\prod\limits_{i\in I}\bP(L(\omega_i)).
\label{T-tors}
\end{equation}
We consider the map 
$$G\ra\prod\limits_{i\in I} L(\omega_{i}),\quad
g\mapsto (g\cdot v_i)_{i\in I},$$
It is $U$-equivariant and factors  through the left hand side of \eqref{T-tors}.
Since both functors are $T$-torsors over $G/B$, we get the following isomorphism of $\bZ$-ind-schemes of ind-ft
\begin{equation}
G/U\cong G/B\times_{\prod\limits_{i\in I}\bP(L(\omega_i))}\prod\limits_{i\in I} (L(\omega_{i})\setminus\{0\}).
\label{G/U}
\end{equation}

\bprop\label{U-rep}
\hfill
\begin{enumerate}[label=$\mathrm{(\alph*)}$,leftmargin=8mm]
\item
The obvious map $G/U\ra\hG/\hU$ is an isomorphism.
\item
The quotients $G/U$, $G/U^{-}$ and $\hG/\hU$ are representable by quasi-compact locally closed in ind-affine ind-ft ind-schemes
 over $\bZ$.
\item
If $G$ is untwisted affine, then $G/U$ is a presheaf quotient: for any ring $R$ we have $$G(R)/U(R)=(G/U)(R).$$
\eenum
\eprop

\bpf
To prove (a), note that the map $G/U\ra\hG/\hU$ is an isomorphism, 
because both sides are $T$-torsors over $G/B\cong\hG/\hB$ by Proposition \ref{minZ}.
Now, by (a), it is enough to prove (b) for $G/U$, the other case is similar. 
Then, the assertion follows from \eqref{G/U}.
Finally, we prove (c). Let $R$ be a ring, and $x\in (G/U)(R)$.
Taking again the notations of Corollary \ref{ces-2}, let $\bar{x}$ be the image of $x\in (G/P)(R)$, 
using Theorem \ref{ces-spl}, we find $y\in G(R)$ such that 
\[y^{-1}x\in (P/U)(R)=(\bG_{f}/U_{f})(R),\]
and the latter lifts as $U_{f}$ is split unipotent.
\epf

\brem
\hfill
\begin{enumerate}[label=$\mathrm{(\alph*)}$,leftmargin=8mm]
\item
It would be useful to know whether (c) holds beyond the untwisted affine case.
\item
Let $I_w$ be the ideal of the ring $R_w$ in \eqref{Rw} given by $I_w=\bigoplus\limits_{\omega\in\La_{++}}L_{w}(\omega)^{\vee}$.
Following \cite[(4.6)]{MR}, we expect that
\[G/U\times_{G/B}S_{w}\cong \Spec(R_w)\setminus V(I_w).\]
This should be useful if one to consider the basic affine space $\overline{G/U}$ in the Kac-Moody case.
\eenum
\erem

We need the variant for the Kashiwara flag scheme.

\bprop\label{U-quot}
Assume that $G$ is untwisted affine. The quotient $\hG/U^{-}$ is a presheaf quotient,  i.e., for any algebra $R$ we have 
$$\hG(R)/U^{-}(R)=(\hG/U^{-})(R).$$
\eprop

\bpf
The proof is the same proof as for Proposition \ref{U-rep}(c), replacing $P$ by $P^{-}$ and using Proposition \ref{ces-3}.
\epf



\section{Affine Grassmannians of Kac-Moody groups}
\subsection{Relative representability}
\subsubsection{Definition of the affine Grassmannian}\label{AGR}
Let $G$ be an ind-affine group ind-scheme over the commutative ring $k$.
Set $S=\Spec(k)$.
The affine Grasmmannian of $G$ is the $S$-étale sheaf
\begin{align}\label{Gr}\Gr_{G}=[G_K/G_\cO].\end{align}
For each $k$-algebra $R$ the groupoid $\Gr_{G}(R)$ consists of étale $G$-torsors over $\ab^{1}_{R}$, 
with a trivialization over $\bG_{m,R}$, that are étale locally trivial over $\Spec(R)$.
The sheaf $\Gr_G$ is not representable in general.
We also need the Laurent series version $\Gr_{G}^\form$
where we replace $\ab^{1}_{R}$ by $\Spec(R\llb t\rrb)$, and $\bG_{m,R}$ by $\Spec(R\llp t\rrp)$, i.e., we set
\[\Gr_{G}^\form=[G_{\hat K}/G_{\widehat{O}}].\]

\bprop\label{ind-quot}
Let $G$ be an ind-affine group ind-scheme over $k$.
Let $H\subset G$  be a closed subgroup. 
\begin{enumerate}[label=$\mathrm{(\alph*)}$,leftmargin=8mm]
\remi
If $G/H$ is either ind-quasi-affine (resp.~ind-strongly quasi-affine) or
open $G$-equivariantly in an ind-affine ind-scheme $Z$,
then the map $\Gr_{H}\ra\Gr_{G}$ is locally closed (resp.~fp locally closed).
\remi
If $G/H$ is ind-affine or ind-fp-affine, then this map is  closed or fp closed respectively.
\eenum
Moreover, we have the same conclusion with $\Gr_G$ replaced with $\Gr_{G}^\form$.
\eprop
\brem
Recall that a $k$-scheme is quasi-affine (resp.~strongly quasi-affine) if it is quasi-compact open of an affine scheme (resp.~finitely presented affine scheme). 
An open $V$ in an affine scheme is not necessarily quasi-compact, but any quasi-compact open of $V$ is quasi-affine.
\erem

\bpf
This lemma is well-known for algebraic groups \cite[Prop.~1.2.6]{Zhu}. We start with the  case  of Laurent polynomials.
Let $(\cE,\beta)$ be a $k$-point of $\Gr_G$ represented by a morphism $S\ra\Gr_{G}$. 
We must prove that the morphism $\cF\to S$ with
\[\cF=S\times_{\Gr_{G}}\Gr_{H}\]
is locally closed. 
Let $\pi:\cE\ra\ab^{1}_{k}$ be the structural morphism.
The trivialization $\beta$ is given by a section of $\pi$ over $\bG_{m,k}$.
Consider the étale quotient 
$\bar{\pi}:[\cE/H]\ra\ab^{1}_{k}.$
Note that $[\cE/H]$ is étale locally isomorphic to $G/H$.
If $G/H$ is ind-quasi-affine (resp.~ind-strongly quasi-affine, ind-affine, ind-fp-affine), by \cite[Lem.~3.12]{HR}, effectivity of descent for quasi-affine schemes (resp.~affine schemes) and descent of finite presentation(\cite[Tag. 0245, 0247, 041V]{Sta}), $\cE_{Z}$ is represented by an ind-quasi-affine scheme (resp.~ind-strongly quasi-affine, ind-affine, ind-fp-affine).

If $G/H$ is open $G$-equivariantly in an ind-affine ind-scheme $Z$, we form the twisted quotient $\cE_{Z}=\cE\times^{G}Z$, which is ind-affine. Consequently $[\cE/H]$ is open in $\cE_{Z}$ by descent (\cite[Tag. 041V]{Sta}).

The section $\beta$ yields a section 
$\bar\beta$
of $\bar{\pi}$ over $\bG_{m,k}$.
An $H$-reduction of $\cE$ is the same as a section of $\bar{\pi}$ over $\ab^{1}_{k}$. 
Consider the presheaf $\cF$ over $S$ that assigns to each $k\ra R$ the set of sections 
$\beta'$
of $\bar{\pi}$ over $\ab^{1}_{R}$ such that $\beta'|_{\bG_{m,R}}=\bar{\beta}|_{\bG_{m,R}}$.

Fix a filtered presentation $[\cE/H]\cong\colim V_{a}$ with $V_{a}$ open in an affine (resp.~quasi-affine, resp.~strongly quasi-affine, resp.~affine, resp.~fp affine) over $S$. 
The section $\bar{\beta}$ factors through some $V_{a}$ and further through some qc subset $V_{aa'}$ of $V_a$ 
 which is quasi-affine by the remark above. In (b) we have $V_a=V_{aa'}$.
We consider the subpresheaf $\cF_a\subset\cF$ 
that consists of sections $\beta'\in\cF(R)$ such that $\beta'$ factors through $V_{aa'}$ for a $k$-algebra $R$.
We claim that $\cF_a=\cF.$

To prove this, let $\beta'\in\cF(R)$ for a $k$-algebra $R$.
It is a section of $\bar\pi$.
We must prove that it factors through $V_{a}$. 
Over $\bG_{m,R}$, we have $\beta'_{R}=\bar{\beta}_{R}$.
Since $\bG_{m,R}$ is schematically dense in $\ab^{1}_{R}$ and $[\cE/H]$ is ind-separated, 
we have $\beta'=\bar{\beta}$. The claim follows.
To conclude, we must prove that $\cF_a$ is locally closed in $S$.
This follows from the
following lemma.

\blem\label{sect-sheaf}
Let $p:V\ra \ab^{1}_k$ be a quasi-affine (resp.~strongly quasi-affine, resp.~affine,  fp-affine) scheme, 
with $\sigma$ a section of $p$ over $\bG_{m,k}$.
Then the presheaf over $S$ that assigns to each morphism $k\ra R$ the set of sections $\sigma'$ of $p_R$ such that 
$\sigma'\vert_{\bG_{m,R}}=\sigma\vert_{\bG_{m,R}}$ is represented by a locally closed 
(resp.~fp locally closed, closed, fp closed) subscheme of $S$.
\elem

\bpf
Assume first that $V$ is affine.  
We can suppose that $V\subset\bA_\cO^{(I)}$ is closed, with the set $I$ finite if $V$ is fp. 
We have
$\sigma=(\sigma_i(s))_{i\in I}$ with
$\sigma_i(s)=\sum \sigma_{ij}s^j$ in $K.$
The presheaf considered is then defined by the closed subscheme of $S$ defined by 
$\sigma_{ij}=0$ for all $i\in I$ and $j<0.$
For the quasi-affine (resp.~strongly quasi-affine) case, let $W$ be affine (resp.~fp-affine) such that $V$ is qc open in $W$. The affine case gives a closed subscheme 
$\Spec(R)\subset S$ and a tautological map $\ab^{1}_R\ra W\times_{\bA^1_k}\bA_R^1$, 
whose restriction to $\bG_{m,R}$ is $\sigma\vert_{\bG_{m,R}}$.
A base change along the 0 section
$\Spec(R)\ra\bA_R^1$
gives a section 
$$\tau:\Spec(R)\ra W\times_{\ab^1_{k}}\Spec(R).$$
The presheaf is represented by the quasi-compact open subscheme $\tau^{-1}(V\times_{\ab^{1}_k}\Spec(R))$ of $\Spec(R)$.
\epf
Finally for the Laurent series case, the proof goes through mutatis mutandis, replacing $\ab^1_R$ by $\Spec(R\llb t\rrb)$ and $\bG_{m,R}$ by $\Spec(R\llp t\rrp)$.
\epf
\blem\label{mono-fact}
If an immersion of prestacks  $\cY\hra\cZ'$ factors through a monomorphism of prestacks $\cZ\hra \cZ'$, 
then the map $\cY\hra \cZ$ is an immersion. If the map $\cY\hra\cZ'$ is fp then the map $\cY\hra\cZ$ is also fp.
\elem

\bpf
Let $S\ra\cZ$ with $S=\Spec(k)$.
 We must prove that the map $\cY\times_{\cZ}S\ra S$ is locally closed (resp.~locally closed and fp).
By assumption, the map $\cY\times_{\cZ'}S\ra S$ obtained through $S\ra\cZ\ra \cZ'$ is an immersion 
(resp.~fp immersion).
But as $\cZ\ra\cZ'$ is a monomorphism, we get have $\cY\times_{\cZ'}S\cong\cY\times_{\cZ}S$.
\epf

\bprop\label{Gr-sep}
Let $G$ be an ind-affine group  ind-scheme over $k$.
\begin{enumerate}[label=$\mathrm{(\alph*)}$,leftmargin=8mm]
\remi
The functor $\Gr_{G}$ has a closed diagonal.
\remi
If further $G$ is ind-fp-affine then $\Gr_{G}$ has an fp closed diagonal.
\eenum
\eprop

\bpf
The map 
${\Delta}:G_K\times G_\cO\ra G_K\times G_K$
such that
$(g,h)\ra (g,gh)$ is closed (resp.~fp closed), 
being the composition of the multiplication map and the base change of the inclusion  $G_\cO\hra G_K$ which is closed (resp.~fp closed).
The map $[{\Delta}/(G_\cO)^{2}]$ is the diagonal of $\Gr_{G}$.
It is closed (resp.~fp closed) by Lemma \ref{quotH}. 
\epf

\subsubsection{Modular description of the affine Grassmannian}\label{mod-desc}
Let $G$ be an ind-affine group ind-scheme over a ring $k$.
Let $\Gr'_{G}$ be the functor from $k$-algebras to sets taking $R$ to 
the set of isomorphism classes of $G$-torsors on $\ab^{1}_R$ with a trivialization over $\bG_{m,R}$.
By \cite[Thm.~3.4, 3.6]{Ces},
if $G$ is  a quasi-split reductive group scheme over  $k$, 
then the functor $\Gr'_{G}$ 
is isomorphic to the presheaf quotient $G_{K}/G_{\co}$, i.e.,  
 there are bijections
\begin{equation}
\Gr'_{G}(R)\cong\Gr_{G}(R)\cong G(K_R)/G(\cO_R).
\label{kes}
\end{equation}
In particular, a $G$-torsor over $\ab^{1}_{R}$ trivial over $\bG_{m,R}$ is trivial.
For a KM group $G$ the functors $\Gr'_{G}$ and $\Gr_{G}$ may differ.
Nevertheless, we have the following result in the untwisted affine case.

\bthm\label{Gr-presheaf}
Let $G$ be a untwisted affine minimal KM group of affine type over a ring $k$.
Let $R$ be a $k$-algebra. 
Any étale $G$-torsor over $\ab^{1}_{R}$ trivial over $\bG_{m,R}$ is trivial.
In particular, we have $\Gr_{G}'\cong\Gr_{G}$ and $\Gr_{G}$ is a presheaf quotient.
\ethm

\bpf
As in Definition \ref{aff-def}, let  $G=\widetilde{\bG_{f,K}}\rtimes H$ where $\widetilde{\bG_{f,K}}$ is a central extension of $\bG_{f,K}$ 
by a split $k$-torus $Z$, the group $\bG_{f}$ split reductive over $k$, and $H$ is a split $k$-torus. 
Let $E$ be a $G$-torsor over $\ab^{1}_{R}=\Spec(R[s])$ trivial over $\bG_{m,R}=\Spec(R[s,s^{-1}])$. 
We want to prove that $E$ is trivial.
Considering the map $p:G\ra H$, the torus case,  and the fact that the map $p$ splits, 
we deduce that the torsor $E$ comes from a 
$\widetilde{\bG_{f,K}}$-torsor over $\ab^{1}_{R}$ which is trivial over $\bG_{m,R}$.
Pushing $E$ along the central extension
$\widetilde{\bG_{f,K}}\to \bG_{f,K}$, we get a $\bG_{f,K}$-torsor over $\ab^{1}_{R}$ which is trivial over $\bG_{m,R}$.
This defines an element in $\Gr_{\bG_{f}}(R[t,t^{-1}])$.
Thus, using \eqref{kes} for the group $\bG_{f}$ and the ring $R[t,t^{-1}]$, we deduce  that 
this $\bG_{f,K}$-torsor is trivial.
Hence $E$ comes from a $Z$-torsor over $\ab^{1}_{R}$ which over $R[s,s^{-1}]$ comes from a section of $\bG_{f,K}(R[s,s^{-1}])$.
The composed map
$Z\ra\widetilde{G_{f,K}}\ra\bG_{f,K}$
factors through the isomorphism
$\widetilde{\bG_{f,\cO}}\cong\bG_{f,\cO}\times Z.$
Hence, the class of $E$ in $H^{1}(R[s,s^{-1}],Z)$  has 
trivial image in $H^{1}(R[s,s^{-1}],\widetilde{\bG_{f,K}})$.
Thus it comes from an element
\[x\in(\widetilde{\bG_{f,K}}/\widetilde{\bG_{f,\cO}})(K_R)=(\Gr_{\bG_{f}})(R[s,s^{-1}]).\]
Using again \eqref{kes} we get that $x$ lifts as a point in $\bG_{f,K}(R[s,s^{-1}])$. 
Hence $x$ has a trivial image in $H^{1}(K_R,\widetilde{\bG_{f,\cO}})$ that contains $H^{1}(K_R,Z)$ as a direct 
summand, thus we end up with an element in $\Gr_{Z}(R)$ and apply one more time \eqref{kes} to get that $E$ is trivial.
\epf

\subsection{Semi-infinite orbits}
\subsubsection{The Iwasawa decomposition}
In this section $G$ is a simply connected minimal  KM group over $\bZ$.
The next property is fundamental for the Iwasawa decomposition. The key subtlety is that we have a global curve.

\bprop\label{glift}
Let $R$ be a principal ring.
Let $G$ be a minimal  KM group over $\bZ$, and $B\subset G$ a Borel subgroup. 
Any map $x:\Spec(R)\to G/B$  lifts to a map $\Spec(R)\ra G$.
\eprop


\bpf
We write $G/B=\colim_{w\in W}(S_w)$ with $S_w=\overline{BwB}/B$.
The map $x$ factors through some $S_w$.
We prove by induction on $\ell(w)$ that $x$ lifts to a map $\Spec(R)\to G$.
If $w=1$, it is clear. Assume that $\ell(w)\geq 1$. We consider a reduced decomposition 
$w=s_{i_1}s_{i_2}\dots s_{i_n}$. 
We consider the associated Demazure scheme 
$D(\underline{w})=E(\underline{w})/B$
with
$E(\underline{w})=P_{i_1}\times^{B}\dots\times^{B}P_{i_n}.$ 
The map 
\[m:D(\underline{w})\ra S_{w}\]
is proper, surjective, and birational over the open subset
$BwB/B$ of $S_{w}$ by \cite[Lem.~29]{Mat2}. 
Note that this is proved for the formal version, but, since $P/B\cong \hat{P}/\hat{B}$, the formal version and the minimal versions of 
$D(\underline{w})$ are isomorphic.
By the valuative criterion of properness, the map $x$ lifts uniquely to a map
$\Spec(R)\ra D(\underline{w}).$
We have a commutative diagram
$$\xymatrix{E(\underline{w})\ar[d]\ar[rr]^-{m}&&G\ar[d]\\
D(\underline{w})\ar[r]^m&\Fl_{\leq w}\ar[r]&G/B}$$
Since $R$ is principal, we have $\Pic(R)=0$.
Hence, since $T$ is split, we can lift $x$ to
a map 
$$\ti{x}:\Spec(R)\to E(\underline{w})/U.$$
Note that  $E(\underline{w})/U$ is the quotient of $P_{\underline{w}}:=P_{i_{1}}\times\dots \times P_{i_{n}}$ 
by $B^{n-1}\times U$ acting by
\[(p_1,\dots,p_n)\cdot (b_{1},\dots, b_n)=(p_{1}b_{1},b_{1}^{-1}p_{2}b_{2},\dots,b^{-1}_{n-1} p_{n}b_{n}).\]
Let $H$ be a closed subgroup of $\bigcap_{i=1}^{n} R_{u}(P_{i})$ which is normal in every $P_i$ 
and such that $U/H$ is split unipotent.
Set
\[P_{\underline{w}}/H^{n}=P_{i_1}/H\times\dots\times P_{i_n}/H.\]
Since $H$ is normal, there is a map
\[\theta:P_{\underline{w}}/H^{n}\ra E(\underline{w})/U\]
that takes the $H^n$-orbit of the tuple $(p_{1},\dots,p_n)$ to its $B^{n-1}\times U$-orbit.
The map $\theta$ is thus an $(B^{n-1}\times U)/H^n$-torsor.
The group $(B^{n-1}\times U)/H^n$ is split solvable.
Thus the torsor is trivialisable because $\Pic(R)=0$. 
Thus we can assume that the map $x$ lifts to a map
$$\ti{x}:\Spec(R)\to P_{\underline{w}}/H^{n}.$$
By Lemma \ref{Un2}, we can choose $H=V_m:=U\cap \hV_m$,
where $\hV_m$ is
the pro-unipotent subgroup of $\hU$ consisting of roots $\al\in R^{+}$ with 
$\hgt(\al)\geq m$ for a big enough integer $m$.
We decompose $\ti{x}$ as a tuple $\ti{x}=(\ti{x}_1,\dots,\ti{x}_n)$ with
$\ti{x}_s:\Spec(R)\to P_{i_s}/H$.
We have $x=\ti{x}_1\dots \ti{x}_n$ in $(G/B)(R)$.
We abbreviate $P=P_{i_{1}}$.
Consider the first projection to $P/H$.
 We have
\[P\ra P/H\ra P/R_{u}(P).\]
The composed map is split by a Levi factor. 
The second map is a torsor over a split unipotent group,
hence it is trivial over any ring $R$.
Thus we can find $y_{1}\in P(R)$ such that 
\[y_{1}\ti{x}_{1}\in (R_{u}(P)/H)(R)\subset (U/H)(R).\]
Set $v=s_{i_2}s_{i_3}\dots s_{i_n}$.
We get a tuple
\[(y_{1}\ti{x}_{1},\dots, \ti{x}_n)\in (U\times P_{\underline{v}})/H^{n}\]
whose image  $y_{1}x$ in $G/B$ lands in $S_v$.
Thus $y_{1}x$ lifts to $G(R)$ by induction.
Since $y_1\in P(R)$, the $R$-point $x$ lifts also to $G(R)$.
\epf

The following proposition is well-known in the affine case. 
It seems to be new in our generality.
Note that a similar statement in  \cite{GR08} uses group theory.

\bprop\label{iwa}
Let $G$ and $\hG$ be the minimal and formal KM groups over a field $k$.
We have
$$G(K)=B(K)\cdot G(\cO)
,\quad
G(K)=B^{-}(K)\cdot G(\cO)
,\quad 
\hG(K)=\hB(K)\cdot\hG(\cO).$$
\eprop

\bpf
The rings $K$ and $\cO$ are both principal. Since $G/B\cong\hG/\hB$ and $G/B^{-}$ are ind-projective,
the valuative criterion of properness for  $\Spec(\co)$ yields 
$$(\hG/\hB)(K)=(\hG/\hB)(\co)
,\quad
(G/B)(K)=(G/B)(\co)
,\quad
(G/B^-)(K)=(G/B^-)(\co).$$
By Proposition \ref{glift}, we have 
$$G(K)/B(K)=(G/B)(K)=(G/B)(\co)=G(\co)/B(\co).$$
The formal case follows because as $\hU$ is pro-unipotent and $T$ is split
\[H^{1}(K,\hU)=H^{1}(K,T)=H^{1}(\co,\hU)=H^{1}(\co,T)=\{1\}.\]
\epf

\subsubsection{The semi-infinite orbits}
Let $G$ and $\hG$ be the minimal and formal KM groups over $\bZ$.
By \eqref{Gr} we can consider the affine Grassmannians $\Gr_G$ and $\Gr_{\hG}$ of
$G$ and $\hG$ and the obvious morphism $\Gr_G\to\Gr_{\hG}$.
For any cocharacters $\mu,\nu\in \LLa$, there are subfunctors $S_\mu$, $T_\nu$, $S_{\leq\mu},$ $T_{\geq\nu}$ of $\Gr_{G}$ given by 
\begin{align}\label{ST}
S_{\mu}=s^{\mu}\cdot\Gr_{U},\quad T_{\nu}=s^{\nu}\cdot\Gr_{U^-}
,\quad
S_{\leq\mu}=\bigsqcup_{\mu'\leq\mu} S_{\mu'}
,\quad
T_{\geq\nu}=\bigsqcup_{\nu'\geq\nu} T_{\nu'}.
\end{align}
Let 
$S'_{\mu}$, $T'_{\nu}$, $S'_{\leq\mu}$ and $T'_{\geq\nu}$ be the subfunctors of $G_K$ 
given by the inverse images 
by the obvious map $\pi:G_{K}\ra \Gr_{G}$.
We define the closures  $\ov{S_{\mu}}$ and $\ov{T_{\nu}}$ of $S_{\mu}$ and $T_{\nu}$ in $\Gr_{G}$ to be the quotients
$$\ov{S_{\mu}}=[\ov{S'_{\mu}}/G_{\cO}]
,\quad
\ov{T_{\nu}}=[\ov{T'_{\nu}}/G_{\cO}]$$
where $\ov{S'_{\mu}}$ and $\ov{T'_{\nu}}$ are the closures in the ind-scheme $G_K$ with the reduced structures.
The formal variants are the subfunctors $\hS_\mu$, $\hT_\nu$ of $\Gr_{\hat{G}}$ given by
\begin{align}\label{hST}
\hS_{\mu}=s^{\mu}\cdot\Gr_{\hU}
,\quad
\hT_{\nu}=s^{\nu}\cdot\Gr_{U^-}\cong T_\nu.
\end{align}
There are obvious projections 
$$\Gr_{B},\Gr_{\hB}\ra\Gr_{T}=\Gr_{T}=\bigsqcup_{\la\in\LLa}\Gr_{T}^{\la}.$$
We define
\begin{align}\label{GrBla}\Gr_{B}^{\la}=\Gr_B\times_{\Gr_T}\Gr_T^\la,\quad
\Gr_{\hB}^{\la}=\Gr_{\hB}\times_{\Gr_T}\Gr_T^\la.\end{align}

\blem\label{Gr-ST}Let $\mu, \nu\in \LLa$. 
\hfill
\begin{enumerate}[label=$\mathrm{(\alph*)}$,leftmargin=8mm]
\item
The obvious map is a bijective fp closed immersion $S_\mu\to\Gr_B^\mu$.
We have $(S_\mu)_\red\cong(\Gr_B^\mu)_\red$.
\item
The obvious map is a bijective fp closed immersion $T_\nu\to\Gr_{B^-}^\nu$.
We have $(T_\nu)_\red\cong(\Gr_{B^-}^\nu)_\red$.
\eenum
\elem

\bpf
Since $B/U\cong T$ is fp and affine, by Proposition \ref{ind-quot} the map
$\Gr_{U}\ra\Gr_{B}$ is fp closed.
Since $S_\mu=s^\mu\cdot\Gr_{U}$ by \eqref{hST}, by left translation by $s^\mu$ 
the map $S_\mu\ra\Gr_{B}$ is also fp closed.
Part (a) follows, because $(\Gr_T)_\red=X_*(T)$. Part (b) is proved similarly.
\epf

\brem
We do not know if $S_\mu$ and $T_\nu$ are reduced.
\erem

\bprop\label{iwa2}
Let $k$ be a field. We have
$$\Gr_G(k)=\bigsqcup_\mu S_\mu(k)=\bigsqcup_\nu T_\nu(k),
\quad
\Gr_{\hG}(k)=\bigsqcup_\mu \hS_\mu(k).$$
\eprop

\bpf
Follows from Proposition \ref{iwa}. 
\epf

There is a Plücker description of the semi-infinite orbits. 
For any $\omega\in \Lambda_+$, let 
$$\eta_{\omega}:L(\omega)_\bZ\to\ell_\omega=\bZ\cdot v_\omega$$ 
be the projection to the highest weight line.
We consider the closed subsets of $G_{K}$ given by
\begin{align}
\label{S1}
Y(\mu)&=\bigcap_{i\in I}\{g\in G_{K}\,;\, g^{-1}(\ell_{\omega_i}\otimes\co)\subset s^{-\langle\omega_i,\mu\rangle}L(\omega_i)\otimes\co\},\\
Z(\nu)&=\bigcap_{i\in I}\{g\in G_{K}\,;\, \eta_{\omega_i}(g L(\omega_i)\otimes\co)
\subset s^{\langle\omega_i,\nu\rangle}\ell_{\omega_i}\otimes{\co}\}.
\label{S2}
\end{align}
The set $Y(\mu)$ is $ U_{K}\times G_{\cO}$-invariant and contains $s^{\mu'}$ for each $\mu'\leqslant\mu$.
The set $Z(\nu)$ is $U_{K}^{-}\times G_{\cO}$-invariant and contains $s^{\nu'}$ for each $\nu'\geqslant\nu$.
Hence, we have 
$$S'_{\leq\mu}\subseteq Y(\mu)
,\quad
T'_{\geq\nu}\subseteq Z(\nu).$$
In addition, 
for each cocharacter $\la$ such that $\la\not\leqslant\mu$ there is a fundamental weight $\omega_i$ such that 
$$\langle\omega_i,\la\rangle\geqslant\langle\omega_i,\mu\rangle+1,$$
hence $s^\la\notin Y(\mu).$
The Iwasawa decomposition in Proposition \ref{iwa} yields
$(S'_{\leq\mu})_\red= Y(\mu)_\red$.
Similarly, we have
$(T'_{\geq\nu})_\red= Z(\nu)_\red.$
Hence
\begin{align}\label{ST1}
\ov{S_\mu}\subset S_{\leq\mu}
,\quad
\ov{T_\nu}\subset T_{\geq\nu},
\end{align}
Finally, let
$\hU_{\pi^-}$ be the kernel of the evaluation $\hU_{\cO^-}\ra\hU$ at $\infty$.

\bprop\label{ST-comp1}
Let  $\mu,\nu\in \LLa$.
\begin{enumerate}[label=$\mathrm{(\alph*)}$,leftmargin=8mm]
\item
We have fp locally closed embeddings 
$S_{\mu}\,,\,T_{\nu}\ra\Gr_{G}$ and
$\hS_{\mu}\ra\Gr_{\hG}.$
\item
We have 
$\overline{S_{\mu}}\subset S_{\leq\mu}$ and
$\overline{T_{\nu}}\subset T_{\geq\nu}$, with equalities if $G$ is symmetrizable.
\item
The functor $\hS_{\mu}$ is representable by an ind-affine ind-scheme. We have $\hS_{0}\cong\hU_{\pi^-}$. 
\end{enumerate}
\eprop

We need a preliminary result.

\bprop\label{roots}
Let $G$ be a minimal symmetrizable KM group over $\bZ$. 
Let $\Lal,\Lbe\in R^{\vee}_{+}$ be such that $\langle \Lal,\beta\rangle<0$.
Then $\Lal+\Lbe\in R^{\vee}_{+}$. 
Hence, for any pair of dominant cocharacters $\mu<\la$, there is a positive 
coroot $\Lal\in R_{+}^{\vee}$ such that $\mu\leq\la-\Lal<\la$.
\eprop

\bpf
The roots are the same over $\bZ$ and $\bQ$. Thus it is enough to check the assertion 
over $\bQ$ where it follows from \cite[Lem. 3.6-3.7]{Mar2}.
Once we have the first assertion, the proof becomes identical to \cite[Lem.~2.3]{Rap}.
\epf

We can now prove Proposition \ref{ST-comp1}.

\bpf 
For (a), it is enough to prove the claim for $\nu=0$ up to a left translation by $s^\nu$.
The morphisms $S_0\,,\,T_0\ra\Gr_{G}$ and $\hat S_0\ra\Gr_{\hG}$ are fp locally closed embeddings 
by Propositions \ref{ind-quot} and \ref{U-rep}.
Now, we prove (b). Both assertions can be checked on $k$-points for an arbitrary field $k$.
Since $G$ is symmetrizable, by Proposition \ref{roots} it is enough to check that $S_{\mu-\al}\subset\overline{S}_\mu$ 
and $T_{\nu+\al}\subset\overline{T}_\nu$ for any positive coroot $\al$.
Let us concentrate on the inclusion $S_{\mu-\al}\subset\overline{S}_\mu$ because the other one is very similar.
Set  $m=\langle \mu,\Lal\rangle-1$.
The inclusion follows from the construction of the curve $C_{\mu,\al}$ 
in the proof of Proposition \ref{Cartan}. Indeed, with the notation there, we have
\begin{align*}
 U_K\cap i_{\al}(K_m)&\cong
\begin{pmatrix}
1&s^m\cO\\0&1
\end{pmatrix}\\
\Ad_{s^{\mu}}( G_\cO)\cap U_K\cap i_{\al}(K_m)&\cong
\begin{pmatrix}
1&s^{m+1}\cO\\0&1
\end{pmatrix}
\end{align*}
For (c), note that the multiplication map
$\hU_{\pi^-}\times\hU_{\cO}\ra\hU_{K}$
is a monomorphism, because $\hU_\cO\cap\hU_{\pi^-}=\{1\}$.
It yields a monomorphism
$f:\hU_{\pi^-}\ra\hS_{0}.$
We must prove surjectivity on $R$-points for any ring $R$.
Since $\hU$ is split pro-unipotent, by \cite[Prop.~A.6]{RS} we have $H^{1}(S,\hU)=\{0\}$ for any affine scheme $S$.
Thus there is a bijection
\[H^{1}(\bP^{1}_{R},\hU)\cong\hU_{\cO^-}(R)\backslash\hU_K(R)/\hU_\cO(R).\]
We consider the presentation $\hU=\lim U^{(n)}$. 
We have $H^{1}(\bP^{1}_{R},\bG_a)=0$. Hence the maps 
$H_0(\bP^{1}_{R},U^{(n+1)})\to H_0(\bP^{1}_{R},U^{(n)})$ are surjective.
We deduce that ${}^1\lim H_0(\bP^{1}_{R},U^{(n)})=0$.
Hence $H^{1}(\bP^{1}_{R},\hU)=0$ by \cite[Lem.~A.3]{RS}.
Since $H^0(\bP^{1}_R,\hU)=\hU(R)$, we deduce that the map $f$ above is an isomorphism.
Finally $\hU_{\pi^-}$ is an ind-affine ind-scheme.
\epf

We need the following result, which strengthens Proposition \ref{ST-comp1}.

\bthm\label{Bfond}
The obvious map $\Gr_{B}\ra\Gr_{G}$ restricts to an fp immersion $\Gr_{B}^{\mu}\ra\Gr_G$
for each $\mu\in \LLa$.
\ethm

\bpf
The map $\Gr_{B}\ra\Gr_{G}$ is already a monomorphism.
The isomorphism $T_{K}\times U_{K}\cong B_{K}$ yields an isomorphism $T_{K}\times S_{0}\cong B_{K}/U_{\co}$. 
On the left hand side we have an action of $T_{\co}$ given by $t\cdot (h,u)=(ht^{-1},tut^{-1})$.
The twisted product $T_{K}\times^{T_{\co}} S_{0}$ yields a commutative diagram
$$\xymatrix{T_{K}\times S_{0}\ar[d]\ar[r]&B_{K}/U_{\co}\ar[d]\\T_{K}\times^{T_{\co}} S_{0}\ar[r]&\Gr_{B}}$$
The vertical maps are $T_{\co}$-torsors and the top horizontal one is an isomorphism.
So the bottom horizontal map is an isomorphism.
We must prove that the composed map $T_{K}\times^{T_{\co}} S_{0}\ra\Gr_{G}$ is an fp immersion
when restricted to a connected component $T^{\mu}_{K}\times^{T_{\co}} S_{0}$.
We can be check this étale locally, after pulling back to $G_{K}$ and prove that the following map  is an fp-immersion
\begin{equation}
T^{\mu}_{K}\times^{T_{\co}} S'_{0}\ra G_{K},
\label{indB}
\end{equation}
with $S'_{0}=S_{0}\times_{\Gr_{G}}G_K$. By left multiplication by $s^{-\mu}$ we can further assume that $\mu=0$.
By Proposition \ref{ST-comp1} the map $S'_{0}\ra G_{K}$ is fp locally closed.
Since $\Gr^{0}_{T}$ is an ind-thickening of $\Spec(k)$, the map \eqref{indB} is ind-locally closed.
Indeed, if the maps $X\stackrel{f}{\rightarrow}Y\stackrel{g}{\rightarrow} Z$ are such that $g\circ f$ is fp locally closed, 
the map $f$ is a nilpotent closed immersion, and $g$ is a monomorphism of ft schemes, 
then the map $g$ is fp locally closed
(after restricting to an open 
$U\subset Z$, we can replace locally closed by closed. 
Then we use \cite[Tag. 03GN]{Sta} and the fact that a proper monomorphism is  
a closed immersion \cite[Tag. 04XV]{Sta}).
Therefore, we must prove that for any $\la\in \LLa_+$ the map
$$(T^0_{K}\times^{T_{\co}} S'_{0})\times_{G_K}G_{\leq\la}\to G_{\leq\la}$$
is locally closed.
Let $R$ be a $k$-algebra and consider
\[tug\in T_{K}(R)U_{K}(R)\cap G_{\cO}(R)\cap G_{\leq\la}(R).\]
We have $t u\in G_{\leq\la}(R)$, because $G_{\leq\la}$ is $G_{\co}$-invariant.
Hence, for any dominant character $\omega\in\Lambda^+$, we have
\[t u\cdot v_{\omega}=\omega(t)\cdot v_{\omega}\in s^{\langle\omega,\la\rangle}\co_R\cdot v_\omega.\]
Thus $t$ is an $R$-point of the fiber product $T^{0}_{\leq\la}=T^{0}_K\times_{G_{K}}G_{\leq\la}$.
Since $T^{0}_{\leq\la}$ is closed in  $T_K$,
the quotient $T^{0}_{\leq\la}/T_{\co}$ is a finite nilpotent scheme.
So the first map in the chain of maps
\[S'_{0}\times_{G_{K}} G_{\leq\la}\ra (T^{0}_{\leq\la}\times^{T_{\co}} S'_{0})\times_{G_{K}} G_{\leq\la}\ra G_{\leq\la},\]
is a nilpotent closed immersion. The second one is schematic, and the composite is fp locally closed.
Thus the second map is also fp locally closed.
Finally, the ind-ft ind-scheme $$(T^{0}_{K}\times^{T_{\co}} S'_{0})\times_{G_{K}}G_{\leq\la}$$ is sandwiched between 
$S'_{0}\times_{G_{K}} G_{\leq\la}$ and $(T^{0}_{\leq\la}\times^{T_{\co}} S'_{0})\times_{G_{K}} G_{\leq\la}$.
Thus, we conclude by the following lemma.

\blem\label{closed}
Let $f:X\ra Y$ and $g:Y\ra Z$ be morphisms of ind-ft ind-schemes such that $f$ and $g  f$ are nilpotent closed immersions and  
$g$ is ind-locally closed. Then $g$ is closed. 
\elem

\bpf
By pulling-back to any closed subscheme $S\subset Z$, we are reduced to the case where $X$ and $Z$ are schemes of ft and $g  f$ is 
defined by a coherent sheaf of ideals $\cI$. We write $Y\cong\colim Y_{\al}$ with $Y_{\al}$ locally closed in $Z$.
For some $\alpha$ we have the commutative diagram of schemes
$$\xymatrix{X\ar[r]^-f\ar[rd]&Y_\alpha\ar[d]^-{g|_{Y_\alpha}}\\&Z}
$$ 
As $f$ is nilpotent, hence surjective, by \cite[Tag.~03GN]{Sta} 
the subscheme $Y_{\al}$ is closed in $Z$. As $\cI$ is coherent, 
there exists $\al_0$ such that for every $\beta\geq\al_0$ we have $Y_{\alpha_0}\cong Y_{\beta}$.
Thus $g$ is a closed immersion.
\epf	
\epf

\brem\hfill
\begin{enumerate}[label=$\mathrm{(\alph*)}$,leftmargin=8mm]
\item
The ind-scheme $\hS_0$ is wild: the transition maps are closed but may not be fp. 
\item
The proof above implies that
$(\ov{S'_{\mu}})_\red= Y(\mu)_\red$ and
$(\ov{T'_{\nu}})_\red= Z(\nu)_\red.$
\item
The map $\hT_{\nu}\ra\Gr_{\hG}$ may not be locally closed and
$\Gr_{\hG}(k)\neq\bigsqcup_\nu \hT_\nu(k)$, see \cite[\S2.3]{BFK}.
\eenum
\erem

\bprop\label{ST-comp2} Let $\mu,\nu\in\LLa$.\hfill
\begin{enumerate}[label=$\mathrm{(\alph*)}$,leftmargin=8mm]
\item
 $S_\mu\cap T_\nu\neq\emptyset\iff\mu\geq\nu$.
\item
$S_{\mu}\cap T_{\mu}=\Spec(\bZ)$.
\eenum
\eprop

\bpf
If $S_\mu\cap T_\nu\neq\emptyset$, then the combination of \eqref{S1} and \eqref{S2} yields
$\langle\omega_i,\mu\rangle\geq\langle\omega_i,\nu\rangle$ for each $i\in I$,
hence $\mu\geq\nu$. 
To prove (b), note that (a) and Proposition \ref{ST-comp1} imply that the open immersions
\begin{equation}
S_{\mu}\cap T_{\mu}\hra S_{\mu}\cap\ov{T_{\mu}}\hra \ov{S_{\mu}}\cap\ov{T_{\mu}}
\label{ser-incl}
\end{equation}
are bijective. Thus the intersection $S_{\mu}\cap T_{\mu}\cong \ov{S_{\mu}}\cap\ov{T_{\mu}}$ is closed in $\Gr_G$.
Further $S_{\mu}$ and $T_{\mu}$ are contained in the attractor and repeller of the point $[s^{\mu}]$
for the action by $2\Lrho$ on $\Gr_{G}$.
So any element $x\in(S_{\mu}\cap T_{\mu})(R)$ yields a $\bG_m$-equivariant map $\bP^{1}_{R}\ra\Gr_{G}$ 
that factors through $S_{\mu}\cap T_{\mu}$, because it is closed.
Composing this map with the obvious map $S_\mu\to\hS_\mu$, 
we get a map $\bP^{1}_{R}\ra\hS_{\mu}$ that is constant equal to $[s^{\mu}]$,
 because $\hS_{\mu}$ is ind-affine by Proposition \ref{ST-comp1}.
\epf

We will prove in Proposition \ref{Z-fin} that the intersection $(S_\mu\cap T_{\nu})_\red$ is a ft-scheme over $\bZ$.
Finally, we define the Laurent-series counterpart of our objects, i.e., we consider the functor $\Gr_{G}^\form=[G_{\hat K}/G_{\hat\cO}]$
and the subfunctors
\begin{align}\label{STf}
S_{\mu}^\form=s^{\mu}\cdot\Gr_{U}^\form
,\quad 
T_{\nu}^\form=s^{\nu}\cdot\Gr_{U^-}^\form,\quad \hat{S}_{\mu}^{\form}=s^{\mu}\cdot\Gr_{\hU}^\form.
\end{align}

\blem\label{f-fp}
The obvious maps $S^{\form}_{\nu},T^{\form}_{\nu}\ra\Gr_{G}^\form$ are fp locally closed.
We have $\hat{S}_{0}\cong\hat{S}_{0}^{\form}.$ 
\elem

\bpf
The first claim follows from Proposition \ref{ind-quot} and Proposition \ref{U-rep}.
For second one, by Proposition \ref{ST-comp1}(c) it is sufficient to prove that $\hat{S}_{0}^{\form}\cong\hU_{\pi^{-}}$.
By \cite[Lem.~2.2.11(b)]{BC}, we have a Beauville-Laszlo gluing for $\hU$, which yields
\[H^{1}(\bP^{1}_{R},\hU)\cong\hU(R[t^{-1}])\backslash\hU(R\llp t\rrp)/\hU(R\llb t\rrb)=\{0\},\]
as in the proof of Proposition \ref{ST-comp1}(c). 
We then conclude as in loc.~cit.~that $\hat{S}_{0}^{\form}\cong\hU_{\pi^{-}}$.
\epf

\subsubsection{Affineness of semi-infinite orbits}
Let $G$ be a minimal KM group over $\bZ$. 

\bthm\label{sinf-aff}
The immersions $S_{\mu}\hra\Gr_{G}$ and $T_{\nu}\hra \Gr_{G}$ are fp affine.
\ethm

First, we consider the formal setting.

\blem\label{A-st1}
The morphism $\Gr_{\hU}\ra\Gr_{\hG}$ is affine and fp locally closed.
\elem

\bpf
By Proposition \ref{ST-comp1}, we already know that this morphism is fp and locally closed.
By Lemma \ref{quotH}, to check that the morphism $\Gr_{\hU}\ra\Gr_{\hG}$ is affine it is enough to prove it after pulling to 
$\hG_{K}$. 
Let $$\hS'_{0}=\Gr_{\hU}\times_{\Gr_{\hG}}\hG_{K}.$$
Since the morphism $\hG_K\times\hG_\cO\to\hG_K$ is ind-affine, by Lemma \ref{quotH}, the morphism $\hG_K\ra\Gr_{\hG}$ is 
also ind-affine.
Thus by base change and Proposition \ref{ST-comp1}
we deduce that $\hS'_0$ is ind-affine, fp and locally closed in $\hG_{K}$. 
Since $\hG_{K}$ is ind-affine, the map $\hS_0'\ra\hG_K$ is affine by \cite[Tag.~01SG]{Sta}.
\epf

\blem\label{fib1}\hfill
\begin{enumerate}[label=$\mathrm{(\alph*)}$,leftmargin=8mm]
\item
The canonical maps
$\Gr_{U}\ra\Gr_{\hU}\times_{\Gr_{\hG}}\Gr_{G}$
and
$\Gr_{B}\ra\Gr_{\hB}\times_{\Gr_{\hG}}\Gr_{G}$
are bijective fp closed immersions. In particular they are affine.
\item
The map
$S_{\mu}\ra \hS_\mu\times_{\Gr_{\hG}}\Gr_{G}$
is a bijective fp closed immersion.
\eenum
\elem

\bpf
Closed immersions are affine. Let us prove the remaining statements.
Part (b) follows from (a) and \eqref{ST}.
By Proposition \ref{ST-comp1}, the maps $\Gr_{\hU}\ra\Gr_{\hG}$ and $\Gr_{U}\ra\Gr_{G}$ are fp locally closed.
Thus, the map
\[\Gr_{U}\ra\Gr_{\hU}\times_{\Gr_{\hG}}\Gr_{G}\]
is fp locally closed. Hence, it is enough to prove that $k$-points are the same for $k$ an algebraically closed field.
Let $E$ be a $k$-point of $\Gr_{\hU}\times_{\Gr_{\hat{G}}}\Gr_{G}$.
By Proposition \ref{U-rep}, we have $G/U\cong\hG/\hU$.
Thus the $G$-torsor $E$ over $\ab_k^1$ has a $U$-reduction, see \S\ref{AGR} and \S\ref{mod-desc}. 
We must prove that it defines an $k$-point of $\Gr_{U}$.
The $G$-torsor $E$ is trivial over $\bG_{m,k}$.
Thus it yields an element in $(G/U)(K)$. 
It lifts to $G(K)$ by Proposition \ref{glift}.
This implies that  the $U$-reduction of $E$ is also trivial over $\bG_{m,k}$, proving the claim.
The proof for $\Gr_{B}$ is the same using Theorem \ref{Bfond} instead.
\epf

We can now finish the proof of Theorem \ref{sinf-aff}.

\bpf
We can assume that $\nu=0$.
By base change and composition, Lemmas \ref{A-st1} and \ref{fib1} imply that the map
$S_0=\Gr_{U}\ra\Gr_{G}$ is affine and fp.
The corresponding statement for $T_0$ is proved in a similar way using $\hG^{\op}$.
\epf

\brem
Lemma \ref{fib1} is finer than saying that we have equivalences on reduced stacks, because the map 
$\cX_\red\ra\cX$ may not be schematic.
\erem

\subsection{The Cartan decomposition}
\subsubsection{The Cartan semigroup}
If $G$ is a minimal KM group over a field $k$, then the Cartan decomposition may not hold for $G_K$. 
One must introduce the Cartan sub-semigroup of $G_K$.
Its definition is inspired from \cite[Appendix A]{BKP}.
Let $g\in G(K)$.
We say that $g$ is bounded if for each representation $L(\omega)$ of highest weight $\omega\in \Lambda_+$, 
there exists an integer $N_\omega$ such that
\[g(L(\omega)_\cO)\subset s^{-N_{\omega}}L(\omega)_\cO.\]
We consider the sub-semigroups $G_b(k)$ and $G_c(k)$ of $G(K)$ given by
\begin{equation}
G_b(k)=\{g\in G(K)\,;\,g~ \text{is bounded}\}
,\quad
G_c(k)=G_b(k)^{-1}.
\label{sem-cart}
\end{equation}
They are stable by left and right action by $G(\cO)$, because $G(\cO)$ preserves $L(\omega)_\cO$.
For  $v\in L(\omega)_{K}$ we define
$$\ord(v)=\min\{n\in\bZ\,;\,s^nv\in L(\omega)_\cO\}. $$
By definition the element  $v_{\omega}$ has the order 0.

\blem\label{la1}
Let $\la\in \LLa$. We have $\la\in\LLa_+$ if and only if $s^{\la}\in G_c(k)$.
\elem

\bpf
For any vector $v\in L(\omega)$ of weight $\mu$, we have $s^{\la}v=s^{\langle\mu,\la\rangle}v.$
Hence $s^\la$ is bounded if and only $\langle\mu,\la\rangle\geq -N_\omega$ for some $N_\omega\in\NN$ and all $\mu\leq\omega$ and
all
$\omega$.
Hence, if and only if the set $\{\langle\al,\la\rangle\,;\,\alpha\in Q^+\}$ is bounded above, thus $\la\in-\La_+$.
\epf

By Lemma \ref{la1}, for any field  $k$ we have
\begin{equation}
G_c(k)\supset\bigcup_{\la\in \LLa_+} G(\cO)s^{\la}G(\cO).
\label{inc-cart}
\end{equation}
\bprop\label{Cart}
For any field $k$, we have an equality
\[G_c(k)=\bigcup_{\la\in \LLa_+} G(\cO)s^{\la}G(\cO).\]
In particular, the right hand side is a semigroup.
\eprop

\bpf
The proof in \cite[Appendix]{BKP}, which is based on the Iwasawa decomposition, applies.
We recall this proof, because loc.~cit. considers only the untwisted affine case with coefficients in $\hat\cO$. 
We will use the Iwasawa decomposition in the form $G(\co)T(K)U(K)$.
Let $g\in G_b(k)$.
By Lemma \ref{la1} it is enough to prove that $g=k_1 s^{-\la}k_2$ for $k_1, k_2\in G(\cO)$ and $\la\in \LLa_+$.
Fix a dominant weight $\omega\in \La_+$.
Since $g\in G_b(k)$, we may
choose $k\in G(\cO)$ such that the element $(gk)\cdot v_{\omega}$ in $L(\omega)_\cO$
has maximal order.
By Proposition \ref{iwa} there a	re $k_1$, $u$ and $t$ such that
\[gk=k_{1}tu,\quad k_1\in G(\cO),\quad u\in U(K),\quad t\in T(K).\]
We must prove that $u\in U(\cO)$. 
Assume that $u\notin G(\cO)$.
The Iwasawa decomposition yields $u^-$, $\mu$ and $k_2$ such that
\[u=u^{-}s^{\mu}k_2,\quad u^-\in U^-(K),\quad\mu\in\LLa,\quad k_2\in G(\cO)\]
We have $u\in T'_\mu\cap S'_0$.
By Proposition \ref{ST-comp2}, we have
$T'_0\cap S'_0= G(\cO)$.
Since $u\notin G(\cO)$, we deduce that $\mu\neq 0$. 
By Proposition \ref{ST-comp2}, we also have 
$$T'_{\mu}\cap S'_0\neq\emptyset\iff\mu\leq 0.$$
Thus $\mu<0$.
Now we have
\begin{align}
\begin{split}
 \ord((gkk_{2}^{-1})\cdot v_{\omega})
  & =\ord(k_1tu^{-}s^{\mu}\cdot v_{\omega})\\
  &=\ord(tu^{-}s^{\mu}\cdot v_{\omega})\\
   & \geq \ord(tv_{\omega})-\langle\omega,\mu\rangle\\
   &=\ord(gkv_{\omega})-\langle\omega,\mu\rangle.\end{split}
\end{align}
This contradicts the maximality of the order. So $u\in G(\cO)$.
\epf

\subsubsection{The cell decomposition of the Cartan semigroup}
Let $G$ be the minimal KM group over $\bZ$.
By Proposition \ref{Cart}, for any field $k$ and any element $g\in G(K)$, there is an element
\[\inv_k(g)\in \LLa_+\cup\{\infty\},\]
such that $\inv(g)$ is the image of $g$ in the double quotient 
$G(\cO)\backslash G_c(k)/G(\cO)$ if $g\in G_c(k)$, and $\inv(g)=\infty$ otherwise.
Now, for any ring $R$ any $g\in G_K(R)$ and any $x\in\Spec(R)$ of residue field $k(x)$,
we set
\[\inv_{x}(g)=\inv_{k(x)}(g\vert_{\Spec(k(x))})\] 

\blem\label{Schub1}
For any  $g\in G_K(R)$ and $\la\in \LLa_+$, the following subset of $\Spec(R)$ is Zariski closed
\[\Spec(R)_{\leq \la}=\{x\in\Spec(R)\,;\,\inv_{x}(g)\leq\la\}.\]
\elem

\bpf
For each dominant character $\omega$,
the $G$-action \eqref{rho-omega} yields a morphism of ind-schemes 
$$\rho_\omega:G_K\to\End^\ind(L(\omega)_K).$$
We consider the functor $X(\la,\omega)$ such that
\begin{align}\label{S0}
\begin{split}
X(\la,\omega)(R)
&=\{g\in G_K(R)\,;\, \rho_\omega(g)(L(\omega)_{\cO_{k(x)}})
\subset s^{\langle\omega,\la\rangle}L(\omega)_{\cO_{k(x)}}\,,\,\forall x\in\Spec(R)\}.
\end{split}
\end{align} 
Given an $R$-point $g$ of $G_K$,
the following fiber product is  closed in $\Spec(R)$
$$X(\lambda,\omega)\times_{G_K}\Spec(R)$$
We claim it coincides with $\Spec(R)_{\leq \la}$.
Since $X(\lambda,\omega)$ is $G_\cO\times G_\cO$-invariant and contains $s^{\mu}$ for any $\mu\leq\la$, we have
$$\Spec(R)_{\leq\la}\subseteq X(\lambda,\omega)\times_{G_K}\Spec(R).$$
Set $X(\la)=\bigcap_{i\in I}X(\lambda,\omega_i)$.
Then 
$$\Spec(R)_{\leq\la}\subseteq X(\lambda)\times_{G_K}\Spec(R).$$
If $g\in X(\la)$ then $g\in G_c(k(x))$ for each $x\in\Spec(R)$.
Further, for each cocharacter $\mu$ such that $\mu\not\leqslant\lambda$ there is a fundamental weight $\omega_i$ such that $\langle\omega_i,\mu\rangle>\langle\omega_i,\la\rangle$.
Hence $s^\mu\notin X(\lambda,\omega_i)$. Thus
\begin{equation}
X(\lambda)\times_{G_K}\Spec(R)_{\leq\mu}=\emptyset
,\quad
 \Spec(R)_{\leq\la}=X(\lambda)\times_{G_K}\Spec(R).
\label{plu-cart}
\end{equation}
\epf

By Lemma \ref{Schub1}, there is a  reduced closed ind-subscheme  $G_{\leq\la}\subset G_K$ such that
for any field $k$ we have
\begin{align}\label{gleqla}
G_{\leq\la}(k)=\{g\in G_K(k)\,;\,\inv_k(g)\leq\la\}.
\end{align}
Let $G_\la\subset G_{\leq\la}$ be the open subset given by
\[G_{\la}= G_{\leq\la}\setminus\bigcup_{\mu<\la} G_{\leq\mu}.\]
We formulate the following proposition over fields. To have it over $\bZ$, we need finer information.

\bprop\label{dg-plus}
Let $G$ be a minimal KM group over an algebraically closed field $k$.
The subfunctor $G_c\subset G_K$ is a closed semigroup ind-subscheme.
\eprop

\bpf
For any closed subscheme $Z\subset G_K$, we have
\[Z\times_{G_K}G_c=\bigcup_{\la\in S}Z\times_{G_K} G_{\leq\la}\]
for some finite $S\subset\Lambda_+$.
Thus $Z\times_{G_K}G_c$ is closed in $Z$ by Lemma \ref{Schub1}.
It remains to prove that $G_c$ is a semigroup ind-scheme.
For $\la,\mu\in \LLa_+$ the multiplication gives a map
$G_{\leq\la}\times G_{\leq\mu}\ra G_K.$
We must prove that it factors through $G_c$.
As we considered the reduced structure and work over an algebraically closed field, a product of reduced is reduced.
Thus proving that the map factors is a statement on closed points.
This statement is clear, because $G_c(k)$ is a semigroup.
\epf

\bprop\label{Cartan}
Let $G$ be a  minimal symmetrizable KM group. 
The open subset $G_{\la}$ of $G_{\leq\la}$ is dense.
\eprop

\bpf
We must prove that the closure of $G_{\la}$ contains each stratum $G_{\mu}$ for $\mu\leq\la$.
The argument adapts \cite[Prop.~2.1.5]{Zhu} in the KM case.
Let $\mu<\la$ with $\lambda,\mu\in\LLa_+$. 
Since $G$ is symmetrizable, by Lemma \ref{roots}, there exists a positive coroot $\al$ such that $\mu\leq\la-\al<\la$.
It suffices to prove that $G_{\lambda-\alpha}\subset\overline{G_{\lambda}}$. 
To do this, we construct a curve 
$C_{\la,\al}\cong\bP^1$ in $\Gr$ with 
$\infty\in \Gr_{\la-\alpha}$ and $\ab^1\subset \Gr_{\la}$.
First, for any integer $m$ let 
$$\theta_m=\begin{pmatrix}
s^{m}&0\\0&1
\end{pmatrix}\in PGL_2(K).$$ 
Set
$K_{m}=\Ad_{\theta_m}(SL_2(\cO))$
and
$H_{m}=\Ad_{\theta_m}(I^{-})$, with $I^{-}\subset SL_2(\cO^-)$ the opposite Iwahori subgroup.
We have $K_m/H_m\cong\mathbb{P}^{1}$.
Let $i_\al:SL_2(K)\ra G_K$ be the  morphism associated with $\al$.
Set $m=\langle \la,\Lal\rangle-1$. 
We consider the orbit
$C_{\la,\al}=i_{\al}(K_m)\cdot [s^{\la}].$
Since $i_\alpha(H_m)\subset \Ad_{s^{\la}}(G_\cO)$, we have $C_{\la,\al}\cong\bP^{1}.$
Since $m\geqslant 0$, we have
\begin{align*}
G_\cO\cap i_{\al}(K_m)&\cong
\begin{pmatrix}
\cO^\times&s^m\cO\\\cO&\cO^\times
\end{pmatrix}\\
\Ad_{s^{\la}}(G_\cO)\cap G_\cO\cap i_{\al}(K_m)&\cong
\begin{pmatrix}
\cO^\times&s^{m+1}\cO\\\cO&\cO^\times
\end{pmatrix}
\end{align*}
Hence, there is an isomorphism
\[(G_\cO\cap i_{\al}(K_m))\cdot [s^{\la}]\cong\ab^{1}.\]
We set
$$\sigma_{m}=
\begin{pmatrix}0&-s^{m}\\s^{-m}&0\end{pmatrix}=
\Ad_{\theta_m}
\begin{pmatrix}
0&-1\\1&0
\end{pmatrix}
\in K_m.$$
We have 
$$s^{-\la+\al}\cdot i_{\al}(\sigma_m)\cdot[s^\la]=
i_\alpha(\sigma_0)
\in G_\cO.$$ 
We deduce that
$i_{\al}(\sigma_m)\cdot [s^{\la}]= \{\infty\}$
under the isomorphism $C_{\la,\al}\cong\bP^{1}$, 
and that $i_{\al}(\sigma_m)\cdot [s^{\la}]\in\Gr_{\lambda-\alpha}$.
\epf

\subsection{The Schubert cells}
Let $G$ be a simply connected minimal KM group over $\bZ$, with the associated formal group $\hG$ over $\bZ$.
We define the Schubert cell $\Gr_{\leqslant\la}$ and the open Schubert cell $\Gr_{\la}$ to be the substacks
of $\Gr_G$ given by
\begin{align}\label{Grla}\Gr_{\leqslant\la}= [G_{\leq\la}/G_\cO],\quad \Gr_{\la}= [G_{\la}/G_\cO].\end{align}
The subfunctor $\Gr_\la\subset\Gr_{G}$ is locally closed because $G_\la\subset G_K$ is locally closed.
Let $i_\la$ denote the immersion
\begin{align}\label{ila}i_\la:\Gr_\la\to\Gr_G.\end{align}
Note that the set of points $\Gr_{\la}(k)$ is a $G(\cO)$-orbit for any field $k$.

\subsubsection{The formal group case}
We first consider the formal group case, where we have better representability statements.
Recall that $\hU_{\pi^-}$ is the kernel of the evaluation $\hU_{\cO^-}\ra\hU$ at $\infty$.
Let $\hG_\pi$ and $\hI$ be the kernel of the evaluation $\hG_\cO\ra \hG$ at $0$ 
and the inverse image of $\hB$.
For each cocharacter $\lambda\in\LLa$ we set
\begin{align}\label{VVJ}
\hU_\la^-=\hU_\cO\cap \Ad_{s^{\la}}(\hU_{\pi^-})
,\quad
\hU_\la=\hU_\cO\cap \Ad_{s^{\la}}(\hU_\cO)
,\quad
\hat J_\la=\hI\cap \Ad_{s^\la}(\hI)\end{align}

\blem\label{hU-orb}\hfill
\begin{enumerate}[label=$\mathrm{(\alph*)}$,leftmargin=8mm]
\item
The multiplication yields an  isomorphism of $\bZ$-ind-schemes
$\hU_\la^-\times \hU_\la\cong\hU_\cO.$
\item
The map $\hU_\cO\ra\Gr_{\hG}$ given by $g\ra g\cdot [s^{\la}]$ yields an immersion
$[\hU_\cO/\hU_\la] \cong \hU_\la^-\ra\Gr_{\hG}.$
\eenum
 \elem

\bpf
We have the scheme isomorphism 
$\hU\cong\Spec(\Sym(\kn^\vee))$.
By \cite[\S3.2]{Rou} it yields compatible group ind-scheme isomorphisms
\begin{align}
\begin{split}\label{hu}
\xymatrix{
\Ad_{s^{\la}}(\hU_\cO)\ar@{^{(}->}[r]\ar@{=}[d]&\hU_K\ar@{=}[d]&\ar@{_{(}->}[l]\Ad_{s^{\la}}(\hU_{\pi^-})\ar@{=}[d]\\
\prod_\al s^{\langle\la,\al\rangle}(\kg_\al)_\cO\ar@{^{(}->}[r]&\prod_\al(\kg_\al)_K&
\ar@{_{(}->}[l]\prod_\al s^{\langle\la,\al\rangle-1}(\kg_\al)_{\cO^-}
}
\end{split}
\end{align}
The claim (a) follows. 
To prove claim \eqref{hU-orb}, note that by Proposition  \ref{ST-comp1} the obvious map
$\Gr_{\hU}\hra\Gr_{\hG}$ is an fp immersion and $\Gr_{\hU}\cong \hU_{\pi^-}$ whicn induces that $S_{\la}\cong s^{\la}\hU_{\pi^-}s^{-\la}$.
Further, by \eqref{hu} the obvious inclusion
$\hU_\la^-\ra\hU_{\pi^-}$
is a closed immersion.
\epf

\blem\label{I-orb}\hfill
\begin{enumerate}[label=$\mathrm{(\alph*)}$,leftmargin=8mm]
\item
The inclusion $\hU_\cO\subset\hI$ yields an isomorphism of stacks
$[\hU_\cO/\hU_\la]\cong[\hI/\hat J_\la].$
\item
The map $\hI\ra\Gr_{\hG}$ given by $g\mapsto g\cdot[s^{\la}]$ yields an immersion 
$[\hI/\hat J_\la]\hra\Gr_{\hG}$.
\eenum
\elem

\bpf
We must prove that the monomorphism
$[\hU_\cO/\hU_\la] \to [\hI/\hat J_\la]$
is surjective.
We will use $\hat\cO$ instead of $\cO$, because
if $W\hra X$ is a qc open immersion of ind-schemes, 
then we have an isomorphism
\begin{align}\label{key}W_{\hat\cO}\cong X_{\hat\cO}\times_{X}W\end{align}
(the proof reduces to the scheme case, which is well-known).
We define $\hat{\hat{I}}\subset\hG_{\hat\cO}$ and
$U_{\widehat\pi}^-\subset U_{\hat\cO}^-$
as above.
Since $\hOm=U^{-}\hB$ is open in $\hG$, from \eqref{key} we deduce that $\hat{\hat{I}}=U_{\widehat\pi}^-\hU_{\hat\cO} T_{\hat\cO}$. As $U_{\widehat\pi}^-$ fixes $s^{\la}$,  we have 
$$[\hU_{\hat\cO}/(\hU_{\hat\cO}\cap \Ad_{s^{\la}}(\hU_{\hat\cO}))] \cong[\hat\hI/(\hat\hI\cap \Ad_{s^\la}(\hat\hI))]$$
 In addition, we have a monomorphism 
 $$[\hI/\hat J_\la]\to[\hat\hI/(\hat\hI\cap \Ad_{s^\la}(\hat\hI))]$$
and a similar computation as in Lemma \ref{hU-orb} gives 
\begin{align*}
[\hU_{\hat\cO}/(\hU_{\hat\cO}\cap \Ad_{s^{\la}}(\hU_{\hat\cO}))]
\cong\hU_{\hat\cO}\cap \Ad_{s^{\la}}(\hU_{\pi^-})
=\hU_\la^-
\cong[\hU_{\cO}/\hU_\la].
\end{align*}
The lemma follows.
\epf

By \S\ref{opp-J}, the parabolic $\hat P_\lambda$ and the opposite parabolic $P_\lambda^-$
represent the functors
$$R\mapsto\{g\in \hG(R)\,;\,s^\la g s^{-\la}\in \hG_\cO(R)\}
,\quad
R\mapsto\{g\in \hG(R)\,;\,s^{-\la} g s^{\la}\in \hG_\cO(R)\}$$
We have $P^{-}_{\la}=\hG\cap \hat K_\la$ with
$\hat K_\la=\hG_\cO\cap \Ad_{s^{\la}}(\hG_\cO)$.
We consider the quotient stacks
$$\widehat\Gr_{\la}=[\hG_\cO/\hat K_\la]
,\quad
\widehat{\underline\Gr}_\la=[\hG/P_\la^-].$$
The evaluation $\hG_\cO\to\hG$ at 0 yields a morphism of stacks
\begin{align}\label{phi}
\phi:\widehat\Gr_{\la}\ra \widehat{\underline\Gr}_\la.
\end{align}

\bprop\label{sch-form}\hfill
\begin{enumerate}[label=$\mathrm{(\alph*)}$,leftmargin=8mm]
\item
The stack $\widehat{\underline\Gr}_\la$ is representable by a scheme.
\item
We have a Cartesian diagram
$$\xymatrix{[\hI/\hat J_\la]\ar[d]\ar[r]&\widehat\Gr_{\la}\ar[d]^-\phi\\ 
\hU_{\la}\ar[r]&\hG/P_{\la}^{-}}$$ 
\item
We have $[\hI/\hat J_\la]\cong\widehat\Gr_{\la}\cap\hS_{\la}$.
It is a qc open subset of $\widehat\Gr_\lambda$. 
We have $\widehat\Gr_\la=\bigcup_{w\in W/W_\la}w\cdot[\hI/\hat J_\la].$
\item
The stack $\widehat\Gr_{\la}$ is representable by a weak ind-scheme. It is locally closed in $\Gr_{\hG}$. 
\eenum
\eprop

\bpf
Part (a) is obvious because $\hG/P^{-}_{\la}$ is a partial thick flag manifold.
To prove (b) we consider the open cell $\hU_\la$ in $\hG/P_{\la}^{-}$.
By (a) we have the map $\phi:\widehat\Gr_\la\ra \hG/P_{\la}^{-}$ such that
$\phi^{-1}(1)=\hG_\pi\cdot[s^{\la}]$.
Part (b) follows because the evaluation at zero sends surjectively $[\hI/\hat J_\la]$ to $\hU_\la$ (seen as the $\hB$-orbit of $s^{\la}$. 
To prove (c), note that Lemma \ref{I-orb} yields the inclusion
$$[\hI/\hat J_\la]\subset\widehat\Gr_{\la}\cap \hS_{\la}.$$
Further, the intersection $\widehat\Gr_{\la}\cap \hS_{\la}$ is contained in the attractor for the $\bG_m$-action on $\widehat\Gr_{\la}$
given by conjugation by $\la$, while, by \S\ref{opp-J}, the group $\hU_\la$ is the strict attractor of the $\bG_m$-action on $\hG/P_{\la}^{-}$
given by conjugation by $\la$. 
Thus the map $\phi$
takes $\widehat\Gr_{\la}\cap \hS_{\la}$ into $\hU_\la$.
Thus the reverse inclusion follows from (b).
The last claims in (c) follows form the open cover
\[\hG/P_{\la}^{-}=\bigcup_{w\in W/W_\la}w\cdot\hU_\la.\]
The representability statement in part (d) follows from (c).
By Lemma \ref{I-orb}, there is a locally closed embedding $[\hI/\hat J_\la]\ra\Gr_{\hG}$.
Thus the map 
$\widehat\Gr_\la\ra\Gr_{\hG}$
is also locally closed by \cite[Tag.~0FCZ]{Sta}, proving (d).
\epf

The $W$-translates of the open subset $\widehat\Gr_{\la}\cap \hat S_{\la}$
yield the open cover
\begin{align}\label{hGrcirc}\widehat\Gr_\lambda=\bigcup_w\widehat\Gr_{w\la}^\circ.\end{align}
Similarly, let $\Gr_{w\la}^\circ\subset\Gr_\la$ be the 
$w$-translate of the intersection $\Gr_{\la}\cap S_{\la}$ 
which yield the open cover
\begin{align}\label{Grcirc}\Gr_\lambda=\bigcup_w\Gr_{w\la}^\circ.\end{align}

\subsubsection{The minimal group case}
\bprop\label{g-cart}\hfill
\begin{enumerate}[label=$\mathrm{(\alph*)}$,leftmargin=8mm]

\item
The obvious map yields a bijective closed immersion
$\widehat\Gr_{\la}\times_{\Gr_{\hG}}\Gr_{G}\ra\Gr_{\la}.$

\item
The stack $\Gr_{\la}$ is locally closed in $\Gr_G$.

\item
The obvious map $\Gr_{\la}^\circ\ra\Gr_{\la}$ is a qc open immersion such that
$\Gr_{\la}=\bigcup_{w\in W}\Gr_{w\la}^\circ$.
\eenum
\eprop

\bpf
For (a) and (b), working étale locally and using \eqref{Grla} and the Plücker description of $G_{\leq\la}$ given by \eqref{S0}, we get an inclusion
\[\widehat\Gr_{\la}\times_{\Gr_{\hG}}\Gr_{G}\subset\Gr_{\la}\]
which is bijective on closed points.
This map is also locally closed by base change and Proposition \ref{sch-form}(d).

The map $\Gr_\la\to\Gr_G$ is a locally closed immersion by base change, (a) and Proposition \ref{sch-form}(d).
Now we prove (c). By (a) and because the open cells of $\hG/P^{-}_{\la}$ and $G/P_{\la}^{-}$ match, we have
\begin{equation}
\widehat\Gr^\circ_{\la}\times_{\Gr_{\hG}}\Gr_{G}=\Gr_{\la}^\circ.
\label{bla-for1}
\end{equation}
The map $\widehat\Gr_{\la}^\circ\ra\widehat\Gr_{\la}$ is qc and open by Proposition \ref{sch-form}(c).
Hence, the map $\Gr_{\la}^\circ\ra\Gr_{\la}$ is also qc and open by base change.
The last claim is proved similarly.
\epf

We abbreviate 
\begin{align}\label{KV}
K_\la=G_{\cO}\cap\Ad_{s^{\la}}(G_\cO)
,\quad
U_\la=U_\cO\cap\Ad_{s^{\la}}(U_\cO).
\end{align}

\bprop\label{unif}\hfill
\begin{enumerate}[label=$\mathrm{(\alph*)}$,leftmargin=8mm]
\item
The map $G_{\cO}\ra\Gr_{G}$, $g\mapsto g\cdot [s^{\la}]$ yields an isomorphism
$[G_{\cO}/K_\la]_\red\cong(\Gr_\la)_\red.$
\item
The isomorphism in $\mathrm{(a)}$ factors through an isomorphism  
$[U_\cO/U_\la]_\red\cong(\Gr^{\circ}_\la)_\red.$
\end{enumerate}
\eprop

\bpf
We have monomorphisms
$[G_{\cO}/K_\la]\subset\Gr_{\la}$
and
$[U_\cO/U_\la]\subset\Gr^{\circ}_{\la}.$
We have to prove that functors are the same on reduced algebras. The claim is étale local. 
By Proposition \ref{g-cart}, for any ring $R$, 
we can assume that the map $G_{\cO}(R)\ra\Gr_{G}(R)$, $g\mapsto g\cdot [s^{\la}]$ lands in some open set 
$\Gr_{w\la}^{\circ}(R)$ Zariski locally on $R$.
So we are reduced to prove (b).
Both functors in (b) commute with filtered colimits, being lft prestacks as defined in \S\ref{ss-fonct}, as quotients of ind-ft group ind-schemes.
Indeed, the source is a quotient in the category of ind-schemes of ind-ft, 
and the target is fp locally closed in such a quotient by Proposition \ref{g-cart}.
So it is enough to compare the functors on reduced local strictly henselian  rings. 
Let $R$ be such a ring and $g\in\Gr_{\la}^\circ(R)$. 
Since $R$ is strictly local henselian, we can assume that $g$ admits a lift $\ti{g}\in G_K(R)$ and that 
\[\ti{g}=us^{\la}k,\quad u\in U_K(R), k\in G_\cO(R).\]
As $g\in \Gr_{\la}(R)$, we can also assume that $k=1$.
We must prove that $u\in U_\cO(R)$. 
Since $U_{\cO}\ra U_{K}$ is closed and $R$ is reduced, by \cite[Tag.~056B]{Sta}, 
the map $u:\Spec(R)\ra U_{K}$ factors through $U_{\cO}$ if it does on each closed point $x\in\Spec(R)$.
So we can assume  that $R=k$ is a field, that $u\in U(K)$, and it is enough to check that $u\in \hU(\co)$. 
By Proposition \ref{sch-form} and Lemma \ref{I-orb}, we have 
\[us^{\la}=v s^{\la}h,\quad v\in\hU(\co), h\in\hG(\co).\]
We deduce that
$$h=s^{-\la}v^{-1}us^{\la}\in \hG(\cO)\cap\hU(K)=\hU(\cO).$$
Since $\la$ is dominant, we have $\Ad_{s^{\la}}(\hU_{\co})\subset \hU_{\co}$.
We deduce that
$u=vs^{\la}hs^{-\la}$ lies in $\hU(\cO).$
\epf

\section{The Mirkovic-Vilonen cycles}

\subsection{The space of maps}
\subsubsection{The space of maps}
Consider a scheme $S$, a proper fppf morphism $X\ra S$ and an $S$-algebraic stack $\cY$.
The functor $\Map_{S}(X,\cY)$ 
on the category of $S$-schemes is such that
\[\Map_{S}(X,\cY)(T)=\Hom_{T}(X_{T},\cY_{T})\]
Here $X_T$, $\cY_T$ are the corresponding base change and morphisms 
$X_T\ra\cY_T$ are morphisms of algebraic stacks over $T$.

\blem\label{mor1}Assume that $\cY=Y$ is a separated $S$-scheme.\hfill
\begin{enumerate}[label=$\mathrm{(\alph*)}$,leftmargin=8mm]
\item
The functor $\Map_{S}(X,Y)$ is representable by an $S$-algebraic space.
\item
Assume that $X$ is projective and $Y$ is a union of qc open subsets that admit a presentation 
as a filtered limit of fp quasi-projective schemes over $S$ with affine transition maps.
Then  $\Map_{S}(X,Y)$ is representable by an $S$-scheme.
\eenum
\elem

\bpf
The assertion is local on $S$, so we can assume that $S$ is affine.
Assume first that $Y$ is qc.
By \cite[Tag.~0GS1]{Sta}, we have $Y\cong\lim Y_a$ with $Y_a\to S$ fp and with affine transition maps,
and by \cite[Tag.~01ZQ]{Sta} we can assume $Y_a\to S$ to be separated.
Then, by  \cite[\S0DPL]{Sta} the functor $\Map_{S}(X,Y_a)$ is representable by a lfp algebraic space over $S$.
For the general case, write $Y=\bigcup U$ as the union of qc open subsets.
For $U\subset U'$
the space of maps $\Map_{S}(X,U)$ is open in  $\Map_{S}(X,U')$.
Hence $\Map_{S}(X,Y)$ is an increasing union of open subsets.
Thus it is representable by an $S$-algebraic space, proving (a).

Assume now that $X$ is projective and $Y_a$ is quasi-projective fp over  $S$.
By noetherian approximation
we can assume that $S$ is noetherian. Then, by \cite[Thm.~5.23]{FGA}, the functor $\Map_{S}(X,Y_a)$ 
is representable by a lfp $S$-scheme.
Thus, we have $\Map_{S}(X,Y)\cong \lim\Map_{S}(X,Y_a)$, and by
\cite[Tag.~05Y6]{Sta} the transition maps are affine.
For the general case, write $Y=\bigcup U$ as the union of its qc open subsets 
that are limit of quasi-projective fp schemes over $S$.
Then part (b) follows as above.
\epf

\bexa\label{P-proj}\label{ex:Pproj}
An example of scheme which is an union of qc pro-quasi-projective
open subsets is $\bP(V)$ for $V$ a pro-finite dimensional vector space. 
Indeed, write $V\cong\lim V_a$. Then $\bP(V)$ is the union of the open subsets 
$\bP(V)\,\setminus\,\bP(\Ker(V\ra V_a))$,
and each of them is the limit of 
$\bP(V_b)\,\setminus\,\bP(\Ker(V_b\ra V_a))$ for $b\geq a$.
Consequently, any fp closed subset in $\bP(V)$ also satisfies this property.
\eexa

Let $\cY$ be a lfp algebraic stack over $S$ with quasi-affine diagonal.
By \cite[Tag.~0CMG]{Sta}, the diagonal of $\cY$ is  fp.
Thus $\cY$ has affine stabilizers, 
as a quasi-affine algebraic group over a field is affine by \cite[Exp.~VI, Prop.~11.11]{SGA3}.
By \cite[Thm.~1.2]{HaR} the functor
$\Map_{S}(X,\cY)$ is representable by a lfp algebraic stack over $S$  with affine diagonal.
Fix an open substack $\bullet\subset\cY$. 
Let 
\begin{align}\label{Hom-bullet}
\Map_{S,\bullet}(X,\cY)\subset\Map_{S}(X,\cY)
\end{align}
be the sub-functor 
which consists of the maps that generically land in $\bullet$. 
We consider the evaluation map
$$\text{can}:X\times\Map_{S}(X,\cY)\ra \cY,\,
(x,f)\mapsto f(x)$$
and the projection 
$$p:X\times\Map_{S}(X,\cY)\ra\Map_{S}(X,\cY).$$ 
The map $p$ is fppf. Thus it is open. We have
\begin{equation}
\Map_{S,\bullet}(X,\cY)=p(\text{can}^{-1}(\bullet)) 
\label{eqcU}
\end{equation}
Hence the map \eqref{Hom-bullet} is open.
The subfunctor 
$\Map_{S}(X,\bullet)$
is also open.
Indeed, for any $S$-scheme $T$ and any map $\sigma_T:X_T\ra\cY_T$, 
the locus where $\sigma_T$ maps to $\bullet$ is $T\,\setminus\, p(X_{T}\,\setminus\,\sigma_T^{-1}(\bullet))$,
which is open in $T$ because the map $p:X_T\ra T$ is proper.
Note that the openess holds without assuming that $X\ra S$ is fppf.

\subsubsection{Maps to pointy stacks}
Let $S$ be a scheme.
Let $\cY$ be a lfp Artin $S$-stack and $\cD\subset\cY$ a closed substack with complement 
$\bullet=\cY\,\setminus\,\cD$. Assume that $\bullet$ 
is isomorphic to $S$ and is dense open in $\cY$.
In \cite{Dr}, such a stack is called a pointy stack.
Consider  a diagram
\begin{equation}
\begin{split}
\xymatrix{X\ar[r]^{j}\ar[dr]_{\pi^{0}}&\ov{X}\ar[d]_-{\pi}&H\ar[dl]\ar[l]_{i}\\&S}
\label{X-diag}
\end{split}
\end{equation}
where $i$ is closed with complement $j$, and $\pi$ is proper fppf.
Let $\Map_{S,\bullet}(X,\cY)$ be the functor that assigns to any $S$-scheme $T$ the groupoid of maps 
$f:X_{T}\ra\cY_T$ that generically land in $\bullet$ and such that $f^{-1}(\cD_T)$ is proper over $T$. 
By \eqref{eqcU}, the locus $U_T$ where $f$ lands in $\bullet$ surjects on $T$.

\blem\label{sep1}
Assume that the diagonal of $\cY$ is separated and the map $\pi$ has geometrically integral fibers.
For any $S$-scheme $T$ the groupoid $\Map_{S,\bullet}(X,\cY)(T)$ is a set.
\elem

\bpf
Write $\cY$ as a quotient of  a smooth $S$-groupoid in algebraic spaces $[U/R]$, see \cite[Tag.~04T5]{Sta}.
Assume that the diagonal of $\cY$ is separated and the map $\pi$ has geometrically integral fibers.
The map $R\ra U\times_{S}U$ is also separated, see \cite[Prop.~4.3.2]{LMB}.
Further, the stabilizer groupoid $G=U\times_{U\times_{S} U}R\ra U$ is separated by base change. 
Let $\cI_{\cY}$ be the inertia stack. 
By \cite[Tag.~06PR]{Sta}, we have an cartesian diagram
$$\xymatrix{G\ar[d]\ar[r]&U\ar[d]\\\cI_{\cY}\ar[r]&\cY}$$
Thus, by descent, the map $\cI_{\cY}\ra\cY$ is separated.
Fix $\sigma:X_T\ra\cY$.
Set $I_{X_{T}}=\sigma^{*}\cI_{\cY}$.
We must prove that $$I_{X_{T}}\cong X_T.$$
Consider a section $\eps:X_T\to I_{X_{T}}$ given by $\id_{x_{T}}$ for each $x_{T}\in\cY(X_T)$.
Let $T'\ra T$ and $y\in(I_{X_T})(T')$.
By assumption, over the open subset $U_{T'}=\sigma^{-1}(\bullet)_{T'}$ of $X_{T'}$, the morphisms $y$ and $\eps$ agree. 
Since the morphism $X_T\ra T$ is fppf with geometrically integral fibers, and the open set $U_{T'}$ is fiberwise non-empty, 
we deduce that $U_{T'}$ is schematically dense fiberwise.
Thus by \cite[Prop.~11.10.10]{EGAIV3}, the open subset $U_{T'}$ is schematically dense in $X_{T'}$.
Then claim follows, because the map $I_{X_T'}\ra X_{T'}$ is separated.
\epf

\blem\label{ext1}
Let $S$ be a scheme. 
Let $W$ be an $S$-scheme, and $W_\circ$ an open subscheme of $W$.
Let $\cY$ be a pointy $S$-stack.
There is a bijection between the set of $S$-maps $\sigma_\circ:W_\circ\ra\cY$ such that 
$(\sigma_\circ)^{-1}(\cD)\subset W_\circ$ is closed, 
and the set of $S$-maps $\sigma:W\ra\cY$ such that $\sigma^{-1}(\cD)\subset W_\circ$.
\elem

\bpf
Recall that $S\cong\bullet=\cY\setminus\cD$.
Given $\sigma_\circ$, the scheme $W$ admits an open cover given by $W_\circ$ and $W\setminus(\sigma_\circ)^{-1}(\cD)$.
Let $\sigma\vert_{ W_\circ}=\sigma_\circ$ and let $\sigma\vert_{ W\setminus(\sigma_\circ)^{-1}(\cD)}$ be
 the structural map to $S=\bullet$. This defines the map $\sigma$. The converse is obvious.
\epf

The following conjecture is made by Drinfeld \cite[Conj. 4.2.3]{Dr}.
Note that Drinfeld only considers the case $S=\Spec(k)$ for $k$ a field, and $C$ a smooth curve.
So the assumption on the relative compactification is automatic.

\bconj
With the assumptions above, assume in addition that $\cY$ has a separated diagonal, 
then $\Map_{S,\bullet}(C,\cY)$ is representable by a lfp $S$-algebraic space.
\econj

We prove the following weaker statement.

\bprop\label{Dr-rep}
Assume that the diagonal of $\cY$ is quasi-affine and we are in the situation of \eqref{X-diag}.
The stack $\Map_{S,\bullet}(X,\cY)$ is representable by a lfp $S$-algebraic space.
\eprop

\bpf
The stack $\Map_{S,\bullet}(\ov{X},\cY)$ is representable by a lfp algebraic stack over $S$ by \cite[Thm.~1.2]{HaR}.
By \cite[04SZ]{Sta} and Lemma \ref{sep1}, 
we deduce that it is representable by an $S$-algebraic lfp space.
By Lemma \ref{ext1}, the functor 
$$\Map_{S,\bullet}(X,\cY)\subset\Map_{S,\bullet}(\ov{X},\cY)$$
consists of 
maps $\sigma_1:\ov{X}\ra\cY$ such that $\sigma_1^{-1}(\cD)\subset X$.
Restricting $\sigma_1$ to $H$ yields a map of functors:
\[\Map_{S,\bullet}(\ov{X},\cY)\ra\ \Map_{S}(H,\cY).\]
Thus, we have the following isomorphism
\[\Map_{S,\bullet}(X,\cY)\cong \Map_{S,\bullet}(\ov{X},\cY)\times_{\Map_{S}(H,\cY)} \Map_{S}(H,\bullet)\]
The claim follows because  $\Map_{S}(H,\bullet)$ is open in $\Map_{S}(H,\cY)$.
Note that here $H\ra S$ is only proper but as mentioned after \eqref{eqcU}, openess does not need the fppf part.
\epf

\subsection{Affine Zastavas}
In this section we work over the ring $k=\bZ$, because the next section will involve a counting argument over finite fields. 
We prove that the Zastava are representable by finite type schemes in Proposition \ref{Z-fin}.
So far, it was only proved that they have a finite number of points over any finite field.
\subsubsection{Definition of the Zastavas}
We define the Birkhoff stacks to be the double quotients 
$$\Bk_G=[B^{-}\backslash G/B]=[B^{-}\backslash \hG/\hB]
,\quad
\wbk_G=[U^{-}\setminus\hG/\hB]$$

\bprop\label{fond-stack}
The $\infty$-stack $\Bk_G$ is representable by a smooth Artin stack over $\bZ$ with affine diagonal. 
\eprop

\bpf
The proof is taken from \cite[Lem.~6]{Fal}.
Recall the open cell $\hOm/\hB\subset \hG/\hB$ which is isomorphic to $U^-$.
For any finite set $F\subset W$ we form the following open subset 
\[\Om_{F}=\bigcup_{w\in F} w\hOm/\hB.\]
Set 
$$U_{w}^{-}=U^{-}\cap w(U^{-})
,\quad
U^{-,w}=U^{-}\cap w(\hU).$$
Choose an integer $n$ such that for each $w\in F$ we have 
\[U_{w}^{-}\cong U_{(n)}^-\times W_n\]
with $W_n$ being the split unipotent group generated by negative roots that occur in $U_{w}^{-}$ but not in $U_{(n)}^-$.
Here $U_{(n)}^-\subset U^-$ is defined as in Proposition \ref{Un}.
By Proposition \ref{uw-comp}, the set $\Omega_{F}$ has a finite cover by open subsets isomorphic to 
$U_{(n)}^-\times\ab^{m}$ for some integers $m$, $n$.
We deduce that $[U_{(n)}^-\backslash\Omega_{F}]$ is representable by a smooth ft scheme. 
Now, since $U_{(n)}^-\subset B^{-}$ has a finite dimensional quotient, we have
\[\Bk_G=\bigcup_{F\subset W}[B^{-}\backslash\Omega_{F}]\]
is an increasing union of qc opens that are representable by smooth ft Artin stacks. 
Set $$\cU_F=[B^{-}\backslash\Omega_{F}].$$
For the last claim, since $\Delta_{\Bk_G}=\bigcup\Delta_{\cU_F}$, it is enough to prove 
that $\cU_F$ has an affine diagonal. 
By Lemma \ref{quotH}, as $\Omega_F$ is ind-separated, we obtain that $\cU_{F}$ has ind-affine diagonal.
On the other hand by considering
\[[U_{(n)}^{-}\backslash\Omega_{F}]\ra\cU_{F}\]
The source is a smooth qc scheme.
Thus it is quasi-separated.
Hence $\Delta_{\cU_{F}}$ is qc schematic by Lemma \ref{quotH2}. 
The diagonal $\Delta_{\cU_F}$ being qc schematic and ind-affine, it is affine.
\epf


By Proposition \ref{fond-stack}, the stack $\wbk_G$ is fp-smooth over $\bZ$.
The substack
$$\bullet=[U^{-}\backslash\hOm/\hB]\subset \wbk_G$$ 
is qc, dense and open.
By Proposition \ref{Dr-rep}, the functor $\Map_\bullet(\bP^{1},\wbk_G)$ 
is representable by a lft algebraic space.
The evaluation at $0$ yields a morphism
\[\ev_{0}:\Map_\bullet(\bP^{1},\wbk_G)\ra \wbk_G.\]
We define the Zastava space to be the fiber 
\begin{align}\label{Zast}\cZ=\{f\in\Map_\bullet(\bP^{1},\wbk_G)\,;\,f(0)=\bullet\}\end{align}
The Zastava space is an open subfunctor of $\Map_\bullet(\bP^{1},\wbk_G)$.
Hence it is representable by a $\bZ$-lft algebraic space.
By Theorem \ref{fs-Kash} and Example \ref{P-proj},
the Kashiwara flag scheme $B^{-}\backslash \hG$ satisfies the 
assumptions of Lemma \ref{mor1}.
Hence the functor 
$$\Map_\bullet(\bP^{1}, B^{-}\backslash \hG),\quad\bullet=B^-\backslash\hOm$$ 
is representable by a $\bZ$-scheme as well as the subfunctor of based maps
\begin{align}\label{bmap}
\Map_b(\bP^{1},B^{-}\backslash \hG)=\{f\in\Map_\bullet(\bP^{1}, B^{-}\backslash \hG)\,;\,f(0)=B^-\backslash B^-\}.
\end{align}

\bcor\label{Zfond}\hfill
\begin{enumerate}[label=$\mathrm{(\alph*)}$,leftmargin=8mm]
\item
The obvious map yields an isomorphism
$\Map_b(\bP^{1}, B^{-}\backslash \hG)\cong \cZ.$
\item
The functor $\cZ$ is representable by a lft scheme over $\bZ$.
\eenum
\ecor

\bpf
The following morphism is an $\hU$-torsor
\[B^{-}\backslash \hG\ra [B^{-}\backslash \hG/\hU]=\wbk_G\]
Since $\hU$ is split pro-unipotent, the proof of Proposition \ref{ST-comp1} yields
$H^0(\bP^{1}_k,\hU)=\hU(k)$ and $H^{1}(\bP^{1}_{k},\hU)=0$  for any ring $k$.
Thus the obvious morphism $\Map_b(\bP^{1}, B^{-}\backslash \hG)\to\cZ$ is surjective on $k$-points.
Since we considered based maps, it is an isomorphism.
Recall that $\Map_b(\bP^{1}, B^{-}\backslash \hG)$ is representable by a scheme and $\cZ$
by a lft algebraic space by Proposition \ref{Dr-rep}. 
Thus the functor  $\cZ$ is representable by a lft scheme.
\epf

Set 
$\ov{T}=\Spec(\bZ[\Lambda_+])$
and
\begin{align}\label{AD}
\ab_\Delta=\prod\limits_{\al\in\Delta}\ab^{1}
,\quad
\bG_\Delta=\prod\limits_{\al\in\Delta}\bG_m
\end{align} 
The scheme $\ov{T}$ is isomorphic to
$\ab_\Delta$ and contains $T$ as an open dense subset isomorphic to $\bG_\Delta$.
We set 
$$\Div_{\ov{T}}=\Map_\bullet(\bP^{1},[T\backslash\ov{T}])
,\quad
\bullet=[T\backslash T]$$
The set $\Div_{\ov{T}}$ is the set of effective Cartier divisors on $\bP^1$ colored by $\LLa_+$.
We have
$$\Div_{\ov{T}}=\bigsqcup\limits_{\la\in \LLa_+}\bP^{\la}
,\quad
\bP^{\la}=\prod_{i\in I}\bP^{\langle\omega_i,\la\rangle}.$$
The map
$\hG/\hU\ra\ov{T}$ such that
$g\mapsto(\langle v_i^\vee, g\cdot v_i\rangle)$
factors through a map 
$$\wbk_{G}\cong[B^{-}\backslash \hG/\hU]\ra[T\backslash\ov{T}]$$ 
It yields a morphism
$\cZ\ra\Div_{\ov{T}}.$
For each $\la\in \LLa_+$ we define
$$\cZ^{\la}=\cZ\times_{\Div_{\ov{T}}}\bP^{\la}.$$

\bprop\label{Zfond2}
For any cocharacter $\la\in \LLa_+$ the $\bZ$-scheme $\cZ^{\la}$ is lft, smooth and connected.
\eprop

\bpf
Once smoothness is proved, it is enough to check the remaining claims over the generic fiber, hence over $\bC$.
Thus the result follows from \cite[Prop.~2.25]{BFG}. 
In loc.~cit. it is assumed that $G$ is symmetrizable. This is not needed as soon as the necessary results on thick flag schemes 
for arbitrary KM groups are established, as explained at the beginning of \S 2 in loc.~cit. This is done in \S\ref{KFM}.
Now we prove smoothness.
The argument is the same as in \cite[Prop.~2.24]{BFG}.
We briefly recall it because loc.~cit. is over $\bC$.
Using Corollary \ref{Zfond}, let $\bar{S}\ra S$ be a square zero extension of Artinian schemes, 
and let 
$\sigma_{0}:\bP^{1}_{\bar{S}}\ra B^{-}\backslash\hG$ be a based map that we must lift.
Since $B^{-}\backslash\hG$ is formally smooth,
the map $\sigma_0$ lifts Zariski locally on $\bP^{1}_{\bar{S}}$.
The obstruction to lift $\sigma_0$ globally lies in the cohomology group
$$H^{1}(\bP^{1}_{\bar{S}}\,,\,(\sigma_0^{*}T_{B^{-}\backslash\hG})(-1)).$$
The tangent sheaf $T_{B^{-}\backslash\hG}$ is flat over $\bZ$. 
The Lie algebra $\kg$ surjects to $T_{B^{-}\backslash\hG}$ at every point.
So we have a surjective morphism of sheaves 
$$\cO_{\bP^{1}_{\bar{S}}}(-1)\otimes\kg\thra(\sigma_{0}^{*}T_{B^{-}\backslash\hG})(-1).$$
Hence, there is a surjection of the first cohomology groups.
Finally, we have 
$$H^{1}(\bP^{1}_{\bar{S}},\cO_{\bP^{1}_{\bar{S}}}(-1)\otimes\kg)=H^{1}(\bP^{1}_{\bar{S}},\cO_{\bP^{1}_{\bar{S}}}(-1))\otimes\kg=0.$$
\epf

\subsubsection{Zastavas and intersections of semi-infinite orbits}
We consider the diagonal embedding 
$$\{\infty\}\hra\bigsqcup\limits_{\la\in\LLa_+}\bP^{\la}=\Div_{\ov{T}}.$$
We define the central fibers of $\cZ$ and $\cZ^{\la}$ to be the schemes
\begin{align}\label{Z}
Z=\cZ\times_{\Div_{\ov{T}}}\{\infty\},\quad
Z^{\la}=Z\times_{\cZ}\cZ^{\la}.
\end{align}
Recall the functors $\Gr_B^\la$ and $\Gr_{\hB}^\la$ introduced in \eqref{GrBla}.
For  future use, we compare the intersection 
$S_{\mu}\cap T_{\nu}$ of the subfunctors of $\Gr_{G}$ with the intersection $S^\form_{\mu}\cap T^\form_{\nu}$ 
of the subfunctors of $\Gr_{G}^\form$ introduced in \eqref{STf}.

\bprop\label{ST-compGr}
\hfill
\begin{enumerate}[label=$\mathrm{(\alph*)}$,leftmargin=8mm]
\item
$Z_\red\cong(\Gr_{B}\times_{\Gr_{G}}\Gr_{U^{-}})_\red$.
\item
$(S_{\mu}\cap T_{\nu})_\red=(S^\form_{\mu}\cap T^\form_{\nu})_\red$.
\eenum
\eprop

\bpf 
By Corollary \ref{Zfond}, the scheme $Z$ is the moduli space of maps $\bP^{1}\ra B^{-}\backslash\hG$ 
such that $f\vert_{\ab^{1}}$ factors through $\hU$ and $f(0)=B^{-}$.
Since the scheme $B^{-}\backslash\hG$ is separated, 
the Beauville-Laszlo gluing applies to $B^{-}\backslash\hG$ 
by \cite[Cor.~4.4]{MB}. 
Hence, we can replace in the moduli problem above $\bP^{1}$ by $\Spec(\bZ\llb s\rrb)$ and use the variant of 
Corollary \ref{Zfond} for $\bZ\llb s\rrb$ to go backwards.
Also, we have
$G/B=\Spec(\bZ)\times_{\bB G}\bB B$.
Quotienting by the group $U^{-}$ we deduce that
$$\wbk_G=\bB B\times_{\bB G}\bB U^-.$$
Recall the functors $\Gr$, $\Gr'$ introduced in \S\ref{AGR}-\ref{mod-desc}.
This implies that
\begin{equation}
Z\cong (\Gr^\form_{\hB})'\times_{(\Gr^\form_{\hG})'}(\Gr^\form_{U^{-}})'
\label{form-pri}
\end{equation}
Next, we must remove the upperscript $(-)'$, the $f$'s and pass from $\hG$ to $G$.
The scheme $Z$ is locally of finite type. Thus it commutes to filtered colimits. 
Also, the fiber product $\Gr_{B}\times_{\Gr_{G}}\Gr_{U^{-}}$ is a prestack locally of finite type,
see Remark \ref{rem-lft}, thus it commutes with filtered colimits. 
Therefore, it is sufficient to check that they have the same points on reduced strictly henselian local rings.
We still have $\Gr^\form_{T}=(\Gr_{T}^\form)'$ and $\Gr^\form_{\hB}=(\Gr_{\hB}^\form)'$.
Since $\Gr_{\hB}^\form\ra(\Gr_{\hG}^\form)'$ factors through $\Gr^\form_{\hG}$, we first prove that:
\[((\Gr^\form_{U^{-}})'\times_{(\Gr_{\hG}^\form)'}\Gr^\form_{\hG})(R)=(\Gr^\form_{U^{-}})(R).\]
As $R$ is strictly henselian, it amounts to prove that a $U^{-}$-torsor $E$ over $R\llb s\rrb$ whose induced $\hG$-torsor is trivial over $R\llb s\rrb$, is trivial. It amounts to an $R\llb s\rrb$-point of $\hG/U^{-}$, that lifts by 
Proposition \ref{w-cov} and as $\Pic(R\llb s\rrb)=0$.
We thus obtain that
\[Z(R)=(\Gr^\form_{\hB}\times_{\Gr^\form_{\hG}}\Gr^\form_{U^{-}})(R).\]
To pass to polynomial loops, we use Lemma \ref{f-fp} to get that $\Gr_{\hB}=\Gr^\form_{\hB}$ and,
 as $\Gr_{\hB}\ra\Gr_{\hG}^\form$ factors through $\Gr_{\hG}$, we need to have
\begin{equation}
(\Gr_{\hG}\times_{\Gr_{\hG}^\form}\Gr_{U^{-}}^\form)(R)=\Gr_{U^{-}}(R).
\label{hG-U}
\end{equation}
Since $\Gr_{U^{-}}^\form\ra\Gr_{\hG}^\form$ is a monomorphism, by base change and Proposition \ref{ST-comp1}, 
the map
\[\iota:\Gr_{U^{-}}\ra\Gr_{\hG}\times_{\Gr_{\hG}^\form}\Gr_{U^{-}}^\form\]
is fp locally closed. In particular, to get \eqref{hG-U}, it is thus sufficient to check that the $k$-points are the same for $k$ algebraically closed, because $\iota$ will induce an isomorphism on the reduced stacks.
Then, the assertion follows from Iwasawa decomposition.
Finally, using Lemma \ref{fib1}, we obtain an element in $(\Gr_{B}\times_{\Gr_{G}}\Gr_{U^{-}})(R)$ as wished.
For the second assertion, by multiplying by $s^{\nu}$, we reduce to the case $\nu=0$.
Using \eqref{form-pri}, we then have a factorization
\[S_{\mu}\cap T_{0}\ra S_{\mu}^{\form}\cap T_{0}^{\form}\ra Z^{\mu}=(\Gr^{\mu,\form}_{\hB})'\times_{(\Gr^\form_{\hG})'}(\Gr^\form_{U^{-}})'\]
and the previous assertion with Lemma \ref{Gr-ST} 
gives that the composite is an isomorphism on the reduced stacks and as all the maps are monomorphisms, this concludes.
\epf
We deduce the following which strengthens Proposition \ref{ST-comp2} and
\cite[Thm.~1.9]{BKGP}, \cite[Thm.~5.6]{Heb2}.

\bprop\label{Z-fin}
For each $\mu,\nu\in\LLa$ the intersection $(S_\mu\cap T_{\nu})_\red$ is a ft scheme over $\bZ$.

\eprop
\bpf
If $S_{\mu}\cap T_{\nu}\neq\emptyset$ then $\mu\geq\nu$ by Proposition \ref{ST-comp2}.
The multiplication by $s^{\nu}$ gives an isomorphism
$$S_{\mu-\nu}\cap T_{0}\cong S_{\mu}\cap T_{\nu}.$$
We can thus assume that $\mu\in \LLa_+$ and $\nu=0$.
By Lemma \ref{Gr-ST} the maps $S_\mu\ra\Gr_{B}$ and $T_0\ra\Gr_{B^-}$ are fp and closed.
Thus the following map is also fp closed
\begin{align}\label{STGr}
S_{\mu}\cap T_{0}\hra \Gr_{B}^\mu\times_{\Gr_{G}}\Gr_{U^-}.
\end{align} 
By Lemma \ref{ST-compGr} we have
$(\Gr_{B}^\mu\times_{\Gr_{G}}\Gr_{U^-})_\red=(Z^\mu)_\red$.
From Proposition \ref{Zfond2} and \eqref{Z}, we deduce that $Z^\mu$ is a scheme.
Thus $(S_{\mu}\cap T_0)_\red$ is a  ft scheme over $\bZ$.
\epf

\subsection{Finite dimensional Mirkovic-Vilonen cycles}
The goal of this section is to prove the following theorem.

\bthm\label{thm:GT}
Assume that $G$ is symmetrizable. Let $\la\in \LLa_+$ and $\nu\in \LLa$. 
\hfill
\begin{enumerate}[label=$\mathrm{(\alph*)}$,leftmargin=8mm]
\item
The intersection $(\Gr_\la\cap T_\nu)_\red$ is a  ft scheme over $\bZ$  of relative dimension $\langle\rho,\la-\nu\rangle$.
\item
The number of irreducible components of $\Gr_\la\cap T_\nu$ of maximal dimension is $\dim L(\la)_\nu$.
\eenum
\ethm

The proof of Theorem \ref{thm:GT} involves a counting argument and combinatorics involved need symmetrizability. For the dimension part, 
there is an other approach, using the dimension formula of Zastavas \cite[Conj. 2.27, Cor. 2.28]{BFG}, but the conjecture in loc.~cit.~is only 
settled in the affine case.

\subsubsection{Preliminary lemmas}
To prove the theorem we need more material.
Recall the cover of $\Gr_\la$ by the open subsets $\Gr_{w\la}^\circ$ introduced in \eqref{Grcirc}.

\blem\label{GrT1}
Let $\la\in \LLa_+$ and $\nu\in \LLa$. 
\hfill
\begin{enumerate}[label=$\mathrm{(\alph*)}$,leftmargin=8mm]
\item
If $\Gr_\la\cap T_\nu\neq\emptyset\Rightarrow\nu\leq\la$.
\item
If $\Gr_{w\la}^\circ\cap T_\nu\neq\emptyset\Rightarrow \nu\leq w\la$.
\item
$(\Gr_\la\cap T_\nu)_\red$ is a ft scheme over $\bZ$.
\eenum
\elem

\bpf
Composing the Cartan involution with the inverse, we get an anti-automorphism $\theta$ of the group
 $G$ which fixes the set $\LLa$
and switches the positive and negative unipotent subgroups.
Claim (a) follows by applying $\theta$ to \cite[Thm.~1.9]{BKGP}.

To prove (b) note that 
if $\Gr_{w\la}^\circ\cap T_\nu\neq\emptyset$ then $(wS_\la)\cap T_\nu\neq\emptyset.$
Since $T_\nu$ is preserved by left translations by elements of the torus $T$ and since $wS_\la$ is the attractor of
the element $[s^{w\lambda}]$ in the affine Grassmannian for the action  of the cocharacter $2w\Lrho$ of $T$, 
by the Iwasawa decomposition, we deduce that 
for each closed point $x$ in $(wS_\la)\cap T_\nu$ the limit of $2w\Lrho(z)\cdot x$ as $z\to 0$ is equal to $[s^{w\lambda}]$
and belongs to the closure of $T_\nu$. 
By Proposition  \ref{ST-comp1}, we deduce that $w\lambda\geqslant\nu$.

Now, we concentrate on (c).
The functor $\Gr_\la$ is covered by the open subsets $\Gr_{w\la}^\circ$ with $w\in W$.
 The set $\{w\lambda\,;\,w\in W\,,\,w\lambda\geqslant\nu\}$ is finite.
Hence it is enough to prove that 
$(\Gr_{w\la}^\circ\cap T_\nu)_\red$ is a ft scheme over $\bZ$ whenever it is non empty.
To do this, we will prove that $((wS_\la)\cap T_\nu)_\red$ is a ft scheme.
If $w=1$ this follows from Proposition \ref{Z-fin}. The general case is proved in a similar way.
More precisely, we have an obvious isomorphism $(wS_\la)\cap T_\nu\cong (wS_{\la-w^{-1}\nu})\cap T_0.$
Hence, according to the discussion above, we can assume that $w\lambda\in \LLa_+$ and $\nu=0$.
Further, we identify 
$$(wS_{\la})\cap T_0\cong S_{\la}\cap (w^{-1}T_0)$$ 
and we observe as in \eqref{STGr} that
the obvious inclusion $w^{-1}(U^-)\subset G$ yields a closed immersion
\begin{align}\label{Zw} 
S_{\la}\cap (w^{-1}T_0)\hra\Gr^{\la}_{B}\times_{\Gr_{G}}\Gr_{w^{-1}(U^{-})}. \end{align}
The automorphism of $G$ given by the left translation by $w$ yields a stack isomorphism
$$[w^{-1}(U^{-})\backslash G/B]\cong \wbk_G.$$
In particular, the substack
$$\bullet=[w^{-1}(U^{-})\backslash w^{-1}\Omega]\subset [w^{-1}(U^{-})\backslash G/B]\cong \wbk_G$$ 
is qc, dense and open.
The proof of Lemma \ref{ST-compGr} yields an isomorphism
\begin{align*}
(\Gr_B\times_{\Gr_{G}}\Gr_{w^{-1}(U^{-})})_\red \cong Z_\red
\end{align*}
The map
$\hG/\hU\ra\ov{T}$ such that
$g\mapsto(\langle v_i^\vee, wg\cdot v_i\rangle)$
factors through a map 
$$\wbk_{G}\cong[w^{-1}(B^{-})\backslash \hG/\hU]\ra[T\backslash\ov{T}]$$ 
which yields a morphism
$Z\to\cZ\ra\Div_{\ov{T}}.$
Recall that $w\lambda$ is dominant.
Let $Z^\lambda$, $\cZ^\lambda$ denote the base changes
of $Z$, $\cZ$  from $\Div(\ov{T})$ to $\bP_\bZ^{w\lambda}$ as above.
We have an isomorphism
\begin{align*}
(\Gr_{B}^\lambda\times_{\Gr_{G}}\Gr_{w^{-1}(U^{-})})_\red \cong Z_\red^\lambda
\end{align*}
Further $Z^\lambda$ is of ft over $\bZ$ by the proof of Propositon \ref{Zfond2}.
Hence the intersection $(S_{\la}\cap (w^{-1}T_0))_\red$ is also of ft.

\epf


\subsubsection{Hall-Littlewood functions}
We first recall some combinatorics following \cite{Vis}. 
We assume in this section that $G$ is symmetrizable.
Set
$$\mathcal{E}_{t}=\Big\{f=\sum\limits_{\mu\in \Lambda}c_{\mu}(t)\,e^{\mu}
\,;\,
c_{\mu}(t)\in\bC\llb t\rrb\Big\}$$
where $c_{\mu}(t)=0$ outside the union of a finite number of sets of the form 
$D(\omega)=\{\nu\in \Lambda\,;\, \nu\leq\omega\}.$
The $e^{\mu}$ are formal exponentials with $e^0=1$ and $e^{\mu+\nu}=e^{\mu}e^{\nu}$.
Consider the partially defined operator 
$$J=\sum\limits_{w\in W}(-1)^{w}w\qxq{on}\cE_t.$$ 
If the stabilizer $W_\mu$ of $\mu$ in $W$ is finite then the infinite sum
$$J(e^\mu)=\sum_{w\in W}(-1)^{\ell(w)}e^{w\mu}$$
 is well-defined in $\cE_{t}$ 
but $J(f)$ is not in general. For $\la\in\LLa_+$, we set
$$f_{\la}=e^{\la+\rho}\prod\limits_{\al\in \Delta_+}(1-te^{-\al})^{m_\al}\in\cE_{t}$$
with  $m_\alpha$ is the root multiplicity of $\alpha$.
By \cite[Prop.~1]{Vis}, the infinite sum $J(f_{\la})$ is well-defined in $\cE_t$.
Let $W_{\la}(t)$ be the formal series in $\bC\llb t\rrb$ given by
$$W_{\la}(t)=\sum\limits_{\sigma\in W_\la}t^{\ell(\sigma)}.$$
The series $W_{\la}(t)$ is invertible, because $W_{\la}(0)=1$.
We define the \emph{Hall-Littlewood function} to be
\begin{equation}
P_{\la}(t)=W_{\la}(t)^{-1}J(e^\rho)^{-1}J(f_\la).
\label{HLW}
\end{equation}
Note that the Weyl-Kac character formula yields
$$J(e^{\rho})^{-1}=e^{-\rho}\prod\limits_{\al\in R^+}(1-e^{-\al})^{-m_{\al}},$$
where we interpret 
$$(1-e^{-\al})^{-1}=1+e^{-\al}+e^{-2\al}+\dots.$$
Thus, we have
$P_{\la}(t)\in\cE_t.$
By \cite[(3.1)]{Vis}, there are formal series $c_{\la\mu}\in\bC\llb t\rrb$ such that
\begin{align}
P_{\la}(t)=\sum\limits_{\mu\in \Lambda^+}c_{\la\mu}(t)\chi_{\mu}.
\label{HL}
\end{align}
We need an explicit description of the coefficients $c_{\la\mu}(t)$.
Let $\cA$ be the set of all finite multisets of positive roots such that each $\al\in R^+$ occurs at most $m_{\al}$ times. 
By \cite[(3.2)]{Vis}, we have
\begin{align}
c_{\la\mu}(t)=W_{\la}(t)^{-1}\sum_{A\in\cA}\sum_{w\in W}(-1)^{\ell(w)}(-t)^{\sharp A},
\label{c-vis}
\end{align}
where the sum runs over all $w\in W$ and all elements $A\in\cA$
such that 
$$w\big(\la+\rho-\sum_{\al\in A}\al\big)=\mu+\rho.$$ 
We have the following properties.

\blem\label{c-prop}\hfill
 
\begin{enumerate}[label=$\mathrm{(\alph*)}$,leftmargin=8mm]
\item
 $c_{\la\mu}(t)\in\bZ[t]$ and when non-zero $\mu\leq\la$.
\item
$c_{\la\la}(t)=1$.
\item
$\mu<\la\Rightarrow c_{\la\mu}(0)=0$. In particular, $P_{\la}(0)=\chi_{\la}$.
\end{enumerate}
\elem

\bpf
The claim (a), (b) are proved in \cite[\S 3.1, 7.1]{Vis}.
For Part (c), we have $W_{\la}(0)=1$. Thus, in \eqref{c-vis} the only contributing $A$ if $t=0$
is $A=\emptyset$. This forces $w(\la+\rho)=\mu+\rho$. As $\mu$ is dominant, we get $\la=\mu$.
\epf

\subsubsection{Affine MacDonald formula}
We still assume that $G$ is symmetrizable.
Since the references we use work over $\bF_{q}\llp t\rrp$ rather that $\bF_{q}[t,t^{-1}]$, we must check that 
the countings over both rings give the same answer.

\blem\label{comp-Gr}
For any $\la\in\LLa_{+}$ and $\mu\in\LLa$, we have $(\Gr_{\la}\cap T_{\nu})(\bF_q)=(\Gr^\form_{\la}\cap T^\form_{\nu})(\bF_q)$.
\elem

\bpf
By Proposition \ref{ST-comp2}, we have a finite stratification of $\Gr_{\la}\cap T_{\nu}$ 
by triple intersections $\Gr_{\la}\cap S_{\mu}\cap T_{\nu}$ for $\la\geq\mu\geq\nu$.
By Lemma \ref{ST-compGr}, we know that
$T_{\nu}\cap S_{\mu}=T^\form_{\nu}\cap S^\form_{\mu}.$
By \eqref{S0} the triple intersections match (on both sides intersecting with $Gr_{\la}$ is a condition on poles using Plücker description).
So the strata are in bijection and we get an equality
\[(\Gr_{\la}\cap T_{\nu})(\bF_q)=(\Gr^\form_{\la}\cap T^\form_{\nu})(\bF_q).\]
\epf

Following  \cite[(4.1)]{Vis} we consider the series $\Delta(t)\in\cE_t$ given by
$$\Delta(t)=\prod\limits_{\al\in R^+}\left(\frac{1-te^{-\al}}{1-e^{-\al}}\right)^{m_\al}.$$
By \cite[\S 7.3]{BKP} in the affine untwisted case, 
and \cite[Thm.~7.3]{BGR} in the symmetrizable one, for each dominant cocharacter $\la\in \LLa_+$ the series 
$$H_{\la}(t)=W_{\la}(t)^{-1}\sum_{w\in W}w(\Delta)~ e^{w\la}$$
and $H_{0}^{-1}$ are well-defined in $\cE_t$ and the coefficients lie in $\bZ[t,t^{-1}]$.
We set
\begin{align}
\begin{split}
\Sat(1_{G_\lambda (\bF_q)})&=q^{ \langle\rho,\lambda\rangle}\frac{H_{\la}(q^{-1})}{H_0(q^{-1})}.
\label{sat2}
\end{split}
\end{align}
We now compare $H_{\la}$ with the Hall-Littlewood function $P_{\la}$.
By \cite[p.~172, \S10.2.2]{Kac}, for any $w\in W$ we have
$$wJ(e^{\rho})=(-1)^{\ell(w)}J(e^{\rho}).$$
We deduce that
\begin{align*}
P_{\la}(t)	
&=W_{\la}(t)^{-1}\sum_{w}(-1)^{\ell(w)}e^{w(\la+\rho)}J(e^\rho)^{-1}\prod\limits_{\al\in R^+}(1-te^{-w\al})^{m_{\al}},\\
&=W_{\la}(t)^{-1}\sum_{w}w\Big(e^{\la+\rho}J(e^\rho)^{-1}\prod\limits_{\al\in R^+}(1-te^{-w\al})^{m_{\al}}\Big),\\
&=W_{\la}(t)^{-1}\sum_{w}w\Big(e^{\la}\prod\limits_{\al\in R^+}\frac{(1-te^{-w\al})^{m_{\al}}}{(1-e^{-\al})^{m_{\al}}}\Big),\\
&=H_{\la}(t).
\end{align*}
In particular,  we have
\begin{align}
\Sat(1_{G_\lambda (\bF_q)})=q^{ \langle\rho,\la\rangle}\frac{P_{\la}(q^{-1})}{P_0(q^{-1})}.
\label{sat3}
\end{align}
We need an alternative formula for $\Sat(1_{G_\lambda (\bF_q)})$ that involves the affine MV cycles. 
For an arbitrary KM group (not necessarily symmetrizable), the Satake transform involves Hecke paths.
By \cite[\S2.8.(1)]{BGR}, see also \cite[prop.~5.2]{GR2}, we have 
\[\Sat(1_{G_\lambda (\bF_q)})=\sum\limits_{\mu\leq\la}n_{\la}(\mu)q^{ \langle\rho,\mu\rangle}e^{\mu},\]
where $n_{\la}(\mu)=\sum_{\pi}S_{\infty}(\pi,\mu)$ is the sum over all Hecke paths $\pi$ of type $\la$ from 0 to $\mu$
of some integers $S_{\infty}(\pi,\mu)$ that involve line segments in the masure.
The only thing that matters for us is that Muthiah in \cite[\S4.5 and \S4.5.5]{Mut}  re-interprets this sum in the following way
\begin{align}
\Sat(1_{G_\lambda (\bF_q)})=\sum\limits_{\mu\in \LLa}
\sharp(\Gr_\lambda \cap T_\mu)_{_{GR}}(\bF_q) q^{\langle\rho,\mu\rangle} e^\mu
\label{sat0}
\end{align}
Note that loc.~cit.~is written is the untwisted affine case, but for \S 4, this restriction is not relevant.
The subscript $GR$ is for Gaussent-Rousseau and means that the numbers are computed inside $G(\bF_{q}\llp t\rrp)$ 
where double orbits are with respect to the group $G_x$  in \S\ref{Elem}.
By Lemma \ref{comp-Gr} and Proposition  \ref{int-comp}, we have
\[(\Gr_\lambda \cap T_\mu)_{_{GR}}(\bF_q)=(\Gr_\lambda \cap T_\mu)(\bF_q).\]
Thus, in \eqref{sat0} we can freely replace the GR's sets by ours.

\subsubsection{Proof of Theorem $\ref{thm:GT}$}
We can now prove the theorem.
By Lemma \ref{GrT1} the intersection $(\Gr_\la\cap T_\mu)_\red$ is empty if $\la\not\geqslant\mu$, and, else,
it is ft scheme over $\bZ$.
We must check that the dimension of $(\Gr_\la\cap T_\mu)_\red$ is $ \langle\rho,\la-\mu\rangle$ and that
the number of irreducible components of maximal dimension is $\dim L(\la)_\mu$.
By \eqref{HL}, Lemma \ref{c-prop}, \eqref{sat3} and \eqref{sat0} we have
$$\sharp(\Gr_\lambda\cap T_\nu)(\bF_q) q^{ \langle\rho,\nu-\lambda\rangle}=
\frac{1}{P_0(q^{-1})}\big(\dim L(\lambda)_\nu+\sum_{\mu<\lambda}c_{\lambda\mu}(q^{-1})\dim L(\mu)_\nu\big)$$
and the coefficient $c_{\lambda\mu}(q^{-1})$ is a polynomial in $q^{-1}$ without constant term.
Since $P_0(q^{-1})\to 1$ as $q\to\infty$, we deduce that
$$\lim_{q\to\infty}\sharp(\Gr_\lambda\cap T_\nu)(\bF_q) q^{ \langle\rho,\nu-\lambda\rangle}=\dim L(\lambda)_\nu$$
from which the theorem follows.

\subsection{Finite codimensional Mirkovic-Vilonen cycles}

The finite codimensional Mirkovic-Vilonen cycles satisfy the following analogue of Lemma \ref{GrT1}.

\blem\label{GrS1}
Let $w\in W$, $\la\in \LLa_+$ and $\mu\in \LLa$. 
\hfill
\begin{enumerate}[label=$\mathrm{(\alph*)}$,leftmargin=8mm]
\item
If $\Gr_\la\cap S_\mu\neq\emptyset\Rightarrow \mu\leq\la$.
\item
If $\Gr_{w\la}^\circ\cap S_\mu\neq\emptyset\Rightarrow \mu\geq w\la$.
\item
The formal analogues of $\mathrm{(a)}$ and $\mathrm{(b)}$ hold.
\eenum
\elem

\bpf
The proof is the same as for Lemma \ref{GrT1}.
Let us recall the argument for (b).
If
$\Gr_{w\la}^\circ\cap S_\mu\neq\emptyset$ then $(wS_\la)\cap S_\mu\neq\emptyset.$
Since $S_\mu$ is preserved by left translations by elements of the torus $T$ and since $wS_\la$ is the attractor of
the element $[s^{w\lambda}]$ in the affine Grassmannian $\Gr$, for the action  of the cocharacter $2w\Lrho$ of $T$, we deduce that 
for each closed point $x$ in $(wS_\la)\cap S_\mu$ the limit of $2w\Lrho(z)\cdot x$ as $z\to 0$ is equal to $[s^{w\lambda}]$
and belongs to the closure of $S_\mu$, hence $\mu\geqslant w\lambda$ by Proposition \ref{ST-comp1}.
\end{proof}

Following \S\ref{UU}, we define
$$U_{w}=U\cap w(U)
,\quad
\hU_{w}=\hU\cap w(\hU)
,\quad
U^w=U\cap w(U^-)
,\quad
\hU^w=\hU\cap w(U^-).$$
Proposition \ref{uw-comp} yields the following scheme and ind-scheme isomorphisms
\begin{align}\label{UhU}
\hU\cong \hU_w\times \hU^w
,\quad
U\cong U_w\times U^w
,\quad
U^w\cong U/U_w
\cong
\hU^w\cong \hU/\hU_w.
\end{align}
We abbreviate $\hU_{w,\cO}=(\hU_w)_\cO$ and $\hU_\cO^w=(\hU^w)_\cO$.
Define $U_{w,\cO}$ and $U_\cO^w$ similarly.
From \eqref{UhU} we deduce the following ind-scheme isomorphisms
\begin{align}\label{UhUO}
\hU_\cO\cong \hU_{w,\cO}\times \hU_\cO^w
,\quad
U_\cO\cong U_{w,\cO}\times U_\cO^w
,\quad
U_\cO^w\cong \hU_\cO^w.
\end{align}
Recall the notation $\hU^-_\la$ and $\hU_\la$ in \eqref{VVJ}.
Let $ H\subset w(\hU_\la)$ be the closed subgroup  given by
$$
 H=w(\hU_\la)\cap\hU_\cO=\hU_{w,\cO}\cap w\Ad_{s^{\la}}(\hU_{\pi^-}).$$
Note that $H=w(\hU_\la)\cap\hU_K$ because $w(\hU_\la)\subset \hG_{\co}$, and that
$$w(\hU_\la)/ H\cong w(\hU^{w^{-1}}_\cO\cap \Ad_{s^{\la}}(\hU_{\pi^-})).$$

\blem\label{GrS2}
Let $\la\in \LLa_+$ and $w\in W$. 
\hfill
\begin{enumerate}[label=$\mathrm{(\alph*)}$,leftmargin=8mm]
\item
There is a fp locally closed embedding $\widehat\Gr_\la\cap\hat S_{w\la}\subset \widehat\Gr_{w\la}^\circ$.
\item
The $w(\hU_\la)$-action on $\widehat\Gr_{w\la}^\circ$ by left translation on $[s^{w\la}]$ 
yields an isomorphism $w(\hU_\la)\cong \widehat\Gr_{w\la}^\circ$.
\item
The $ H$-action on $\widehat\Gr_\la\cap\hat S_{w\la}$ by left translation yields an isomorphism 
$ H\cong \widehat\Gr_\la\cap\hat S_{w\la}$.
\item 
$w(\hU_\la)/ H\cong\bA^{\langle\rho,\la-w\la\rangle}$.
\eenum
\elem

\bpf
By Proposition \ref{sch-form}, to prove (a) we must check that the image of $\widehat\Gr_\la\cap \hat S_{w\la}$
by the map $\phi$ in \eqref{phi} is contained $w\hU_\la P_\la^-/P_\la^-$.
The intersection $\widehat\Gr_\la\cap \hat S_{w\la}$ is the attractor locus of $[s^{w\la}]$ in $\widehat\Gr_\la$
for the cocharacter $2\Lrho$ of the torus $T$. 
Since the map $\phi$ commutes with the $T$-action, we deduce that
$\phi(\widehat\Gr_\la\cap \hat S_{w\la})$ is contained in the attractor locus $\hU_\la wP_\la^-/P_\la^-$
of the closed point $w$ in $ \hG/P_{\la}^{-}$
for the cocharacter $2\Lrho$, from which the inclusion 
$\phi(\widehat\Gr_\la\cap \hat S_{w\la})\subset w\hU_\la P_\la^-/P_\la^-$ follows, because
$\hU_\la wP_\la^-/P_\la^-\subset w\hU_\la P_\la^-/P_\la^-$.
In particular, we get $\widehat\Gr_\la\cap \hat S_{w\la}=\widehat\Gr_{w\la}^{\circ}\cap \hat S_{w\la}$.
Thus, we have a fp locally closed embedding $\widehat\Gr_\la\cap\hat S_{w\la}\subset \widehat\Gr_{w\la}^\circ$
by base change and Proposition \ref{ST-comp1}.
Part (b) follows from  Lemmas \ref{hU-orb}, \ref{I-orb} and Proposition \ref{sch-form}.
Part (c) is obvious, because
$$\widehat\Gr_\la\cap\hS_{w\la}=\widehat\Gr_{w\la}^\circ\cap\hS_{w\la}=
(w(\hU_\la)\cdot[s^{w\la}])\cap(\hU_K\cdot[s^{w\la}])=(w(\hU_\la)\cap\hU_K)\cdot[s^{w\la}]\cong H,$$
where the third equality follows from the fact that 
$(w(\hU_\la)\cdot[s^{w\la}])\cap(\hU_K\cdot[s^{w\la}])$ is the attractor locus of $[s^{w\la}]$ for the cocharacter $2\Lrho$ and
$w(\hU_\la)\cdot[s^{w\la}]\cong w(\hU_\la)$ by (b), hence the intersection is identified with 
the attractor locus of $1$ in $w(\hU_\la)$ for the cocharacter $2\Lrho$, which is $w(\hU_\la)\cap\hU_K$.
Parts (b) and (c) identify the inclusion $\widehat\Gr_\la\cap\hat S_{w\la}\subset \widehat\Gr_{w\la}^\circ$
with the inclusion $ H\subset w(\hU_\la)$, which is closed, thus parts (a), (b) and (c) are proved.
For part (d), note that 
$$w(\hU_\la)\cong \prod_{\al\in w(R^+)}\kg_{\al,\cO}\cap\Ad_{s^{w\la}}(\kg_{\al,\pi^-})
,\quad
 H\cong \big(\prod_{\al\in R^+}\kg_{\al,\cO}\big)\cap w(\hU_\la).
$$
Hence, we have
$$w(\hU_\la)/ H\cong \prod_{\al\in w(R^+)\cap R^-}\kg_{\al,\cO}\cap\Ad_{s^{w\la}}(\kg_{\al,\pi^-})
\cong\prod_{(\al,n)\in\Gamma}U_{\al,n}$$
where $\Gamma$ is the following set 
$$\Gamma=\{(\al,n)\in R^+\cap w^{-1}(R^-)\times\bN\,;\,n<\langle\al,\la\rangle\}$$
and, for each $\alpha\in R_{\re}^+$ and $n\in\bZ$, we define
\begin{align}\label{Uan}
U_{\al,n}=x_{\al}(s^n\bG_a).
\end{align}
Further, we have
\begin{align*}
\sharp\Gamma&=\sum_{\al\in w(R^+)\cap R^-}\langle\al,w\la\rangle
=\langle\!\!\!\!\sum_{\al\in w(R^+)\cap R^-}\!\!\!\!\!\!\!\al\,,\,w\la\rangle
=\langle w\rho-\rho\,,\,w\la\rangle=\langle\rho\,,\,\la-w\la\rangle
\end{align*}
\epf

Since $ H\subset\hU_\cO$, the $ H$-action on $\widehat\Gr_{w\la}^\circ$ preserves
the intersection $\widehat\Gr_{w\la}^\circ\cap \hat S_\mu$ for each  $\mu\in\LLa$
with $w\la\leq\mu\leq\la$. 
We consider the quotient stacks $\hX_\mu$ and $\hX$ given by
\begin{align}\label{XX}
\hX_\mu=[\widehat\Gr_{w\la}^\circ\cap \hat S_\mu/ H]\subset \hX=[\widehat\Gr_{w\la}^\circ/ H].
\end{align}
By the above, the stack  $\hX$ is representable by the affine space
$\bA^{\langle\rho,\la-w\la\rangle}$, while $\hX_\mu$ is a locally closed susbscheme and
$\hX_{w\la}\cong\Spec(\bZ).$
For a future use, let us record the following well-known fact.

\blem\label{GrS0}
Let $X_1$ be a separated scheme of ft, and $X_0\subset X_1$ a closed subscheme such that $X_1\setminus X_0$ is affine.
Let $Z_1\subset X_1$ be an irreducible component such that $Z_1\not\subset X_0$ and $Z_1\cap X_0\neq\emptyset$.
Let $Z_0\subset Z_1\cap X_0$ be an irreducible component.
Then $\dim Z_1=\dim Z_0+1$.
\qed
\elem

We can now prove the following.

\blem\label{GrS3}
For each $w\la\leq\mu\leq\la$,
the subscheme $\hX_\mu\subseteq\hX$
is equidimensional of codimension $\langle\rho,\la-\mu\rangle$.
\elem

\bpf
Let 
$\hX_d=\bigsqcup_{ \langle\rho,\mu\rangle\leqslant d} \hX_\mu$.
We have a sequence of closed embeddings of schemes
$$\Spec(\bZ)=\hX_{\langle\rho,w\la\rangle}\subseteq \hX_{1+\langle\rho,w\la\rangle}\subseteq\cdots
\subseteq \hX_{ \langle\rho,\la\rangle}=\bA^{\langle\rho,\la-w\la\rangle}$$
where the last equality follows from Lemma \ref{GrS1}(a) and the Iwasawa decomposition.
For all $d$ the open subset
$$\hX_{1+d}\setminus \hX_d=\bigsqcup_{ \langle\rho,\mu\rangle=1+d} \hX_\mu$$ 
is affine  by Theorem \ref{sinf-aff}.
For any irreducible component $Z_d\subset \hX_d$, either $Z_d$ is still an irreducible component of $\hX_{d+1}$, or there is
an irreducible component $Z_{d+1}$ containing strictly $Z_d$, in which case Lemma \ref{GrS0} applies.
Hence, in both cases, by descending induction we deduce that the codimension  of $Z_d$ in $\bA^{\langle\rho,\la-w\la\rangle}$ is smaller than 
$ \langle\rho,\lambda\rangle-d$.

Now, we prove that the codimension of $Z_d$ is greater than $ \langle\rho,\lambda\rangle-d$ by ascending induction.
If $d=\langle\rho,w\la\rangle$ the claim is obvious.
Assume that $d>\langle\rho,w\la\rangle$ and fix any irreducible component $Z_d\subset \hX_d$.
If $Z_d\subset \hX_{d-1}$ we are done by induction. 
Else, we claim that $Z_d\cap \hX_{d-1}\neq\emptyset$, and we conclude by applying Lemma \ref{GrS0}. 

To prove the claim we consider the map $\phi$ in \eqref{phi}.
By Proposition \ref{sch-form} we have $\phi(\widehat\Gr_{w\lambda}^\circ)=w\hU_\lambda P_\la^-$ in $\hG/P_\la^-$.
Let $Z'_d$ be the inverse image of $Z_d$ by the obvious projection
$\widehat\Gr_{w\lambda}^\circ\to \hX_{ \langle\rho,\la\rangle}$.
Since $Z'_d$ is preserved by the $\hU_\cO$-action, we deduce that $\phi(Z'_d)$ is a union of $\hU$-orbits
and is contained in the big cell $w\hU_\lambda P_\la^-$. 
The big cell contains a unique $\hU$-orbit.
Hence $\phi(Z'_d)=\hU w P_\la^-$.
Further, by definition of the map $\phi$, the image $\phi(x)$ of any point $x\in \widehat\Gr_\lambda$
is the limit $t\cdot x$ as $t\to 0$ for the  loop action of $\bG_m$.
Such a limit is indeed well defined by Lemma \ref{I-orb} and the open cover of $\widehat\Gr_\lambda$
by the $\widehat\Gr_{w\lambda}^\circ$'s.
Since $Z'_d$ is closed, this implies that the image of
$\phi(Z'_d)$ by the obvious immersion $\hG/P_\lambda^-\to\widehat\Gr_\la$, which is a section of $\phi$,
is contained in $Z'_d$.
Since $\hU w[s^\la]\subset\hat S_{w\la}$, we deduce that $Z'_d\cap \hat S_{w\la}\neq\emptyset$.
Hence $Z_d\cap \hX_{\langle\rho,w\la\rangle}\neq\emptyset$, and, a fortiori, 
we have $Z_d\cap \hX_{d-1}\neq\emptyset$, proving the claim and the lemma.
\epf

Now, we consider the non formal setting.
Lemma \ref{fib1} and Proposition \ref{g-cart} yield the isomorphisms
\begin{equation}
\widehat\Gr_{w\la}^\circ\times_{\Gr_{\hG}}\Gr_{G}\cong\Gr_{w\la}^\circ
,\quad
(\widehat\Gr_{w\la}^\circ\cap \hat S_{\mu})\times_{\Gr_{\hG}}\Gr_{G}\cong\Gr_{w\la}^\circ\cap S_{\mu}
\label{bla-for2}
\end{equation}
Recall the notation $\hU_\la$ and $U_\la$ in \eqref{VVJ}, \eqref{KV}. Lemma \ref{I-orb} and Proposition \ref{unif} yield
$$\widehat\Gr_{w\la}^\circ\cong [w\hU_\cO/w\hU_\la]
,\quad
(\Gr_{w\la}^\circ)_\red\cong [wU_\cO/wU_\la]_\red.$$ 
By \eqref{UhUO}, we have the following ind-scheme isomorphisms
\begin{align*}
wU_\cO
&\cong w(U_{w^{-1},\cO}\times \hU_\cO^{w^{-1}})\\
&\cong(U_\cO\cap wU_\cO)\times w (\hU_\cO^{w^{-1}})\\
&\cong(U_\cO\cap wU_\cO)\times(U_\cO^-\cap w\hU_\la)\times
(U_\cO^-\cap w\hU_\cO\cap w\Ad_{s^\la}(\hU_{\pi^-}))\\
&\cong(U_\cO\cap wU_\cO)\times (U_\cO^-\cap w\hU_\la)\times\hX
\end{align*}
Let $\hX$ and $\hX_\mu$ be as in \eqref{XX}.
We consider the group ind-schemes of ind-ft $H_1$, $H_2$ given by
$$H_1=wU_\la
,\quad
H_2=(U_\cO\cap wU_\cO)\times (U_\cO^-\cap w\hU_\la).$$
Let $p_1:wU_\cO\to\Gr_{w\la}^\circ$ and $p_2:wU_\cO\to\hX$ be the quotients by the groups $H_1$ and $H_2$.
We have the following commutative diagram with Cartesian squares, that we will need in the proof of Theorem \ref{t-exact}
\begin{align}\label{diag2}
\begin{split}
\xymatrix{
\Gr_{w\la}^\circ\cap S_\mu\ar[d]_-{i_0}&\ar[l]_-{q_1}Y_\mu\ar[r]^-{q_2}\ar[d]_-{i_1}&\hat X_\mu\ar[d]^-{i_2}\\
\Gr_{w\la}^\circ&\ar[l]_-{p_1}wU_\cO\ar[r]^-{p_2}&\hat X
}
\end{split}
\end{align}
Note that the group ind-schemes
$H_1$ and $H_2$ have a $\bG_m$-equivariant contracting presentation,
as in Lemma \ref{pul-push} below.
The vertical maps $i_0$, $i_1$, $i_2$ are the obvious locally closed immersions.

\section{Perverse sheaves on the affine Grassmannian of a KM group}

Let $k$ be an algebraically closed field.
In this section all prestacks are defined over the field $k$.

\subsection{Category of \texorpdfstring{$\ell$}{l}-adic sheaves}
This section is a reminder on the category of sheaves on $\infty$-stacks. 
The main reference is \cite{BKV}.
Let $\ell$ be a prime different from the characteristic of $k$.

\subsubsection{Definition of the category of \texorpdfstring{$\ell$}{l}-adic sheaves}\label{l-ad}

For an algebraic space $Y$ of ft, we have 
an $\infty$-category $\cD_{c}(Y)=\cD^{b}_{c}(Y,\bql)$ whose homotopical category is the derived category
$D^{b}_{c}(Y,\bql)$ of bounded complexes of sheaves with constructible cohomology,
see \cite{LZ1}, \cite{LZ2} or \cite{GL}.
We will use the formalism of six functors for the $\infty$-categories $\cD_c$.
This yields a functor
\begin{equation}
\cD_{c}:(\AlgSp^\ft_k)^{\op}\rightarrow\StCat
\label{tfdef1}
\end{equation}
which  to each morphism $f:X\rightarrow Y$ associates $f^{!}:\cD_{c}(Y)\rightarrow\cD_{c}(X)$.
We define the functor
\begin{equation}
\cD:(\AlgSp^\ft_k)^{\op}\rightarrow\PrCat
\label{tfdef2}
\end{equation}
to be $\cD=\Ind\circ\cD_{c}$. 
The embedding 
$\iota:\Aff_k\hra\AlgSp_k$ yields an equivalence of  $\infty$-categories 
$$\iota^{*}:\Sh(\Aff_k)\ra\Sh(\AlgSp_k).$$
Thus, to construct the $\infty$-category $\Stk_k$ we could use
the category $\AlgSp_k$ instead of the category $\Aff_k$ as in \eqref{PSt}.
The same applies to the category $\AlgSp_k^\qcqs$ of qcqs algebraic spaces.
By doing left Kan extensions, the functors $\cD_c$ and $\cD$ yield the  functors
\begin{equation}
\cD_{c}:(\AlgSp_k^\qcqs)^\op\rightarrow\StCat
,\quad
\cD:(\AlgSp_k^\qcqs)^\op\rightarrow\PrCat
\label{kan1}
\end{equation}
By doing right Kan extension we get the  functors
\begin{equation}
\cD_{c}:\PreSh(\AlgSp_k^\qcqs)\rightarrow\StCat
,\quad
\cD:\PreSh(\AlgSp_k^\qcqs)\rightarrow\PrCat
\label{kan2}
\end{equation}
Both Kan extensions exist by \cite[Thm.~5.1.5.6]{Lu1},
because the category $\PrCat$ is bicomplete, see, e.g., \cite[\S2.2]{RS}.
Let $\cD_{\bullet}$ be either $\cD$ or $\cD_c$.

\bprop\label{sheaf}
The functor $\cD_\bullet$ is a sheaf for the étale topology, i.e., it factors through  
$\Sh(\AlgSp_k)$ via  the sheafifying functor.
\eprop

\bpf
Let $\pmb\Delta$ be the semi-simplicial category, whose objects are the $[m]$'s with $m\in\bN$
and whose morphisms $[m]\to[n]$ are the injective monotone maps $\{0,\dots,m\}\to\{0,\dots,n\}$.
To any fp étale covering of qcqs algebraic spaces
$\pi:X\ra Y$  we associate a semi-simplicial algebraic space $(X^{[n]})$
such that 
\begin{align}\label{[n]}\pi^{[n]}:X^{[n]}=X\times_{Y}\dots \times_{Y}X\to Y\end{align} 
is the obvious projection and $X^{[n]}$ is the $(n+1)$-th fiber power of $X$ over $Y$.
We must prove that the functor
\begin{align}\label{upi}\underline{\pi}^!:\cD_\bullet(Y)\to\lim_{t^!}\cD_\bullet(X^{[n]})\end{align}
given by the system of functors $(\pi^{[n]})^!$ 
is an equivalence. 
Here $[n]$ runs over the objects of the category $\pmb\Delta^\op$ opposite to $\pmb\Delta$.
If the map $\pi$ has a section the claim is standard. 
The functor $\underline{\pi}^!$ has a right adjoint $\underline{\pi}_*$.
Indeed, the map $\pi^{[n]}$ is fp \'etale.
So $(\pi^{[n]})^!$ has a right adjoint $(\pi^{[n]})_*$. 
Therefore $\underline{\pi}^!$ has a right adjoint $\underline{\pi}_*$, which sends the object
$\underline{K}=(K_n)$ in $ \lim_{[n]}\cD(X^{[n]})$ to the limit $\lim_{n}(\pi^{[n]})_*(K_n)$.

We claim that the unit $K\to\underline{\pi}_*\underline{\pi}^!K$ is an isomorphism, i.e., the map
$K\to \lim_{[n]}(\pi^{[n]})_*(\pi^{[n]})^!(K)$ is an isomorphism. Since $\pi^!$ is faithful, it suffices to check the isomorphism after we
apply $\pi^!$. Since $\pi^!$ commutes with limits, because it has a left adjoint $\pi_!$, and with $(\pi^{[n]})_*$, by \cite[Prop.~5.2.7.(a)]{BKV},
we are reduced to the corresponding assertion for the projection $\pi^+:X\times_Y X\to X$. 
Since it has a section, which is the diagonal $X\to X\times_Y X$,
we are done.

Finally, we claim that the counit $\underline{\pi}^!\underline{\pi}_*(\underline{K})\to \underline{K}$ is an isomorphism. It suffices to show that
the map $\pi^!(\lim_{[n]}(\pi^{[n]})_*(K_n))\to K_0$ is an isomorphism. As above, the assertion follows from the commutativity of
$\pi^!$ commutes with limits and $(\pi^{[n]})_*$. 
The claim is proved for $\cD$. To prove it for $\cD_c$, note that the functor $\pi^{!}$ preserves $\cD_c$.
\epf

Proposition \ref{sheaf} allows us to see $\cD_{\bullet}$ as a functor on $\Stk_k$ such that
$$\cD_\bullet(\cX)=\lim\cD_\bullet(X_a)$$
where the limit runs over all schemes $i_a:X_a\to\cX$ over $\cX$.
Unwinding the definition, an object $\cE$ in $\cD_\bullet(\cX)$ is the datum of a sheaf $\cE_{i_a}$ in $\cD(X_a)$ for each $i_a$ as above,
with an equivalence of sheaves $\cE_{f(i_a)}\stackrel{\sim}{\rightarrow} f^!\cE_{i_a}$ for each morphism of schemes $f$, 
satisfying a homotopy-coherent system of compatibilities.

For any stack $\cX\in\Stk_k$ the inclusion $\cD_c(\cX)\to\cD(\cX)$ yields a functor
$\Ind(\cD_c(\cX))\to\cD(\cX)$. It is an equivalence if $\cX\in\AlgSp_k^\qcqs$ but may not be an equivalence in general.

\bexa\hfill
\begin{enumerate}[label=$\mathrm{(\alph*)}$,leftmargin=8mm]
\item
If $X\in\AlgSp_k^\qcqs$ has a presentation as a cofiltered limit $X\cong\lim X_a$ with $X_a\in\AlgSp^\ft_k$ 
and affine transition morphisms, 
then we have 
$$\cD_{\bullet}(X)\cong\colim_{t^!}\cD_\bullet(X_a).$$
\item
Let $X=\colim X_a$ be an ind-algebraic space with $X_a\in\AlgSp_k$.
Then we have 
\begin{align}\label{D-ind-sch1}\cD_{\bullet}(X)\cong\lim_{t^!}\cD_{\bullet}(X_a).\end{align}
\eenum
\eexa

\subsubsection{The bar construction}
Given a group ind-scheme $H$ acting on an ind-scheme $X$ the quotient $\infty$-stack $[X/H]$ is equipped with
the stable $\infty$-category $\cD(X/H)$. By definition the $\infty$-prestack $[X/H]$ is the (homotopy) colimit of the bar construction
$\text{Bar}(H,X)$, which is the semi-simplicial $\infty$-prestack built out of action and projection maps 
$$\text{Bar}(H,X)=\Big(\xymatrix{
	\dots \ar@<-1.5ex>[r] \ar@<-0.5ex>[r] \ar@<0.5ex>[r] \ar@<1.5ex>[r]
        &
	H^2 \times X
	\ar@<-1ex>[r]
	\ar[r]
	\ar@<1ex>[r]
	&
	H \times X
	\ar@<0.5ex>[r]^{\phantom{hh}a}
	\ar@<-0.5ex>[r]_{\phantom{hh}p_2}
	&
	X
	}\Big).$$
There is an equivalence
\[\cD(X/H)=\lim_{t^!}\cD(\text{Bar}(H,X)).\]
Thus, an $H$-equivariant complex on $X$ is a collection of objects
$\cE_n\in\cD(H^n\times X)$ with $n\in\NN$, together with equivalences 
$a^!\cE_0\cong \cE_1$, $p_2^!\cE_0\cong\cE_1$, etc., one for any map in $\text{Bar}(H,X)$, 
subject to the compatibility conditions given by the relations between the maps.

\subsubsection{Lurie's adjunction}
Let $\cI$ be a small category and $\cI\rightarrow\PrCat$ a functor.
For each $i\in \cI$, we are given an $\infty$-category $\cC_{i}$ and for each morphism $\alpha:i\to j$
a continuous functor $\phi_\al:\cC_{i}\to\cC_{j}$.
Suppose that for each morphism $\al$ the functor $\phi_{\al}$ has a continuous right adjoint $\psi_{\al}$.
Since adjoints are compatible with compositions, the datum $(\cC_i,\psi_{\al})$ extends to a functor
$\cI^{\op}\rightarrow\PrCat$, see \cite[Cor.~5.5.3.4]{Lu1}.
We consider the  limit
$\wh{\cC}=\lim_{i} \cC_{i}$
with the evaluation functor
$\ev_{i}:\wh{\cC}\to\cC_{i}$
for each $i\in\cI$.

\bthm \label{T:lu1}
\hfill
\begin{enumerate}[label=$\mathrm{(\alph*)}$,leftmargin=8mm]
\item
The colimit $\cC=\colim_{i} \cC_{i}$ exists in $\PrCat$. It is equivalent to $\wh{\cC}$.
\item
The equivalence $\cC\cong \wh{\cC}$ is characterized by the condition that the evaluation functor
$\ev_{i}:\wh{\cC}\to\cC_{i}$ is the right adjoint to the tautological functor
$\ins_{i}:\cC_{i}\to\cC$ for each $i\in \cI$.
\eenum
\qed
\ethm

\brem
Assume $\cI$ is filtered.
In this case \cite{Roz} gives another description of the equivalence $\cC\cong \wh{\cC}$, because
for each $i,j\in\cI$ the composition $\ev_j\circ\ins_i:\cC_i\to\cC\cong\wh{\cC}\to\cC_j$ coincides with the colimit
\[\ev_j\circ\ins_i\cong\colim_{\al,\beta}\phi_{\beta}\circ\psi_{\al}\] 
over all $\al:i\to k$ and $\beta:j\to k$.
\erem

\bcor \label{C:lurie}
For each object  $c\in\cC$, the assignment $i\mapsto \ins_{i}\circ\ev_{i}(c)$ yields a functor $\cI\to\cC$.
The obvious map 
$\colim_{i}\ins_{i}\circ\ev_{i}(c)\to c$
is an isomorphism.
\qed
\ecor

\bexa\label{exa-ind}
If $\cX\cong\colim\cX_a$ is a colimit of prestacks then we have
$$\cD(X)\cong\lim_{t^!}\cD(X_a).$$
If the morphisms $i_a:\cX_a\to\cX$ are such that the functor
$(i_a)^!$ has a left adjoint $(i_a)_!$, then we also have
\begin{align}\label{D-ind-sch2}\cD(X)\cong\colim_{t_!}\cD(X_a).\end{align}
Further, for each $\cE\in\cD(X)$ we have
$$\cE\cong\colim (i_a)_!(i_a)^!(\cE).$$
In particular, assume that $X=\colim X_a$ is a reasonable ind-algebraic space of ind-fp type, see Definition \ref{d-indsch}.
For each fp closed embedding $i$ in $\AlgSp_k^\qcqs$ the functor $i^!$ has a left adjoint $i_!$. 
Hence  \eqref{D-ind-sch1} and Theorem \ref{T:lu1} imply that  we have an equivalence
as in \eqref{D-ind-sch2}.

\eexa

Finally, we consider some adjunctions in limits or colimits, following \cite[\S5.1.7]{BKV}.
Let $\Cat$ \label{N:catell} be either $\StCat$ or $\PrCat$. 
Let $\cI$ be a small $\infty$-category.
Next, let $\cD_{.}$ and $\cC_{.}$ be two functors $\cI\to\Cat$ and
let $\Phi_{.}:\cC_{.}\to\cD_{.}$ be a morphism of functors. It is given by
a functor  $\Phi_i:\cC_i\to\cD_i$ for each $i\in\cI$ and
an equivalence $\Phi_{\al}:\cD_{\al}\circ\Phi_i\cong \Phi_j\circ\cC_{\al}$ for each morphism
$\al:i\to j$ in $\cI$. 
Let $\widehat{\Phi}=\lim_{i}\Phi_i$ be the limit functor $\wh\cC\to \wh\cD$.
If $\cI$ is filtered, let also $\Phi=\colim_i\Phi_i$ be the colimit functor $\cC\to\cD$.
Assume that
\begin{enumerate}[label=$\mathrm{(\alph*)}$,leftmargin=8mm]
\item
for each $i\in\cI$ the morphism $\Phi_i:\cC_i\to\cD_i$ has a left adjoint $\Psi_i$,

\item 
for each morphism $\al:i\to j$ in $\cI$ the base change morphism
$bc_{\al}:\Phi_j\circ \cD_{\al}\to\cC_{\al}\circ \Psi_i$, obtained from the counit map by adjointness
$$\cD_{\al}\to \cD_{\al}\circ\Phi_i\circ\Psi_i\cong  \Phi_j\circ\cC_{\al}\circ \Psi_i$$  is an equivalence
(Beck--Chevalley condition).
\eenum

\bprop\label{adjlu}
Suppose the assumptions above hold.\hfill
\begin{enumerate}[label=$\mathrm{(\alph*)}$,leftmargin=8mm]
\item
The collection of $\Psi_i$'s and $bc_{\al}$'s define a morphism of functors $\Psi:\cD\to \cC$.

\item
The limit functor $\widehat{\Phi}$ has the left adjoint $\wh{\Psi}$.
For each $i\in\cI$ the base change morphism is an equivalence
$\Psi_i\circ\ev_i^{\cD}\rightarrow\ev_i^{\cC}\circ\widehat{\Psi}$.

\item Assume that $\cI$ is filtered. 
The colimit functor $\Phi$ has the left adjoint $\Psi$. 
For each $i\in\cI$ the base change morphism is an equivalence 
$\Psi\circ\ins_i^{\cD}\rightarrow\ins_i^{\cC}\circ~\Psi_{i}$
\qed
\end{enumerate}
\eprop

\subsection{Cohomological operations}
\subsubsection{General functoriality}\label{funct}
For each morphism of $\infty$-stacks $f:\cX\ra\cY$, we have a functor $f^!:\cD_{\bullet}(\cY)\to\cD_\bullet(\cX)$.
If $f$ is a topological equivalence then the functor $f^!$ is an equivalence of categories by \cite[Cor.~5.3.6]{BKV}.
If $f$ is an ind-fp ind-proper morphism, then the functor $f^!$ admits a left adjoint $f_!$ 
that satisfies base change by \cite[Prop.~5.3.7]{BKV}.
If $f$ is a topological fp locally closed immersion, then we have a functor $f_*$ by \cite[\S5.4.4]{BKV}.
We are interested in the existence of a left adjoint $f^*$.
To do so, we consider an intermediate class between ind-schemes and $\infty$-stacks that contains $[X_K/H]$ 
for any group ind-scheme $H$ of ind-ft over $k$ which acts on $X_K$, for any $X\in\Aff^\ft_k$.

\bdefi\label{glu}
Following \cite[Def.~5.5.1]{BKV},
we say that an $\infty$-stack $\cY$ satisfies gluing if for each topological fp  immersion 
$f:\cX\hra\cY$, the functor $f_{*}$ has a left adjoint $f^*$.
\edefi 

\brem\label{glustk}\hfill
\begin{enumerate}[label=$\mathrm{(\alph*)}$,leftmargin=8mm]
\item
If $\cY$ satisfies gluing, then so does each topologically fp-immersion 
$f:\cX\hra\cY$ by \cite[Lem.~5.5.5.(a)]{BKV}.
\item
If $\cX\cong\colim \cX_{\al}$ where each $\cX_{\al}$ satisfies gluing and transition maps are qc open immersions
or closed fp immersions, then  $\cX$ satisfies gluing, see \cite[Lem.~5.5.5.(b)]{BKV}. 
Note that any open immersion is of finite presentation if and only if it is qc by \cite[Tag 01TU]{Sta}.
In particular, it is the case of ind-ft-schemes.
\item
The quotient $[X/H]$ of an ind-ft-algebraic space by an ind-ft group ind-scheme $H$ satisfies gluing,
see \cite[Prop.~5.5.7]{BKV}.  In particular, the Grassmannian $\Gr_{G}$ statisfies gluing for any minimal KM group $G$.
\eenum
\erem

\subsubsection{Functoriality for lft prestacks}\label{ss-fonct}
We will restrict to a class of prestacks for which we have more operations. 

\bdefi\label{lft}
A prestack $\cX$ is locally of finite type (lft)  if 
$\cX=\colim_{S\ra \cX} S$
with
$S\in\Aff^\ft_{k}.$
\edefi

\brem\label{rem-lft}
A quotient $[X/H]$ is a  lft prestack if $X$ is an ind-ft ind-scheme and $H$ is an ind-ft 
group ind-scheme acting on $X$, due to the bar construction.
\erem

\bprop\label{prop-lft}
Let $f:\cX\ra\cY$ and $g:\cZ\ra\cY$ be morphisms of lft prestacks.
\hfill
\begin{enumerate}[label=$\mathrm{(\alph*)}$,leftmargin=8mm]
\item
The functor $f^{!}$ has a left adjoint $f_!$.
\item
If $f$ is ft schematic, there is a continuous adjoint pair $(f^*,f_*)$
with base change equivalences for $(f_*,g^{!})$ and $(f^{*},g_!)$.
\item
If $f$ is  ind-ft ind-schematic, there is a continuous functor $f_*:\cD(\cX)\to\cD(\cY)$ 
with base change equivalences for $(f_*,g^{!})$.
We have a left adjoint at the level of pro-categories
\[f^{*}:\cD(\cY)\ra\Pro(\cD(\cX)).\]
\eenum
\eprop

\bpf
Part (a) is \cite[Cor.~1.4.2]{Ga}.
Part (b) is \cite[\S2.6-2.8]{Ho}.
Note that all stacks are lft in Gaitsgory's work by  \cite[\S0.8.1]{Ga}, as well as in
Ho's work which uses the same conventions as \cite{Ga}, 
see \cite[\S 2.1]{Ho}.
To extend $f_*$ from ft-schematic to ind-ft-schematic, if $\cY$ is a ft-scheme, then using Lurie's adjunction, we write 
$\cD(\cX)\cong\colim_{S\hra \cX,i_*}\cD(S)$ where $S$ runs over the ft-closed subschemes of $\cX$ and we get a functor 
$f_*$ that satisfies base change  $(f_*,g^{!})$, thus it extends to a functor $f_*$ for any ind-ft-schematic morphism of 
lft prestacks that continues to satisfy base change.
The second one is proved in \cite[App.~A3]{DG1}.
\epf

\brem
Given lft prestacks $\cX$ and $\cY$, the Lurie adjunction yields equivalences
$$\cD(\cX)\cong\colim\cD(X_a),\quad \cD(\cY)\cong\colim\cD(Y_a)$$ 
where the colimits run over all ft schemes $\la_a:X_a\to\cX$ and $\mu_a:Y_a\to \cY$.
Let $f:\cX\to\cY$ be a morphism such that $f=\colim f_a$ where $(f_a)$ is a system of morphisms of schemes  $f_a:X_a\to Y_a$.
By  \cite[Cor.~2.8.5]{Ho} we have $$f^*=\colim (\la_a)_!(f_a)^*(\mu_a)^!.$$
Further, if $f:\cX\to Y$ is a morphism from a lft prestack to a ft scheme and $f_a=f\mu_a$, then 
 \cite[Lem.~2.5.9]{Ho} and \cite[Proof of Prop.~1.5.2]{Ga} yield
$$f_!=\colim (f_a)_!(\mu_a)^!.$$
\erem

\subsubsection{Monoidal structure}
Let $f_i:X_i\ra Y_i$ be morphisms of ft schemes with $i=1,2$.
Let $\star$ denote either $!$ or $*$.
The Künneth formula \cite[Exp.~III, (1.6.1), Prop.~1.7.4]{SGA5} yields an isomorphism
\begin{equation}
(f_1\times f_2)^\star(K_1\boxtimes K_{2})\cong f_{1}^\star K_1\boxtimes f_{2}^\star K_{2}
\label{box1}
\end{equation}
and \cite[Exp.~III, \S 1.6.2, (1.7.1)]{SGA5} yields an isomorphism
\begin{equation}
(f_{1}\times f_2)_\star(K_1\boxtimes K_2)\cong(f_1)_\star K_1\boxtimes (f_2)_\star K_2.
\label{box2}
\end{equation}
Using the definition of $\cD$ as a colimit, this allows us to define 
for arbitrary affine schemes $Y_1$, $Y_2$ and for any objects $K_i\in\cD(Y_i)$ the external tensor product $K_{1}\boxtimes K_2$. 
For any prestacks $\cY_i$, not necessarily lft, we define the external tensor product such that \eqref{box1} holds for pull-back 
by morphisms $S_i\ra\cY_i$ for $S_i$ being affine schemes, see also \cite[\S1.3.4]{Ga}
or \cite[\S2.10]{Ho}.
Taking the first projection $p:\cX\times\cY\ra\cX$, we obtain 
\begin{equation}
p^{!}K\cong K\boxtimes\omega_{\cY}.
\label{kun1}
\end{equation}
Indeed it is enough to check \eqref{kun1} on schemes after pulling back to $S_1\times S_2\ra\cX\times \cY$.
Let $ \widehat\boxtimes$ denote the external tensor product of pro-sheaves. 

\blem
Let $f_i:\cX_i\ra \cY_i$ be   ind-ft ind-schematic morphisms, between lft prestacks for $i=1,2$.
\begin{enumerate}[label=$\mathrm{(\alph*)}$,leftmargin=8mm]
\item
We have an isomorphism
\begin{equation}
(f_{1}\times f_2)^{*}(K_1\boxtimes K_2)\cong f_{1}^{*}K_1\widehat\boxtimes f_{2}^{*}K_2.
\label{kun3}
\end{equation}
\item
If in addition $\cY_i$ are ind-schemes of ind-ft, we have an isomorphism
\begin{equation}
(f_{1}\times f_2)_!(K_1\boxtimes K_2)\cong(f_1)_!K_1\boxtimes (f_2)_!K_2.
\label{kun2}
\end{equation}
\eenum
\elem

\bpf
Let first assume that $\cY_i=Y_i$ are ft schemes.
Thus $\cX_i=X_i$ are ind-schemes of ind-ft.
We write $X_i\cong\colim_a X_{i,a}$ for a system of maps $\la_{i,a}:X_{i,a}\to X_i$ where $X_{i,a}$ are ft-schemes.
Hence, for any $K_i\in\cD(Y_i)$, the formula \eqref{box1} yields
\begin{align}
\label{kun3a}
\begin{split}
(f_1\times f_2)^{*}(K_1\boxtimes K_2)
&\cong\quotlim((f_{1,a}\times f_{2,a})^{*}(K_1\boxtimes K_2))\\
&\cong\quotlim((f_{1,a})^{*}K_1\boxtimes (f_{2,a})^*K_2)\\
&\cong f_{1}^{*}K_1\widehat\boxtimes f_{2}^{*}K_2
\end{split}
\end{align}
with $f_{i,a}=f_i\circ\la_{i,a}$ a morphism of ft schemes.
For the second claim,
Lurie's adjunction yields
\[\cD(X_i)\cong\colim_{(\la_{i,a})_!}\cD(X_{i,a}),\quad
\cD(X_1\times X_2)\cong\colim_{(\la_a)_!}\cD(X_{1,a}\times X_{2,a})\]
where $\la_a=\la_{1,a}\times\la_{2,a}$.
Any complex $K_i\in\cD(X_i)$ admits the presentation
$$K_i\cong\colim_a (\la_{i,a})_!(\la_{i,a})^{!}K_i.$$
Since taking $!$-pushforwards is a left adjoint functor, it is continuous.
Thus, we have
\begin{align*}
(f_1\times f_2)_!(K_1\boxtimes K_2)
&\cong(f_1\times f_2)_!\colim_a(\la_a)_!\la_a^!(K_1\boxtimes K_2)\\
&\cong\colim_a(f_{1,a}\times f_{2,a})_!\la_a^!(K_1\boxtimes K_2)\\
&\cong\colim_a(f_{1,a}\times f_{2,a})_!(\la_{1,a}^!K_1\boxtimes \la_{2,a}^!K_2)\\
&\cong\colim_a((f_1)_!(\la_{1,a})_!\la_{1,a}^!K_1\boxtimes (f_2)_!(\la_{2,a})_!\la_{2,a}^!K_2)\\
&\cong(f_1)_!K_1\boxtimes (f_2)_!K_2
\end{align*}
where the third isomorphism follows from \eqref{box1} and the fourth one from \eqref{box2}.

Consider now the general case. 
Fix a presentation 
$$\cY_i\cong\colim_a Y_{i,a},$$ for a system of maps
$\mu_{i,a}:Y_{i,a}\ra\cY_i$ with $Y_{i,a}$ a scheme of ft.
This yields the presentation 
$$\cX_i\cong\colim_a X_{i,a}$$ 
for the system of maps $\la_{i,a}:X_{i,a}\to\cX_i$
where $X_{i,a}=\cX_i\times_{\cY_i}Y_{i,a}$
is an ind-ft ind-scheme.
Lurie's adjunction yields
\begin{align*}
K_{1}\boxtimes K_2
&\cong \colim_a(\la_a)_{!}(\la_a)^{!}(K_1\boxtimes K_2)\\
&\cong \colim_a(\la_a)_{!}((\la_{1,a})^{!}K_1\boxtimes (\la_{2,a})^{!}K_2).
\end{align*}
The first isomorphism is as in Example \ref{exa-ind}, the second one
is \eqref{box1} extended to lft prestacks.
The claim \eqref{kun3} follows by proper base change and \eqref{kun3a}.
Indeed, since taking $*$-pullbacks is a left adjoint functor by Proposition \ref{prop-lft},
it commutes with colimits by \cite[Cor.~5.5.2.9]{Lu1}.
Since we work with pro-sheaves, to check \eqref{kun3} we check  it term by term, hence the claim follows from \eqref{kun3a}.
\epf

\subsubsection{Projection formula}
Let $\cX$ be any prestack. For any $K,L\in\cD(\cX)$, we set
\[K\overset{!}{\otimes}L=\Delta^{!}(K\boxtimes L).\]
By construction, for any morphism of prestacks $\phi:\cY\ra\cX$, we have
\begin{equation}
\phi^{!}(K\overset{!}{\otimes}L)=\phi^{!}K\overset{!}{\otimes}\phi^{!}L
\label{tir-tens}
\end{equation}

\blem\label{p-form}
Let $\phi:\cY\ra\cX$ be any ind-ft ind-schematic morphism of lft prestacks.
For any $K\in\cD(\cX)$ and $L\in\cD(\cY)$ we have un isomorphism
$\phi_*(\phi^!K\overset{!}{\otimes}L)\cong K\overset{!}{\otimes}\phi_*L.$
\elem

\bpf
Using base change and \eqref{tir-tens}, 
we reduce to the case where $\cX$ is a ft scheme.
Then the assertion is standard and follows by Verdier duality from the usual projection formula, as in this case $K\overset{!}{\otimes}L=\bD(\bD(K)\otimes\bD(L))$.
\epf


\subsubsection{Base change identities}

\bprop\label{star-calc}
Let $H$ be a group ind-scheme of ind-ft,
$f:\uucX\ra\cX$ be an $H$-torsor between lft prestacks, and
$\phi:\cY\ra\cX$ be an ind-ft ind-schematic morphism of lft prestacks.
We form the Cartesian diagram
\begin{align}\label{CD1}
\begin{split}
\xymatrix{\uucY\ar[r]^{\uuphi}\ar[d]_{g}&\uucX\ar[d]^\form\\\cY\ar[r]^{\phi}&\cX}
\end{split}
\end{align}
\begin{enumerate}[label=$\mathrm{(\alph*)}$,leftmargin=8mm]
\remi
The obvious morphism is an isomorphism
\begin{equation}
 \uuphi^{*}f^{!}\cong g^{!}\phi^{*}.
\label{star-1}
\end{equation}
\remi
If $\uucX$ is an ind-ft ind-scheme, then
\benumr
\remi
the canonical map is an equivalence
\begin{equation}
\uuphi^{!}f^{*}\cong g^{*}\phi^{!},
\label{wsm2}
\end{equation}
\remi
the base change gives equivalences
\begin{equation}
 f^{!}\phi_{!}\cong\uuphi_{!}g^{!}
,\quad 
f^{*}\phi_{*}\cong \uuphi_{*}g^{*}.
\label{shriek-1}
\end{equation}
\eenum
\eenum
\eprop

\bpf
The bar-resolution gives a presentation of the stack $\cX=[\uucX/H]$.
Hence 
$$\cD(\cX)\cong\lim_{[n],t^!} \cD(H^n\times\uucX).$$
Similarly, we have a stack isomorphism $\uucX=[H\times\uucX/H]$ which yields an 
equivalence 
$$\cD(\uucX)=\lim_{[n],t^!} \cD(H^{n+1}\times\uucX).$$
The transitions morphisms $H^n\times\uucX\to H^m\times\uucX$ decompose as isomorphisms and projections of the form
$H\times\cZ\to\cZ$ as in \cite[Prop.~5.5.7]{BKV}.
Thus \eqref{star-1} follows from the base change property for the Cartesian 
diagram of lft prestacks 
$$\xymatrix{H\times\uucY\ar[d]_{q}\ar[r]^{\id\times\uuphi}&H\times\uucX\ar[d]^{p}\\\uucY\ar[r]^{\uuphi}&\uucX}$$
More precisely, by  \eqref{kun1} and \eqref{kun3}, for each $K\in\cD(\uucX)$ we have
\[(\id\times\uuphi)^{*}p^!K=\omega_H\boxtimes\uuphi^{*}K=q^!\uuphi^{*}K.\]
This proves (a). To prove (b), we assume that $\uucX$ is an ind-ft ind-scheme. Then so is $\uucY$ by base change.
Write $X=\uucX$ and $Y=\uucY$.
Note that $\cX=[X/H]$ and $\cY=[Y/H]$.
For any $K\in\cD(Y)$, using \eqref{kun2} we obtain
\[p^!\uuphi_{!}K=\omega_H\boxtimes\uuphi_{!}K=(\id\times\uuphi)_{!}(\omega_{H}\boxtimes K)=(\id\times\uuphi)_{!}q^!K.\]
Next, we consider the Cartesian diagram of ind-ft schemes
\begin{align}\label{cart-diag0}
\begin{split}
\xymatrix{H\times Y\ar[d]_{q}\ar[r]^{\id\times\phi}&H\times X\ar[d]^{p}\\Y\ar[r]^{\phi}&X}
\end{split}
\end{align}
We must check that
\begin{align}\label{wsma}
(\id\times\phi)^{!}p^{*}\cong q^{*}\phi^{!}
,\quad
p^{*}\phi_{*}\cong (\id\times\phi)_{*}q^{*}
,\quad
p^{!}\phi_{!}\cong (\id\times\phi)_{!}q^{!}.
\end{align}
Writing $H\cong\colim H_{a}$ as a colimit of ft schemes, one has an equality of pro-systems
\[(\id\times\phi)^{!}p^{*}K=\quotlim (\id\times\phi)^{!}p_{a}^{*}K\in \Pro(\cD(H\times Y)),\]
with $p_{a}:H_a\times Y\ra Y$. Similarly, we have
\[q^{*}\phi^{!}K=\quotlim p_{a}^{*}\phi^{!}K\in \Pro(\cD(H\times Y)).\]
So we reduce the first isomorphism in \eqref{wsma} to the case where $H$ is of ft.
We deduce that  
\[(\id\times\phi)^{!}p^{*}K\cong(\id\times\phi)^{!}(\bql\boxtimes K)\cong\bql\boxtimes\phi^!K\cong q^{*}(\phi^{!}K).\]
For the second one, we are reduced to prove the base change property for the Cartesian diagram 
\eqref{cart-diag0}.
Again, we first reduce to the case where $H$ is of ft.
By continuity of the functors, we can assume that everything is of ft. 
Using \eqref{box2}, we get
\[p^{*}\phi_{*}K\cong\bql\boxtimes \phi_{*}K\cong(\id\times\phi)_{*}(\bql\boxtimes K)\cong(\id\times\phi)_{*}q^{*}K.\]
\epf

\subsubsection{Contractive morphisms}

\bdefi\label{def-contracting}\hfill
\begin{enumerate}[label=$\mathrm{(\alph*)}$,leftmargin=8mm]
\item
A $\bG_m$-action contracts a prestack $\cX$ to $\cX^0$ if the attractor, 
repeller and fixed point locus are such that
$\cX^+=\cX$ and $\cX^-=\cX^0$.
\item
A $\bG_m$-equivariant presentation $H=\colim H_a$ of a group ind-scheme is contracting 
if the $\bG_m$-action contracts  $H$  to $\{1\}$.
\eenum
\edefi

\blem\label{pul-push}
Let $H$ be a group ind-scheme of ind-ft with a $\bG_m$-equivariant contracting presentation.
Let $X$ be an ind-ft ind-scheme with an $H$-action.
Set $\cX=[X/H]$ and $f:X\ra\cX$.
Then, the functors $f^{*}$ and $f^{!}$ are fully faithful.
Equivalently, we have $f_{!}f^{!}\cong\id\cong f_{*}f^{*}$.
\elem

\bpf
Since $H$ admits a contracting $\bG_m$-equivariant presentation, 
we have $H\cong\colim H_{a}$ with the $\bG_m$-action on each $H_{a}$ contracting to $\{1\}$.
Let $\pi^{[n]}:H^n\times X\ra\cX$ be the obvious projection.
Let $K\in\cD(\cX)$.
Since $\cX=[X/H]$, we have an equivalence 
$$\cD(\cX)=\lim_{[n],t^!} \cD(H^{n}\times X).$$
Under this equivalence, the complex 
$K$ identifies with the projective system $(K^n)$ given by $K^{n}=(\pi_n)^{!}K$.
Since $X=[H\times X/H]$, we also have an equivalence 
$$\cD(X)=\lim_{[n],t^!} \cD(H^n\times H\times X).$$
Let $p_n$ be the obvious map $H^{n+1}\times X\to H^n\times X$.
The functor $(p_n)_*$ satisfies base change by Proposition \ref{prop-lft}.
Hence Proposition \ref{star-calc} yields
\[f_{*}f^{*}K=((p_n)_{*}(p_n)^{*}K^{n}),\quad f_{!}f^{!}K=((p_n)_{!}(p_n)^{!}K^{n})\]
In particular, we are reduced to prove that for each $n$, we have
\[(p_n)_{*}(p_n)^{*}\cong(p_n)_{!}(p_n)^{!}\cong\id.\]
Fix a $\bG_m$-equivariant presentation $H\cong\colim H_{a}$.
For each complex $K\in\cD(H^n\times X)$ we have an equality of pro-systems
\[(p_n)_{*}(p_n)^{*}K=\quotlim (p_{n,a})_*(p_{n,a})^{*}K,\]
We are thus reduced to the case $H=H_{a}$. 
In this case, the functors $(p_{n,a})_*$ and $(p_{n,a})^*$ are both continuous
by Proposition \ref{prop-lft}.
Taking a presentation $X=\colim X_a$ and using Lurie's adjunction again, we obtain
$\cD(X)=\colim\cD(X_a)$.
Thus we can assume $H^n\times X$ to be a ft $k$-scheme.
Since the action of $\bG_m$ is contracting, the contraction principle for algebraic spaces in \cite[Prop.~3.2.2.(a)]{DG1} 
implies that for each $n,$ $a$ we have
\[(p_{n,a})_*(p_{n,a})^{*}K=\eps_{a}^{*}(p_{n,a})^{*}K\cong K,\]
where $\eps_a:1\ra H_{a}$ is the unit.
The proof of the second assertion is similar, using instead the contraction principle for $\eps_a^{!}$ in \cite[Prop.~3.2.2.(b)]{DG1}.
\epf

We now turn Lemma \ref{pul-push} and Proposition \ref{star-calc} into a definition.

\bdefi\label{dwsm}
Let $f:\uucX\ra\cX$ be a morphism of lft prestacks.
We say that
\begin{enumerate}[label=$\mathrm{(\alph*)}$,leftmargin=8mm]
\item
$f$ satisfies universal descent if the functor $\underline f^!$ in \eqref{upi} yields an equivalence
$$\underline f^!:\cD(\cX)\cong\lim_{[n],t^!}\cD(\uucX^n),$$
as well as after any base change $\cX'\ra\cX$ of lft prestacks,
\item
$f$ is weakly smooth if it is  ind-ft ind-schematic, 
 satisfies universal descent and for each ind-ft ind-schematic morphism $\phi:\cY\ra\cX$ with Cartesian diagram \eqref{CD1}, 
 we have
\benumr
\remi
$\uuphi^{!}f^{*}\cong g^{*}\phi^{!}$ and $\uuphi^{*}f^{!}\cong g^{!}\phi^{*},$
\puteqnum \label{shriek-star}
\remi
$f^{!}\phi_{!}\cong \uuphi_{!}g^{!}$ and $f^{*}\phi_{*}\cong \uuphi_{*}g^{*}$, 
\puteqnum \label{wsm-3}
 \eenum
 as well as after any  ind-ft-schematic base change $\cX'\ra\cX$ of lft prestacks,
\remi
$f$ is contractive, if it is weakly smooth and $f^{*}$, $f^{!}$ are fully faithful and stay so after any  ind-ft-schematic base change 
$\cX'\ra\cX$ of lft prestacks.
\eenum
\edefi

\blem\label{lem-contract}
Let $H$ be a group ind-scheme of ind-ft, and
$f:\uucX\ra\cX$ be an $H$-torsor between lft prestacks.
\begin{enumerate}[label=$\mathrm{(\alph*)}$,leftmargin=8mm]
\item
If $H$ is formally smooth then $f$ is also formally smooth.
\item
If $\uucX$ is an ind-ft-scheme then $f$ weakly smooth.
 \eenum
\elem

\bpf
Formal smoothness of morphisms of stacks is \'etale-local on the target.
Since $f$ is an \'etale $H$-torsor, there are sections \'etale locally, see \S\ref{sec-ind-sch}.
Part (a) follows. Part (b) follows from Proposition \ref{star-calc}.
\epf

\brem\label{rem-contract}\hfill
\begin{enumerate}[label=$\mathrm{(\alph*)}$,leftmargin=8mm]
\item
Ind-ft-schematic morphisms are stable by ind-ft-schematic base change and by composition. 
Indeed, for such $X\ra Y\ra Z$, we can assume $Z$ is a scheme.
Then, write $Y\cong \colim Y_{a}$ with $Y_a$ of finite type over $Z$. 
We have $X\cong \colim X\times_{Y}Y_a\cong\colim X_{aa'}$ with $X_{aa'}$ of finite type over $Y_{a}$, hence the claim.
\item
Universal descent,  weakly smoothness and contractiveness are also stable by ind-ft ind-schematic base change 
and by composition.
\eenum
\erem

We need stability by base change of these notions for a larger class of morphisms.

\
\bprop\label{gm-bc}
Consider a Cartesian diagram of lft prestacks
\begin{equation}
\begin{split}
\xymatrix{\uucZ\ar[d]^{h}\ar[r]^{\uupsi}&\uucY\ar[d]^{g}\ar[r]^{\uuphi}&\uucX\ar[d]^\form\\\cZ\ar[r]^{\psi}&\cY\ar[r]^{\phi}&\cX}
\label{gt1}
\end{split}
\end{equation}
Assume that $\psi$ is weakly smooth and $\phi\psi$ is ind-ft-schematic.
If  $f$ is contractive, then $g$ is contractive.
\eprop

\bpf
By base change, we already know that $g$ is ind-ft-schematic and satisfies universal descent.
We claim that $g^{!}$ and $g^{*}$ are fully faithful. We start with $g^{!}$.
We must prove that the counit $g_!g^{!}\ra\Id$ is an equivalence.
Since $\psi$ satisfies universal descent, the functor $\psi^{!}$ is conservative, thus we can check the assertion after applying $\psi^{!}$.
Since $f$ is contractive and $\phi\psi$ is  ind-ft ind-schematic, the morphism $h$ is contractive by base change.
Thus $h^!$ is fully faithful. Since $\psi$ is weakly smooth, we apply \eqref{wsm-3} to $\psi$ to get
\[\psi^{!}g_{!}g^{!}\cong h_{!}\uupsi^{!}g^{!}\cong h_{!}h^{!}\psi^{!}\cong \psi^{!},\]
 as wished.
The claim for $g^{*}$ is similar: one check it after applying $\psi^!$, using base change, full faithfulness of $h^{*}$ 
and \eqref{shriek-star} for $\psi$.
Let us prove now that $g$ is weakly smooth.
Let $\la:\cW\ra\cY$ be an ind-ft morphism.
We consider the Cartesian diagram
$$\xymatrix{\underline{\cW}\ar[d]^{\uug}\ar[r]^{\uula}&\uucY\ar[d]^{g}\\\cW\ar[r]^{\la}&\cY}$$
Let us prove that $\uula^{!}g^{*}\cong \uug^{*}\la^{!}$, the $g^{!}$-assertion is similar.
We base change by $\la$ the diagram \eqref{gt1} to get
$$\xymatrix{&\cW_{\uucZ}\ar[rrrr]^{\ti{\la}_{\cZ}}\ar[dl]_{\ti{\uupsi}}\ar[dr]^{\ti{h}}&&&&\uucZ\ar[dr]^{h}\ar[dl]_{\uupsi}\\
\underline{\cW}\ar[dr]_{\uug}&&\cW_{\cZ}\ar[dl]^{\ti{\psi}}&&\uucY\ar[dr]_g&&\cZ\ar[dl]^{\psi}\\&\cW\ar[rrrr]^{\la}&&&&\cY}$$
By base change by ind-ft-schematic morphisms, the maps $\uupsi$ and $\ti{\uupsi}$ are still weakly smooth.
Thus the functors $\uupsi^!$ and $\ti{\uupsi}^!$ are conservative, and it is enough to prove the claim after applying $\ti{\uupsi}^!$.
We obtain
\[\ti{\uupsi}^{!}\uula^{!}g^{*}=\la_{\cZ}^{!}\uupsi^{!}g^{*}=\la_{\cZ}^{!}h^{*}\psi^{!}=\ti{h}^{*}\uula^{!}\psi^{!}.\]
Here the second and third equalities follow from weakly smoothness of $\psi$ and $h$. 
Finally, by base change again $\ti{\psi}$ is weakly smooth, thus we get
\[\ti{h}^{*}\uula^{!}\psi^{!}=\ti{h}^{*}\ti{\psi}^{!}\la^{!}=\ti{\uupsi}^{!}\uug^{*}\la^{!},\]
as wished.
The second claim about proper base change is proved along the same lines.
\epf

\subsubsection{Homotopically ind-schematic morphisms}

We must define $*$-pushforwards for a lft prestack. It may not exist. 
We define some pushforwards using Lemma \ref{pul-push}.

\bdefi\label{d-bullet}
Let $\phi:\cY\ra\cX$ be a morphism of lft prestacks.\hfill
\begin{enumerate}[label=$\mathrm{(\alph*)}$,leftmargin=8mm]
\item
$\phi$ is homotopically ind-schematic if there is a morphism of lft
prestacks $g:\uucY\ra\cY$ such that
\benumr
\remi
the composition 
$\phi \,g$
is  ind-ft ind-schematic,
\remi
$g$ is contractive.
\eenum
\item
If $\phi$ is homotopically ind-schematic  
we define the functor $\phi_{\bullet}=(\phi\,g)_{*}g^{*}:\cD(\cY)\ra\Pro(\cD(\cX))$.
\eenum
\edefi

\blem\label{ind-bullet0}\hfill
\begin{enumerate}[label=$\mathrm{(\alph*)}$,leftmargin=8mm]
\item
The functor $\phi_\bullet$ is well-defined.
If $\phi:\cY\ra\cX$ is ind-ft ind-schematic, then $\phi_{*}$ exists and  $\phi_{\bullet}=\phi_{*}$.
\item
Homotopically ind-schematic morphisms are stable by base change by morphisms 
$\phi$ of lft prestacks such that there exists a weakly smooth $\psi$ with $\phi\,\psi$  ind-ft ind-schematic.
\eenum
\elem

\bpf
To prove (a) note that, if there are two contractive morphisms $g_1:\uucY_1\to\cY$ and $g_2:\uucY_2\to\cY$ as in the definition, 
then we can form the morphism $g_3:\uucY_3\to\cY$ where $\uucY_3=\uucY_{1}\times_{\cY}\uucY_2$.
Since $g_1$ and $g_2$ are contractive, the $*$-pullbacks by the maps 
$f_1:\uucY_{3}\ra\uucY_1$ and $f_2:\uucY_{3}\ra\uucY_2$ given by ind-ft base change are fully faithful.
Hence 
$$(\phi\,g_1)_*g_1^*=(\phi\,g_1)_*(f_1)_*f_1^*g_1^*=(\phi\,g_3)_*g_3^*=(\phi\,g_2)_*(f_2)_*f_2^*g_2^*=(\phi\,g_2)_*g_2^*$$
In particular, the functor $\phi_{\bullet}$ does not depend of the choice of the contractive morphism $g:\uucY\to\cY$. 
Part (b) follows from Proposition \ref{gm-bc}, and the fact that ind-ft ind-schematic morphisms are preserved by base change.
\epf

\blem\label{rem-contract2}
Let $H$ be a group ind-scheme of ind-ft with a $\bG_m$-equivariant contracting presentation.
\begin{enumerate}[label=$\mathrm{(\alph*)}$,leftmargin=8mm]
\remi
Any $H$-torsor $g:\uucY\ra\cY$ with $\uucY$ an ind-ft ind-scheme is contractive.
\remi
Any morphism $\phi:\cY\ra\cX$ such that 
$\phi g$ is ind-ft-schematic is homotopically ind-schematic.
\eenum
\elem

\bpf
The torsor $g$ is weakly smooth by Lemma \ref{rem-contract}.
The functors $g^*$, $g^!$ are fully faithful by Lemma \ref{pul-push},
and they stay so after any ind-ft ind-schematic base change 
$\cY'\ra\cY$ of lft prestacks, because $H$-torsors are stable by base change.
Thus $g$ is contractive. 
\epf

\blem\label{hom-comp}
The composition of two homotopically ind-schematic morphisms 
$\psi:\cZ\to\cY$ and $\phi:\cY\to\cX$ is also homotopically ind-schematic.
\elem

\bpf
Consider the Cartesian diagram
$$\xymatrix{\uucZ\times_{\cY}\uucY\ar[d]_{h'}\ar[r]&\uucZ\ar[d]^{h}\\
\cZ\times_{\cY}\uucY\ar[r]^-{g'}\ar[d]&\cZ\ar[d]\\
\uucY\ar[r]^{g}\ar[dr]&\cY\ar[d]^{\phi}\\
&\cX}$$
As ind-ft-schematic morphisms are preserved by composition and base change, remark \ref{rem-contract} gives that  $g'\, h'$ is contractive. 
Moreover, by base change, the morphism $\uucZ\times_{\cY}\uucY\ra\uucY$ is ind-ft ind-schematic and,
by composition, the morphism $\uucZ\times_{\cY}\uucY\ra\cY$ is also ind-ft ind-schematic.
\epf
We claim the functor $(-)_\bullet$ is compatible with base changes, compositions, and the projection formula.

\bprop\label{bul-comp-proj}
Let $\psi:\cZ\to\cY$ and $\phi:\cY\to\cX$ be homotopically ind-schematic morphisms.
\begin{enumerate}[label=$\mathrm{(\alph*)}$,leftmargin=8mm]
\item
For any ind-ft-schematic morphism $\pi:\cX'\to\cX$ of lft prestacks, let
$\pi':\cY'\to\cY$ and $\phi':\cY'\to\cX'$ be the base changes of $\pi$ and $\phi$.
There is an isomorphism of functors $\pi^{!}\phi_{\bullet}\ra (\phi')_{\bullet}(\pi')^{!}.$
\item
There is an isomorphism of functors
$(\phi\,\psi)_{\bullet}\cong\phi_\bullet\psi_\bullet$.
\item
For each $K\in\cD(\cX)$ and $L\in\cD(\cY)$ we have an isomorphism
$\phi_\bullet(\phi^! K\overset{!}{\otimes} L)\cong K\overset{!}{\otimes} \phi_\bullet L$.
\eenum
\eprop

\bpf
For part (a) we consider the  following Cartesian diagram
\begin{equation}
\xymatrix{\uucY'\ar[d]_{g'}\ar[r]^{\pi''}\ar@/_2pc/[dd]_{\uuphi'}&\uucY\ar[d]^{g}\ar@/^2pc/[dd]^{\uuphi}\\
\cY'\ar[d]_{\phi'}\ar[r]^{\pi'}&\cY\ar[d]^{\phi}\\\cX'\ar[r]^{\pi}&\cX}
\label{B-cart}
\end{equation}
The  base changes yields a chain of isomorphisms
\[\pi^{!}\phi_{\bullet}=\pi^{!}\uuphi_{*}g^{*}\cong
(\uuphi')_{*}(\pi'')^{!}g^{*}\stackrel{\eqref{shriek-star}}{\cong}(\uuphi')_{*} (g')^{*}(\pi')^{!}=(\phi')_{\bullet}(\pi')^{!}.\]
To prove (b), note that, by Lemma $\ref{hom-comp}$, the composed map $\phi\,\psi$ is homotopically ind-schematic.
We have the following Cartesian diagram
$$\xymatrix{\uucZ\times_{\cY}\uucY\ar@/_2pc/[dd]_{\uupsi'}\ar[d]_{h'}\ar[r]^-{\uug}&\uucZ\ar[d]^{h}\ar@/^2pc/[dd]^{\uupsi}\\\cZ\times_{\cY}\uucY\ar[r]^-{g'}\ar[d]&\cZ\ar[d]^{\psi}\\\uucY\ar[r]^{g}\ar[dr]_{\uuphi}&\cY\ar[d]^{\phi}\\&\cX}$$
By Lemma \ref{hom-comp} and base change, to compute the functor $(\phi\,\psi)_{\bullet}$ we can use the map
$$\uuphi\uupsi':\uucZ\times_{\cY}\uucY\ra\cX.$$
We have
\[\phi_\bullet\psi_\bullet=\uuphi_*g^{*}\uupsi_*h^{*}=
\uuphi_*\uupsi'_*\uug^{*}h^*=\uuphi_*\uupsi'_*(g'h')^{*}=(\phi\,\psi)_{\bullet},\]
where the second equality follows from the fact that the map $g$ satisfies \eqref{wsm-3}.

To prove (c), let $g:\uucY\ra\cY$ be a contractive morphism such that the map $\phi\,g$ is ind-ft ind-schematic.
The projection formula holds for the maps $\phi\,g$ and $g$.
Further,  Lemma \ref{ind-bullet0} yields $(\phi\,g)_*=(\phi\,g)_\bullet$ and  $g_*=g_\bullet$.
Finally, since $g$ is contractive we have $g_*g^*=\id$.
Thus Lemma \ref{p-form} and Proposition \ref{bul-comp-proj} yield
\begin{align*} 
K\overset{!}{\otimes} \phi_\bullet L
&=K\overset{!}{\otimes} (\phi\,g)_{*}g^{*}L\\
&=(\phi\,g)_*((\phi\,g)^{!}K\overset{!}{\otimes}g^{*}L )\\
&=(\phi\,g)_\bullet((\phi\,g)^{!}K\overset{!}{\otimes}g^{*}L )\\
&=\phi_\bullet g_\bullet((\phi\,g)^{!}K\overset{!}{\otimes}g^{*}L )\\
&=\phi_\bullet g_{*}(g^{!}\phi^{!}K\overset{!}{\otimes}g^{*}L)\\
&=\phi_{\bullet}(\phi^{!}K\overset{!}{\otimes}g_{*}g^{*}L )\\
&=\phi_{\bullet}(\phi^{!}K\overset{!}{\otimes} L).
\end{align*}
\epf

\subsubsection{Formalism of kernels}\label{kern}
We follow \cite[Cor.~4.2.3]{DG}. 
Let $\cY_1,$ $ \cY_2$ be lft prestacks.
Let $p_1,p_2:\cY_1\times\cY_2\ra\cY_1$ be the obvious projections.
Assume that  $p_2$ is homotopically ind-schematic. 
For any sheaf $K\in\cD(\cY_1\times\cY_2)$ we have the functor
$$\Phi_{K}:\cD(\cY_1)\ra\cD(\cY_2),\quad
 \cM\mapsto (p_2)_\bullet(p_{1}^{!}\cM\overset{!}{\otimes}K)$$
We say that $K$ is the kernel of the functor $\Phi_{K}$.
More generally,  let
$$\xymatrix{&\cX\ar[d]^{r}\ar[dr]^{q}\ar[dl]_{p}\\
\cY_{1}&\cY_{1}\times\cY_2\ar[r]^{p_2}\ar[l]&\cY_{2}}$$
be a commutative diagram of lft prestacks.
Assume that $p_2$ and $r$ are homotopically ind-schematic.
Let 
$$K=r_{\bullet}(\omega_{\cX})\in\cD(\cY_1\times\cY_2).$$
The projection formula in Proposition \ref{bul-comp-proj} implies that the functor $\Phi_{K}$ identifies with $q_{\bullet}p^{!}$.
The functor $K\ra \Phi_{K}$ may not be fully faithful. 
Nevertheless, the morphisms $\Phi_{K}\ra\Phi_{L}$ that we will encounter come from morphisms of sheaves $K\ra L$.
In particular, we will encounter the following setting.
Consider a commutative diagram
\begin{equation}
\begin{split}
\xymatrix{
&\cU\ar[d]^{j}\ar[dr]\ar[dl]\\
\cY_1&\cX\ar[d]^{r}\ar[r]^{q}\ar[l]_{p}&\cY_2\\
&\cY_{1}\times\cY_2\ar[ur]\ar[ul]&
}
\label{ker-form}
\end{split}
\end{equation}
Assume that $j$ is a qc open immersion.
Hence, it is ind-ft ind-schematic.
We consider the kernel 
$$L= (rj)_{\bullet}(\omega_{\cU})\in\cD(\cY_1\times\cY_2).$$
Proposition \ref{bul-comp-proj}, the equality $j_{\bullet}=j_{*}$ in
Lemma \ref{ind-bullet0} and the equality $j^{!}=j^{*}$
yield a  morphism  of kernels
\begin{equation}
K\ra r_{\bullet}j_{*}j^{*}\omega_{\cX}\ra (r j)_{\bullet}\omega_{\cU}=L.
\label{ker-form2}
\end{equation}
By functoriality, this yields an obvious map
\begin{equation}
\Phi_{K}\ra\Phi_{L}.
\label{ker-form3}
\end{equation}

\subsection{The t-structures}
Consider an $\infty$-stack $\cY$ and a collection of non-empty topological fp 
locally closed reduced substacks $\cY_{\al}$ with $\al\in\cI$, such that $\cY_{\al}\cap\cY_{\beta}=\emptyset$ whenever $\alpha\neq\beta$.
\begin{enumerate}[label=$\mathrm{(\alph*)}$,leftmargin=8mm]
\item
The collection $\{\cY_{\al}\}$ is a {finite constructible} stratification if $\cI$ is finite and there is a total order 
$\al_1<\dots<\al_n$ of $\cI$ and an increasing sequence $\emptyset=\cY_0\subset\dots\subset\cY_n=\cY$ such that 
$\cY_{\al_i}\subset\cY_i\backslash \cY_{i-1}$ and the embedding is a topological equivalence for each $i=1,\dots, n$,
see \cite[\S 2.4.5]{BKV}.
\remi
A substack $\cX\subset\cY$ is adapted to the collection $\{\cY_{\al}\}$ if for each $\al$, 
either $\cX\cap\cY_{\al}=\emptyset$ or $\cY_{\al}\subset\cX$. Then, we set $\cI_\cX=\{\al\in\cI\,;\,\cY_\al\subset\cX\}$.
\remi
The collection $\{\cY_{\al}\}$ is a bounded constructible stratification if there is a presentation $\cY\cong\colim_\beta \cU_\beta$ 
as a filtered colimit where $\cU_\beta$ is an fp open adapted substack of $\cY$ and  the induced stratification 
on each $\cU_\beta$ is finite constructible.
\remi
The collection $\{\cY_{\al}\}$ is  a constructible stratification if there is a presentation $\cY\cong\colim_\beta \cY_\beta$ as a filtered colimit 
where $\cY_\beta$ is a topologically fp closed adapted substack of $\cY$ and  the induced stratification on each $\cY_\beta$ is bounded 
constructible.
\eenum
A perversity on a stratified stack $(\cY\,,\,\{\cY_{\al}\})$ is a function 
$p_\nu:\cI\ra\bZ$, $\al\mapsto\nu_{\al}$.
If $\cY$ satisfies gluing, then we have pullback functors 
$$(i_\al)^*,(i_\al)^!:\cD(\cY)\ra\cD(\cY_{\al}).$$
If each $\cD(\cY_{\al})$ has a perverse $t$-structure $(\mathstrut^{p}\cD^{\leq 0}(\cY_{\al})\,,\,\mathstrut^{p}\cD^{\geq 0}(\cY_{\al}))$ 
then \cite[Prop.~6.4.2]{BKV} yields a unique $t$-structure $(\mathstrut^{p}\cD^{\leq 0}(\cY)\,,\,\mathstrut^{p}\cD^{\geq 0}(\cY))$ 
on $\cD(\cY)$ such that
\begin{equation}
\begin{split}
\mathstrut^{p}\cD^{\geq 0}(\cY)=\{K\in\cD(\cY)\,;\, (i_\al)^!K\in\mathstrut^{p}\cD^{\geq -\nu_{\al}}(\cY_{\al}), \forall~\al\in\cI\}.
\end{split}
\label{t-st1}
\end{equation}
If the stratification $\{\cY_{\al}\}$ is bounded constructible, then we also have
\begin{equation}
\begin{split}
\mathstrut^{p}\cD^{\leq 0}(\cY)=\{K\in\cD(\cY)\,;\, (i_\al)^*K\in\mathstrut^{p}\cD^{\leq -\nu_{\al}}(\cY_{\al}), \forall~\al\in\cI\}.
\end{split}
\label{t-st2}
\end{equation}
Let $\cD(\cY)^\heartsuit$ be the heart of the $t$-structure. We identify the heart with its homotopy category as in \cite[Rmk.~1.2.1.12]{Lu2}.
In particular, we view it as an Abelian category.
If $j:\cU\hra\cY$ is an fp open immersion adapted to the stratification $(\cY_{\al})$, 
and $i:\cZ\hra\cX$ is its closed complement, 
we equip $\cD(\cU)$ and $\cD(\cZ)$ with the induced $t$-structures. 
For any $K\in \cD(\cU)^\heartsuit$
one can define the intermediate extension $j_{!*}K\in\cD(\cY)^\heartsuit$ as in  \cite[\S6.4.8.(b)]{BKV}.
If further the stratification $\{\cY_{\al}\}$ is bounded, then by \cite[Lem.~6.4.11.(d)]{BKV} the sheaf $j_{!*}K$
is the unique sheaf $\ti{K}\in\cD(\cY)$ such that
\begin{equation} 
(i_\al)^!\ti{K}\in \mathstrut^{p}\cD^{> -\nu_{\al}}(\cY_{\al})
,\quad
(i_\al)^*\ti{K}\in\mathstrut^{p}\cD^{< -\nu_{\al}}(\cY_{\al})
,\quad
\forall~\al\in \cI\smallsetminus\cI_{\cU}
\label{t3}
\end{equation}

\blem\label{Irr-Perv}
Let $\cY$ be a bounded stratified stack with a perversity $p_\nu$ such that $\cD(\cY_a)$ has a $t$-structure.
The simple objects in $\cD(\cY)^\heartsuit$ are the intermediate extensions of the simple objects in
$\cD(\cY_\al)^\heartsuit$ as $\al$ runs into $\cI$.
\elem

\bpf
Choose a partial order on the set $\cI$ such that $\cY_{\leq\al}=\bigcup_{\beta\leq\al}\cY_\beta$ is
an fp closed substack of $\cY$ and $\{\beta\in\cI\,;\,\beta>\al\}$ is finite for each $\alpha\in\cI$. 
Set $\cZ_\al=\cY_{\leq\al}$ and $\cU_\al=\cY\setminus\cZ_\al$. 
Let $i_\al:\cZ_\al\to\cY$ and $j_\al:\cU_\al\to\cY$ be the obvious maps.
By general properties of glueing, we have a short exact sequences of Abelian categories
$$\xymatrix{0\ar[r]&\cD(\cZ_\al)^\heartsuit\ar[r]^{(i_\al)_*}&\cD(\cY)^\heartsuit\ar[r]^{(j_\al)^!}&\cD(\cU_\al)^\heartsuit\ar[r]&0}$$
where $(i_\al)_*$ is exact and fully faithful and embeds 
$\cD(\cZ_\al)^\heartsuit$ in $\cD(\cY)^\heartsuit$ as a Serre subcategory, and $(j_\al)^!$ is exact and yields an equivalence
$\cD(\cY)^\heartsuit/\cD(\cZ_\al)^\heartsuit\cong\cD(\cU_\al)^\heartsuit$ such that simple in $\cD(\cU_\al)^\heartsuit$ lift uniquely to
$\cD(\cY)^\heartsuit$ via $(j_\al)_{!*}$. The claim follows by descending induction.
\epf

\subsection{The t-structure on the affine Grassmannian}
Recall that $\Gr_{G}=[G/G_\cO]$, $\Gr_{G_c}=[G_c/G_\cO]$ and that
both stacks $\Gr_{G}$, $\Gr_{G_c}$ are quotients of an ind-ft ind-scheme by a group ind-scheme of ind-ft.
Thus they both satisfy gluing.
Further, the Cartan decomposition yields  a  stratification on $\Gr_{G_c}$.

\blem\label{b-constr}\hfill
\begin{enumerate}[label=$\mathrm{(\alph*)}$,leftmargin=8mm]
\item
The decomposition $\Gr_{G_c}=\bigsqcup_{\la\in \LLa_+}\Gr_\la$ 
is a constructible stratification.
\item
The decomposition $\Gr_{\leqslant\la}=\bigsqcup_{\mu\leqslant\la}\Gr_\mu$ 
is a bounded constructible stratification.
\eenum
\elem

\bpf
The Cartan decomposition implies that it is a stratification. We must check that the stratification of $\Gr_{\leqslant\la}$ is bounded.
The open sets $U_{\mu}=\Gr_{\leq\la}\setminus\Gr_{\leq\mu}$ with $\mu\leq\la$ and $\mu\in \LLa_+$ yield
the open covering
\[\Gr_{\leq\la}=\bigcup_{\mu\leq\la} U_{\mu}.\]
We must check that $U_{\mu}$ contains only a finite number of strata $\Gr_{\nu}$,
 which by the closure relations ends to check that the set of cocharacters 
 $\{\nu\in \LLa_+\,;\,\mu\leq\nu\leq\la\}$ is finite.
\epf

The $t$-structure we now construct holds only on the category of $G(\cO)$-equivariant sheaves on $\Gr_{G_c}$. 
We'll need the following contraction lemma.

\blem\label{contract}
Let $f:X\ra S$ be a morphism of ind-schemes of ind-ft over $k$. 
Assume that $X$ has a $\bG_m$-equivariant presentation and that the $\bG_m$-action contracts $X$
to $X^0\cong S$. Then $f_{!}\omega_{X}\cong \omega_{S}$.
\elem

\bpf
Since the assertion commutes with filtered colimits, 
considering a $\bG_m$-equivariant presentation of $X$ 
we reduce to the case 
where $X$ and $S$ are ft $k$-schemes and $X$ is contracted to $S$.
Then, the isomorphism $f_{!}\omega_{X}\cong\omega_{S}$
is a particular case of hyperbolic localization, 
with attractor $X^{+}=X$ and repeller and fixed point locus $X^{-}=X^{0}=S$, see \cite[Thm.~B]{Ric}.
\epf

Recall that  $K_{\la}$ 
is the stabilizer in $G_{\co}$ of the point $[s^{\la}]\in \Gr_G$ and $P_\la^-=G\cap K_\la$.

\bprop\label{t-strat}
Let $\la\in \LLa_+$. Assume that $G$ is of affine type (untwisted or not). 
\hfill
\begin{enumerate}[label=$\mathrm{(\alph*)}$,leftmargin=8mm]
\item
There is an equivalence of stacks
$[G_{\co}\backslash\Gr_{\la}]\cong \bB K_{\la}.$
\item
There is an equivalence $\infty$-categories $\cD(\bB K_{\la})\cong\cD(\bB L_\la)$ 
for some reductive group $L_\la$.
\eenum
\eprop

\bpf
Part (a) is clear, because by Proposition  \ref{unif} we have $[G_{\co}/K_{\la}]_\red\cong(\Gr_\la)_\red$
 and $\cD$ takes topological equivalences to equivalences of $\infty$-categories.
Let us prove (b).
The obvious map $p_\la:\Spec(k)\ra \bB K_{\la}$ is a morphism of  lft prestacks, see Remark \ref{rem-lft}.
The functor $(p_\la)^{!}$ is continuous by construction.
It admits a left adjoint $(p_\la)_!$ by Proposition \ref{prop-lft}, which is also continuous.
The functor $(p_\la)^{!}$ is conservative, having sections locally for the étale topology.
By Barr-Beck-Lurie \cite[Thm.~4.7.3.5]{Lu2}, we have an equivalence
\[\cD(\bB K_{\la})\cong (p_{\la})^{!}(p_\la)_!\bql\text{-}\text{mod}\cong H_{*}(K_{\la},\bql)\text{-}\text{mod}.\]
The last equality follows from the base change formula \eqref{shriek-1}, applied to the Cartesian diagram
$$\xymatrix{K_\la\ar[r]\ar[d]&\Spec(k)\ar[d]^{p_\la}\\\Spec(k)\ar[r]^{p_\la}&\bB K_\la}$$
The loop action contracts $K_{\la}$ to $P_{\la}^-$.
So Lemma \ref{contract} yields an isomorphism 
$H_{*}(K_{\la},\bql)\cong H_{*}(P_{\la}^-,\bql).$
By Proposition \ref{sch-form} the group $P_{\la}^-$ is a parahoric of $G$.
Let $L_\la$ be  the Levi factor. 
Since $G$ is of affine type, the group $L_\la$ is reductive.
By \S\ref{opp-J}, we have
$P_{\la}^-\cong L_\la\ltimes U_{\la}^-$
where $U_{\la}^-$ is the unipotent radical. 
In particular, the cocharacter $\la$ contracts $P_{\la}^-$ to $L_\la$ and Lemma \ref{contract} yields an isomorphism
\[H_{*}(K_{\la},\bql)\cong H_{*}(L_\la,\bql).\]
Applying Barr-Beck-Lurie, we obtain an equivalence
$\cD(\bB K_{\la})\cong\cD(\bB L_\la).$
\epf

\brem
Using some ongoing work of Y. Varshavsky, it should be possible to construct a $t$-structure 
for any simply connected KM group, using an equivalence of categories
$\cD(\bB G)\cong\lim_{P}\cD(\bB P),$
where $P$ runs over all fp-parabolics.
\erem

Since $\bB L_\la$ is a smooth ft Artin stack, this proposition yields a  $!$-shifted $t$-structure 
on $\cD_{G_\cO}(\Gr_\la)$
such that the dualizing sheaf $\omega_{\Gr_{\la}}$ is perverse.  
A sheaf $\cE$ is perverse for this $!$-shifted $t$-structure if $\cE[-\dim(\bB L_\la)]$ is perverse for the usual $t$-structure.
To get the right $t$-structure $\cD_{G_\cO}(\Gr_{\leq\la})$, 
we need an extra shift $[-2 \langle\rho,\la\rangle]$ over each strata, 
whose justification will appear when we prove the $t$-exactness of the constant term functor.
On $\cD_{G_\cO}(\Gr_{\leq\la})$ we now define a $t$-structure by gluing, 
using Lemma \ref{b-constr} and \eqref{t-st1}, \eqref{t-st2} with the perversity function 
$$p_\nu:\{\eta\in\LLa_+\,;\,\eta\leqslant\la\}\to\bZ
,\quad
\eta\mapsto 2 \langle\rho,\la-\eta\rangle.$$ We get
\begin{align*}
\begin{split}
K\in \mathstrut^{p}\cD^{\leq 0}_{G_\cO}(\Gr_{\leq\la})
&\Longleftrightarrow (i_\eta)^*K\in\mathstrut^{p}\cD^{\leq -2 \langle\rho,\la-\eta\rangle }_{G_\cO}(\Gr_{\eta})\,,\, \forall\eta\leq\la,
\\
K\in\mathstrut^{p}\cD^{\geq 0}_{G_\cO}(\Gr_{\leq\la})
&\Longleftrightarrow (i_\eta)^!K\in\mathstrut^{p}\cD^{\geq -2 \langle\rho,\la-\eta\rangle}_{G_\cO}(\Gr_{\eta})\,,\, \forall\eta\leq\la.
\end{split}
\end{align*}
Thus, we can consider the object $\IC_\la\in\cD(\Gr_{\leq\la})^\heartsuit$ given by
\[\IC_{\la}=(i_\la)_{!*}\omega_{\Gr_{\la}}[-2 \langle\rho,\la\rangle].\]
Note that  Lurie's adjunction yields
\[\cD_{G_\cO}(\Gr_{G_c})\cong\lim_{\la}\cD_{G_\cO}(\Gr_{\leq\la})\cong\colim_{\la}\cD_{G_\cO}(\Gr_{\leq\la}).\]
Closed immersions are $t$-exact.
The proof is as in \cite[Lem.~6.4.11]{BKV}.
Hence, we get a $t$-structure on $\cD_{G_\cO}(\Gr_{G_c})$.

\bprop\label{t-gen}
Assume that $G$ is of affine type.
\hfill
\begin{enumerate}[label=$\mathrm{(\alph*)}$,leftmargin=8mm]
\item
The category $\mathstrut^{p}\cD^{\leq 0}_{G_\cO}(\Gr_{G_c})$ is generated by the sheaves
$(i_\la)_!\omega_{\Gr_{\la}}[m-2 \langle\rho,\la\rangle]$ with $m\in\NN$.
\item
The category  $\mathstrut^{p}\cD^{\geq 0}_{G_\cO}(\Gr_{G_c})$ is generated by the sheaves
$(i_\la)_*\omega_{\Gr_{\la}}[-m-2 \langle\rho,\la\rangle]$ with $m\in\NN$.
\eenum
\eprop

\bpf
First, note that
$(i_\la)_!\omega_{\Gr_{\la}}[-2 \langle\rho,\la\rangle]\in\mathstrut^{p}\cD^{\leq 0}_{G_\cO}(\Gr_{G_c})$ 
because closed immersions are t-exact.
Let $K\in\mathstrut^{p}\cD^{\leq 0}_{G_\cO}(\Gr_{G_c})$ be a complex such that
$$\Hom((i_\la)_!\omega_{\Gr_{\la}}[m-2 \langle\rho,\la\rangle]\,,\,K)=0,\quad m\in\NN.$$
We must prove that $K\cong 0$.
By Lemma \ref{b-constr}, the stratification
$(\Gr_{\la})_{\la\in\Lambda_{+}}$ is constructible.
Since $\Gr_{G_c}$ satisfies gluing, an induction using the fiber sequence in \cite[Lem.~5.4.1]{BKV}
implies that it is enough to check that for each $\la$ we have $(i_\la)^!K=0$.
Thus we are reduced to the case of a single strata and $K\in {}^{p}\cD^{\leq 0}$. 
 By Proposition \ref{t-strat}, we have an equivalence
 $\cD_{G_{\co}}(\Gr_{\la})\cong\cD(\bB L_{\la})$ given by 
 $!$-pullback, where $L_{\la}$ is a connected reductive group.
Up to a shift by $[\left\langle 2\rho,\la\right\rangle]$, we can suppose that the $t$-structure on  $\cD_{G_{\co}}(\Gr_{\la})$ 
is the pullback of the $!$-$t$-structure on $\cD(\bB L_{\la})$.
The map $\pi:\Spec(k)\ra\bB L_{\la}$ is smooth and $\pi^{!}$ is $t$-exact.
Thus by conservativity, it is sufficient to get the claim for $\pi^{!}K$, where it is clear.
The assertion follows.

If $K\in\mathstrut^{p}\cD^{\geq 0}_{G_\cO}(\Gr_{G_c})$ is orthogonal to all 
 $(i_\la)_*\omega_{\Gr_{\la}}[-m-2 \langle\rho,\la\rangle]$, it is enough to check that for every 
$\la$ we have $(i_{\la})^{*}K=0$ using the fiber sequence in \cite[Lem.~5.5.3]{BKV}.
We conclude in a similar way.
\epf

Since the group $L_\la$ is connected for each dominant cocharacter $\la$, 
Lemma \ref{Irr-Perv} and Proposition \ref{t-strat} yield the following.

\bcor\label{simple}
Assume that $G$ is of affine type.
The set of isomorphism classes of simple objects in $\cD_{G_{\co}}(\Gr_{G_c})^\heartsuit$ is $\{\IC_\la\,;\,\la\in\LLa_+\}$.
\qed
\ecor

\section{The hyperbolic localization}

In this section we gather some material on hyperbolic localization on stacks.
We will apply it in \S\ref{HLGr} to the affine Grassmannian to prove
Theorem \ref{Braden}

\subsection{Interpolating family}

\subsubsection{Attractors and repellers}\label{att-rep}
Given two $\infty$-prestacks $\cY$ and $\cZ$ over $k$ with $\bG_m$-actions, let 
$\Hom^{\bG_m}(\cY,\cZ)$
be the functor of $\bG_m$-equivariant morphisms.
We define the functors
\begin{align}\label{ZZZ}
\cZ^{+} =\Hom^{\bG_m}(\ab^{1},\cZ),\quad \cZ^{-}=  \Hom^{\bG_m}(\ab_{-}^{1},\cZ),\quad \cZ^{0}=\Hom^{\bG_m}(\Spec(k),\cZ),
\end{align}
where $\ab^{1}$ is equipped with the standard $\bG_m$-action and $\ab^{1}_{-}$ the opposite $\bG_m$-action.
The functor $\cZ^{+}$ is called the {attractor}, $\cZ^{-}$ the {repeller} and $\cZ^{0}$ the {fixed point locus}.
We say that the $\bG_m$-action contracts $\cZ$ to $\cZ^0$ if $\cZ^+=\cZ$ and $\cZ^-=\cZ^0$.
We have a commutative diagram
$$\xymatrix{&\cZ^{+} \ar[dl]^{q^{+}}\ar[dr]^{p^{+}}&\\\cZ^{0}\ar@/_1pc/[dr]_{i^-}\ar@/^1pc/[ur]^{i^+}&&\cZ\\&\cZ^{-}\ar[ul]_{q^{-}}\ar[ur]_{p^{-}}}$$
The map $q^{\pm}$ is the evaluation at 0, 
the map $p^{\pm}$ is the functoriality with respect to the inclusion $\bG_m\subset\ab^1$,
 and $i^{\pm}$ is the functoriality with respect to the structural map
$\ab^1\ra\Spec(k)$ or $\ab_{-}^1\ra\Spec(k)$.
Recall that a prestack is separated if the diagonal is  closed.

\blem\label{att-mon}
Let $\cZ$ be a separated prestack.
The maps $p^{\pm}:\cZ^{\pm}\ra \cZ$ are monomorphisms.
\elem

\bpf
Given a test scheme $Y$ and two maps $a, b:Y\ra\cZ$,
the locus $Z\subset Y$ where $a|_Z=b|_Z$ is $\cZ\times_{\cZ\times\cZ}Y.$ 
Since $\cZ$ is separated, the functor $Z$ is closed in $Y$.
The map $p^{+}$ is identified with the restriction map
\[\Hom^{\bG_m}(\ab^1,\cZ)\ra \Hom^{\bG_m}(\bG_m,\cZ),\]
and similarly for $p^{-}$. 
Since $\bG_{m,R}$ is schematically dense in $\ab^{1}_{R}$ for any ring $R$, a point $\phi\in\cZ^{+}(R)$ is completely determined by its restriction.
\epf

A morphism $\cX\ra\cY$ of prestacks admits a relative deformation theory if Definition \ref{defT} holds.

\bprop\label{att-smooth}
Assume that $\cZ$ is formally smooth with a relative deformation theory.
Then $\cZ^{0}$ and $\cZ^{+}$ are formally smooth.
\eprop

\bpf
Let us prove that $\cZ^{+}$ is formally smooth, formal smoothness of $\cZ^{0}$ is analog.
Let $I\subset R$ be a square zero ideal $\bar{R}=R/I$.
Let $\bar{f}:\ab^{1}_{\bar{R}}\ra\cZ$ be a $\bG_m$-equivariant morphism.
By formal smoothness of $\cZ$, we lift it non-equivariantly as $f:\ab^{1}_{R}\ra\cZ$.
Let $L_{\cZ/k}$ be the pro-cotangent complex.
By Proposition \ref{equiv-lift}, the obstruction to the existence of a $\bG_m$-equivariant map sits in
$H^{1}(\bG_m,M)$ with 
\[M=\Map(\bar{f}^{*}(L_{\cZ/k}),\cJ),\]
where $\cJ$ the sheaf of ideals defining $\ab^{1}_{\bar{R}}\ra\ab^{1}_{R}$. 
Since $M$ is a quasi-coherent sheaf,  this $H^{1}$ is zero because $\bG_m$ is diagonalizable.
\epf

\subsubsection{The tilde functor}\label{ti-f}
Let $\bX=\ab^{2}$ with the hyperbolic $\bG_m$-action given by $$\la\cdot(x,y)=(\la x,\la^{-1}y).$$
We view $\bX$ as a scheme over $\ab^{1}$, via the multiplication map 
$$m:\bX\ra\ab^1,\quad (x,y)\mapsto xy.$$ 
For any $\infty$-prestack $\cZ$ over $k$ with a $\bG_m$-action we 
define $\widetilde\cZ$ to be the $\infty$-prestack over $\ab^1$ such that
\begin{align}\label{ti-Z}
\Hom_{\ab^{1}}(S,\tilde\cZ)=\Hom^{\bG_m}(\bX_S,\cZ)
,\quad
\bX_S:=\bX\times_{\ab^1} S
\end{align}
for any affine scheme $S$ over $\ab^{1}$. 
The $\infty$-groupoid
$\ti{\cZ}(S)$ consists of all $\bG_m$-equivariant maps 
$\bX_S\ra\cZ.$
Let $\str$ be the structure map $\ti{\cZ}\to\ab^1$.
For each section $\sigma:\ab^{1}\ra\bX$ there is a morphism $\sigma^{*}:\ti{\cZ}\ra\cZ$.
Let 
$$\pi^-=(\sigma_1)^*,\quad\pi^+=(\sigma_2)^*$$
where $\sigma_1(t)=(t,1)$ and $\sigma_2(t)=(1,t)$.
We define the map 
\begin{equation}
\ti{p}=(\pi^-,\pi^+,\str):\ti{\cZ}\ra\cZ\times\cZ\times\ab^{1}
\label{ti-p}
\end{equation}
The action of $(\bG_m)^2$ on $\bX$ given by
\[(\la,\mu)\cdot(x,y)= (\la x,\mu y),\quad (x,y)\in\bX,\quad \la,\mu\in\bG_m,\]
 yields an action of $(\bG_m)^2$ on $\ti{\cZ}$.
 The map $\ti{p}$ is equivariant for the action of 
 $(\bG_m)^2$ on $\cZ\times\cZ\times\ab^1$ given by
\begin{equation}
(\la,\mu)\cdot(x,y,t)=(\la x,\mu y,(\la\mu)^{-1}t).
\label{gm-eq}
\end{equation}
In the non-affine case, the map $\ti{p}$ may not be closed, but  we have the following.

\blem\label{ti-mono}
\hfill
\begin{enumerate}[label=$\mathrm{(\alph*)}$,leftmargin=8mm]
\item
If $\cZ$ is an ind-ft ind-affine ind-scheme, then $\ti{p}$ is a closed immersion.
\item
If $\cZ$ is a separated prestack, then $\ti{p}$ is a monomorphism.
\eenum
\elem

\bpf
If $\cZ$ is an affine scheme of ft, then, by \cite[Prop.~2.3.7]{DG1}, 
the stack $\ti{\cZ}$ is representable by an affine scheme of ft and $\ti{p}$ is a closed immersion. 
The functor $\cZ\mapsto\ti{\cZ}$ commutes with filtered colimits.
So if $\cZ$ is an ind-ft ind-affine ind-scheme, then the map $\ti{p}$ is again a closed immersion, proving (a).
To prove (b) set $\bX^{\circ}=\bX\setminus\{(0,0)\}$.
For a map $S\ra\ti{\cZ}$, the corresponding 
$\bG_m$-equivariant map $\bX_{S}^{\circ}\ra\cZ$ is completely determined by the composition
\[S\ra\ti{\cZ}\stackrel{\ti{p}}{\rightarrow}\cZ\times\cZ\times\ab^1.\]
Now, the scheme $\bX^{\circ}_{S}$ is schematically dense in $\bX_{S}$.
Since $\cZ$ is separated, we conclude as in Lemma \ref{att-mon}.
\epf

\subsubsection{Some  fiber products}
For any prestack $\cZ$ we form the fiber product $\cZ^\pm\times_{\cZ}\ti{\cZ}$ using $\pi^\pm$.
We want to construct two  maps
\begin{equation}
r^\pm:\cZ^\pm\times\ab^1\to\cZ^\pm\times_{\cZ}\ti\cZ.
\label{r+}
\end{equation}
To do this, following Drinfeld \cite[Def.~3.1.8]{Dr2},  we define two schemes $\bX^\pm$ over $\ab^1$ 
 with a $\bG_m$-action.
Let
\begin{equation}
\ab^1\times\bG_m\ra\bX,\quad (t,\la)\mapsto (t\la ,\la^{-1}),
\label{sigma1}
\end{equation}
\begin{equation}
\ab^1\times\bG_m\ra\bX,\quad (t,\la)\mapsto (\la ,t\la^{-1}). 
\label{sigma2}
\end{equation}
Both maps are open embeddings. Let $\bX^+$ be the pushout of the diagram of open embeddings
\[\ab^1\times\ab^1\hookleftarrow\ab^1\times\bG_m\hra\bX,\]
where the last arrow is  \eqref{sigma1}, and
let $\bX^-$ be the pushout of the diagram 
\[\ab^1\times\ab_{-}^1\hookleftarrow\ab^1\times\bG_m\hra\bX,\]
where the last arrow is  \eqref{sigma2}.
Both diagrams hold in the category of schemes over $\ab^1$ with a $\bG_m$-action.
For $\ab^1\times\ab^1$ and $\ab^1\times\ab^1_-$ the structural map is the first projection.
So $\bX^\pm$ are also schemes over $\ab^1$ with a $\bG_m$-action.
By \cite[(3.5)-(3.6)]{Dr2} we have
\begin{align}\label{isom33}
\Hom_{\ab^{1}}(S,\cZ^\pm\times_{\cZ}\ti\cZ)\cong\Hom^{\bG_m}(\bX_{S}^\pm,\cZ)
,\quad 
\bX_{S}^\pm:=\bX^\pm\times_{\ab^1}S.
\end{align}
In addition, there are  $\bG_m$-equivariant morphisms over $\ab^1$
$$\sigma^+:\bX^+\ra\ab^1\times\ab^1
,\quad
\sigma^{-}:\bX^-\ra\ab^1\times\ab^{1}_{-}.$$
The first one is the blow-up at $(0,0)$ and the second one  the blow-up at $(0,\infty)$.
See \cite[Lem.~3.1.10]{Dr2} for details.
Thus for any test scheme $S\ra \ab^1$, 
we have morphisms 
$$\sigma_{S}^\pm:\bX_{S}^\pm\ra S\times\ab^1_\pm$$
where $\ab^1_+=\ab^1$.
Since
$$
\Hom_{\ab^1}(S,\cZ^\pm\times\ab^1)=\Hom^{\bG_m}(S\times\ab^1_\pm,\cZ),$$
pulling back the isomorphisms \eqref{isom33} by $\sigma_{S}^\pm$ yields the maps
\[(\sigma_{S}^\pm)^{*}:\Hom_{\ab^1}(S,\cZ^\pm\times\ab^1)\ra\Hom_{\ab^{1}}(S,\cZ^\pm\times_{\cZ}\ti\cZ)\]
from which we obtain the maps $r^\pm$ in \eqref{r+}.
By \cite[Prop.~3.1.3]{Dr2},  these maps are open embeddings if $\cZ$ is a ft algebraic space over $k$.

\subsubsection{Interpolating family over \texorpdfstring{$\ab^1$}{A1}}
Let $\ti{\cZ}_{t}$ be the fiber over $t\in\ab^1$. 
By \cite[Prop.~2.2.6]{DG2}, the map $\ti{p}$ in \eqref{ti-p} is an isomorphism from 
$\ti\cZ\times_{\ab^{1}}\bG_m$ 
to the graph $\Gamma$ of the action map  $\bG_m\times \cZ\ra\cZ$.
The fiber over 0 is
$$\ti{\cZ}_{0}=\Hom^{\bG_m}(\bX_{0},\cZ)$$
where
$$\bX_{0}=\{(x,y)\in\ab^{2}\,;\,xy=0\}=\bX_{0}^{+}\cup\bX_{0}^{-}$$
with $\bG_m$-equivariant isomorphisms 
$$\ab^1\stackrel{\sim}{\rightarrow}\bX^+_{0},\, t\mapsto (t,0)
,\quad
\ab_{-}^1\stackrel{\sim}{\rightarrow}\bX^-_{0},\, t\mapsto(0,t^{-1}).$$ 
In particular, we have a  morphism
\begin{align}\label{TZ0}\iota:\ti{\cZ}_{0}\ra\cZ^{-}\times_{\cZ^{0}}\cZ^{+}.\end{align}
It is an isomorphism if $\cZ$ is a scheme by \cite[Prop.~2.2.9]{DG1}.
Since the tilde-functor preserves closed immersions by \cite[Prop.~2.3.2]{DG}, 
the map $\iota$ is also an isomorphism if $\cZ$ is an ind-scheme with a $\bG_m$-equivariant presentation.
In general, we have the following.

\bprop\label{int-quot}
For any separated prestack $\cZ$ the map $\iota$ is a monomorphism.
\eprop

\bpf
By construction, we have a commutative diagram
$$\xymatrix{\ti{\cZ}_{0}\ar[d]_{\iota}\ar[r]^-{\ti{p}_{0}}&\cZ\times\cZ\\\cZ^{-}\times_{\cZ^{0}}\cZ^{+}\ar[r]&\cZ^{-}\times\cZ^{+}\ar[u]}$$
By Lemma \ref{ti-mono}, the map $\ti{p}_0$ is a monomorphism.
Thus $\iota$ is also a monomorphism.
\epf

\blem\label{ti-fsmooth}
Let $f:\cZ_1\ra\cZ_2$ be a $\bG_m$-equivariant ind-schematic formally smooth morphism of prestacks.
\hfill
\begin{enumerate}[label=$\mathrm{(\alph*)}$,leftmargin=8mm]
\item
The following induced morphisms are formally smooth
 $$\ti{f}:\ti{\cZ}_1\ra\ti{\cZ}_2,\quad
 f^{+}:(\cZ_1)^{\pm}\ra(\cZ_2)^{\pm},\quad
 f^0:(\cZ_1)^{0}\ra(\cZ_2)^{0}.$$ 
 \item
If $\cZ_2=\Spec(k)$, then the map $\ti{\cZ}_{1}\ra\ab^1$ is formally smooth.
\eenum
\elem

\bpf
Let $I\subset R$ an ideal  such that $I^2=0$.
Set $\bar{R}=R/I$.
We abbreviate $\bX_R=\bX_{\Spec(R)}$ and $\bX_{\bar R}=\bX_{\Spec(\bar R)}$.
Consider a diagram 
$$\xymatrix{\bX_{\bar{R}}\ar[r]^{\bar{\phi}}\ar[d]&\cZ_1\ar[d]^\form\\\bX_{R}\ar[r]^{\phi}&\cZ_2}$$
By formal smoothness of $f$, we can lift $\phi$ to a map $\ti{\phi}:\bX_R\ra\cZ_1$ such that $\ti{\phi}=\bar{\phi}$ modulo $I$.
By Proposition \ref{equiv-lift}, the obstruction to lift $\phi$ in a $\bG_m$-equivariant way lies in $H^{1}(\bG_m, M)$, 
where
$$M=\Map_{\cO_{\bX_{\bar{R}}}}(\bar{\phi}^{*}L_{\cZ_{1}/\cZ_{2}},\cI)$$ 
and  $L_{\cZ_{1}/\cZ_{2}}$ is the relative cotangent complex.
Since $\bG_m$ is diagonalizable,
the group $H^{1}(\bG_m, M)$ vanishes \cite[3.7]{DeG}.
Thus, the map $\ti{f}:\ti{\cZ}_1\ra\ti{\cZ}_2$ is formally smooth by \eqref{ti-Z}.
Note that  the morphism $f$ has relative deformation theory because it is ind-schematic, see Remark \ref{rem-defT1}.
\epf

\subsection{The specialization map}
Given a lft prestack $\cY$, let 
$$\pi:\cY\times\ab^1\ra\cY
,\quad
i_0, i_1:\cY\ra\cY\times\ab^1$$
be the projection and the fp closed embeddings associated with the inclusions $\{0\},\{1\}\subset \ab^1$.
Let 
$$\cD(\cY\times\ab^1)^{\bG_m\text{-}\mono}\subset\cD(\cY\times\ab^1)$$ 
be the full subcategory generated by the essential image of the $!$-pullback
by the obvious map 
$$\cY\times\ab^1\to\cY\times[\ab^1/\bG_m].$$
An object of $\cD(\cY\times\ab^1)^{\bG_m\text{-}\mono}$
is called a $\bG_m$-monodromic complex.
For any $\bG_m$-monodromic complex $K$ on $\cY\times\ab^1$ we abbreviate
$$K_{0}=(i_0)^!K,\quad K_1=(i_1)^!K.$$
We will need the following lemma in \S\ref{finalCT}.

\blem\label{sp}
For any lft prestack $\cY$ and any $K\in\cD(\ab^1\times\cY)^{\bG_m\text{-}\mono}$ there is a functorial specialization map
$\Sp_{K}:K_{1}\ra K_0.$
\elem

\bpf
First, assume that $\cY=Y$ is a $k$-scheme of ft.
The natural transformation
\begin{equation}
(i_0)^!\ra (i_0)^!\pi^!\pi_!\cong\pi_!
\label{io-nat}
\end{equation}
is an equivalence by \cite[\S4.1]{DG1}.
We define $\Sp_K$ to be the obvious map
\[K_1\cong\pi_!(i_1)_!(i_1)^!K\to\pi_{!}K\cong K_0.\]
This map is functorial in the following way. 
Given a morphism of $k$-schemes of ft $f:Y'\ra Y$, we set $K'=(f\times\id_{\ab^1})^{!}K$.
Then, we have the commutative diagram
$$\xymatrix{K_1'\ar[r]^{\Sp_{K'}}&K'_{0}\\f^{!}K_1\ar[u]^{\sim}\ar[r]^{f^{!}\Sp_{K}}&f^{!}K_0\ar[u]^{\sim}}$$
Now, for any lft prestack $\cY$ we define $\Sp_{K}$
to be the projective system $(\Sp_{f^{!}K})$ of its $!$-restriction to affine ft-schemes $f:Y\ra\cY$.
\epf

If $K=L\boxtimes\omega_{\ab^1}$ with $L\in\cD(\cY)$, 
then the specialization $\Sp_K$ is the identity morphism 
\begin{equation}
K_1\cong L\cong K_0.
\label{sp0}
\end{equation}
For any homotopically ind-schematic morphism $f:\cY'\to\cY$ of lft prestacks and any 
$K'\in\cD(\cY'\times\ab^1)^{\bG_m\text{-}\mono}$ we write $K=(f\times\id_{\ab^1})_{\bullet}K'$.
Then, we have the commutative diagram
$$\xymatrix{K_1\ar[d]^{\sim}\ar[r]^{\Sp_{K}}&K_{0}\ar[d]^{\sim}\\
f_{\bullet}K'_1\ar[r]^{f_{\bullet}\Sp_{K'}}&f_{\bullet}K'_0}$$

\subsection{The hyperbolic diagram for the affine Grassmannian}\label{sec-Hyp}
Let  $k$ be an algebraically closed field.
Recall that $G$ is  a minimal simply connected KM group over $k$.
We assume that the characteristic of $k$ is zero, or that it is arbitrary if $G$ is an affine group in order to use Theorem \ref{fsm}.
Applying the previous construction with $\cZ=G$ 
and the $\bG_m$-action given by the adjoint action of  $2\Lrho$, from \eqref{ZZZ} we get
$$\cZ^{+}=B,\quad\cZ^{0}=T,\quad\cZ^{-}=B^{-}.$$
Next, we apply \eqref{ZZZ} with $\cZ=\Gr_G$ 
and the $\bG_m$-action given by the adjoint action of  $2\Lrho$.

\bprop\label{id-att}
The following obvious inclusions are isomorphisms
$$
\Gr_{B}\ra(\Gr_{G})^{+},\quad
\Gr_{T}\ra(\Gr_{G})^{0},\quad
\Gr_{B^{-}}\ra(\Gr_{G})^{-}$$
\eprop

\bpf
The ind-scheme $G_\cO$ is formally smooth by Lemma \ref{loop-rep} and Theorem \ref{fsm}.
Thus, the $\bG_m$-equivariant map $G_K\ra\Gr_{G}$ is formally smooth by Lemma \ref{lem-contract}.
By Lemma \ref{ti-fsmooth} it yields formally smooth maps
\[(G_K)^{+}\ra(\Gr_{G})^{+},\quad
(G_K)^{0}\ra(\Gr_{G})^{0}.\]
Since $G_K$ is ind-affine, we have  
\begin{align}\label{FPGK}
(G_K)^{+}=B_K,\quad(G_K)^{0}=T_K.
\end{align}
Hence, the formally smooth maps
\[B_K\ra (\Gr_{G})^{+},\quad
T_K\ra(\Gr_{G})^{0},\]
factor through $\Gr_{B}$ and $\Gr_{T}$.
Since $B_K\to\Gr_B$ and $T_K\to\Gr_T$ are \'etale torsors with groups $T_\cO$ and $B_\cO$, 
there are sections \'etale locally, see \S\ref{sec-ind-sch}.
Thus, by Proposition \ref{et-lift}, the morphisms 
$$\Gr_{B}\ra(\Gr_{G})^{+}
,\quad
\Gr_{T}\ra(\Gr_{G})^{0}$$ 
are formally smooth. 
We also have a bijection on $k$-points.

By Theorem \ref{Bfond} for each $\mu\in \LLa$ the map $\Gr_{B}^{\mu}\ra\Gr_{G}$ is fp locally closed.
Thus, since $(\Gr_{G})^{+}\ra\Gr_{G}$ is a monomorphism, 
the morphism $\Gr_{B}^{\mu}\ra(\Gr_{G})^{+}$ is fp locally closed and formally smooth, hence it is a qc open immersion.
Thus the morphism $\Gr_{B}\ra(\Gr_{G})^{+}$ is a qc open immersion which is a bijection on $k$-points.
Hence it is an isomorphism.

Since $\Gr_{T}\ra\Gr_{B}$ is fp closed by Proposition \ref{ind-quot},
for each $\nu\in \LLa$ the composed map 
$$\Gr_{T}^{\nu}\hra\Gr_{B}^{\nu}\hra\Gr_{G}$$ is fp locally closed by Theorem \ref{Bfond} again.
So the morphism $\Gr_{T}^{\nu}\ra(\Gr_G)^{0}$ is a qc open immersion,
and $\Gr_{T}\ra(\Gr_G)^{0}$ is an isomorphism.
\epf

\blem\label{clopen}
Let $j:\Gr_{T}\ra\Gr_{B}\times_{\Gr_{G}}\Gr_{B^{-}}$ be the diagonal embedding.
\begin{enumerate}[label=$\mathrm{(\alph*)}$,leftmargin=8mm]
\item
The map $j$ is an fp closed and open immersion. We have
\[\Gr_T\cong\bigsqcup_{\nu\in \LLa}\Gr_{B}^{\nu}\times_{\Gr_{G}}\Gr_{B^{-}}^{\nu}
=\Gr_{T}\times_{(\Gr_{T})^{2}}(\Gr_{B}\times_{\Gr_{G}}\Gr_{B^{-}}).\]
\item
We have the following Cartesian diagram, where $\Delta$ is the diagonal,
$$\xymatrix{\Gr_{T}\ar[r]^-{j}\ar[d]&\Gr_{B}\times_{\Gr_{G}}\Gr_{B^{-}}\ar[d]\\\Gr_{G}\ar[r]^-{\Delta}&\Gr_{G}\times\Gr_{G}}$$
\eenum
\elem

\bpf
By Proposition \ref{ind-quot} the maps $\Gr_{T}\ra\Gr_{B}$ and $\Gr_{T}\ra\Gr_{B^{-}}$ are fp closed.
Thus, since the map $$\Gr_{B}\times_{\Gr_{G}}\Gr_{B^{-}}\ra\Gr_{B}\times\Gr_{B^{-}}$$ is a monomorphism, the map $j$ is fp closed.
It is enough to check that $j$ is formally smooth. 
As $j$ is equivariant for left multiplication on the left hand side and diagonal action on the right hand side, 
it is enough to prove formal smoothness at $1$.
Since $\Gr_G$ is a sheaf, by Proposition \ref{id-att}, we have 
\[\Hom^{\bG_m}(\bP^{1},\Gr_{G})=\Gr_{B}\times_{\Gr_{G}}\Gr_{B^{-}}.\]
Since $G_K$ is ind-affine, by \eqref{FPGK} we have 
$$\Hom^{\bG_m}(\bP^{1},G_{K})=B_K\times_{G_K}B_K^{-}=T_K$$
The obvious map $G_K\ra\Gr_{G}$ is $\bG_m$-equivariant, hence yields a map
\[\phi:\Hom^{\bG_m}(\bP^{1},G_K)\ra\Hom^{\bG_m}(\bP^{1},\Gr_{G})\]
which factors as
\[T_K\ra\Gr_{T}\stackrel{j}{\rightarrow}\Gr_{B}\times_{\Gr_{G}}\Gr_{B^{-}}.\]
Since $T_K\to\Gr_T$ is an \'etale $T_\cO$-torsor, there are sections \'etale locally, see \S\ref{sec-ind-sch}.
Thus, by Proposition \ref{et-lift}, it is enough to prove that $\phi$ is formally smooth to prove (a).
We first prove the formal smoothness of the map
$$\phi':\Hom(\bP^{1},G_K)\ra\Hom(\bP^{1},\Gr_{G}).$$
Let $R$ be a ring and $I\subset R$ a square zero ideal.
Let $y:\bP_R^{1}\ra\Gr_{G}$, such that the map $\bar{y}= y\mod I$ lifts as $\ti{y}:\bP^{1}_{R/I}\ra G_K$.
We form the pullback
$$\xymatrix{E\ar[d]\ar[r]&G_K\ar[d]\\\bP^{1}_R\ar[r]^{y}&\Gr_{G}}$$
It yields a $G_\cO$-torsor $E$ over $\bP_R^{1}$
for the étale topology, representable by an ind-affine ind-scheme with a section $\bar{\sigma}$ modulo $I$.
The ind-scheme $G_\cO$ is formally smooth by Lemma \ref{loop-rep} and Theorem \ref{fsm}.
Hence $\bar{\sigma}$ lifts Zariski locally on $\bP^{1}_{R}$ and by writing $E\cong\colim E_a$ and 
applying deformation theory for enoughly large $E_a$ the obstruction to lift $\bar{\sigma}$ globally sits in the cohomology group
\[ H^{1}(\bP^{1}_{R/I}, \colim\Map(\bar{x}^{*}(\Omega^{1}_{E_a/\bP^{1}_{R/I}}), I))\cong 
H^{1}(\bP^{1}_{R/I},\bar{x}^{*}(T_{E/\bP^{1}_{R/I}})\otimes I),\]
where the last equality follows from formal smoothness and \cite[Prop.~7.11.8]{BD}.
Since $E\vert_{\bP^{1}_{R/I}}$ is trivial, we have
$$\bar{x}^{*}(T_{E/\bP^{1}_{R/I}})=\kg[s]\otimes_{R/I}\cO_{\bP^{1}_{R/I}}.$$
Since $H^{1}(\bP^{1}_{R/I},\cO_{\bP^{1}_{R/I}})=0$, the obstruction vanishes.
Using Proposition \ref{equiv-lift} similarly to Proposition \ref{att-smooth}, 
we can choose this lift to be $\bG_m$-equivariant, because $\bG_m$ is diagonalizable.
In particular, the map $\phi'$ is formally smooth.
We have a decomposition into open and closed subfunctors
\[\Gr_{B}\times_{\Gr_{G}}\Gr_{B^{-}}=\bigsqcup_{(\nu,\mu)\in (\LLa)^{2}}\Gr^{\nu}_{B}\times_{\Gr_{G}}\Gr_{B^{-}}^{\mu},\]
obtained by pullback through the map 
\[\Gr_{B}\times_{\Gr_{G}}\Gr_{B^{-}}\ra\Gr_{T}\times\Gr_T.\]
By Proposition \ref{ST-comp2}, the map $j$ identifies $\Gr_{T}$ with the component labelled by the diagonal pairs 
$(\nu,\nu)$.

To prove the claim (b), 
taking only the connected components indexed by $(\nu,\nu)$ with $\nu\in X_{*}(T)$
yields an open closed immersion 
$$\Gr_{B}\times_{\Gr_{T}}\Gr_{B^{-}}\hra\Gr_{B}\times\Gr_{B^{-}}.$$
We also have 
$$\Gr_{G}\times_{(\Gr_{G})^{2}}(\Gr_{B}\times\Gr_{B^{-}})=\Gr_{B}\times_{\Gr_{G}}\Gr_{B^{-}}.$$ 
Further, if we restrict the map 
$$\Gr_{B}\times_{\Gr_{G}}\Gr_{B^{-}}\ra(\Gr_T)^2$$
to the diagonal $\Gr_{T}\hra(\Gr_T)^2$, 
we get an isomorphism 
\begin{align} \label{isom3}\Gr_{G}\times_{(\Gr_{G})^{2}}(\Gr_{B}\times_{\Gr_{T}}\Gr_{B^{-}}) 
\cong\Gr_T\times_{(\Gr_{T})^2}(\Gr_{B}\times_{\Gr_{G}}\Gr_{B^{-}})\end{align}
and (b) follows from (a).
\epf

\subsection{Interpolating family of groups}\label{sec-Inter}
We keep the same notation as in \S\ref{sec-Hyp}. 
We first set $\cZ=G$.
The ind-ft ind-scheme $\ti G$ is closed in $G\times G\times\ab^1$ by Lemma \ref{ti-mono}. 
Let $\Gm$ be
the graph of the $\bG_m$-action on $G$.
Since $\ti{\cZ}$ is functorial in $\cZ$, the ind-scheme $\ti{G}$ is a group ind-scheme such that
\begin{align}\label{ZG}
\ti{G}\times_{\ab^1}\bG_m\cong\Gm\cong G\times \bG_m
,\quad
\ti{G}_0\cong B\times_{T}B^{-}.
\end{align}
Let $\ov{\Gm}$ be the closure of $\Gm$ in $\ti G$.

\blem\label{adh-G}\hfill
\begin{enumerate}[label=$\mathrm{(\alph*)}$,leftmargin=8mm]
\item
The closure $\ov{\Gm}$ is isomorphic to $\ti{G}$.
\item
The ind-scheme $\ti{G}$ is ind-flat over $\ab^{1}$.
\eenum
\elem

\bpf
By \cite[Prop.~2.5.5]{DG1} for a $k$-smooth affine scheme $Z$, we have the same description for $\ti{Z}$. 
The ind-scheme $G$ is formally smooth, but it may not be ind-smooth.
So we can not aply loc.~cit.~to $G$.
Consider the open Bruhat cell $S=U^{-}\times T\times U$.
By \cite[Prop.~2.3.2]{DG1}, since the functor $\cZ\mapsto\ti \cZ$ commutes with filtered colimits, 
we have that $\ti{S}$ is open in $\ti{G}$.
Since the functor $\cZ\mapsto\ti \cZ$ commutes also with finite products and filtered colimits, 
by \cite[Prop.~2.4.4,~2.4.6]{DG1} we have
\[\ti{S}=\ti{U}^{-}\times_{\ab^1}\ti T\times_{\ab^1}\ti{U}=\ab^{1}\times S.\]
In particular $\ti{S}\times_{\ab^1}\bG_m$ is dense in $\ti{S}$. 
Thus it is enough to prove that $\ti{S}$ is schematically dense in $\ti G$.
Fix a presentation  $G=\colim X_a$ as an ind-ft scheme.
We have $S=\colim S_a$,
and there is an open immersion $\ti{S}_a=\ab^{1}\times S_a\hra\ti{X}_a$
that is schematically dense over each fiber over $\ab^{1}$. By \cite[Thm.~11.10.9]{EGAIV3}, 
we deduce that $\ti{S}_a$ is schematically dense in $\ti{X}_a$.
Thus $\ti{S}$ is schematically dense in $\ti{G}$.
Finally, over $\bG_m$ the ind-scheme $\Gm$ is ind-flat.
Flatness over $\ab^1$  is equivalent to torsion freeness 
and the schematic closure  of a flat scheme over $\bG_m$ is flat over $\ab^1$.
\epf

One of the main input is the following proposition.

\bprop\label{ti-quot}
The quotient stack $[(G\times G\times\ab^1)/\ti{G}]$ is $G\times G$-equivariantly open
in an ind-affine ind-scheme over $\ab^1$.
\qed
\eprop

The proof will be given in \S\ref{PfPROP}.

\bcor\label{cor-quot}
The map $\la:\Gr_{\ti{G}}\ra\Gr_G\times\Gr_G\times\ab^1$ is an fp immersion.
\ecor

\bpf
By Propositions  \ref{ind-quot} and \ref{ti-quot}, the map $\la$ is locally closed.
By Proposition \ref{ti-mono}, after a base change by the map 
$$G_K\times G_K\times\ab^1\ra\Gr_G\times\Gr_G\times\ab^1,$$
the fiber product is an immersion in an ind-ft scheme.
Thus $\la$ is an fp immersion by Lemma \ref{quotH}.
\epf

Now, we apply \eqref{ZZZ} with $\cZ=\Gr_G$. 
We want to compare $\wti{\Gr}_{G}$ with $\Gr_{\ti{G}}$.
We first define a map
$\Gr_{\ti{G}}\ra \wti{\Gr}_{G}.$
Since the assignment $\cZ\mapsto\ti\cZ$ is functorial by \S\ref{ti-f}, we have a map
$\wti{G_K}\ra\wti{\Gr}_{G}.$
Further, we have $\wti{G_K}\cong(\ti{G})_K$ because 
$$\Hom^{\bG_m}(\bX,G_{K})=\Hom^{\bG_m}(\bX[t,t^{-1}],G)=\Hom(\Spec(k[t,t^{-1}]),\Hom^{\bG_m}(\bX,G)).$$
This isomorphism is $(\ti{G})_\cO=\wti{G_\cO}$-equivariant.
Thus it factors to the quotient and we get a morphism
\begin{equation}
\eta:\Gr_{\ti{G}}\ra \wti{\Gr}_{G}.
\label{eta-}
\end{equation}

\bprop\label{ti-gr}
The map \eqref{eta-} is an isomorphism.
\eprop

\bpf
Since the stack $\Gr_{G}$ is separated, Lemma \ref{ti-mono} implies that the map
$$\wti{\Gr}_{G}\ra\Gr_G\times\Gr_G\times\ab^1$$ is a monomorphism.
So Corollary \ref{cor-quot} implies that $\eta$ is fp locally closed.
Moreover, it is bijective on $k$-points. 
Above a point $t\in\bG_m$ this is obvious.
Above the point $t=0$ we have an isomorphism
\begin{equation}
\Gr_{\tilde{G}_0}=\Gr_{B\times_{T}B^{-}}\cong\Gr_{B}\times_{\Gr_T}\Gr_{B^{-}}
\label{0-fib}
\end{equation}
which follows from the isomorphism
\[\bB(B\times_{T}B^{-})\cong B\backslash T/B^{-}\cong \bB B\times_{\bB T}\bB B^{-}\]
(note that the isomorphism
$$\bB(G_1\times_{G_2}G_3)\cong\bB G_1\times_{\bB G_2}\bB G_3$$ 
does not hold in general, as can be seen from the following example:
$G_1=B,$ $G_3=B^{-}$ and $G_2=G$).
Thus \eqref{0-fib} identifies $\eta_0$ with a map
\[\Gr_{\ti{G}_{0}}=\Gr_{B}\times_{\Gr_{T}}\Gr_{B^{-}}\ra(\wti{\Gr}_{G})_{0}.\]
By Proposition \ref{int-quot}, we also have a monomorphism
\begin{equation}
(\wti{\Gr}_{G})_{0}\ra\Gr_{B}\times_{\Gr_{T}}\Gr_{B^{-}}.
\label{eta0}
\end{equation}
Thus the fiber $\eta_0$ of $\eta$ is an isomorphism.
Now, to prove that $\eta$ is an isomorphism, it suffices to prove that it is formally smooth. 
As $\ti{G}_K\ra\Gr_{\ti{G}}$ has sections étale locally and is ind-fp, by Proposition \ref{et-lift}, 
it suffices to prove that the map
$\ti{G}_K\ra\wti{\Gr_{G}}$
is formally smooth.
Since $\ti{G}_K=\wti{G_K}$, this follows from Lemma \ref{ti-fsmooth}, 
by applying the tilde functor to $G_K\ra\Gr_{G}$ that is already formally smooth.
\epf

By \eqref{r+}, for any prestack $\cZ$ with a $\bG_m$-action we have morphisms
\begin{equation}
r^\pm:\cZ^\pm\times\ab^1\to\ti\cZ^\pm
\end{equation}
where $\ti\cZ^\pm$ is the fiber product relative to the map $\pi^\pm$ in \eqref{ti-p} given by
\begin{align}\label{tiZ+-}
\ti\cZ^\pm=\cZ^\pm\times_{\cZ}\ti{\cZ}
\end{align}
Set $\cZ=\Gr_{G}$ with the $\bG_m$-action given by $2\Lrho$.
Propositions \ref{id-att} and \ref{ti-gr} yield
\begin{align}\label{ZZZGr}
\cZ^{+}=\Gr_{B},\quad
\cZ^{-}=\Gr_{B^{-}},\quad
\ti{\cZ}=\Gr_{\ti{G}}.
\end{align}
So the morphisms $r^\pm$  become
\begin{equation}
r^+:\Gr_B\times\ab^1\ra\Gr_{B}\times_{\Gr_{G}}\Gr_{\ti G},\quad 
r^-:\Gr_{B^{-}}\times\ab^1\ra\Gr_{B^{-}}\times_{\Gr_{G}}\Gr_{\ti{G}}
\label{r-Gr}
\end{equation}
where the fiber product $\Gr_{B}\times_{\Gr_{G}}\Gr_{\ti G}$ is relative to the map $\pi^+$ and 
the fiber product $\Gr_{B^-}\times_{\Gr_{G}}\Gr_{\ti G}$ is relative to the map $\pi^-$.

\bprop\label{fib-prod}
Let $\cZ=\Gr_{G}$ with the $\bG_m$-action coming from $2\Lrho$.
The maps $r^\pm$ are qc open embeddings.
\eprop

\bpf
Let us prove the assertion for $r^+$.
The case of $r^-$ is similar.
First we claim that $r^+$ is a monomorphism.
Consider the chain of maps
\begin{equation}
\Gr_B\times\ab^1\stackrel{r^+}{\rightarrow}
\Gr_{\ti{G}}\times_{\Gr_{G}}\Gr_{B}\stackrel{(\ti{p},\id)}{\rightarrow}
(\Gr_G\times\Gr_G\times\ab^1)\times_{\Gr_{G}}\Gr_{B}
=\Gr_G\times\Gr_B\times\ab^1
\label{mono-comp}
\end{equation}
where the fiber product $(\Gr_G\times\Gr_G\times\ab^1)\times_{\Gr_{G}}\Gr_{B}$ is relative to the second projection
$\Gr_G\times\Gr_G\to\Gr_G$.
The composed map is the base change to $\ab^1$ of the composition of the diagonal of $\Gr_{B}$ 
with the map $\Gr_{B}\times\Gr_{B}\ra\Gr_G\times\Gr_B$ which is the identity on the second factor. 
So it is a monomorphism. So $r^+$ is also a monomorphism.
Next, we decompose
$$\Gr_{B}=\bigsqcup_{\mu\in \LLa} \Gr_{B}^{\mu}.$$
We claim that the map
\[r^{+,\mu}:\Gr_{B}^{\mu}\ra\Gr_{\ti{G}}\times_{\Gr_{G}}\Gr_{B}^{\mu}.\]
 is an fp immersion.
Indeed, if we use the same composition as in \eqref{mono-comp}, we obtain a map
\[\Gr_B^\mu\times\ab^1\ra\Gr_G\times\Gr_B^\mu\times\ab^1.\]
By Theorem \ref{Bfond}, the map $\Gr_{B}^{\mu}\ra\Gr_{G}$ is an fp immersion, as well as the diagonal of $\Gr_{B}^{\mu}$.
By  Corollary \ref{cor-quot}, the morphism 
$\Gr_{\ti{G}}\ra\Gr_G\times\Gr_G\times\ab^1$ is an fp immersion.
We deduce that $r^{+,\mu}$ is an fp immersion.
To conclude, it is enough to prove that $r^+$ is formally smooth. Let first check formal smoothness on the fibers. 
Over a point $t\in\bG_m$ this is clear.
Over the point $t=0$, using \eqref{eta0}, we must prove that the map
\[r_0^+:\Gr_{B}\times_{\Gr_T}\Gr_T
=\Gr_{B}\ra\Gr_B\times_{\Gr_G}(\Gr_{B^-}\times_{\Gr_{T}}\Gr_{B})
=\Gr_{B}\times_{\Gr_T}(\Gr_{B}\times_{\Gr_{G}}\Gr_{B^{-}}),\]
is a qc open immersion. 
This map is just the base change to $\Gr_{B}$ of the map
\[\Gr_{T}\ra\Gr_{B}\times_{\Gr_{G}}\Gr_{B^{-}},\]
which is fp closed and open by Lemma \ref{clopen}. 
Now, to deduce formal smoothness for $r^+$ it is enough to check it after pulling back by the étale torsor
$B_K^{\mu}\times G_K\ra\Gr_{B}^{\mu}\times\Gr_{G}$, where $B_K^{\mu}$ is the connected component of $B_K$
containing $s^\mu$. 
Then, we get that the pullback is an open immersion and we just apply Lemma \ref{quotH} to go back to the map  $r^+$. 
After this pullback, we get an fp immersion between ind-ft ind-schemes over $\ab^1$ and a formally smooth source.
Thus by  
Proposition \ref{ind-smooth}, it is enough to check formal smoothness on fibers over $\ab^{1}$. This follows from the above.
\epf

\section{The constant term}\label{HLGr}

In this section we introduce and study the constant term functor for the affine Grassmannian. 
The main results are Theorems \ref{Braden} and \ref{t-exact}, which are proved in \S\S\ref{defCT}-\ref{finalCT}
and \S\ref{exactCT} respectively.

\subsection{The definition of the constant term functors}\label{defCT}
We consider the following diagram of $\infty$-stacks
\begin{align}\label{cart1}
\begin{split}\xymatrix{
\Gr_T\ar[dr]^-{j}\ar@/^1pc/[drr]^{i^{+}}\ar@/_1pc/[ddr]_-{i^-}&&&\\
&\Gr_B\times_{\Gr_G}\Gr_{B^-}\ar[d]^{\ti{p}^+}\ar[r]_-{\ti{p}^-}&\Gr_B\ar[d]^-{p^+}\ar[r]_-{q^+}&\Gr_T\\
&\Gr_{B^-}\ar[r]_-{p^-}\ar[d]^-{q^-}&\Gr_G&\\&\Gr_T}
\end{split}
\end{align}
By Theorem \ref{Bfond}, the maps $p^{+}$, $p^-$
are fp immersions when restricted to each connected component $\Gr_B^\nu$, $\Gr_{B^-}^\nu$ of $\Gr_B$, $\Gr_{B^-}$.
The affine Grassmannian $\Gr_{G}$ satisfies gluing by Remark \ref{glustk}.
Hence, since $p^\pm$ is a fp immersion, the functors $(p^\pm)_{*}$, $(p^\pm)^{*}$ and $(p^\pm)^{!}$ are well-defined by \S\ref{funct}.
By base change, the map $\ti{p}^\pm$ is also a fp immersion when restricted to the connected components of $\Gr_{B}$ (resp.~$\Gr_{B^{-}}$).
Hence the functors $(\ti{p}^\pm)_{*}$, $(\ti{p}^\pm)^{*}$ and $(\ti{p}^\pm)^{!}$ are also well-defined.
By Proposition \ref{ind-quot}, since $B/T=U$ and $B^{-}/T=U^{-}$ are 
ind-ft ind-affine ind-schemes,
the maps $i^\pm$ are fp closed immersions.
Since $\Gr_{B}$ satisfies gluing, the functors $(i^\pm)_{*}$, $(i^\pm)^{*}$ and $(i^\pm)^{!}$ are well-defined by \S\ref{funct}.
Using the bar-complexes for the affine Grassmannians, we prove that all prestacks are lft.
Hence the functors 
$(\ti{p}^\pm)_{!}$, $(p^\pm)_{!}$, $(i^\pm)_{!}$, $(q^{\pm})_{!}$ are well-defined by Proposition \ref{prop-lft}.
We define the constant term functors
\begin{align}\label{CT}
\begin{split}
\CT_{*}&=(i^{+})^{*}(p^{+})^{!}:\cD(\Gr_{G})\ra\cD(\Gr_{T}),\\
\CT_{!}^{-}&=(i^{-})^{!}(p^{-})^{*}:\cD(\Gr_{G})\ra\cD(\Gr_{T}).
\end{split}
\end{align}
There is a morphism of functors
\begin{equation}
\CT_*\ra\CT_{!}^{-},
\label{adj-map}
\end{equation}
which is defined as follows  
\begin{align*} 
(i^{+})^{*} (p^{+})^{!}&\ra (i^{+})^{*} (p^{+})^{!} (p^{-})_* (p^{-})^{*}\\
&\cong (i^{+})^{*} (\ti{p}^{+})_{*} (\ti{p}^{-})^! (p^{-})^{*}\\
&\ra(i^{+})^{*} (\ti{p}^{+})_{*}  j_{*}  j^{*} (\ti{p}^{-})^! (p^{-})^{*}\\
&\cong (i^+)^{*} (i^+)_* (i^-)^{!} (p^-)^{*}\\
&\cong (i^-)^{!} (p^-)^{*}
\end{align*}
Here, we used the base change for the Cartesian square in \eqref{cart1},
the fact that $j$ is closed and open by Lemma  \ref{clopen}, hence $j^!\cong j^*$,
and the full faithfulness of $(i^+)_*$ which follows from \cite[Lem.~5.4.1]{BKV}.

Let $\cD(\Gr_{G})^{\bG_m\text{-}\mono}$ be the full subcategory of $\bG_m$-monodromic objects in $\cD(\Gr_{G})$, i.e., 
the full subcategory strongly generated by $\bG_m$-equivariant complexes, i.e., 
the complexes given by a finite iteration of taking a cone of a map in $\cD(\Gr_G)$.
We define $\cD(\Gr_{B})^{\bG_m\text{-}\mono}$ and $\cD(\Gr_{B^-})^{\bG_m\text{-}\mono}$ similarly.

\bthm\label{Braden}
The morphism of functors $\CT_*\ra\CT_{!}^{-}$ in \eqref{adj-map} is an equivalence on $\cD(\Gr_{G})^{\bG_m\text{-}\mono}$.
\ethm
To prove the theorem, we follow the strategy of \cite{DG1}, \cite{DG2}.

\subsection{The contraction principle}\label{sec-cont-principle}
To prove Theorem \ref{Braden}, we first need a contraction principle.
Since $\Gr_{B}=\bigsqcup_{\nu\in \LLa}\Gr_{B}^{\nu}$, the map $q^+$ decomposes as
$q^+=\bigsqcup q_{\nu}$ with
$q_{\nu}:\Gr_{B}^{\nu}\ra\Spec(k).$
Let $T_K^0$ be the neutral component of $T_K$.
By Lemma \ref{rem-contract2}, the $U_\cO$-torsor 
$$U_K\cdot s^{\nu}\cdot T_K^0\ra \Gr_{B}^{\nu}.$$
 is contractive and the map $q^+$ is homotopically ind-schematic.
Thus the functor $(q^+)_\bullet$ is well-defined. The discussion is the same for $q^{-}$.
Let us change our notation. Following \eqref{ZZZGr} we now write
$$Z=\Gr_G,\quad Z^{+}=\Gr_B,\quad Z^{0}=\Gr_T,\quad Z^{-}=\Gr_{B^{-}},\quad \ti Z=\Gr_{\ti G}.$$

\bprop\label{adj}
The following morphisms of functors are equivalences
\begin{enumerate}[label=$\mathrm{(\alph*)}$,leftmargin=8mm]
\item
$(i^{-})^{!}\ra(i^{-})^{!} (q^{-})^{!}  (q^-)_!=(q^{-}  i^{-})^{!} (q^{-})^{!}\ra (q^-)_!$ on $\cD(\Gr_{B^-})^{\bG_m\text{-}\mono}$, 
\item
$(q^+)_\bullet\ra (q^+)_\bullet  (i^{+})_{*} (i^{+})^{*}\cong(q^+)_\bullet (i^{+})_{\bullet} (i^{+})^{*}\cong(i^{+})^{*}$
on $\cD(\Gr_{B})^{\bG_m\text{-}\mono}$.
\eenum
\eprop

\bpf
The first isomorphism in (b) follows from Lemma \ref{ind-bullet0}, because $i^+$ is ind-ft ind-schematic,
and the second one from Proposition \ref{bul-comp-proj}. 
To prove (a) we abbreviate $i=i^{-}$ and $q=q^{-}$.
Consider the commutative triangle
$$\xymatrix{&\ov{Z}^-\ar[d]^{\psi}\ar@/_1pc/[dl]_-{\ov{q}}\\
Z^{0}\ar[ur]_-{\ov{\imath}}\ar[r]^-{i}&Z^{-}\ar@/^1pc/[l]^-{q}}$$
with 
$\ov{Z}^-=\bigsqcup_{\nu\in \LLa} U_K\cdot s^{\nu}\cdot T_K^0$.
Note that $\ov{Z}^-$ is an ind-ft ind-scheme with a $\bG_m$-equivariant presentation that contracts to $s^{\nu}$.
Consequently, using \cite[Prop.~3.2.2 (b)]{DG1} and a colimit argument, 
we get that $\ov{i}^{!}\cong\ov{q}_!$ when restricted to the $\bG_m$-monodromic objects.
Thus, the equality $\ov{q}=q \psi$ yields
$$i^{!}\cong\ov{\imath}^{!}\circ\psi^{!}\cong\ov{q}_!\circ\psi^!\cong q_{!}\circ\psi_{!}\circ\psi^{!}\cong q_!,$$
where the last equality is Lemma \ref{pul-push}.
The proof of (b) is similar, applying the contraction principle upstairs to the pair of functors
$(\ov{\imath}^{+})^{*}$ and $(\ov{q}^+)_*$ and \cite[3.2.2 (a)]{DG1} instead, 
because $(q^+)_\bullet=(\ov{q}^+)_* \psi^{*}$.
\epf

In particular, the constant term functors in \eqref{CT}, when restricted to $\bG_m$-monodromic complexes,  
are given by the following formulas
\begin{align}\label{bis-CT}\CT_{*}=(q^{+})_\bullet (p^{+})^{!},\quad\CT_{!}=(q^{-})_! (p^{-})^{*}\end{align}
The contraction principle in Proposition \ref{adj} and \eqref{adj-map} give a morphism of $\bG_m$-monodromic complexes
$$(q^{+})_{\bullet} (p^{+})^{!}\cong (i^{+})^{*} (p^{+})^{!}\to(i^{-})^{!} (p^{-})^{*}\cong (q^{-})_{!} (p^{-})^{*}$$
The adjoint pair of functors $((q^{-})_{!} (p^{-})^{*},(p^{-})_{*} (q^{-})^{!})$ yields a map
\begin{equation}
(q^{+})_{\bullet} (p^{+})^{!}  (p^{-})_{*} (q^{-})^{!}\ra\id_{\cD(Z^{0})},
\label{unit1}
\end{equation}
On the other hand, the left adjoint of $(p^{-})_{*} (q^{-})^{!}$ is $(q^-)_{!}(p^{-})^{*}$.
Thus, Theorem  \ref{Braden} follows from the following statement which yields an isomorphism of functors
$(q^-)_{!}(p^{-})^{*}\cong(q^{+})_{\bullet} (p^{+})^{!}$.

\bthm\label{brad2}
The morphism \eqref{unit1} is the restriction to the $\bG_m$-monodromic categories of
the co-unit of an adjunction for the pair of functors $((q^{+})_{\bullet} (p^{+})^{!},(p^{-})_{*} (q^{-})^{!})$.
\ethm

To prove Theorem  \ref{brad2}, we need the unit of this adjunction. 
The morphism \eqref{unit1} can be obtained from the morphism \eqref{ker-form3} obtained from to the diagram
$$\xymatrix{&Z^{0}\ar[d]^{j}\ar[dr]\ar[dl]\\
Z^0&Z^{+}\times_{Z}Z^{-}\ar[d]^{r}\ar[r]^-{\al_-}\ar[l]_-{\al_+}&Z^0\\
&Z^{0}\times Z^{0}\ar[ur]\ar[ul]&}$$
Here $\al_\pm$ are the composed maps
$$Z^{+}\times_{Z}Z^{-}\ra Z^{\pm}\ra Z^{0}.$$
Using this, we construct this unit via the interpolation.

\subsection{Equivariant version}
Recall that $Z=\Gr_G$.
As in \cite[\S 3.4]{DG1}, we consider the $\bG_m$-equivariant version.
Set 
$$\cZ^{0}=Z^{0}/\bG_m,\quad
\cZ^{\pm}= Z^{\pm}/\bG_m,\quad
\cZ= Z/\bG_m,\quad
\ti{\cZ}=\ti{Z}/\bG_m$$
where the last quotient is relative to the antidiagonal embedding $\bG_m\ra(\bG_m)^2$ and \eqref{gm-eq}.
We have the obvious morphisms
\[\kp^{\pm}:\cZ^{\pm}\ra\cZ,\quad \kq^{\pm}:\cZ^{\pm}\ra\cZ^0.\]
The morphism \eqref{unit1} lifts to the $\bG_m$-equivariant categories.
This yields a morphism
\begin{equation}
(\kq^{+})_{\bullet} (\kp^{+})^{!} (\kp^{-})_{*} (\kq^{-})^{!}\ra\id_{\cD(\cZ^{0})}
\label{brad-equiv}
\end{equation}
Theorem \ref{brad2} follows from the following as in \cite[Thm.~3.4.3]{DG1}. 

\bthm\label{brad3}\hfill
\begin{enumerate}[label=$\mathrm{(\alph*)}$,leftmargin=8mm]
\item
The morphism \eqref{brad-equiv} is the co-unit of an adjunction for the pair of functors
$((\kq^{+})_{\bullet} (\kp^{+})^{!}\,,\,(\kp^{-})_{*} (\kq^{-})^{!})$.
\item
The equivariant version implies Theorem $ \ref{brad2}$.
\eenum
\ethm

We now focus on the construction of the unit of the adjunction in this setting.

\subsection{Construction of the unit}

By Proposition \ref{ti-gr} and Corollary \ref{cor-quot} we have an fp immersion
\[\la:\widetilde{\Gr_G}=\Gr_{\ti{G}}\ra\Gr_G\times\Gr_G\times\ab^1.\]
By \eqref{gm-eq}, the map $\la$  yields a map
\[\la:\ti{\cZ}\ra\cZ\times\cZ\times\ab^1\]
We consider the sheaf $\cQ\in\cD(\cZ\times\cZ\times\ab^1)$ given by
$$\cQ=\la_{*}\omega_{\ti{\cZ}}.$$
The sheaf $\cQ$ descends to an object from $\cD(\cZ\times\cZ\times(\ab^1/\bG_m))$.
This follows from the Cartesian diagram
\begin{equation}
\begin{split}
\xymatrix{\ti{\cZ}\ar@{=}[r]\ar[d]^{\ti{\kp}}&\ti{Z}/\bG_m\ar[d]^{\ti{p}/\bG_m}\ar[r]&\ti{Z}/(\bG_m\times\bG_m)\ar[d]\\
\cZ\times\cZ\times\ab^1\ar@{=}[r]&Z/\bG_m\times Z/\bG_m\times\ab^1\ar[r]&Z/\bG_m\times Z/\bG_m\times\ab^1/\bG_m}
\end{split}
\label{gm-Q}
\end{equation}
where the $(\bG_m)^2$-action in $\ti\cZ$ is as in  \S\ref{ti-f}.
Note that \eqref{ZG} and Proposition \ref{ti-gr} yield isomorphisms
$$(\ti{\cZ}_1,\ti{p}_1)\cong(\cZ,\Delta)
,\quad
(\ti{\cZ}_0,\ti{p}_0)\cong(\cZ^{+}\times_{\cZ^{0}}\cZ^{-},\kp^{+}\times\kp^{-}).$$
 In particular, using the specialization in Lemma \ref{sp}, we get a specialization map
\begin{equation}
\Sp_{\ti{\cQ}}:\cQ_1=(i_{1})^{!}\ti{\cQ}=\Delta_{*}\omega_{\cZ}\ra\cQ_0=(i_{0})^{!}\ti{\cQ}=(\kp^{+}\times\kp^{-})_{*}\omega_{\cZ^{-}\times_{\cZ^{0}}\cZ^{+}}.
\label{spQ}
\end{equation}
For nice spaces, this map suffices to construct the unit of the adjunction using the formalism of kernels. 
We cannot apply this formalism because it involves considering the functor $p_*$ for the map
$p:\Gr_G\times\Gr_G\ra\Gr_G$ which may not be defined. 
Thus, we first restrict to a smaller substack where such functor is defined and then apply pushforward.
We must construct a map
\begin{equation}
\id_{\cD(Z)}\ra(\kp^-)_{*}  (\kq^{-})^{!} (\kq^{+})_{\bullet} (\kp^{+})^{!}.
\label{adj-5}
\end{equation}
Using base change for $(\kq^{-})^{!} (\kq^{+})_{\bullet}$, 
the right hand side can be interpreted as a pull-push using the following diagram
\begin{equation}
\begin{split}
\begin{tikzpicture}[>=Stealth, every node/.style={inner sep=1pt, align=center, font=\small}]
\node (X) at (0,2.6) {$\cZ^{+}\times_{\cZ^{0}}\cZ^{-}$};
\node (Y) at (-1.3,1.3) {$\cZ^{+}$};
\node (Z) at (1.3,1.3) {$\cZ^{-}$};
\node (A) at (-2.6,0) {$\cZ$};
\node (B) at (0,0) {$\cZ^{0}$};
\node (C) at (2.6,0) {$\cZ$};
\draw[->] (X) -- (Y) node[midway,left,xshift=1pt,yshift=5pt] {$'\kq^{-}$};
\draw[->] (X) -- (Z) node[midway,right,xshift=-1pt,yshift=5pt] {$'\kq^{+}$};
\draw[->] (Y) -- (A) node[midway,left,xshift=2pt,yshift=5pt] {$\kp^{+}$};
\draw[->] (Y) -- (B) node[midway,xshift=1pt,yshift=5pt,right] {$\kq^{+}$};
\draw[->] (Z) -- (B) node[midway,xshift=2pt,yshift=5pt,left] {$\kq^{-}$};
\draw[->] (Z) -- (C) node[midway,xshift=2pt,yshift=5pt,right] {$\kp^{-}$};
\draw[->, bend left=40] (X) to node[above,pos=0.5,xshift=4] {$\kq$} (C);
\draw[->, bend right=40] (X) to node[above,pos=0.5,xshift=-4] {$\kp$} (A);
\end{tikzpicture}
\label{gros-carré}
\end{split}
\end{equation}
The maps $\kq^+$ and $\kq^{-}$ are homotopically ind-schematic by the discussion in \S\ref{sec-cont-principle}.
The maps $\kp^{+}$ and $\kp^-$
are fp immersions when restricted to each connected component $\Gr_B^\nu/\bG_m$ and $\Gr_{B^-}^\nu/\bG_m$,
see \S\ref{defCT}.
By base change and composition, 
the map $\kq$ is also homotopically ind-schematic.
In particular, we have 
\begin{equation}
(\kp^{-})_*  (\kq^{-})^{!} (\kq^{+})_{\bullet} (\kp^{+})^{!}=\kq_{\bullet} \kp^{!}.
\label{bpul-push}
\end{equation}
From \eqref{gros-carré}, we have a commutative diagram
$$\xymatrix{&\cZ^{+}\times_{\cZ^{0}}\cZ^{-}\ar[dd]^{\kp}\ar[dl]_{'\kq^{-}}\ar[dr]^{\kappa}\\\cZ^{+}\ar[dr]^{\kp^{+}}&&\cZ\times\cZ\ar[dl]_{p_1}\\&\cZ}$$
with $p_{1}:\cZ\times\cZ\ra\cZ$ the first projection.
Consider $\cM\in\cD(\cZ)$. Tensoring  \eqref{spQ} by $\cM$ yields a map
\[(p_{1})^{!}\cM\overset{!}{\otimes} \Delta_{*}\omega_{\cZ}\ra (p_{1})^{!}\cM\overset{!}{\otimes} \kappa_{*}\omega_{\cZ^{-}\times_{\cZ^{0}}\cZ^{+}}.\]
Since we quotiented by $\bG_m$, the map $\Delta$ is no more fp closed, but it is still affine fp and 
$\kappa=\kp^{+}\times\kp^{-}$ is an fp immersion when restricted to its connected components. 
Since  $p_1\Delta=\id_{\cZ}$, the projection formula for $\Delta$ and $\kappa$ yields a map
\begin{equation}
\Delta_{*}\cM\ra \kappa_{*}\kappa^{!}p_{1}^{!}\cM,
\label{adj-2}
\end{equation}
So  we get from \eqref{adj-2} a map
\begin{equation}
\Delta_{*}\cM\ra\kappa_*\kp^{!}\cM.
\label{adj-3bis}
\end{equation}
Using  Lemma  \ref{clopen}, the isomorphism \eqref{isom3} 
and taking $\bG_m$-quotients, we get a Cartesian square and a lower commutative triangle
\begin{equation}
\begin{split}
\xymatrix{\cZ^{0}\ar[d]_{i}\ar[r]^-{j}&\cZ^{+}\times_{\cZ^{0}}\cZ^{-} \ar[ddr]^{\kappa}\ar[d]_-{\kappa_1}\\
\cZ\ar[r]^-{\Delta_1} \ar[drr]_{\Delta}&[Z\times Z/\bG_m]\ar[dr]^-{p}&\\
&&\cZ\times\cZ}
\end{split}
\label{diag1}
\end{equation}
Since $p$ is a $\bG_m$-torsor we get an isomorphism
$$p^{*}  p_{*}\cong-\otimes\RGm(\bG_m,\bql)[2].$$
Thus the counit $p^{*}  p_{*}\ra\id$ splits by the unit section $1\ra\bG_m$.
Now \eqref{adj-3bis} and \eqref{diag1} yield
\[p_*(\Delta_1)_*\cM\ra p_{*}(\kappa_1)_*\kp^{!}\cM.\]
Using the splitting this gives a map
\[(\Delta_1)_*\cM\ra(\kappa_1)_*\kp^{!}\cM.\]
Since $\Delta_1$ is fp closed, using the adjunction and the base change along \eqref{diag1}, this gives a map
$$j_{*}i^{*}\cM\ra \kp^{!}\cM.$$
Applying the functor $\kq_{\bullet}$ we  get a map
\begin{equation}
\kq_{\bullet}j_{*}i^{*}\cM\ra \kq_{\bullet}\kp^{!}\cM.
\label{adj-3}
\end{equation}
Since $j$ is fp schematic,  Lemma \ref{ind-bullet0} and Proposition \ref{bul-comp-proj} yield
\begin{align}\label{adj-11}
\kq_\bullet j_{*}=\kq_\bullet j_{\bullet}=i_\bullet=i_*.
\end{align}
So from \eqref{adj-3} and \eqref{adj-11} we obtain a map
\begin{equation}
\cM\ra i_{*}i^{*}\cM\ra \kq_{\bullet}\kp^{!}\cM,
\label{adj-4}
\end{equation}
which is functorial in $\cM$. 
So we get the putative unit of the adjunction \eqref{adj-5} as reformulated in \eqref{bpul-push}.

\subsection{Proof of  Theorem \ref{brad3}}
We have defined the following morphisms of functors in \eqref{brad-equiv} and \eqref{adj-5}
\begin{equation}
\id_{\cD(Z)}\ra(\kp^-)_{*} ( \kq^{-})^{!} (\kq^{+})_{\bullet} (\kp^{+})^{!},
\label{adj-6}
\end{equation}
\begin{equation}
(\kq^{+})_{\bullet} (\kp^{+})^{!} (\kp^{-})_{*} (\kq^{-})^{!}\ra\id_{\cD(\cZ^{0})}.
\label{adj-7}
\end{equation}
To prove Theorem \ref{brad3}, hence Theorems  \ref{brad2} and \ref{Braden},
it is enough to prove that the composed morphisms
\begin{equation}
(\kp^{-})_{*} (\kq^{-})^{!} (\kq^{+})_{\bullet} (\kp^{+})^{!} (\kp^{-})_{*} (\kq^{-})^{!}\ra(\kp^{-})_{*} (\kq^{-})^{!},
\label{adj-8}
\end{equation}
\begin{equation}
(\kq^{+})_{\bullet} (\kp^{+})^{!}\ra(\kq^{+})_{\bullet} (\kp^{+})^{!} (\kp^-)_{*}  (\kq^{-})^{!} (\kq^{+})_{\bullet} (\kp^{+})^{!}
\label{adj-9}
\end{equation}
are isomorphic to the identity. 
We prove the assertion for \eqref{adj-8}.
The proof for \eqref{adj-9} is similar.
The proof follows the argument in \cite[\S 5]{DG1}. 
It relies on Proposition \ref{fib-prod}.

\subsubsection{The kernel for the composition}
We abbreviate
$$\Phi=(\kp^{-})_{*}(\kq^{-})^{!}(\kq^{+})_{\bullet}(\kp^{+})^{!} (\kp^{-})_{*}(\kq^{-})^{!}.$$
We consider the morphism of functors \eqref{adj-8}
$$\Phi\ra(\kp^{-})_{*}(\kq^{-})^{!}.$$
The functor $\Phi$ is given by pull-push along the diagram
\begin{equation}
\begin{split}
\begin{tikzpicture}[>=Stealth, every node/.style={inner sep=1pt, align=center, font=\small}]
\node (A) at (0,4)  {$\cZ^{-}\times_{\cZ}\cZ^{+}\times_{\cZ^{0}}\cZ^{-}$};
\node (B) at (-1.3,2.6) {$\cZ^{-}\times_{\cZ}\cZ^{+}$};
\node (C) at (1.3,2.6)  {$\cZ^{+}\times_{\cZ^{0}}\cZ^{-}$};
\node (D) at (-2.6,1.3) {$\cZ^{-}$};
\node (E) at (0,1.3)  {$\cZ^{+}$};
\node (F) at (2.6,1.3)  {$\cZ^{-}$};
\node (G) at (-3.9,0) {$\cZ^{0}$};
\node (H) at (-1.2,0) {$\cZ$};
\node (I) at (1.3,0)  {$\cZ^{0}$};
\node (J) at (3.9,0)  {$\cZ$};
\draw[->] (A) -- (B);
\draw[->] (A) -- (C);
\draw[->] (B) -- (D);
\draw[->] (B) -- (E);
\draw[->] (C) -- (E);
\draw[->] (C) -- (F);
\draw[->] (D) -- (G) node[midway,left,xshift=1pt,yshift=5pt] {$\kq^{-}$};
\draw[->] (D) -- (H) node[midway,right,xshift=4pt,yshift=5pt] {$\kp^{-}$};
\draw[->] (E) -- (H) node[midway,left,xshift=1pt,yshift=5pt] {$\kp^{+}$};
\draw[->] (E) -- (I) node[midway,right,xshift=4pt,yshift=5pt] {$\kq^{+}$};
\draw[->] (F) -- (I) node[midway,left,xshift=1pt,yshift=5pt] {$\kq^{-}$};
\draw[->] (F) -- (J) node[midway,right,xshift=4pt,yshift=5pt] {$\kp^{-}$};
\end{tikzpicture}
\end{split}
\label{big-diag}
\end{equation}
We want to reinterpret $\Phi$ from a smaller diagram.
Following \eqref{tiZ+-} we set
\[\ti{\cZ}^{-}=\cZ^{-}\times_{\cZ}\ti{\cZ}=\ti Z^-/\bG_m
,\quad
\ti Z^-=Z^{-}\times_{Z}\ti{Z}.\]
By base change, the map $\ti{\kp}$ in \eqref{gm-Q} yields a map
\[\ti{\cZ}^{-}\ra\cZ^-\times\cZ\times\ab^1.\]
Composing it with $\kq^{-}\times\id\times\id:\cZ^-\times\cZ\times\ab^1\ra \cZ^{0}\times\cZ\times\ab^1$ we get the map
\begin{equation}
r:\ti{\cZ}^{-}\ra\cZ^{0}\times\cZ\times\ab^1.
\label{r1}
\end{equation}
The map $r$ is homotopically ind-schematic,
because it is the composite of $\ti{\kp}$ which is fp locally closed and 
 a base change of $\kq^{-}$ by $\cZ^0\times\cZ\times\ab^{1}\ra\cZ_0$, which is homotopically ind-schematic by Proposition \ref{gm-bc}.
For each $t\in\ab^{1}$, let $r_{t}$ be the fiber above $t$.
By Proposition \ref{ti-gr} and \eqref{0-fib} we have an isomorphism 
$$\ti{\cZ}_0\cong \cZ^{+}\times_{\cZ^0} \cZ^{-}.$$
Thus, we have
\[\ti{\cZ}_{0}^{-}=\cZ^{-}\times_{\cZ}\ti{\cZ}_0\cong\cZ^{-}\times_{\cZ}\cZ^{+}\times_{\cZ^{0}}\cZ^{-}.\]
Thus, the upper term of \eqref{big-diag} is isomorphic to $\ti{\cZ}_{0}^{-}$.
As in \cite[(5.8)]{DG1}, 
the functor $\Phi$ is the pull-push of the diagram
\begin{equation}
\xymatrix{\cZ^{0}&\ti{\cZ}_{0}^{-}\ar[r]^{p_{2}  r_{0}}\ar[l]_{p_{1}  r_{0}}&\cZ}
\label{min-diag}
\end{equation}
where $p_1:\cZ^{0}\times\cZ\ra\cZ^{0}$ and $p_2:\cZ^{0}\times\cZ\ra\cZ$ are the obvious projections. 
The morphism $p_{2}  r_{0}$ is homotopically ind-schematic, because it is the composition of 
\[\cZ^{-}\times_{\cZ}\cZ^{+}\times_{\cZ^0}\cZ^{-}\ra\cZ^{+}\times_{\cZ^0}\cZ^{-}\ra\cZ^{-}\stackrel{\kp^-}{\rightarrow}\cZ,\]
all of which are homotopically ind-schematic, see loc.~cit.. 
Set
\begin{equation}
\cS=r_{\bullet}\omega_{\ti{\cZ}^{-}}\in\cD(\cZ^0\times\cZ\times\ab^1)
,\quad
\cS_0=(r_0)_{\bullet}\omega_{\ti{\cZ_0}^{-}}\in\cD(\cZ^0\times\cZ).
\label{r-kernel}
\end{equation}
The same argument as in \eqref{gm-Q} implies that the complex $\cS$ is $\bG_m$-monodromic.
Further, we have
$$\Phi=\Phi_{\cS_0}.$$
Now, we consider the commutative diagram
$$\xymatrix{&\cZ^{-}\ar[dl]_{\kq^{-}}\ar[dr]^{\kp^{-}}\ar[d]\\
\cZ^{0}&\ar[l]_-{p_1}\cZ^{0}\times\cZ\ar[r]^-{p_2}\ar[l]&\cZ}$$
The map $\kq^{-}\times \kp^{-}$ is homotopically ind-schematic, because $\kq^{-}$ is homotopically ind-schematic
and $\kp^{-}$ is ind-schematic of ind-ft. 
Lemma \ref{ind-bullet0} for $\kp^{-}$ implies that 
\begin{equation}
(\kp^{-})_{*} (\kq^{-})^{!}=\Phi_{\cT}
,\quad
\cT=(\kq^{-}\times \kp^{-})_{\bullet}\omega_{\cZ^{-}}.
\label{T-kern}
\end{equation}

\subsubsection{The natural transformations at the level of kernels}
We now want to describe the kernels of the morphisms of functors
$\Phi_{\cS_0}\ra\Phi_{\cT}$ and
$\Phi_{\cT}\ra\Phi_{\cS_0}.$
We start with the first one. 
Recall the closed and open embedding from Lemma  \ref{clopen}
\[j:Z^{0}\hra Z^+\times_{Z^{0}}Z^-.\]
By base change and $\bG_m$-equivariance, we get an open immersion
\begin{align}\label{j-}\kj^{-}:\cZ^{-}\hra\cZ^{-}\times_{\cZ}\cZ^{+}\times_{\cZ^{0}}\cZ^{-}\cong\ti{\cZ}^{-}_{0}.\end{align}
Moreover, we have
\begin{align}\label{qp}\kq^{-}\times\kp^{-}=r_0\kj^-:\cZ^{-}\ra\ti{\cZ}_{0}^{-}\ra\cZ^{0}\times\cZ.\end{align}
The morphism $\Phi_{\cS_0}\ra\Phi_{\cT}$ is obtained from \eqref{adj-7}.
It comes from the morphism of kernels
\begin{equation}
\cS_{0}\ra\cT
\label{map1}
\end{equation}
given by the composition
\[\cS_{0}=(r_0)_\bullet\omega_{\ti{\cZ}^{-}_{0}}\ra (r_0)_\bullet (\kj^{-})_{*}(\kj^{-})^{*}\omega_{\ti{\cZ}^{-}_{0}}\cong 
(r_0)_\bullet  (\kj^-)_{\bullet}\omega_{\cZ^{-}}\cong(\kq^{-}\times \kp^{-})_{\bullet}\omega_{\cZ^{-}}\cong\cT,\]
where we use the equality $(\kj^-)_{\bullet}=(\kj^-)_{*}$ and Proposition \ref{bul-comp-proj}.
This morphism of kernels is also obtained by applying the formalism \eqref{ker-form} to the diagram
$$\xymatrix{
&&\cZ^{-}\ar[d]^{\kj^{-}}\ar[drr]\ar[dll]&&\\
\cZ^0&&\ti{\cZ}_0^{-}\ar[d]^{r_0}\ar[rr]\ar[ll]&&\cZ\\
&&\cZ^{0}\times\cZ\ar[urr]\ar[ull]&&}
$$
Now the isomorphism  $ \cZ\cong\ti{\cZ}_{1}$ of the fiber at 1 gives an isomorphism
\begin{equation}
\cZ^{-}\cong\ti{\cZ}^{-}_{1}.
\label{tfib1}
\end{equation}
Hence the morphism $r_{1}:\cZ^{-}_{1}\ra\cZ^{0}\times\cZ$ is isomorphic to $\kq^{-}\times\kp^{-}$. 
In particular, we have an isomorphism
\begin{equation}
\cT\cong\cS_{1}
\label{taut}
\end{equation}
where $\cS_{1}=(i_{1})^{!}\cS$.
Let us consider now the morphism of functors $\Phi_{\cT}\ra\Phi_{\cS_0}$. 
The specialization map  in Lemma \ref{sp} gives a map
\begin{equation}
\Sp_{\cS}:\cS_{1}\ra\cS_{0}.
\label{sp2}
\end{equation}
By functoriality of the specialization map, the map $\Phi_{\cT}\ra\Phi_{\cS_0}$ is given by the morphism
\begin{equation}
\cT\cong\cS_{1}\ra\cS_{0}
\label{map2}
\end{equation}
equals to the composition of \eqref{sp2} and \eqref{taut}.

\bcor\label{conc-ker}
To prove that \eqref{adj-8} is an isomorphism, it suffices to show that the composed map
\begin{equation}
\cT\cong\cS_1\ra\cS_0\ra\cT
\label{adj10}
\end{equation}
is the identity.
\ecor

\subsubsection{Restriction to an open substack}
Following \eqref{j-}, we consider the open embedding 
$$j^{-}:Z^{-}\hra\ti{Z}^{-}_{0}$$
and the qc open
\[\overbigdot{Z^{-}}:=\ti{Z}^{-}\smallsetminus(\ti{Z}^{-}_{0}\smallsetminus Z^{-}).\]
Let $\overbigdot{\cZ^{-}}$ be the corresponding substack of $\ti{\cZ}^{-}$.
Taking the fiber at 0 yields an isomorphism
\begin{equation}
\cZ^{-}\cong\overbigdot{{\cZ}^{-}_{0}},
\label{arc0}
\end{equation}
as well as, using \eqref{tfib1}, an other identification
\[\cZ^{-}\cong\overbigdot{{\cZ}_{1}^{-}}.\]
We consider the obvious map 
\begin{align}\label{rdot}\overbigdot{r}:\overbigdot{{\cZ}^{-}}\ra\cZ^0\times\cZ\times\ab^1.\end{align}
It yields the sheaf
\[\overbigdot{\cS}=\overbigdot{r}_{\bullet}\omega_{\overbigdot{{\cZ}}^{-}}.\]
The open embedding $\overbigdot{{\cZ}^{-}}\hra\ti{\cZ}^{-}$ gives maps
\[\cS\ra\overbigdot{\cS},\quad\cS_{0}\ra\overbigdot{\cS}_{0},\quad\cS_{1}\ra\overbigdot{\cS}_{1}.\]
As in \eqref{map1}-\eqref{map2}, we have natural transformations
\begin{equation}
\cT\ra\overbigdot{\cS}_1\ra\overbigdot{\cS}_0\ra\cT
\label{adj11}
\end{equation}
yielding the following commutative diagram
$$\xymatrix{\cT\ar[d]_{\id}\ar[r]&\cS_1\ar[d]\ar[r]&\cS_0\ar[r]\ar[d]&\cT\ar[d]^{\id}\\
\cT\ar[r]&\overbigdot{\cS}_1\ar[r]&\overbigdot{\cS}_0\ar[r]&\cT}$$ 
Thus, to prove Corollary \ref{conc-ker} it is enough to show it for \eqref{adj11}.

\subsubsection{Proof of Theorem $\ref{brad3}$}\label{finalCT}
By Proposition \ref{fib-prod}, there is a qc embedding 
$$Z^{-}\times\ab^1\ra\ti{Z}^{-}.$$
By definition\footnote{it is an isomorphism over $\bG_m$ and over 0, use \eqref{arc0}.}, we have an isomorphism $\overbigdot{{Z}^{-}}\cong Z^{-}\times\ab^1$, 
hence an isomorphism
\[\overbigdot{{\cZ}^{-}}\cong\cZ^{-}\times\ab^1.\]
Under this isomorphism we have the following identifications:
\begin{enumerate}[label=$\mathrm{(\alph*)}$,leftmargin=8mm]
\item
the map $\overbigdot{r}$ in \eqref{rdot}
identifies with the map 
$$\cZ^{-}\times\ab^1\ra\cZ^{0}\times\cZ\times\ab^1$$ 
given by $\id_{\ab^1}$ and the map $\kq^{-}\times\kp^{-}$ in \eqref{qp},
\item
the isomorphism $\cZ^{-}\cong\ti{\cZ}^{-}_{1}$ in \eqref{tfib1} identifies with the identity map
\[\cZ^{-}\ra(\cZ^{-}\times\ab^1)\times_{\ab^1}\{1\},\]
\item
the isomorphism $\cZ^{-}\cong\overbigdot{{\cZ}_0^-}$ of \eqref{arc0} identifies with the identity map
\[\cZ^{-}\ra(\cZ^{-}\times\ab^1)\times_{\ab^1}\{0\}.\]
\eenum
Thus the composition \eqref{adj11} identifies with the specialization map
\[\cT\cong (i_{1})^{!}(\ct\boxtimes\omega_{\ab^1})\ra(i_{0})^{!}(\cT\boxtimes\omega_{\ab^1})\cong\cT\]
which is the identity according to the proof of Lemma \ref{sp}.

\subsection{The t-exactness of the constant term}\label{exactCT}
We have proved Theorem \ref{Braden} using Proposition \ref{ti-quot}
which is proved in \S\ref{PfPROP} below.
We now prove that the constant term functor introduced in \eqref{CT} is $t$-exact.
More precisely, we define the normalized constant term functor 
$$\CT_{*}[\deg]=\bigoplus_{\nu\in \LLa}\CT_{*,\nu}[2 \langle\rho,\nu\rangle]:\cD_{G_\cO}(\Gr_{G_c})\ra\cD(\Gr_{T}),$$
where $\CT_{*,\nu}$ is the obvious direct summand of $\CT_*$.

\bthm\label{t-exact}
The normalized constant term functor $\CT_{*}[\deg]$ is $t$-exact.
\ethm

\bpf
For this proof it is convenient to equip $G_K$ with the $\bG_m$-action given by the adjoint action of $-2\Lrho$ instead of $2\Lrho$.
Then, by Proposition \ref{id-att}, the attractor and repulsor locus of the induced $\bG_m$-action on $\Gr_G$ are $\Gr_{B^{-}}$ and $\Gr_{B}$
and all maps in the diagram \eqref{cart1} change accordingly.
In particular, the constant term functors in \eqref{CT} are now given by the following formulas
\begin{align}\label{new-CT}\CT_{*}=(q^{-})_\bullet (p^{-})^{!},\quad\CT_{!}=(q^{+})_! (p^{+})^{*}\end{align}
instead of \eqref{bis-CT}.
We first prove that the normalized constant term functor $\CT_{*}[\deg]$ is left $t$-exact, using the formulas \eqref{new-CT}.
By Proposition \ref{t-gen}, it is enough to prove that for each dominant cocharacter $\la$ we have
\begin{align}\label{CT1}
\CT_{*}[\deg]((i_\la)_*\omega_{\Gr_{\la}}[-2 \langle\rho,\la\rangle])\in \mathstrut^{p}\cD^{\geq 0}(\Gr_{T}).
\end{align}
Recall the diagram \eqref{cart1}.
We also have a Cartesian diagram
$$\xymatrix{
\bigsqcup_{\nu\in \LLa} \Gr_{\la}\cap \Gr_{B^{-}}^{\nu}\ar[d]_{\ti\imath_\la}\ar[r]&\Gr_{\la}\ar[d]^{i_\la}\\
\Gr_{B^{-}}\ar[r]^{p^-}&\Gr_G}
$$
Fix $\nu\in \LLa$.
By base change, we get
\begin{align*}
\CT_*(i_\la)_*\omega_{\Gr_{\la}}
=(q^-)_\bullet(p^{-})^{!}(i_\la)_*\omega_{\Gr_{\la}}
.
\end{align*}
By Proposition \ref{bul-comp-proj}, we deduce that
\begin{align*}
\CT_*(i_\la)_*\omega_{\Gr_{\la}}=\bigoplus_\nu(q^-\ti\imath_\la)_\bullet\omega_{\Gr_{\la}\cap \Gr_{B^{-}}^{\nu}}.
\end{align*}
Now, by Lemma \ref{Gr-ST} we have 
$$(\Gr_{\la}\cap \Gr_{B^{-}}^{\nu})_\red=(\Gr_{\la}\cap T_{\nu})_\red$$ 
and by Theorem \ref{thm:GT} the functor $\Gr_{\la}\cap T_{\nu}$ is a ft scheme. 
Thus there is no difference between $\bullet$-pushforward and $*$-pushforward to $\Gr_T$
by Lemma \ref{ind-bullet0}, and we get
\begin{equation}
\CT_*(i_\la)_*\omega_{\Gr_{\la}}
= \bigoplus_\nu \RGm(\Gr_{\la}\cap T_\nu,\omega_{\Gr_{\la}\cap T_\nu}).
\label{bm-calc}
\end{equation}

By Theorem \ref{thm:GT}, the scheme $\Gr_{\la}\cap T_{\nu}$ has dimension $\langle\rho,\la-\nu\rangle$.
We deduce that
\begin{equation}
\RGm(\Gr_\la\cap T_\nu,\omega_{\Gr_\la\cap T_\nu})[-2\langle\rho,\la-\nu\rangle]\in
\cD^{\geq 0}(\Spec(k)).
\label{bm-calc2}
\end{equation}
Hence, we have the following relation from which \eqref{CT1} follows
\[\CT_{*,\nu}[2 \langle\rho,\nu\rangle]\big((i_\la)_*\omega_{\Gr_{\la}}[-2 \langle\rho,\la\rangle]\big)\in\cD^{\geq 0}(\Spec(k)).\]
We are left to prove right $t$-exactness.
By Proposition \ref{t-gen}, we must prove that
\[\CT_{*}[\deg]\big((i_\la)_!\omega_{\Gr_{\la}}[-2 \langle\rho,\la\rangle]\big)\in \mathstrut^{p}\cD^{\leq 0}(\Gr_{T})\]
or, equivalently, that
\[\CT_{*,\mu}[2 \langle\rho,\mu\rangle]\big((i_\la)_!\omega_{\Gr_{\la}}[-2 \langle\rho,\la\rangle]\big)\in\cD^{\leq 0}(\Spec(k)).\]
By Theorem  \ref{Braden}, we have $\CT_{*}[\deg]\cong\CT^{-}_{!}[\deg].$
We use the right hand side to prove the right $t$-exactness.
We have a Cartesian diagram
$$\xymatrix{
\bigsqcup_{\mu\in \LLa} \Gr_{\la}\cap \Gr_{B}^{\mu}\ar[d]_{\ti\imath_\la}\ar[r]&\Gr_{\la}\ar[d]^{i_\la}\\
\Gr_{B}\ar[r]^{p^+}&\Gr_G}$$
Let $\ti{p}_{\mu}^{+}$ be the obvious map
$$\ti{p}_{\mu}^{+}:\Gr_{\la}\cap S_{\mu}\ra\Gr_{\la}.$$
Formula \eqref{new-CT} and base change yield
\begin{align*}
\CT_!(i_\la)_!\omega_{\Gr_{\la}}
&=(q^+)_!(p^+)^{*}(i_\la)_!\omega_{\Gr_{\la}}\\
&=\bigoplus_\mu \RGm_{c}(\Gr_{\la}\cap S_{\mu}, (\ti{p}_{\mu}^{+})^{*}\omega_{\Gr_{\la}})
\end{align*}
where we used the isomorphism 
$$(\Gr_{\la}\cap \Gr_{B}^{\mu})_\red=(\Gr_{\la}\cap S_{\mu})_\red$$
which follows from Lemma \ref{Gr-ST}.
We must prove that
\begin{equation}
\RGm_{c}(\Gr_{\la}\cap S_{\mu}, (\ti{p}_{\mu}^{+})^{*}\omega_{\Gr_{\la}})[-2 \langle\rho,\la-\mu\rangle]\in\cD^{\leq 0}(\Spec(k)).
\label{gs-mu}
\end{equation}
The problem is Zariski local on $\Gr_{\la}$ and to do that, we consider the commutative diagram \eqref{diag2}
with the maps $i_0$, $i_1$ and $i_2$ there.
By Lemma \ref{pul-push} and base change, we have
\begin{align*}
\RGm_{c}(\Gr_{w\la}^\circ \cap S_\mu, (i_0)^*\omega_{\Gr_{w\la}^\circ})
\cong\RGm_{c}(Y_\mu, (i_1)^*\omega_Y)
\cong\RGm_{c}(\hX_\mu, (i_2)^*\omega_{\hX})
\end{align*}
By Lemma \ref{GrS3}, we have
\begin{align*}
\RGm_{c}(\hX_\mu, (i_2)^*\omega_{\hX})
\cong\RGm_{c}(\hX_\mu, (i_2)^*(\bql)_{\hX}[2\langle\rho,\la-w\la\rangle])\in\cD^{\leqslant -2 \langle\rho,\la-\mu\rangle}(\Spec(k))
\end{align*}
from which the claim follows.
\epf

\subsection{Geometric Satake}
\subsubsection{The dual category}\label{O-cat}
Let $G$ be a minimal KM group over a field of characteristic zero. 
The category $\cO$ consists of the diagonalizable $\kg$-modules $V$ with finite dimensional weight 
spaces and such that the sets of weight $P(V)$ satisfies the following condition
\[P(V)\subset\bigcup_{i=1}^{r}(\la_i-\NN\Delta),\]
see \cite[\S 2.1]{Kum}. It is an Abelian tensor category.
A diagonalizable $\kg$-module $V$ is integrable if $e_i$ and $f_i$ act locally nilpotently for all $i\in I$.
Integrable modules are stable by quotients, subobjects and tensor product.
Let $\Rep(G)$ be the full subcategory of $\kg$-integrable modules in the category $\cO$, which is thus an Abelian tensor category.
The simple objects are $L(\la)$ for $\la\in X_{*}(T)^{+}$, see, e.g., \cite[Cor.~2.1.3, 2.1.8]{Kum}.
Assume that $G$ is symmetrizable, then  $\Rep(G)$ is semisimple, i.e., every $M\in\Rep(G)$ is a direct sum of simple objects with finite 
multiplicities by \cite[\S10]{Kac}, see also \cite[Cor.~2.2.7, 3.2.10]{Kum}.
Given a minimal KM group $G$ defined over an algebraically closed field $k$, one defines its Langlands dual $G^{\vee}$ 
over $\bC$  to be the minimal KM group obtained by exchanging roots and coroots in the KM root datum. If $G$ is symmetrizable, 
then $G^{\vee}$ is symmetrizable.

\subsubsection{The equivalence}
\bprop\label{conserv}
Let $A\in\cD_{G_{\cO}}(\Gr_{G_c})$ such that $\CT_{*}(A)=0$. Then $A=0$. 
So the functor $\CT_*$ is conservative.
\eprop

\bpf
Let $\la$ be the maximal cocharacter such that $(i_{\la})^{!}A\neq 0$. 
By Theorem \ref{thm:GT} we have $(\Gr_{\la}\cap T_{\la})_\red\cong\Spec(k)$.
Thus $\CT_{*,\la}(A)$ is just the stalk of $A$ at $\la$.
It vanishes by assumption, yielding a contradiction.
\epf

For convenience, for any $\la\in\LLa_+$ let $i_\lambda$ be as in \eqref{ila} and set
\[\Delta_{\la}=\mathstrut^{p}(i_\lambda)_{!}\omega_{\Gr_{\la}}[-2\langle\rho,\la\rangle]
,\quad
 \nabla_{\la}= \mathstrut^{p}(i_\lambda)_{*}\omega_{\Gr_{\la}}[-2\langle\rho,\la\rangle]
 ,\quad
 \IC_{\la}=(i_\lambda)_{!*}\omega_{\Gr_{\la}}[-2\langle\rho,\la\rangle].\]
 
\bthm\label{ic-calc}
Assume that $G$ is of affine type over an algebraically closed field. Let $\la\in\LLa_+$.
\begin{enumerate}[label=$\mathrm{(\alph*)}$,leftmargin=8mm]
\remi
We have $\CT_{*}[\deg](\nabla_{\la})\cong L(\la)$ as $\LLa$-graded vector space.
\remi
The obvious maps
$\Delta_{\la}\ra\IC_{\la}\ra\nabla_{\la}$
are all isomorphisms. We have
$\CT_{*}[\deg](\IC_{\la})\cong L(\la).$
\eenum
\ethm

\bpf
For each $\nu\in\LLa$ the functor $\CT_{*,\nu}[2\langle\rho,\nu\rangle]$ is $t$-exact.
Thus
\begin{align*}
\CT_{*,\nu}(\nabla_{\la})
=H^{2\langle\rho,\nu\rangle}(\CT_{*,\nu}((i_\lambda)_{*}\omega_{\Gr_{\la}}[-2\langle\rho,\la\rangle]))
=H_{2\langle\rho,\la-\nu\rangle}^{\BM}(\Gr_{\la}\cap T_{\nu},\bql),
\end{align*}
where the last equality follows from \eqref{bm-calc} and \eqref{bm-calc2}.
By Theorem \ref{thm:GT}, we have 
\[H_{2\langle\rho,\la-\nu\rangle}^{\BM}(\Gr_{\la}\cap T_{\nu},\bql)=\bql[\Top\text{-}\!\Irr(\Gr_{\la}\cap T_{\nu})]=L(\la)_{\nu}\]
where $\Top\text{-}\!\Irr$ is the set of top-dimensional irreducible components.
The claim (a) follows.

To prove (b) we first prove that $\IC_{\la}\cong\nabla_{\la}$. 
By Proposition \ref{conserv}, it is enough to prove the claim after applying $\CT_*[\deg]$.
To do so, we must compute the space $H^{2\langle\rho,\nu\rangle}(T_{\nu},\IC_{\la}|^!_{T_\nu})$.
The characterization of IC-complexes yields
\benumr
\remi
For $\eta>\la$ we have $(i_{\eta})^{!}\IC_{\la}=0$.
\remi
For $\eta=\la$ we have $(i_{\eta})^{!}\IC_{\la}\in\cD_{G_\cO}(\Gr_\la)^\heartsuit$.
\remi
For $\eta<\la$ we have  $(i_{\eta})^{!}\IC_{\la}\in\mathstrut^{p}\cD^{> -2\langle\rho,\la-\eta\rangle}_{G_\cO}(\Gr_{\eta})$.
\eenum

We want to have the same inequalities with $\Gr_{\eta}$ replaced by $\Gr_{\eta}\cap T_{\nu}$.
These inequalities can be checked after taking a smooth cover of $\Gr_{\eta}\cap T_{\nu}$.
We concentrate on the  relation (iii),
otherwise the claim is immediate. Using Proposition \ref{t-strat}(a), we form the Cartesian diagram
\begin{equation}
\begin{split}
\xymatrix{
X_T\ar[d]\ar[r]&E_{\eta}\ar[d]_{\pi_{\eta}}\ar[rr]&&\Spec(k)\ar[d]\\
Y_T=\Gr_{\eta}\cap T_{\nu}\ar[r]&\Gr_{\eta}\ar@/^2pc/[rr]^{h}\ar[r]&[\Gr_{\eta}/G_{\co}]=\bB K_{\eta}\ar[r]&\bB L_{\eta}}
\label{pi-l}
\end{split}
\end{equation}
If we forget the $G_{\co}$-equivariance, the $!$-restriction $(\pi_\eta)^!(i_{\eta})^{!}\IC_{\la}$ 
is a direct sum of copies of $\omega_{E_{\eta}}[-d]$ with $d> -2\langle\rho,\la-\eta\rangle$. 
Since $\omega_{X_T}\in \mathstrut^{p}\cD^{\geq0}(X_T)$, see, e.g., \cite[6.3.5.(a)]{BKV},
restricting further to $X_T$ yields
\[\IC_\la|^{!}_{X_T}\in\mathstrut^{p}\cD^{>-2\langle\rho,\la-\eta\rangle}(X_T).\]
The inequalities (i)-(iii), the dimension estimate in Theorem \ref{thm:GT} 
and a standard spectral sequence argument similar to \cite[Prop.~5.13]{BR} 
imply that only the stratum $\Gr_{\la}$ contributes.
Using (a), we deduce
\[H^{2\langle\rho,\nu\rangle}(T_{\nu},\IC_{\la}|^!_{T_\nu})
=H^{2\langle\rho,\nu\rangle}(Y_T,\IC_{\la}\vert_{\Gr_{\la}})
=H_{2\langle\rho,\la-\nu\rangle}^\BM(Y_T,\bql)=L(\la)_{\nu}.\]
In particular, we have
$$\CT_{*}(\IC_\la)=\CT_{*}(\nabla_\la).$$
By Proposition \ref{conserv}, this yields an isomorphism $\IC_{\la}\cong\nabla_{\la}$.

Next, we prove that $\Delta_{\la}\cong\IC_{\la}$.
We must compute $H_{c}^{2\langle\rho,\mu\rangle}(S_{\mu},\IC_{\la})$.
The strategy is the same.
We must prove that only the stratum $\Gr_{\la}$ contributes.
In that case, using the $t$-exactness and the isomorphisms
$$\Delta_{\la}\vert_{\Gr_{\la}}\cong\IC_{\la}\vert_{\Gr_{\la}}\cong\omega_{\Gr_{\la}}[-2\langle\rho,\la\rangle],$$ 
we get
\[\CT_{*,\mu}(\IC_{\la})=H^{2\langle\rho,\mu\rangle}_{c}(S_{\mu}\cap\Gr_{\la},\IC_{\la})=H_{c}^{2\langle\rho,\mu\rangle}(S_{\mu},\Delta_{\la})=H_{c}^{-2\langle\rho,\la-\mu\rangle}(S_{\mu},s_{\mu}^{*}\omega_{\Gr_{\la}})=\CT_{*,\mu}(\Delta_{\la}).\]
It is sufficient to prove that for all $\eta<\la$
\begin{equation}
 \RGm_{c}(S_{\mu}\cap\Gr_{\eta},\IC_{\la})\in\cD^{<- 2\langle\rho,\la-\mu\rangle}(\Spec(k)).
\label{rl-gm}
\end{equation}
First, by definition of the IC complex \eqref{t3} we have
\[(i_{\eta})^{*}\IC_{\la}\in\mathstrut^{p}\cD^{< -2\langle\rho,\la-\eta\rangle}_{G_\cO}(\Gr_{\eta}).\]
The map $\pi_{\eta}$  is smooth and surjective.
Thus the functor $(\pi_{\eta})^!$ is $t$-exact.
Thus, we have
\[(\pi_{\eta})^{!}(i_{\eta})^{*}\IC_{\la}\in\mathstrut^{p}\cD^{< -2\langle\rho,\la-\eta\rangle}_{G_\cO}(E_{\eta})\]
Further, the sheaf $(i_{\eta})^{*}\IC_{\la}$ being $G(\co)$-equivariant, 
it can be written as $h^{!}K_0$ for some $K_{0}\in\cD(\bB L_{\eta})$.
By the commutativity of \eqref{pi-l}, 
the sheaf $(\pi_{\eta})^{!}(i_{\eta})^{*}\IC_{\la}$ is a direct sum  of $\omega_{E_{\eta}}[-d]$ for $d< -2\langle\rho,\la-\eta\rangle$.
We form the Cartesian square
$$\xymatrix{X_S\ar[d]_{\pi}\ar[r]^{s}&E_{\eta}\ar[d]^{\pi_{\eta}}\\Y_S\ar[r]^-{s_{\mu}}&\Gr_{\eta}}$$
Using smooth descent and Lurie's adjunction, for a sheaf $K\in\cD(Y_S)$, we have an equivalence 
\[K\cong\colim_{[n]}(\pi^{[n]})_{!}(\pi^{[n]})^{!}K,\] 
where the map
$\pi^{[n]}: X_S^{[n]}\ra Y_S$
is as in \eqref{[n]}.
This gives an isomorphism
\[\RGm_{c}(Y_S,K)\cong\colim_{[n]}\RGm_{c}(X_S^{[n]},(\pi^{[n]})^{!}K).\]
Since the category $\cD^{< -2\langle\rho,\la-\eta\rangle}(\Spec(k))$ is stable by colimits, 
to prove \eqref{rl-gm} it is enough to prove it on the various $X_S^{[n]}$ for the sheaf
$$K=\IC_{\la}\vert_{Y_S}=(s_\mu)^*(i_{\eta})^{*}\IC_{\la}.$$
By smoothness, we have
\[(\pi^{[n]})^{!}K\cong(s^{[n]})^{*}((\pi_\eta)^{[n]})^{!}(i_{\eta})^{*}\IC_{\la}\]
The complex $(\pi^{[n]})^{!}K$ is a direct sum of complexes
\[(s^{[n]})^{*}\omega_{(E_\eta)^{[n]}}[-d]=(\pi^{[n]})^{!}(s_{\mu})^{*}\omega_{\Gr_{\eta}}[-d]\]
with $d< -2\langle\rho,\la-\eta\rangle$.
Thus the claim follows from \eqref{gs-mu} and the following lemma.

\blem
If $K\in\cD(Y_S)$ is such that $\RGm_c(Y,K)\in\cD^{<0}(\Spec(k))$, then
$$\RGm_c(X_S^{[n]},(\pi^{[n]})^{!}K)\in\cD^{<0}(\Spec(k)).$$
\elem

\bpf
Considering the tower
\[X_S^{[n]}\ra\cdots\ra X_S\ra Y_S.\]
By induction on $n$, we can assume that $n=1$.
Since the map $\pi$ is smooth, the projection formula gives
\[\pi_{!}\pi^{!}K\cong K\overset{!}{\otimes}\pi_{!}\omega_{X_S}\cong K\overset{!}{\otimes} {f}^!\bar{\pi}_{!}\bql\]
where the maps $\overline{\pi}:\Spec(k)\ra\bB L_{\la}$ and $f:Y\ra\bB L_{\la}$ are the obvious ones.
Here the last base change follows from \eqref{shriek-1}.
Since $L_{\la}$ is geometrically connected, the sheaf $\mathstrut^{p}\cH^{i}(\bar{\pi}_{!}\bql)$ is constant. 
Further, we have $\bar{\pi}_{!}\bql\in\mathstrut^{p}\cD^{\leq 0}(\bB L_{\la})$. 
Filtering the complex $\overline\pi_{!}\bql$, an induction using the smoothness of $\bB L_{\la}$ implies that we can replace
$\overline\pi_{!}\bql$ by the  dualizing sheaf $\omega_{\bB L_{\la}}$.
Thus, the claim reduces to prove that
$\RGm_c(Y_S,K)\in\cD^{<0}(\Spec(k))$ which holds by assumption. 
\epf
\epf

We have the normalized constant term functor:
\[\CT_{*}[\deg]:\cD_{G_\cO}(\Gr_{G_c})^\heartsuit\Ind(\Rep(T^{\vee})),\]
as well as a canonical restriction functor:
\[\Rep(G^{\vee})\ra\Rep(T^{\vee}),\]
that is faithful and injective on objects. In particular, it identifies $\Rep(G^{\vee})$ as a non-full subcategory of $\Rep(T^{\vee})$. 

\bthm\label{satake-fin}
The category $\cD_{G_\cO}(\Gr_{G_c})^\heartsuit$ is semisimple.
The normalized constant term $\CT_{*}[\deg]$ is an equivalence of semisimple Abelian categories
$\cD_{G_\cO}(\Gr_{G_c})^\heartsuit\cong\Ind(\Rep(G^{\vee})).$
\ethm

\bpf
By Theorem \ref{t-exact} and Proposition \ref{conserv}, 
the functor is exact and conservative, thus faithful. 
By Corollary \ref{simple}, the simple objects are the same and $\Rep(G^{\vee})$ is semisimple.
Let us prove that $\cD_{G_\cO}(\Gr_{G_c})^\heartsuit$ is semisimple.
We must prove that for any $\la,\mu\in X_{*}(T)$, we have
\[\Hom_{\cD_{G_{\co}}(\Gr_{G_c})}(\IC_{\la},\IC_{\mu}[1])=0.\]
In the reductive case, the usual argument uses parity vanishing of $\cH^{i}(\IC_{\la})$. 
In our situation, we do not know how to prove 
such a parity result. Nevertheless, to prove semisimplicity we need less.
The argument of \cite[Prop.~3.1]{Ric2} gives that we only need for $\mu<\la$ and $\mu$ dominant that
\begin{equation}
(i_{\mu})^{!}\IC_{\la}\in\mathstrut^{p}\cD^{\geq-2 \langle\rho,\la-\mu\rangle+2}(\Gr_\mu)
,\quad
(i_{\mu})^{*}\IC_{\la}\in\mathstrut^{p}\cD^{\leq -2 \langle\rho,\la-\mu\rangle-2}(\Gr_\mu).
\label{2-st}
\end{equation}
The usual inequality for an IC complex is \eqref{t3} which yields
\[(i_{\mu})^{!}\IC_{\la}\in\mathstrut^{p}\cD^{\geq-2 \langle\rho,\la-\mu\rangle+1}(\Gr_\mu)
,\quad
(i_{\mu})^{*}\IC_{\la}\in\mathstrut^{p}\cD^{\leq -2 \langle\rho,\la-\mu\rangle-1}(\Gr_\mu).\]
However, the stronger inequality follows from Theorem \ref{ic-calc}(b) and \cite[Rmk.~after~Cor.~1.4.24]{BBD}.
Thus, the category $\cD_{G_\cO}(\Gr_{G_c})^\heartsuit$ is semisimple and the functor is fully faithful. 
As both sides are stable by arbitrary direct sums and $\Rep(G^{\vee})$ is already in the image of the functor, we obtain the desired 
equivalence.
\epf

\begin{appendix}

\section{The Vinberg monoid of a Kac-Moody group}\label{vinberg}

The goal of this section is to give a proof of Proposition \ref{ti-quot}.
To do this, we first gather some material on Vinberg monoids of KM groups.
The main result of this section is Proposition \ref{Hyp}.
We will assume that $k$ is an algebraically closed field, and $G$ is a simply connected minimal KM group over $k$.
We also assume that either $G$ is arbitrary and $\car(k)=0$, 
or $G$ is affine and $\car(k)$ is arbitrary.

\subsection{Construction}
We consider the group ind-scheme 
$$G_+=(G\times T)/Z$$ where we embed the center $Z$ of $G$ anti-diagonally.
Let
$$T_+=(T\times T)/Z,\quad Z_+=(Z\times T)/Z\cong T$$
be the maximal torus and the center of $G_+$. 
For each dominant character $\omega$, let
$\rho_\omega$ be as in \eqref{rho-omega}.
Set 
$$
H_G=\prod\limits_{i\in I}\End^\ind(L(\omega_i))\times\ab_\Delta$$
where $\ab_\Delta$ is as in \eqref{AD}.
The map
$G_+  \to  H_G$ given by
$(g,t)\mapsto (\omega_i(t)\rho_{\omega_{i}}(g),\al(t))$
is a monomorphism.
We define $\Vin_G$ to be the scheme theoretic image  of $G_+$ in $H_G$.
The scheme theoretic image commutes with filtered colimits.
The scheme theoretic image of a morphism of schemes $f:X\to Y$ is the smallest closed subscheme $Z\subset Y$ 
through which $f$ factors, see \cite[Tag.~01R7]{Sta}.
Thus the functor $\Vin_G$ is the colimit of the closures of the images of the components of a colimit representing $G$.
The functor $\Vin_G$ has an action of $G_+\times G_+$ that extends the left and right multiplication on $G_+$.
We deduce the following.

\bprop\label{Vin}
$\Vin_G$ is  a $G_+\times G_+$-equivariant ind-affine ind-scheme with a monoid structure.
\qed
\eprop

We call $\Vin_G$ the Vinberg monoid of $G$.
Since $\End^\ind(V)$ may not be of ind-finite type, the Vinberg monoid may also not be of ind-finite type,
see \S\ref{EndInd}.
We identify $G_+$ with its image in $\Vin_G$.
Let $T_\ad=T/Z$.
The group $T_\ad$ embeds in $\ab_\Delta$ via the simple roots as the open subset where all coordinates are nonzero.
Let $T_\Delta\subset \Vin_G$ be the image of the anti-diagonal morphism $T\to T_+$.
Let $\overline{T}_\Delta\subset\Vin_G$ be its closure.
The projection $H_G\to\ab_\Delta$ gives a map
$$\det:\Vin_G\ra\ab_\Delta.$$
This map restricts to an isomorphism
\begin{equation}
T_\Delta\cong T_\ad.
\label{det1}
\end{equation}
For each $t\in T_\Delta$ and each $i\in I$, the endomorphism $\omega_i(t^{-1})\rho_{\omega_{i}}(t)$ is polynomial in the simple roots $\al(t)$.
Hence the isomorphism \eqref{det1} extends to an isomorphism
\begin{equation}
\ov{T}_\Delta\cong\ab_\Delta
\label{det2}
\end{equation}
This yields the section 
\begin{equation}
\sigma=(\det)^{-1}:\ab_\Delta\ra\overline{T}_{\Delta}\subset\Vin_G.
\label{sigma}
\end{equation}

We introduce the open Bruhat cell in $\Vin_G$ following the strategy of Solis in \cite{Sol}.
For a vector space $V$ and an element $v^\vee\in V^\vee$, we consider the qc open subsets in $V$ and $\bP(V)$ given by
$$V_{v}=V\setminus\{v^{\vee}=0\}
,\quad
\bP(V)_v=\bP(V)\smallsetminus\bP(\{v^{\vee}=0\}).$$
We have an $U^-\times U$-equivariant map
$$\psi: H\ra \prod_{i\in I}L(\omega_i)\times L(\omega_i)^{\vee}
,\quad
(f,z)\mapsto (f(v_i),v_i^{\vee}  f)_{i\in I}$$ 
We consider the qc open subset
$H_\Omega\subset H_G$
given by
$$H_\Omega=\psi^{-1}\Big(\prod_{i\in I}L(\omega_i)_{v_i}\times L(\omega_i)_{v_i^\vee}^\vee\Big).$$
As $\Vin_G$ is closed in $H_{G}$, the qc open subset $\Vin_\Omega\subset\Vin_G$ given by
$$\Vin_\Omega=\Vin_G\cap H_\Omega$$
is closed in $H_\Omega$.
Let $\Omega_+=U^{-}T_+ U$ be the open cell of $G_+$.
We have 
$$\Omega_+=G_+\cap \Vin_\Omega
,\quad
\overline{T}_{\Delta}\subset\Vin_\Omega.$$
Quotienting by $\bG_m$, we obtain a map
\[\ov{\psi}:H_\Omega\ra \prod_{i\in I} \bP( L(\omega_i))_{v_i}\times \bP( L(\omega_i)^\vee)_{v_i^\vee}.\]
Further, the orbit maps at the tuples $(v_i, v_i^{\vee})$ yield the map
\[b:G_+\ra \prod_{i\in I} \bP (L(\omega_i))\times \bP( L(\omega_i)^{\vee})\]
The map $b$ coincides with $\ov{\psi}$  over the open cell $\Omega_+$.
By Lemma \ref{stab}, the stabilizer of the tuple $(v_i, v_i^{\vee})$ in $G_+\times G_+$ identifies with $B_+\times B_+^-$.
Hence the map $b$ factors through
\[G_+\stackrel{\Delta}{\rightarrow}G_+\times G_+\ra G_+/B_+\times G_+/B_+^-=G/B\times G/B^{-}.\]
Over  $\Omega_+$ the maps $b$ and $\ov\psi$ are both given by the projection followed by the open immersion 
\begin{equation}
U^{-}\times T_+\times U\rightarrow U^{-}\times U\hra G/B\times G/B^{-}.
\label{op-cell2}
\end{equation}

\blem\label{image}\hfill
\begin{enumerate}[label=$\mathrm{(\alph*)}$,leftmargin=8mm]
\item
$U^{-}\times U$ is closed in $\prod_{i\in I} \bP( L(\omega_i))_{v_i}\times \bP( L(\omega_i)^\vee)_{v_i^\vee}$.
\item
$\ov{\psi}(\Vin_{\Omega})= U^{-}\times U$.
\eenum
\elem

\bpf
Part (a) follows from  \eqref{uminus}.
To prove (b), note that $\Omega_+$ is dense in $G_+$. 
Hence $\Omega_+$ is dense in $\Vin_\Omega$.
By \eqref{op-cell2}, we have $\ov{\psi}(\Omega_+)=U^-\times U$, which is closed in the product
$$\prod_{i\in I} \bP( L(\omega_i))_{v_i}\times \bP( L(\omega_i)^\vee)_{v_i^\vee}.$$
Part (b) follows.
\epf

\bprop\label{bruhat}\hfill
\begin{enumerate}[label=$\mathrm{(\alph*)}$,leftmargin=8mm]
\item
The map 
$U^-\times Z_+\times \overline{T}_{\Delta}\times U\ra\Vin_\Omega$
given by $(u^-,z,t,u)\mapsto u^-ztu$ is an isomorphism. 
\item
$\Omega_+$ is a qc open of $\Vin_\Omega$.
\item
 $\Vin_\Omega$ is an ind-ft-scheme.
\eenum
\eprop

\bpf
(c) follows from (a).
Let $\eta:U^-\times Z_+\times \overline{T}_{\Delta}\times U\ra\Vin_\Omega$ be the map in (a).
By Lemma \ref{image}, we have
$\ov{\psi}(\Vin_\Omega)\subset U^{-}\times U$.
Since the map $\ov{\psi}$ is $U^{-}\times U$-equivariant,
 to prove (a) it is enough to check that
\[Z_+\times \overline{T}_{\Delta}=\ov{\psi}^{-1}(1,1).\]
The inclusion $\subset$ is immediate, because $Z_+$ acts by dilatation on each $L(\omega_i)$. 
Let us prove the converse.
Set 
$$H_0=\Big(\prod_{i\in I}\big(\End^\ind(L(\omega_i))\setminus\{0\}\big)\Big).$$
The group $Z_+$ acts freely on $H_0$, because it acts by dilatation on each factor. 
The functor $H_0/Z_+$ is representable by 
$$\prod_{i\in I}\bP\Big(\End^\ind(L(\omega_i))\Big).$$
We have $\Vin_\Omega\subset H_\Omega\subset H_0$. 
Since $\Vin_\Omega$ is closed in $H_\Omega$, it is locally closed in $H_0$.
Hence, the functor $\Vin_\Omega/Z_+$ is also representable by an ind-scheme.
We have $T_+\cong Z_+\times T_\Delta.$ and as $\ov{T}_{\Delta}$ is closed in $\Vin_\Omega/Z_+$, $Z_+\times \ov{T}_{\Delta}$ is closed in $\Vin_\Omega$.
Hence $Z_+\times \ov{T}_{\Delta}$ is the closure of $T_+$ in $\Vin_\Omega$.
Let  $x$ be a $k$-point in $\ov{\psi}^{-1}(1,1)$.
By \cite[Tag.~02JQ]{Sta}, there is a valuation ring $R$ with fraction field $K$ and
a function $f\in(U^{-}T_+ U)(K)$ that extends to a function $\ti{f}\in\Vin_\Omega(R)$,
and such that $f(\km)=x$ where $\km$ is the closed point of $k$. Let 
$(u^{-},u)=\ov{\psi}(\ti{f})$ in $(U^{-}\times U)(R)$.
We have
$$((u^-,u)^{-1}\ti{f})(\km)=x
,\quad
((u^-,u)^{-1}\ti{f})|_{\Spec(K)}\in T_+(K).$$
Thus the map $x$ is in the closure of $T_+$  which is $Z_+\times \ov{T}_{\Delta}$.
We deduce that the map $\eta$ is a closed immersion which is bijective on $k$-points.
Thus it is an homeomorphism.
We deduce that $\Omega_+$ is open in $\Vin_\Omega$ and thus in $\Vin_G$ and as $\Omega_+$ is ind-affine, we get (b).

Since $\Omega_+$ is schematically dense in $G_+$ and quasi-compact by (b), 
the ind-scheme $\Vin_G$ is also the scheme theoretic image of $\Omega_+$. 
Thus  \cite[Tag.~01R8, 01RD]{Sta} implies that $\Omega_+$ is schematically dense in $\Vin_\Omega$, i.e., 
for any open subset $U\subset \Vin_\Omega$ the scheme theoretic closure of $\Omega_+\cap U$ in $U$ is equal to $U$ and as $\eta$ was already a nilpotent closed immersion, it is an isomorphism proving (a).

\epf

Since $\Vin_\Omega$ is qc open in $\Vin_G$, the subset
\begin{equation}
\Vin_0=(G\times G)(k)\cdot\Vin_\Omega\subset\Vin_G
\label{0-vin}
\end{equation}
is an increasing union of qc open subsets.
The action map $$G\times G\times\Vin_0\ra\Vin_G$$ 
factors through $\Vin_0$, because it factors on $k$-points and $\Vin_0$ is open.
Thus  $\Vin_0$ is $G\times G$-stable.
We have $$\overline{T}_{\Delta}\subset\Vin_0,$$ 
because the coefficient in the highest weight vector of  
$\omega_i(t)\rho_{\omega_{i}}(t^{-1})$ is 1.
We consider the group ind-scheme
\[\Stab_{G\times G}(\sigma)=\{(g_1,g_2,t)\in G\times G\times\ab_\Delta\,;\,g_{1}\sigma(t)g_{2}^{-1}=\sigma(t)\}.\]
It is closed in $G\times G\times\ab_\Delta$, which is 
viewed as a group ind-scheme over $\ab_\Delta$.
The section $\sigma$ in \eqref{sigma} factors through $\ab_\Delta\ra\Vin_0.$

\bprop\label{vin-quotient}
The action map $G\times G\times\ab_\Delta\ra \Vin_0$ on the section $\sigma$ yields an isomorphism of étale sheaves over $\ab_\Delta$
$$[(G\times G\times\ab_\Delta)/\Stab_{G\times G}(\sigma)]\cong\Vin_0$$ 
This isomorphism also holds Zariski locally. 
\eprop

\bpf
By Proposition  \ref{bruhat}, over $\Vin_\Omega$ we have a  section and  $\Vin_0$ is covered by translates of $\Vin_\Omega$.
\epf

\bprop\label{stab-smooth}
The morphism $\Stab_{G\times G}(\sigma)\ra \ab_\Delta$ is formally smooth.
\eprop

\bpf
By Proposition \ref{vin-quotient}, it is  enough to prove the formal smoothness of the action map
$$G\times G\times\ab_\Delta\ra\Vin_0.$$
The group ind-scheme $G$ is formally smooth, due to our assumptions on $k$ or $G$.
Thus the source is formally smooth.
The target is also formally smooth by Proposition  \ref{bruhat}.
The morphism has sections Zariski locally by Proposition \ref{vin-quotient}.
Hence it is formally smooth.
\epf

We want to describe the fibers over $\ab_\Delta$ of this stabilizer group ind-scheme.
For each subset $J\subset\Delta$, let $e_J\in\ov{T}_{\Delta}$ be the element with coordinates
$\al_{j}(e_J)=0$ if $j\in J$ and $\al_j(e_J)=1$ otherwise.
Hence, the linear map $\rho_{\omega_i}(e_J)$ is the projection from $L(\omega_i)$ to the sum 
$L(\omega_i)_J$ of the weight subspaces whose weights belong to the set $\omega_i+\bZ\{\al_j\,;\,j\in J\}$.

\bprop\label{stabI}
For each subset $J\subset\Delta$ we have $\Stab_{G\times G}(e_J)=P_J\times_{L_J} P_J^{-}$.
\eprop

\bpf
Let $R$ be any $k$-algebra $R$.
Let $(g_1,g_2)\in G(R)\times G(R)$ be such that $g_1e_J=e_Jg_2$. 
For each $i\in I$, let $M(\omega_i)_J$ be the $T$-invariant complement subspace of $L(\omega_i)_J$ in $L(\omega_i)$.
We have
\begin{itemize}[leftmargin=6mm]
\item 
$L(\omega_i)_J$ is stable under $\rho_{\omega_i}(g_1)$,
\item
$M(\omega_i)_J$ is stable under $\rho_{\omega_i}(g_2)$,
\item
the maps $L(\omega_i)_J\to L(\omega_i)_J\to  L(\omega_i)/M(\omega_i)_J\cong L(\omega_i)_J$ 
given by 
$\rho_{\omega_i}(g_1)$ and $\rho_{\omega_i}(g_2)$ coincide.
\end{itemize}
The proposition follows from the next lemma.

\blem
\hfill
\hfill
\begin{enumerate}[label=$\mathrm{(\alph*)}$,leftmargin=8mm]
\remi
$\Stab_{G}(L(\omega_i)_J)=P_J$ and $\Stab_{G}(M(\omega_i)_J)=P_J^{-}$.
\remi
The kernel of the $P_J$-action on $L(\omega_i)_J$ is $U_JZ(L_J)$.
\remi
The kernel of the $P_J^-$-action on $M(\omega_i)_J$ is $U_J^-Z(L_J)$.
\eenum
\elem

\bpf
Set $Q_J=\Stab_{G}(L(\omega_i)_J)$. 
We have $P_J\subset Q_J$. 
So $Q_J$ is a standard parabolic.
By \cite[Thm.~5.1.3(g)]{Kum}, since $s_j$ does not 
stabilize $L(\omega_i)_J$ if $j\notin J$, we have a bijection on $k$-points.
By considering the open cell $U_J^{-}P_J$, 
we obtain the same way as in Lemma \ref{stab}, that it is an isomorphism. 
The second claim in (a) is similar.
To prove (b) it suffices to note that $L(\omega_i)_J$  and $M(\omega_i)_J$ are $L_J$-stable.
Hence the claim 
reduces to the Levi, for which it is clear.  
\epf
\epf

\subsection{Proof of Proposition \ref{ti-quot}}\label{PfPROP}
For any positive root $\al$, the map
$\al\circ 2\Lrho:\bG_m\to\bG_m$ extends to a map $\ab^1\ra\ab^{1}$.
Taking the product over all simple roots, we get a map
\begin{equation}
\phi:\ab^{1}\ra\ab_\Delta.
\label{pul-g}
\end{equation}
We introduce the hyperbolic monoid
$$\xymatrix{\Hyp_G\ar[r]\ar[d]_{\det}&\Vin_G\ar[d]^{\det}\\\ab^{1}\ar[r]^-{\phi}&\ab_\Delta}$$
Let $\Hyp_0$ denote the base change of $\Vin_0$.
The base change yields a section  to the  map $\det$
$$\sigma^{\Hyp}:\ab^{1}\ra \Hyp_0.$$

\blem\label{stab-d}\hfill
\begin{enumerate}[label=$\mathrm{(\alph*)}$,leftmargin=8mm]
\item
If $t\neq 0$, then $\Stab_{G\times G}(\sigma^{\Hyp}(t))=G$.
\item
If $t=0$,  then $\Stab_{G\times G}(\sigma^1(t))=B\times_{T}B^{-}$.
\eenum
\elem

\bpf
Part (a) is obvious.
If $t=0$ then $\sigma(\phi(0))=e_{\emptyset}$ in $\Vin^0_{G}$.
Hence (b) follows from Proposition \ref{stab}.
\epf

Proposition \ref{ti-quot} follows from Proposition \ref{Vin} and the following.

\bprop\label{Hyp}\hfill
\begin{enumerate}[label=$\mathrm{(\alph*)}$,leftmargin=8mm]
\item
The action map yields an equivalence of étale sheaves over $\ab^1$
$$[(G\times G\times\ab^1)/\Stab_{G\times G}(\sigma^1)]\cong \Hyp_0$$
The left hand side is open in an ind-affine ind-scheme.
\item
We have $\ti{G}\cong\Stab_{G\times G}(\sigma^{\Hyp})$.
\eenum
\eprop

\bpf
Claim (a) follows from Proposition \ref{vin-quotient} by base change.
To prove (b) we must relate $\ti{G}\to\ab^1$ with the stabilizer of $\sigma^1$.
The fibers of $\ti{G}$ and $\Stab_{G\times G}(\sigma^{\Hyp})$ over any $t\in \ab^1$ are isomorphic by 
Proposition \ref{stab-d}.
The group $\ti{G}|_{\bG_m}$ is the subscheme of $G\times G\times\bG_m$ such that
\[\phi(t)\cdot g_{1}\cdot\phi(t)^{-1}=g_2,\quad\forall g_1,g_2\in G.\]
On the other hand $\Stab_{G\times G}(\sigma^1)|_{\bG_m}$ is defined by the equation
\[g_{1}\cdot\sigma^{\Hyp}(t)\cdot g_{2}^{-1}=\sigma^{\Hyp}(t).\]
By \eqref{sigma} and base change, the element $\sigma^{\Hyp}(t)$ differs from $\phi(t)^{-1}$ by an element of $Z_+$.
Thus the ind-schemes are the same over $\bG_m$.
Since $\Stab_{G\times G}(\sigma^1)$ is closed in $G\times G\times\ab^1$, 
by Lemma \ref{adh-G} the closed immersion $\ti{G}\hra G\times G\times\ab^1$ factors as an fp closed immersion
\[i:\ti{G}\ra\Stab_{G\times G}(\sigma^{\Hyp})\]
which is fiberwise an isomorphism by Lemma \ref{stab-d}.
By Lemma \ref{ti-fsmooth} and Proposition \ref{stab-smooth}, the groups $\ti{G}$ and $\Stab_{G\times G}(\sigma^1)$ are formally smooth over $\ab^1$.
Thus by Proposition \ref{ind-smooth},
the map $i$ is formally smooth.
Thus it is étale.
Since it is a bijective closed immersion, it is an isomorphism.
\epf

\section{Deformation theory for prestacks}

The goal of this section is to prove Proposition \ref{ind-smooth}, from which Propositions \ref{ti-quot} and \ref{Hyp} follow.

\subsection{Quasi-coherent sheaves}
Let $\dAff_k$ be the category of affine derived $k$-schemes, and
$\dPrSt_k=\PrSh(\dAff)$ be the category of derived $\infty$-prestacks over $k$.
Let $\Ind(\dAff_k)\subset\dPrSt_k$ be the full subcategory of ind-objects in  $\dAff_k$ 
with no condition on transition morphisms.
The Yoneda embedding \cite[Prop.~5.1.3.1]{Lu1} yields a fully faithful functor 
$\eta:\dAff_k\ra\dPrSt_k$ that factors through a chain of fully faithful functors
\begin{equation}
\begin{split}
\xymatrix{\dAff_k\ar[r]_-{\eta_1}\ar@/^2pc/[rr]^{\eta}&\Ind(\dAff_k)\ar[r]_{\eta_{2}}&\dPrSt_k}.
\end{split}
\label{faith}
\end{equation}
For any $X\in\dAff_k$, let $\QCoh(X)$ be the stable $\infty$-category 
of quasi-coherent sheaves over $X$,
see \cite[Def.~1.3.5.8]{Lu2}. 
It is equipped with the standard $t$-structure, see \cite[Prop.~ 1.3.5.21]{Lu2}.
Let $\QCoh(X)^{\heartsuit}$ be its heart.
 For each morphism $f:X\ra Y$ in $\dAff_k$, we have a functor 
 $$f^*:\QCoh(X)\ra\QCoh(Y).$$
 Thus, there are  functors
$\Pro(\QCoh),\QCoh:\dAff_k^{\op}\ra\StCat$
that we left Kan-extend to get functors
\[\Pro(\QCoh),\QCoh:\dPrSt_k^{\op}\ra\StCat.\]
In particular, for any prestack $\cX$ and any $X\in\dAff_k$ we have 
\[\QCoh(\cX)=\lim_{X\ra\cX}\QCoh(X),\]
By \cite[\S3.1.5.1]{GRI}, the category $\QCoh(\cX)$ is equipped with a  $t$-structure such that 
\[\QCoh(\cX)^{\leq 0}=\{\cF\in\QCoh(\cX)\,;\,x^*\cF\in\QCoh(S)^{\leq 0}\,,\,\forall~ x:S\to \cX\}.\]
The positive category $\QCoh(\cX)^{\geq 0}$ does not have such a convenient description.

\subsection{Functor of derivations}
We recall the following construction from \cite[2.1.3]{Bt3}.

\bdefi
For each $\infty$-category $\cC$ with finite colimits, we consider the functor  $\Fact_{\cC}:\cC^{\Delta^{1}}\ra\infty$-$\Cat$
which associates to a morphism $x\ra y$ in $\cC$ the $\infty$-category of factorizations $x\ra c\ra y$.
\edefi

For an affine scheme $S=\Spec(R)$ and any $k$-module $M$, let 
$$S[M]=\Spec(R\oplus M)$$ with the multiplication given by
\[(a,m)\cdot(a',m')=(aa,am'+a'm).\]
We have a factorization $S\hra S[M]\ra S$, where the first map is a square zero closed immersion.
We get a morphism of functors $\Aff\ra\infty$-$\Cat$ taking $S$ to
\[S[-]:\QCoh(S)^{\heartsuit}\ra \Fact_{\Aff}(\id_{S}).\]
We abbreviate
\[[-]:\QCoh_{\Aff}^{\heartsuit}\ra \Fact_{\Aff}(\id_{-}).\]
By Lurie \cite[Const.~25.3.1.1]{Lu3}, there exists a unique extension of this functor to $\dAff$ 
and $\QCoh^{\leq 0}$ that commutes with small colimits
\[[-]:\QCoh_{\dAff}^{\leq 0}\ra \Fact_{\dAff}(\id_{-}).\]
We must extend this functor to any prestack.
Applying the $\Pro$ functor, we get a natural transformation
\[[-]:\Pro(\QCoh_{\dAff}^{\leq 0})\ra\Pro(\Fact_{\dAff}(\id_{-}))\cong\eta_{1}^{*}(\Fact_{\Ind(\dAff)}(\id_{-}))\ra \eta^{*}(\Fact_{\dPrSt}(\id_{-})),\]
with  $\eta_1^{*}$ and $\eta^{*}$ the pullback functors induced by restriction to $\dAff$, see \eqref{faith}.
We then apply the left Kan extension $\eta_!$.
It is left adjoint to $\eta^{*}$.
Since $\eta$ is fully faithful, the counit $\eta_!\eta^{*}\ra\id$ is an equivalence.
We obtain
\begin{equation}
[-]:\Pro(\QCoh_{\dPrSt}^{\leq 0})\ra\eta_!\eta^{*}(\Fact_{\dPrSt}(\id_{-}))\cong\Fact_{\dPrSt}(\id_{-}).
\label{foncultim}
\end{equation}

\bdefi
For any morphism of $\infty$-prestacks $\cX\ra\cY$, we define the derivation functor
\[\Der_{\cY}(\cX,-)=\Map_{\cX/./\cY}(\cX[-],\cX).\]
\begin{enumerate}[label=$\mathrm{(\alph*)}$,leftmargin=8mm]
\item
$\cX\ra\cY$ admits a relative pro-cotangent complex if there is an object 
$L_{\cX/\cY}\in\Pro(\QCoh(\cX))$ such that for each $E\in\Pro(\QCoh^{\leq 0}(\cX))$ we have
\[\Der_{\cY}(\cX,E)\cong\Map(E,L_{\cX/\cY}).\]
\item
$\cX$ admits  a pro-cotangent complex if  $\cX\ra\Spec(k)$ has a relative pro-cotangent complex.
\eenum
\edefi

\brem
\hfill
\begin{enumerate}[label=$\mathrm{(\alph*)}$,leftmargin=8mm]
\item
Even if in the sequel we use non-derived prestacks, to construct the cotangent complex we must consider the derived setting.
\item
For a morphism of $\infty$-prestacks $f:\cX'\ra\cX$ with pro-cotangent complexes, functoriality yields a morphism  
 \begin{equation}
f^{*}L_{\cX/\bZ}\ra L_{\cX'/\bZ}.
\label{canL}
\end{equation}
Applying $\Map(-,E)$ for each $E\in\Pro(\QCoh^{\leq 0}(\cX'))$, the map $f$ has a relative pro-cotangent complex
\begin{equation}
L_{\cX'/\cX}\cong\Cofib(f^{*}L_{\cX/\bZ}\ra L_{\cX'/\bZ}).
\label{Omful}
\end{equation}
\eenum
\erem

\bexa\label{Tdef-alg}
\hfill
\begin{enumerate}[label=$\mathrm{(\alph*)}$,leftmargin=8mm]
\item
Let $X$ be an ind-algebraic space.
By  \cite[Prop.~1.2.19]{Hen} it has a pro-cotangent complex $L_{X/\bZ}$, such that for any $X\cong\colim X_{a}$  
and any $x\in X(R)$ we have
\[x^{*}L_{X/\bZ}\cong\lim x^{*}L_{X_{a}/\bZ}.\] 
In particular, we have $L_{X/\bZ}\in\Pro(\QCoh(X)^{\leq 0})$.
\item
If $X$ is a formally smooth ind-scheme of ind-ft and $H$ is a group ind-scheme,
then the quotient $[X/H]$ admits a pro-cotangent complex by \cite[Prop.~2.2.11]{Bt3}.
\eenum
\eexa

In the formally smooth case, we have more constraints on the pro-cotangent complex.

\bprop\label{GR-smooth}
Let $X$ be a formally smooth ind-scheme over a base scheme $S$. 
For each point $x:\Spec(R)\ra X$ we have
\hfill
\begin{enumerate}[label=$\mathrm{(\alph*)}$,leftmargin=8mm]
\item
 $H_0(x^{*}L_{X/S})$ is pro-projective,
\item
$H_1(x^*L_{X/S})=0$.

\eenum
\eprop

\bpf
For (a), by \cite[Lem.~5.2.1]{Bt3}, given a presentation $X\cong\colim X_a$, it is enough to prove that the functor on $R$-modules
$$M\mapsto\Hom_{c,R}(x^{*}\Omega^{1}_{X/S},M):=\colim_{a}\Hom_{R}(x^{*}\Omega^{1}_{X_a/S},M)$$ 
is exact. Denote $\Omega=x^{*}\Omega^{1}_{X/S}$.
The functor is already left exact. For right exactness, if $M\thra N$ and $\phi\in \Hom_{c,A}(\Omega,N)$ 
it amounts to a cofiltered family of  diagrams
$$\xymatrix{&\cO_{X_a}\oplus M\ar[d]\ar[r]&\cO_{X}\ar@{=}[d]\\\cO_{X_{a}}\ar[r]^-{\phi}&\cO_{X_a}\oplus N\ar[r]&\cO_{X_a}}$$
and because $X$ is formally smooth, the section $\phi$ lifts for $a$ big enough.
Part (b) follows from the proof of \cite[Prop.~9.4.2]{GRd}.
\epf

An $\infty$-prestack $\cX$ with a pro-cotangent complex $L_{\cX}$ controls the split square zero extension. 
We want to control all square zero extensions. To do this we introduce an extra condition.

\subsection{Prestacks with a Deformation theory}

Recall the following definition \cite[\S I, Def.~7.1.2]{GRII}.

\bdefi\label{defT}
A morphism $\cX\ra\cY$ of $\infty$-prestacks admits a relative deformation theory if 
\hfill
\begin{enumerate}[label=$\mathrm{(\alph*)}$,leftmargin=8mm]
\item
$\cX$ is convergent, i.e., for each $S\in\dAff$ we have
$$\Map(S,\cX)\cong\lim_{n\in\NN}\Map(\tau^{\leq n}S,\cX)$$ where $\tau^{\leq n}S$ is the $n$-truncation of $S$.
\item
$\cX\ra\cY$ admits  a relative pro-cotangent complex.
\item
$\cX$ is infinitesimally cohesive, see \cite[\S I, \S6]{GRII}.
\eenum
If $\cY=\Spec(\bZ)$, we say that $\cX$ admits a deformation theory.
\edefi

\brem\label{rem-defT1}\hfill
\begin{enumerate}[label=$\mathrm{(\alph*)}$,leftmargin=8mm]
\item
To be  infinitesimally cohesive for a prestack $\cX$ amounts to say that for each square zero extension $\bar S\hra S$, 
the groupoid $\cX(S)$ can be described in terms of $\QCoh(\bar S)$.
\item
Each $n$-truncated prestack is convergent.
In the sequel we will only work with classical prestacks.
\item
An ind-scheme $X$ has a deformation theory by \cite[\S II, Prop.~1.3.2]{GRII}.
By \eqref{Omful}, a morphism of ind-schemes 
$X\ra Y$ has a relative deformation theory. The same holds for an ind-algebraic space, because it is infinitesimally cohesive by 
\cite[Rmk.~17.3.1.7]{Lu3} and we already checked the other conditions.
\item
By \cite[Prop.~2.2.11]{Bt3}, étale quotients $[X/H]$ for $X$ a formally smooth ind-scheme of ind-ft and $H$ an group ind-scheme 
have a deformation theory. 
\eenum
\erem

Now, the main reason for introducing this property is the following, see  \cite[Rmk.~17.3.1.8]{Lu3}.

\bprop\label{lift}
Let $\Spec(\ov R)\hra \Spec(R)$ be a square zero extension of affine schemes of ideal $I$.
Let $f:\cX\ra\cY$ be a morphism of $\infty$-prestacks with a relative deformation theory.
For each commutative diagram
$$\xymatrix{\Spec(\ov{R})\ar[d]\ar[r]^-{\bar{\eta}}&\cX\ar[d]\\\Spec(R)\ar[r]^-{\eta}&\cY}$$
the obstruction to lift $\bar\eta$  belongs to $\Ext_{\Pro(\Mod_{\ov{R}})}^{1}(\bar{\eta}^{*}L_{\cX/\cY},I)$.
If this obstruction vanishes, then the space of liftings is a trivial torsor under 
$\Hom_{\Pro(\Mod_{\ov{R}})}(\bar\eta^{*}L_{\cX/\cY},I)$.
\eprop

\bpf
Let $I=\Ker(R\thra \ov R)$.
By \cite[Tag.~08US, 07BP]{Sta}, we have $\tau_{\geq -1}L_{\ov R/R}=M[1]$.
Thus we get a canonical map $s:L_{\ov R/R}\ra M[1]$ that gives rise to a derivation $s:\ov R\ra \ov R\oplus I[1]$.
Equivalently, the map $s$ is a morphism of rings that splits the projection $\ov R\oplus I[1]\ra \ov R$, 
such that there is a Cartesian diagram
$$\xymatrix{R\ar[r]^\form\ar[d]&\ov R\ar[d]^{i}\\\ov R\ar[r]^-{r}&\ov R\oplus I[1]}$$
where $i$ is the inclusion and $f$ the obvious map.
Since $f$ is cohesive \cite[Def.~17.3.7.1]{Lu3}
there is a pullback square
$$\xymatrix{\cX(R)\ar[d]\ar[r]&\cX(\ov R)\ar[d]\\
\cX(\ov R)\times_{\cY(\ov R)}\cY(R)\ar[r]&\cX(\ov R\oplus I[1])\times_{\cY(\ov R\oplus I[1])}\cY(\ov R)}$$
If we fix a point $\eta\in\cX(\ov R)$ of image $\bar{\eta}\in\cY(\ov R)$ then we get a fiber sequence of spaces
\[\{\eta\}\times_{\cX(\ov R)}\cX(R)\ra \{\bar{\eta}\}\times_{\cY(\ov R)}\cY(R)\ra\Hom(\bar{\eta}^{*}L_{X/Y},I[1])\]
and the desired obstruction, as well as the assertion on torsors.
\epf

\bprop\label{ind-smooth}
Let $R$ be a noetherian ring.
Let $f:X\ra Y$ be an ft schematic morphism of ind-schemes of ind-ft over $S=\Spec(R)$. 
Assume that
\begin{enumerate}[label=$\mathrm{(\alph*)}$,leftmargin=8mm]
\item
$X$ is formally smooth over $S$,
\item
the map $f_s$ is an open immersion for each $s\in S$.
\eenum
Then $f$ is an open embedding.
\eprop

\bpf
Let $x:\Spec(R')\ra X$, we must show that $\tau_{\geq -1 }x^{*}L_{X/Y}=0$. The problem being local, we can assume $R'$ is local of maximal ideal $\km_{R'}$ and $R$ local of maximal ideal $\km$, so that we have $\km\subset\km_{R'}$ 
through the obvious map $R\ra R'$. 
Since $X$ is formally smooth over $S$, by Proposition \ref{GR-smooth}, $\tau_{\geq -1 }x^{*}L_{X/S}$ is pro-projective concentrated in degree zero, thus in $\Pro(\QCoh(\Spec(R')))$, we have an exact sequence:
\[0\ra H_{1}(x^{*}\Omega^{1}_{X/Y})\ra x^{*} f^{*}\Omega^{1}_{Y/S}\ra x^{*}\Omega^{1}_{X/S}\ra x^{*}\Omega^{1}_{X/Y}\ra 0.\]
Since $X$ and $Y$ are of ind-ft over $S$, we deduce that
$x^{*}\Omega^{1}_{X/S}$ and $x^{*} f^{*}\Omega^{1}_{Y/S}$ are $R'$-pro-modules of ft, 
as well as $H_{1}(x^{*}\Omega^{1}_{X/Y})$ because $R$ is noetherian and 
$x^{*}\Omega^{1}_{X/Y}$ is a $R'$-module of finite type because $f$ is schematic.
By Nakayama \cite[Tag.~00DV]{Sta} and because $f_{s}$ is an open immersion for each $s\in S$, we already know that $x^{*}\Omega^{1}_{X/Y}=0$.
Using that $x^{*}\Omega^{1}_{X/S}$ is pro-projective, we obtain that $H_{1}(x^{*}\Omega^{1}_{X/Y})/\km=0$.
Thus, again by Nakayama \cite[Lem.~5.4.9]{Bt3} for pro-modules, we have  $H_{1}(x^{*}\Omega^{1}_{X/Y})=0$.
We thus obtain that $f$ is formally étale. 
We deduce that $f$ is an open embedding because if we write $Y=\colim Y_a$ and $X=\colim X_a=Y_a\times_{Y}X$, we have that 
$f_a:X_a\ra Y_a$ is étale. 
In particular, the map $X_a\ra X_a\times_{Y_a}X_a$ is flat and,
since it is fiberwise an isomorphism, it is an isomorphism. 
In particular, the map $f_a$ is an étale monomorphism.
By \cite[Tag.~025G]{Sta}, it is an open embedding.
\epf

\subsection{Equivariant deformations}
Let $S$ be an affine scheme, $G\ra S$ a fppf group scheme and $\cX\ra S$ a formally smooth prestack with a $G$-action over $S$. 
Assume that $\cX$ admits a deformation theory.
Let 
$$\cX^{G}=\lim(G\times\cX\rightrightarrows\cX\times_{S}\cX)$$ 
be the functor of fixed points, where the morphisms are the action map and the diagonal.
Let $i:S_0\hra S$ be a closed subscheme defined by a quasi-coherent sheaf of ideals $\cI$ of square zero. 
Set $\cX_0=\cX\times_{S}S_0$ and $G_0=G\times_{S}S_0$.
Let $\eps_0\in\cX^{G}(S_0)$. 
We want to define an obstruction to lift $\eps_0$ to a point of $\cX^{G}(S)$. By formal smoothness and Proposition \ref{lift}, 
the space of liftings of $\eps_0$ to $\cX(S)$ is a trivial torsor under the abelian group 
$\Map_{\cO_{S_0}}(\eps_0^{*}L_{\cX/S},\cI).$
Let $L_0$ be the $\cO_{S_0}$-module 
$$L_0=\Map_{\cO_{S_0}}(\eps_0^{*}L_{\cX/S},\cI).$$
Since $\eps_0$ is fixed under $G_0$, the $\cO_{S_0}$-module $L_0$ is indeed a $G_0\ltimes\cO_{S_0}$-module. 
Let $\rho_{0}:G_0\ra\Aut_{S_0}(L_0)$ be the associated representation.

\bprop\label{equiv-lift}
There is a  class $c(\eps_0)$ in $H^{1}(G,i_*L_0)=H^{1}(G_0,L_0)$
whose vanishing is equivalent to the existence of a lift of $\eps_0$ to $\cX(S)$.
\eprop

\brem
Here the cohomology theory considered is Hochschild cohomology see \cite[Exp.~I, \S 5.1]{SGA3}.
\erem

\bpf
The argument follows \cite[Exp.~XII, Lem.~9.4]{SGA3} that we recall. 
The adjunction $H^{1}(G,i_*L_0)=H^{1}(G_0,L_0)$ is \cite[Exp.~III, Lem.~1.1.2]{SGA3}.
We consider the small fppf site on $S$. For any fppf scheme $T\ra S$, we set $G_T$ and $\cX_{T}$ the corresponding base change over $T$. Consider 
the sheaf $A$ on the small fppf site such that for any fppf scheme $T\ra S$, we have
\[A(T)=\{\text{set of liftings of}~ (\eps_0)_{T}~\text{in}~\cX(T)\}.\]
Because formation of $L_0$ commutes with flat base change, we get that $A(T)$ is a torsor under $H_0(T,i_*L_0)$. Because $\eps_0$ is fixed and $T$ 
is fppf on $S$, $g\in G(T)$ acts by affine automorphisms on $A(T)$ compatibly with the action of $G$ on $i_*L_0$, i.e.,
\[g(a_T+v)=g(a_T)+\rho(g)(v), v\in H^{1}(T,i_{*}L_0).\]
As $G$ is fppf, we apply the above discussion $T=G$ and $g=\id_G\in G(G)$. We obtain an action of the fppf sheaf $G$ on $k$.
Let $a$ be a lift $\eps_0$ in $\cX(S)$, let $g^{\sharp}$ be the universal point of $G$, we define $v^{\sharp}\in H^{1}(G,i_*L_0)$ by
\[\rho(g^{\sharp})(v^{\sharp})=g^{\sharp}a_G-a_G.\]
For any $S$-scheme $Y$, we set 
$$z(Y):G(Y)\ra H_0(Y,i_*L_0),\quad
g\mapsto v^{\sharp}(g).$$
This defines an 1-cocycle $z\in Z^{1}(G,L)$ and a class $c(\eps_0)$ in $H^{1}(G,i_{*}L_0)$ that is independent of $a$, see loc.~cit. for details. 
In particular if $\eps_0$ lifts to $\cX^{G}(S)$, we have $c(\eps_0)=0$. 
Conversely if $c(\eps_0)=0$, then there exists $w\in H_0(S,i_*L_0)$ such that $z(Y)(g)=gw_{Y}-w_{Y}$ 
for each $S$-scheme $Y$ and $g\in G(Y)$. By applying it to $Y=G$ and $g^{\sharp}$, we get that $a-w\in\cX^{G}(S)$, as wished. 
\epf

\end{appendix}

\end{document}